\documentclass[11pt]{amsart}

\usepackage{hyperref}
\usepackage{mathtools}
\usepackage{amsthm}
\usepackage{amssymb}
\usepackage{amsfonts}
\usepackage[utf8]{inputenc}
\usepackage{float}
\usepackage{verbatim}
\usepackage{enumitem}
\usepackage{booktabs}

\usepackage{tikz}

\numberwithin{equation}{section}
\newtheorem{theorem}{Theorem}[section]

\newtheorem{proposition}[theorem]{Proposition}
\newtheorem{prop}[theorem]{Proposition}
\newtheorem{claim}[theorem]{Claim}
\newtheorem{corollary}[theorem]{Corollary}
\newtheorem{cor}[equation]{Corollary}
\newtheorem{question}{Question}

\newtheorem{lemma}[theorem]{Lemma}
\newtheorem{conjecture}[theorem]{Conjecture}

\theoremstyle{definition}
\newtheorem{definition}[theorem]{Definition}
\theoremstyle{remark}
\newtheorem{remark}[theorem]{Remark}
\newtheorem{example}[theorem]{Example}
\newtheorem{notation}[equation]{Notation}
\newtheorem{convention}[equation]{Convention}

\newcommand{\bd}{\partial}
\newcommand{\mC}{\mathcal C}
\newcommand{\mU}{\mathcal U}
\newcommand{\mM}{\mathcal M}
\newcommand{\mH}{\mathcal H}
\newcommand{\mQ}{\mathcal Q}

\newcommand{\ra}{\rangle}

\newcommand{\Stab}{\mathrm{Stab}}
\newcommand{\ub}{\underline{b}}
\newcommand{\uv}{\underline{v}}
\newcommand{\eps}{\varepsilon}
\newcommand{\wh}{\widehat}
\newcommand{\wt}{\widetilde}
\newcommand{\Omegaout}{\Omega_{\mathrm{out}}}
\newcommand{\bunder}{\underline{b}}

\DeclareMathOperator{\area}{\mathrm{Area}}
\DeclareMathOperator{\length}{\mathrm{Length}}
\DeclareMathOperator{\ind}{\mathrm{ind}}

\newcommand{\Id}{\mathrm{Id}}

\newcommand{\Sph}{\mathbb{S}}

\newcommand{\N}{\mathbb{N}}
\newcommand{\Z}{\mathbb{Z}}

\newcommand{\del}{\partial}

\newcommand{\aut}{\mathrm{Conf}}

\newcommand{\mP}{\mathcal{P}}
\newcommand{\fC}{\mathfrak{C}}
\newcommand{\fU}{\mathfrak{U}}
\newcommand{\dist}{\mathrm{dist}}
\newcommand{\g}{\gamma}
\newcommand{\Mcal}{\mathcal{M}}

\newcommand{\Rcapunder}{\underline{\mathsf{R}}}
\newcommand{\Ecal}{\mathcal{E}}

\newcommand{\R}{\mathbb{R}}
\newcommand{\B}{\mathbb{B}}

\newcommand{\Ncal}{\mathcal{N}}
\newcommand{\Ccal}{\mathcal{C}}

\newcommand{\Acal}{\mathcal{A}}
\newcommand{\Scal}{\mathcal{S}}

\newcommand{\genus}{\mathrm{genus}}
\newcommand{\Isom}{\mathrm{Isom}}
\newcommand{\Conf}{\mathrm{Conf}}
\newcommand{\Area}{\mathrm{Area}}
\newcommand{\inj}{\mathrm{inj}}
\newcommand{\Diff}{\mathrm{Diff}}
\newcommand{\Met}{\mathrm{Met}}

\newcommand{\firsteigen}{\Ecal}

\newcommand{\lapone}{\Ecal_{\lambda_1}}
\newcommand{\stekone}{\Ecal_{\sigma_1}}
\newcommand{\cut}{\circ}

\newcommand{\Nserif}{\mathsf{N}}

\newcommand{\inn}{\circ}
\newcommand{\out}{-}

\newcommand{\Uhat}{\hat{U}}
\newcommand{\sigmabar}{\overline{\sigma}}
\definecolor{light-gray}{gray}{.95}
\definecolor{dark-gray}{gray}{.7}

\begin{document}
\title[Equivariant Eigenvalue Optimization]{Embedded minimal surfaces in $\mathbb{S}^3$ and $\mathbb{B}^3$ via equivariant eigenvalue optimization}
%Minimal surfaces in the $3$-sphere and $3$-ball via equivariant eigenvalue optimization
\author[M.~Karpukhin]{Mikhail~Karpukhin}
\author[R.~Kusner]{Robert~Kusner}
\author[P.~McGrath]{Peter~McGrath}
\author[D.~Stern]{Daniel~Stern}
\date{}
\address{Department of Mathematics, University College London, 25 Gordon Street, London, WC1H 0AY, UK} \email{m.karpukhin@ucl.ac.uk}
\address{Department of Mathematics, University of Massachusetts,
Amherst, MA, 01003} \email{profkusner@gmail.com, kusner@umass.edu}
\address{Department of Mathematics, North Carolina State University, Raleigh NC 27695} 
\email{pjmcgrat@ncsu.edu}
\address{Department of Mathematics, Cornell University, Ithaca NY, 14853} \email{daniel.stern@cornell.edu}

\begin{abstract}
%In 1970, Lawson solved the topological realization problem for minimal surfaces in the sphere, showing that any closed oriented surface can be minimally embedded in $\Sph^3$. The analogous problem for surfaces with boundary, realizing every topological type as a free boundary minimal surface in $\B^3$, has attracted much attention in recent years, stimulating the development of many new constructions for free boundary minimal surfaces. In this paper, we solve the free boundary realization problem, showing that any compact oriented surface with boundary can be embedded in $\mathbb{B}^3$ as a free boundary minimal surface with area below $2\pi$. More generally, we develop new methods for producing minimal surfaces of prescribed topology in low-dimensional balls and spheres, based on the optimization of Laplace and Steklov eigenvalues in the presence of a discrete symmetry group. As further applications, we show that the number of minimal surfaces in $\Sph^3$ of prescribed topology and area below $8\pi$, and the number of free boundary minimal surfaces in $\mathbb{B}^3$ with prescribed topology and area below $2\pi$, grow at least linearly with the genus. As a key ingredient, we develop new techniques for proving the existence of maximizing metrics for Laplace and Steklov eigenvalues, which can be used to resolve the existence problem in many equivariant situations and provide at least partial existence results for classical eigenvalue optimization problems. 

In 1970, Lawson solved the topological realization problem for minimal surfaces in the sphere, showing that any closed orientable surface can be minimally embedded in $\Sph^3$. The analogous problem for surfaces with boundary was posed by Fraser and Li in 2014, and it has attracted much attention in recent years, stimulating the development of many new constructions for free boundary minimal surfaces. In this paper, we resolve this problem by showing that any compact orientable surface with boundary can be embedded in $\mathbb{B}^3$ as a free boundary minimal surface with area below $2\pi$. Furthermore, we show that the number of minimal surfaces in $\Sph^3$ of prescribed topology and area below $8\pi$, and the number of free boundary minimal surfaces in $\mathbb{B}^3$ with prescribed topology and area below $2\pi$, grow at least linearly with the genus. This is achieved via a new method for producing minimal surfaces of prescribed topology in low-dimensional balls and spheres, based on the optimization of Laplace and Steklov eigenvalues in the presence of a discrete symmetry group.
 As a key ingredient, we develop new techniques for proving the existence of maximizing metrics, which can be used to resolve the existence problem in many symmetric situations and provide at least partial existence results for classical eigenvalue optimization problems. 
\end{abstract}

\maketitle
\bibliographystyle{alpha}

\section{Introduction} 

The study of minimal surfaces in space forms dates back to the early days of differential geometry, with 
$\R^n$ and $\Sph^n$ being the %best-studied 
simplest geometries. The sphere $\Sph^n$ is distinguished %in particular 
as the basic setting for the study of \emph{closed} minimal surfaces, and closed minimal surfaces in $\Sph^n$ arise naturally in a variety of problems, for example, as the links of minimal cones in $\mathbb{R}^{n+1}$, and as examples of Willmore surfaces in $\mathbb{R}^n$ under stereographic projection. For surfaces with boundary, a natural analog of closed minimal surfaces in the sphere are free boundary minimal surfaces in the Euclidean ball $\B^n$, that is,  critical points for the area functional among $2$-cycles relative to the boundary $\partial \mathbb{B}^n$, which have attracted attention in recent decades. %whose study has flourished in recent decades.

In the last thirty years, a striking connection has emerged between natural isoperimetric problems in spectral geometry and minimal surfaces in $\Sph^n$ and $\mathbb{B}^n$. Nadirashvili  discovered \cite{Nadirashvili} that metrics $g$ on a closed surface $M$ which are critical points for the normalized Laplace eigenvalues
$$\bar{\lambda}_k(M,g):=\area(M,g)\lambda_k(M, g)$$
are (up to scaling) precisely those induced by branched minimal immersions of $M$ into $\Sph^n$, where $k$ can be recovered as the smallest integer for which $\lambda_k(M, g)=2$. More recently, Fraser and Schoen discovered an analogous characterization of free boundary minimal surfaces in $\mathbb{B}^n$ as critical points for the normalized Steklov eigenvalues
$$\bar{\sigma}_k(N,g):=\length(\partial N,g)\sigma_k(N,g)$$
on surfaces $N$ with boundary \cite{FSadvances, FraserSchoen}.

Nadirashvili's observation was motivated by the study of metrics maximizing the first nonzero Laplace eigenvalue $\bar{\lambda}_1(M,g)$ on surfaces of fixed topological type, a subject with origins in Hersch's 1970 proof that $\bar{\lambda}_1(\Sph^2,g)$ is maximized by the round metric \cite{Hersch}, later taken up by Berger \cite{Berger}, Yang--Yau \cite{YangYau}, Li--Yau \cite{LiYau} and others. In particular, \cite{Nadirashvili} shows that $\bar{\lambda}_1$-maximizing metrics must be induced by branched minimal immersions $M\to \Sph^n$ whose coordinate functions $x_1,\ldots,x_{n+1}$ are first Laplace eigenfunctions. Here, $n+1$ is bounded above by the multiplicity of $\lambda_1(M,g)$, which a priori could grow linearly with the genus of $M$ \cite{Cheng, Besson}, and $\bar{\lambda}_1$-maximizing metrics are generally expected to  arise from high-codimension surfaces: indeed, the corresponding metrics on the projective plane, torus, %$\mathbb{RP}^2$, $\mathbb{T}^2$,
 and the Klein bottle are induced by minimal embeddings in $\mathbb{S}^4$, $\mathbb{S}^5$, and $\mathbb{S}^4$, respectively \cite{LiYau, Nadirashvili, Jakobson, EGJ}.

Building on the ideas of Nadirashvili, Fraser--Schoen, and others, here we develop the existence theory for metrics maximizing $\bar{\lambda}_1(M,g)$ and $\bar{\sigma}_1(N,g)$ with a prescribed discrete symmetry group, as a tool for producing many new minimal surfaces in $\Sph^n$ and free boundary minimal surfaces in $\mathbb{B}^n$ with prescribed topology and symmetries. In the present paper, we identify a large class of symmetries for which maximizing metrics exist and are induced by \emph{minimal embeddings of codimension one}, generating a wealth of new examples of embedded minimal surfaces in $\mathbb{S}^3$ and free boundary minimal surfaces in $\mathbb{B}^3$ with prescribed topology and symmetries, as we describe in more detail below. Moreover, the symmetric $\bar{\lambda}_1$- and $\bar{\sigma}_1$-maximization techniques developed here can be used to generate many other new families of minimal surfaces in $\Sph^3$, in $\mathbb{B}^3$, and in higher dimensional spheres and balls, well beyond the examples described in Sections \ref{intro.lap.sec} and \ref{intro.stek.sec} below.

\subsection{New minimal surfaces in $\mathbb{S}^3$}\label{intro.lap.sec}

Until the 1960s, the only known closed minimal surfaces embedded in $\mathbb{S}^3$ were the classical examples: the equator $\mathbb{S}^2$, and the Clifford torus $\Sph^1(1/\sqrt{2})\times \Sph^1(1/\sqrt{2})$, 
 up to rotations. This changed with Lawson's landmark paper \cite{Lawson}, which provided a construction associating to each pair of integers $m,k\in \mathbb{N}$ a minimal surface $\xi_{m, k}$ of genus $mk$, demonstrating in particular the existence of at least one example of each orientable topological type.  Since then, many new discrete families of examples have been constructed using a variety of methods, including generalizations of Lawson's construction, based on solving Plateau problems \cite{KPS, ChoeSoretT}, singular perturbation and gluing techniques \cite{KY, Kapouleas, KapMcG, Wiygul,  LDG, KapWtori}, and min-max methods \cite{KetoverTori}.  We refer to Brendle's survey \cite{Brendle:survey} for more about minimal surfaces in $\Sph^3$.

For surfaces of very low genus, the space $\Mcal_\gamma$
of genus $\gamma$ minimal surfaces embedded in $\Sph^3$, modulo isometries, is fully understood: work of Almgren \cite{Almgren} and Brendle \cite{Brendle:lawson} shows the great spheres and Clifford tori are the only examples with genus zero and one, respectively. 
For $\gamma\geq 2$, however, relatively little is known beyond the fact that $\mathcal{M}_{\gamma}$ is nonempty and is compact in the smooth topology \cite{ChoiSchoen}. It is widely believed that
the cardinality $|\mathcal{M}_{\gamma}|$ of $\Mcal_\gamma$ is always finite, but as of this writing this remains an open problem.

Regarding lower bounds for $|\Mcal_\gamma|$,  examples from families mentioned above provide sequences $\gamma_n \rightarrow \infty$ with $|\Mcal_{\gamma_n}| \rightarrow \infty$, and  Ketover recently established \cite{KetoverTori} that $|\Mcal_\gamma| \rightarrow \infty$ in general as $\gamma \rightarrow \infty$.  His %beautiful 
min-max argument %, inspired by techniques from geometric topology, 
constructs a family of doublings of the Clifford torus and establishes the lower bound $|\mathcal{M}_{\gamma,4\pi^2}|> C\gamma / \log \log \gamma$ 
on the cardinality of the subset $\Mcal_{\gamma, 4\pi^2} \subset \Mcal_\gamma$
of surfaces with area $<4\pi^2$. 
For future use, let $\Mcal_{\gamma, A} \subset \Mcal_\gamma$ denote the subset of surfaces with area $< A$. 
 
The methods developed in this article produce many new minimal surfaces embedded in $\Sph^3$ with prescribed genus and symmetry groups, and area $<8\pi$, 
leading to new conclusions about $\Mcal_{\gamma}$ and $\Mcal_{\gamma, 8\pi}$.  In what follows, a finite group $\Gamma$ acting smoothly, properly, and effectively on a surface $M$ is called a \emph{reflection group} acting on $M$ if $\Gamma$ is generated by involutions each of whose fixed-point sets separates $M$. 

\begin{theorem}
\label{Tsph1}
Let $M$ be a closed orientable surface and $\Gamma = \Z_2 \times G$ be a reflection group acting on $M$.  If the quotient $M / \Z_2$ has genus zero, then $M$ admits a $\Gamma$-equivariant minimal embedding into $\Sph^3$, with area less than $8\pi$. 

When $G = \Z_2$ and $M$ has a given genus $\gamma$, there are at least $\lfloor \frac{\gamma-1}{4}\rfloor+1$ distinct such embeddings.  In particular, $|\Mcal_{\gamma, 8\pi}| \geq \lfloor \frac{\gamma-1}{4}\rfloor+1$. 
\end{theorem}

Several remarks are in order.  First, for each fixed $A< 8\pi$, the space $\Mcal_{\gamma, A}$ is known  to be empty \cite{KLS} for all sufficiently large $\gamma$, so Theorem \ref{Tsph1} shows that the number of minimal surfaces of genus $\gamma$ grows at least linearly at the lowest possible area threshold.  Second, the groups $G$ arising in Theorem \ref{Tsph1} are precisely those isomorphic to the finite reflection subgroups of $O(3)$: the symmetry groups of the platonic solids, the dihedral and prismatic dihedral groups $D_k$ and $\Z_2 \times D_k$, the group $\Z_2$, and the trivial group.  Third, each group $\Gamma$ in Theorem \ref{Tsph1} is a reflection group on infinitely many distinct surfaces $M$, up to $\Gamma$-equivariant homeomorphism, so Theorem \ref{Tsph1} provides infinitely many distinct $\Gamma$-invariant minimal surfaces. 

The surfaces in Theorem \ref{Tsph1} are all \emph{doublings} of the equator $\Sph^2 \subset \Sph^3$: following \cite[Definition 1.1]{LDG}, a surface $M$ \emph{doubles} a surface $\Sigma$ if the nearest-point projection $\pi$ to $\Sigma$ is well-defined on $M$ and $M = M_1 \cup M_2$, where $M_1$ is a $1$-manifold, $M_2 \subset M$ is open,  $\pi |_{M_1}$ is a diffeomorphism,  and $\pi|_{M_2}$ is a $2$-sheeted covering map.  Each connected component of $\Sigma \setminus \pi(M)$ is called a \emph{doubling hole}, and the doubling holes for the constructions of Theorem \ref{Tsph1} are all convex.   We refer to Kapouleas's survey  \cite{Kap:survey} for more about constructions of doublings and desingularizations by gluing methods. 

We now discuss some of the surfaces from Theorem \ref{Tsph1}.

When $G$ is the symmetry group of a platonic solid, a $G$-fundamental domain $\Omega \subset \Sph^2$ bounded by fixed-point sets for reflections generating $G$ is a geodesic triangle.  Theorem \ref{Tsph1} allows arbitrary and independent numbers of doubling holes on the interior of $\Omega$, holes centered along each edge of $\partial \Omega$, and holes centered on each corner of $\partial \Omega$.  

When $G$ is a dihedral group $D_k$, one new family has doubling holes in a pyramidal configuration, one centered at a pole $p \in \Sph^2$, and $k$ more symmetrically arranged on a circle of some fixed distance from $p$. 

When $G$ is a prismatic dihedral group $\Z_2 \times D_k$,  Theorem \ref{Tsph1} is interesting only when $k$ is small, because examples from \cite{LDG} already exhaust all $\Gamma$-equivariant homeomorphism classes for $k$ large---indeed, more information is needed to adequately describe the diversity of configurations of doubling holes in examples from \cite{LDG}. One family to note has $k = 3$ and  $m$ doubling holes arranged along each of $3$ equally-spaced meridians.

Finally, when $G = \Z_2$, the configurations of doubling holes are governed by a pair of integers $a,b$, with $a$ holes centered on the circle $C\subset \Sph^2$ fixed by the $\Z_2$-action, and $2b$ more placed $\Z_2$-symmetrically on $\Sph^2 \setminus C$.  Varying $a$ and $b$ and estimating the number of necessarily distinct configurations with a given genus leads to the bound on $|\Mcal_{\gamma, 8\pi}|$ in Theorem \ref{Tsph1}.

\subsection{New free boundary minimal surfaces in $\mathbb{B}^3$}\label{intro.stek.sec}

In manifolds with boundary, the natural analog of closed minimal surfaces are \emph{free boundary minimal surfaces}---critical points for the area functional on the space of relative $2$-cycles with respect to the boundary of the ambient manifold. Though free boundary minimal surfaces have been studied since work of Courant \cite{CourantFBMS}, the subject has been taken up with renewed intensity in recent years, following Fraser-Schoen's discovery \cite{FSadvances} that free boundary minimal surfaces in Euclidean balls arise from a natural analog of the $\bar{\lambda}_1$-optimization problem. In this setting, the Laplacian is replaced by the first-order, nonlocal \emph{Dirichlet-to-Neumann} operator $\mathcal{D}:C^{\infty}(\partial N)\to C^{\infty}(\partial N)$ on the compact surface $N$ with boundary $\partial N$, given by
$$\mathcal{D}(u)=\frac{\partial\hat{u}}{\partial\nu},$$
where $\hat{u}$ is the harmonic extension of $u$ to $N$. The spectrum 
$ 0=\sigma_0(N,g)<\sigma_1(N,g)\leq \cdots$ of $\mathcal{D}$ is called the \emph{Steklov} spectrum of $(N,g)$.
%???, first studied [SOME STATEMENT].
 Fraser-Schoen discovered \cite{FScontemp, FraserSchoen} that, just as $\bar{\lambda}_1$-extremal metrics on closed surfaces are induced by minimal immersions in spheres, extremal metrics for the length-normalized first Steklov eigenvalue
$$\bar{\sigma}_1(N,g):=\length(\partial N,g)\sigma_1(N,g)$$
are induced by free boundary minimal immersions into Euclidean balls $\mathbb{B}^n$.

Just as $\Sph^3$ is the simplest and most natural setting for studying closed %embedded
minimal surfaces, the simplest setting for studying compact free boundary minimal surfaces is the %standard
 unit ball $\mathbb{B}^3$. For many years, the only known examples of such surfaces % free boundary minimal surfaces in $\mathbb{B}^3$ 
were the equatorial disk and the critical catenoid, until Fraser-Schoen showed \cite{FraserSchoen} that $\bar{\sigma}_1$-maximizing metrics on genus zero surfaces must be induced by minimal embeddings in $\mathbb{B}^3$, establishing the existence of genus zero examples with
%embedded free boundary minimal surfaces in $\mathbb{B}^3$ with genus zero 
 any number of boundary components, modulo the existence theory for $\bar{\sigma}_1$-maximizing metrics. Over the last decade, many families of examples with various topological types %free boundary minimal surfaces in $\mathbb{B}^3$ 
have been constructed \cite{Zolotareva, KapWiygul, KapZou, LDG, Carlotto, Carlotto2} by a combination of gluing and min-max methods, and we refer to Martin Li's survey \cite{LiSurvey} and Mario Schulz's website \cite{Schulz} for an overview of previous constructions and other conjectural families.

In spite of this rapid recent progress, %in the existence theory for free boundary minimal surfaces, 
the following basic \emph{realization} problem \cite[Question 1]{FraserLi} has remained open until now.

\begin{question}\label{fbms.q1}
Can every compact, oriented surface with boundary be realized as an embedded free boundary minimal surface in $\mathbb{B}^3$?
\end{question}

As with $\bar{\lambda}_1$-maximizing metrics on closed surfaces, $\bar{\sigma}_1$-maximizing metrics on surfaces of nonzero genus are in general induced by branched free boundary minimal immersions into balls of dimension $>3$, so one cannot hope to resolve Question \ref{fbms.q1} via unconstrained $\bar{\sigma}_1$-maximization. % in the case of positive genus. 
On the other hand, analogous to Yau's conjecture for minimal surfaces in $\mathbb{S}^3$, the coordinate functions of every embedded free boundary minimal surface in $\mathbb{B}^3$ are expected to be \emph{first} Steklov eigenfunctions \cite[Conjecture 3.3]{FraserLi}, \cite[Open Question 7]{LiSurvey}, \cite{KusnerMcGrath}, suggesting that such surfaces can be constructed by alternative variational methods for the $\bar{\sigma}_1$ functional.

We resolve Question \ref{fbms.q1} as follows (see Theorem \ref{thm:group_Sexistence}):
\begin{theorem}
\label{Tfbms1} 
Each compact orientable surface with boundary is realized as an embedded free boundary minimal surface in $\B^3$, with area $<2\pi$.

Moreover, for each $\gamma \geq 0$ and $b \geq 2$, there are at least $\lfloor \frac{\gamma-2}{4} \rfloor + 1$ distinct $\Z_2\times \Z_2$-equivariant such embeddings of the surface with genus $\gamma$ and $b$ boundary components.

\end{theorem}

\pdfsuppresswarningpagegroup=1
\begin{figure}
\includegraphics[width=0.3\textwidth,page=4]{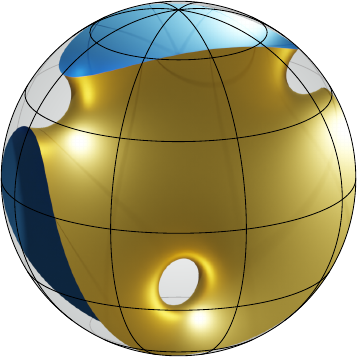}
\hfill
\includegraphics[width=0.3\textwidth,page=5]{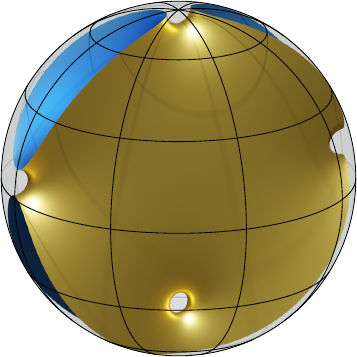}
\hfill
\includegraphics[width=0.3\textwidth,page=6]{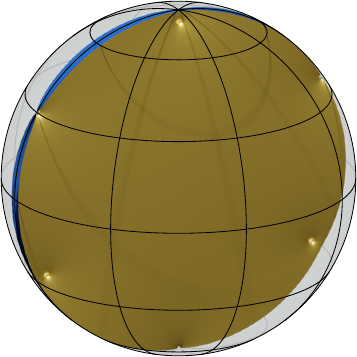}
\caption{{\small Conjectural pictures of some of the surfaces given by Theorem~\ref{Tfbms1}. Image credit: M.~Schulz~\cite{Schulz}.}}
\label{fig:fbms1}
\end{figure}

Independently of our work, we note that Mario Schulz has  conjectured \cite{Schulz} the existence of examples satisfying the area bound from the first part of Theorem \ref{Tfbms1}.  Figure~\ref{fig:fbms1} shows numerical pictures of free boundary minimal surfaces provided by M. Schulz, whose geometry is consistent with our understanding of the surfaces obtained in Theorem~\ref{Tfbms1}.

For genus zero surfaces, we have a result analogous to Theorem \ref{Tsph1}.  
\begin{theorem}
\label{Tfbms2}
For every $b\in \mathbb{N}$, there exist at least $\lfloor \frac{b-2}{4}\rfloor+1$ distinct free boundary minimal surfaces embedded in $\mathbb{B}^3$ of genus $0$ with $b$ boundary components and area $<4\pi$, invariant under a $\mathbb{Z}_2\times \mathbb{Z}_2$ action. More generally, on any compact, genus zero surface with boundary $N$, there is a large class of reflection groups $G$ on $N$ for which $N$ admits a $G$-equivariant free boundary minimal embedding in $\B^3$ of area $<4\pi$.
\end{theorem}

\begin{figure}
\includegraphics[width=0.3\textwidth,page=1]{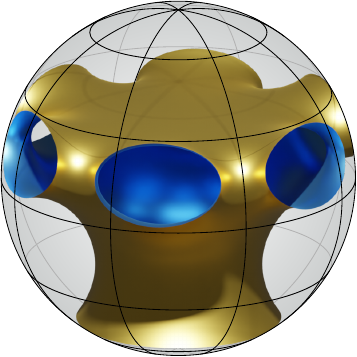}
\hfill
\includegraphics[width=0.3\textwidth,page=2]{figures-fbms}
\hfill
\includegraphics[width=0.3\textwidth,page=3]{figures-fbms}
\caption{{\small Conjectural pictures of some of the surfaces given by Theorem~\ref{Tfbms2}. From left to right: $D_3$-symmetric surface with $7$ boundary components; symmetric surface with boundary components around vertices of the dodecahedron; $\Z_2\times D_{16}$-symmetric surface with $64$ boundary components. Image credit: M.~Schulz~\cite{Schulz}.}}
\label{fig:fbms2}
\end{figure}

 Figure~\ref{fig:fbms2} shows numerical examples of surfaces with symmetry, whose existence is covered by Theorem~\ref{Tfbms2}.
The list of pairs $(N, G)$ admitting equivariant free boundary minimal embeddings in $\mathbb{B}^3$ is not yet exhaustive, as is the case for Theorem \ref{Tsph1}, but is nonetheless extensive, and the reasons for this are detailed in Section \ref{ssbasref} and in the discussion before Theorem \ref{thm:group_Sexistence}.  We refer to Theorem \ref{thm:group_Sexistence} for the complete list of actions, but remark that all cases where $G$ is a tetrahedral, octahedral, or icosahedral symmetry group are included, as well as most cases where $G = \Z_2\times D_k$. 

\subsection{Existence theory for extremal metrics} 
The bulk of the analytic work in the paper is devoted to developing the existence theory for $\bar{\lambda}_1$- and $\bar{\sigma}_1$-maximizing metrics in the presence of a discrete symmetry group. Notably, our methods also provide a new approach to the existence theory for the classical $\bar{\lambda}_1$- and $\bar{\sigma}_1$-maximization problems, as described briefly below.

The existence of a metric maximizing $\bar{\lambda}_1$ in a fixed \emph{conformal} class $c$ on a closed surface $M$ was established by Petrides \cite{Pet1}, with alternative proofs later given in \cite{KNPP} and \cite{KSminmax}, yielding the following.

\begin{theorem}[\cite{Pet1, KNPP, KSminmax}]\label{lap.cf.ex}
For each closed surface $M$ and each conformal class $c$ on $M$, there is a metric $g \in c$, possibly with a finite set of conical singularities, realizing the supremum
$$\Lambda_1(M,c):=\sup \{ \bar{\lambda}_1(M, h) \mid  h \in c \}, $$
and a sphere-valued harmonic map $u\colon(M,g)\to \mathbb{S}^n$ by first eigenfunctions. 
\end{theorem}

With the existence of conformally $\bar{\lambda}_1$-maximizing metrics in hand, to show the existence of metrics realizing the global supremum
$$\Lambda_1(M):=\sup\{\bar{\lambda}_1(M,h)\mid h\in \Met(M)\},$$
it suffices to show that a sequence of conformal classes $c_j$ with  $\Lambda_1(M, c_j)\to \Lambda_1(M)$ does not accumulate at the boundary of the finite-dimensional moduli space of conformal structures on $M$. Petrides \cite{Pet1} reduced this problem  to verifying a strict inequality, which is simplest to state for orientable surfaces (see \cite{MS2} for the nonorientable case); here and below, the closed orientable surface with genus $\gamma$ is denoted by $M_\gamma$. 

\begin{theorem}[\cite{Pet1}]\label{pet.glob1}
Suppose that
\begin{equation}\label{g.mono}
\Lambda_1(M_{\gamma+1})>\Lambda_1(M_{\gamma}).
\end{equation}
Then $\Lambda_1(M_{\gamma+1})$ is achieved by a metric $g$ on $M_{\gamma+1}$, possibly with conical singularities, induced by a branched minimal immersion $u\colon M_{\gamma+1}\to \mathbb{S}^n$ by first eigenfunctions.
\end{theorem}

While the nonstrict version of the inequality \eqref{g.mono} is not difficult to prove (see e.g.~\cite{CES}), the strict inequality is surprisingly subtle. There have been multiple attempts, most notably \cite{MatthiesonSiffert}, to prove \eqref{g.mono} %and its nonorientable analog 
by a perturbation and induction argument, starting from a maximizing metric $g$ on $M_{\gamma}$, and attaching a handle with carefully chosen geometry to directly produce a metric $g'$ on $M_{\gamma+1}$ with $\bar{\lambda}_1(M_{\gamma+1},g')>\bar{\lambda}_1(M_{\gamma},g_{})=\Lambda_1(M_{\gamma})$.  H. Matthiesen and A. Siffert  have communicated to us \cite{Siffert} that the argument in \cite{MatthiesonSiffert} contains a nontrivial gap, and to our knowledge the strict inequality \eqref{g.mono} remains open in general. 

In Section \ref{sec:global_existence}, we develop a new strategy for proving \eqref{g.mono} and its analog in the presence of a prescribed symmetry group. Unlike previous approaches, our approach is based on a contradiction argument, using the bound $\Lambda_1(M_{\gamma+1},c)\leq \bar{\lambda}_1(M_{\gamma},g_{})$ for the conformal suprema of $\bar{\lambda}_1$ on $M_{\gamma+1}$ and the min-max characterization of $\Lambda_1(M_{\gamma+1},c)$ obtained in \cite{KSminmax}. 

Starting from a $\bar{\lambda}_1$-maximizing metric $g_{}$ on the %lower genus
 surface $M_{\gamma}$, we remove small disks centered at a prescribed pair of points $p,q\in M_{\gamma}$, and attach a flat cylinder to obtain a conformal class $c$ on $M_{\gamma+1}$. Assuming---for a contradiction---that $\Lambda_1(M_{\gamma+1})\leq \Lambda_1(M_{\gamma})$, it follows that $\Lambda_1(M_{\gamma+1},c)\leq \bar{\lambda}_1(M_{\gamma},g_{})$.  Combining this inequality with results of \cite{KSminmax} and a few elementary estimates proves the existence of a map $u\colon(M_{\gamma+1},c)\to \mathbb{S}^n$ with quantifiably negligible energy in the handle region, whose component functions are $W^{1,2}$-close to first eigenfunctions on $(M_{\gamma}, g)$. Taking the radius of the disks removed arbitrarily small and letting the conformal type of the attached handle degenerate, we take the limits of these maps on $M_{\gamma}$ to obtain the following.

\begin{proposition}\label{intro.handle}
If $\Lambda_1(M_{\gamma-1})<\Lambda_1(M_{\gamma})=\Lambda_1(M_{\gamma+1})$, then for every pair of points $p,q\in M_{\gamma}$, there exists a map $F=F_{p,q}\colon M_{\gamma}\to \mathbb{S}^n$ by first eigenfunctions such that $F(p)=F(q).$
\end{proposition}

We conjecture that the conclusion of Proposition \ref{intro.handle} can never hold, in which case \eqref{g.mono} must hold, yielding the existence of $\bar{\lambda}_1$-maximizing metrics on any closed oriented surface.  As explained in Section \ref{sec:global_existence}, applying a similar strategy with three points and three cylinders glued together shows that if $\Lambda_1(M_{\gamma-1})<\Lambda_1(M_{\gamma})=\Lambda_1(M_{\gamma+2})$, then for any $p, q, r\in M_{\gamma}$, there exists a map $F\colon M_{\gamma}\to \mathbb{S}^n$ by first eigenfunctions  such that $F(p) = F(q) = F(r)$, a condition which we prove is indeed impossible, leading to the following weaker monotonicity result (see Corollary \ref{cor:gamma2}). 

\begin{theorem} 
\label{thm:gamma+2}
For any $\gamma \geq 0$, one has  $\Lambda_1(M_{\gamma+2})>\Lambda_1(M_{\gamma})$.  In particular, for all $\gamma\geq 0$, either $M_{\gamma}$ or $M_{\gamma+1}$ admits a $\bar{\lambda}_1$-maximizing metric.
\end{theorem}

The current status of the existence theory for $\bar{\sigma}_1$-maximizing metrics on surfaces with boundary $N$ is similar to that of the $\bar{\lambda}_1$-maximization problem on closed surfaces. Work of Petrides \cite{PetridesS} establishes the existence of metrics maximizing $\bar{\sigma}_1$ in any conformal class, assuming the validity of a mild gap condition, which is expected to hold in general.

\begin{theorem}[\cite{PetridesS}]\label{stek.cf.ex}
Let $N$ be a compact surface with boundary and $c$ be a conformal class on $N$. 
If 
$$\Sigma_1(N,c):=\sup\{\bar{\sigma}_1(N,h)\mid h\in c\}> 2\pi,$$
then $\Sigma_1(N, c)$ is realized by a metric $g \in c$, possibly with a finite set of conical singularities, which is induced by a free boundary harmonic map $u\colon(N,\partial N)\to (\mathbb{B}^n,\mathbb{S}^{n-1})$ by first Steklov eigenfunctions.
\end{theorem}

\begin{comment}
\begin{theorem}[\cite{PetridesS}]\label{stek.cf.ex}
For any metric $g$ on the compact surface $N$ with boundary $\partial N$, there exists a metric $g_{max}$ in $[g]$, possibly with a finite set of conical singularities, realizing the supremum
$$\Sigma_1(N,[g]):=\sup\{\bar{\sigma}_1(N,h)\mid h\in [g]\},$$
provided $\Sigma_1(N,[g])>\Sigma_1(\mathbb{D}^2)=2\pi$. Moreover, there is a free boundary harmonic map $u:(N,\partial N)\to (\mathbb{B}^n,\mathbb{S}^{n-1})$ by first Steklov eigenfunctions.
\end{theorem}
\end{comment}

For the global maximization problem, there is an analog of Theorem \ref{pet.glob1}, giving existence of a $\bar{\sigma}_1$-maximizer among all metrics on $N$, assuming a natural gap condition holds.  In the version below---stated for simplicity for oriented surfaces---$N_{\gamma, k}$ denotes the compact surface with genus $\gamma$ and $k$ boundary components. 

\begin{theorem}[\cite{PetridesS}]\label{pet.glob2}
If
\begin{align}
\label{stek.k.mono0}
\Sigma_1(N_{\gamma,k})&>\Sigma_1(N_{\gamma,k-1})\text{ or }k=1; \quad \text{and} \\
\label{stek.g.mono}
\Sigma_1(N_{\gamma,k})&>\Sigma_1(N_{\gamma-1,k+1})\text{ or }\gamma=0,
\end{align}
 there is a metric $g$ on $N_{\gamma,k}$ realizing $\bar{\sigma}_1(N_{\gamma,k},g)=\Sigma_1(N_{\gamma,k})$, induced by a branched free boundary minimal immersion $u\colon N_{\gamma,k} \to \B^n$ by first Steklov eigenfunctions.
\end{theorem}

As in the closed case, the non-strict versions of \eqref{stek.k.mono0} and \eqref{stek.g.mono} are easily established, but proving the strict inequalities is highly nontrivial in general.  As was pointed out in \cite[Appendix A]{GL}, the argument for \eqref{stek.k.mono0} from \cite[Proposition 4.3]{FraserSchoen} contains a gap; also, an approach \cite{MatthiesonPetrides}  analogous to the one from \cite{MatthiesonSiffert} in the closed cases is known \cite{Petridescomm}  to contain a gap: in particular, the lower bound in \cite[inequality (6.35)]{MatthiesonPetrides} is not necessarily positive, with consequences for several key estimates.

In Section \ref{sec:global_Sexistence}, we develop new approaches for proving \eqref{stek.k.mono0} and \eqref{stek.g.mono} via contradiction arguments and a min-max characterization of the conformal suprema $\Sigma_1(N_{\gamma,k}, c)$. In particular, arguments similar to the proof Proposition \ref{intro.handle} in the closed case establish the following.

\begin{proposition}\label{intro.strip}
The following hold.
\begin{enumerate}[label=\emph{(\roman*)}]
\item The inequality
$\Sigma_1(N_{\gamma,k})>\Sigma_1(N_{\gamma,k-1})$ holds, provided $\gamma\geq 0, k\geq 2$.
\item  If $\gamma, k \geq 1$ and $N=N_{\gamma-1,k+1}$ admits a $\bar{\sigma}_1$-maximizing metric, but
$$\Sigma_1(N_{\gamma,k})=\Sigma_1(N_{\gamma-1,k+1}),$$
then for each pair of points $p,q$ in distinct components of $\partial N$, there is a map $F\colon(N,\partial N)\to (\mathbb{B}^n,\Sph^{n-1})$ by first Steklov eigenfunctions with $F(p)=F(q)$.
\end{enumerate}
\end{proposition}

\begin{comment}
\begin{proposition}\label{intro.strip}
The sharp inequality
\begin{equation}\label{stek.k.mono}
\Sigma_1(N_{\gamma,k})>\Sigma_1(N_{\gamma,k-1})
\end{equation}
holds, provided $\gamma\geq 0$ and $k\geq 2$.  If $\gamma, k \geq 1$ and $N=N_{\gamma-1,k+1}$ admits a $\bar{\sigma}_1$-maximizing metric $g_{max}$ but
$$\Sigma_1(N_{\gamma,k})=\Sigma_1(N_{\gamma-1,k+1}),$$
then for each pair of points $p,q$ in distinct components of $\partial N$, there is a map $F:(N,\partial N)\to (\mathbb{B}^n,\Sph^{n-1})$ by first Steklov eigenfunctions for $g_{max}$ with $F(p)=F(q)$.
\end{proposition}
\end{comment}

As argued in \cite{FraserSchoen}, when $\gamma = 0$, the inequality \eqref{stek.k.mono0}  already implies the existence of $\bar{\sigma}_1$-maximizing metrics on genus zero surfaces with any number of boundary components, and these metrics must be induced by free boundary minimal embeddings in $\mathbb{B}^3$, whose coordinate functions span the first Steklov eigenspace. When $\gamma=1$, it follows from Proposition \ref{intro.strip} that if \eqref{stek.g.mono} fails, then the $\bar{\sigma}_1$-maximizing metric on $N=N_{0,k+1}$ is realized by a free boundary minimal surface $N\subset\mathbb{B}^3$ such that for \emph{any} pair of points $p,q$ in distinct components of $\partial N$, the coordinates of the plane normal to $p-q$ satisfy an additional quadratic relation on $\partial N$. However, it is easy to see that this cannot occur, so $\Sigma_1(N_{1,k})>\Sigma_1(N_{0,k+1})$ for every $k\in \mathbb{N}$. Combining these observations with Theorem \ref{pet.glob2} gives the following corollary.

\begin{corollary}
Each compact oriented surface with boundary of genus zero or one admits a
$\bar{\sigma}_1$-maximizing metric. 
\end{corollary}

While the explicit $\bar{\sigma}_1$-maximizing metrics in these cases have only been determined for the disk and the annulus, the results of \cite{GL} and \cite{KS21} give an asymptotic description of their geometry as the number of boundary components becomes large: namely, the $\bar{\sigma}_1$-maximizing metrics on $N_{0, k}$ are induced by free boundary minimal surfaces in $\mathbb{B}^3$ converging to the boundary $\Sph^2$ as $k\to\infty$, while the $\bar{\sigma}_1$-maximizing metrics on $N_{1, k}$ are induced by free boundary minimal surfaces in $\mathbb{B}^6$ converging to the Bryant-Itoh-Montiel-Ros torus in $\mathbb{S}^5=\partial \mathbb{B}^6$ as $k\to\infty$. 

As with the conclusion of Proposition \ref{intro.handle}, we conjecture that the conclusion of Proposition \ref{intro.strip} is never satisfied, 
  so that \eqref{stek.g.mono} always holds, and every orientable surface with boundary admits a $\bar{\sigma}_1$-maximizing metric. 

\subsection{Extremal metrics with prescribed symmetries}

Given a closed surface $M$ and a finite subgroup $\Gamma\leq \Diff(M)$ of the diffeomorphism group, denote by $\Met_{\Gamma}(M)$ the (always nonempty) space
%$$Met_{\Gamma}(M):=\{g\in Met(M)\mid \beta^*g=g\text{ for every }\beta\in \Gamma\}$$
of $\Gamma$-invariant metrics $g$ on $M$, and let $c \subset \Met_{\Gamma}(M)$ be a non-empty conformal class of $\Gamma$-invariant metrics. 
%$$[g]_{\Gamma}:=[g]\cap Met_{\Gamma}(M).$$

In Section \ref{SSprelim}, we show that, under very mild assumptions, every such conformal class $c$ contains a $\bar{\lambda}_1$-maximizing representative and an associated $\Gamma$-equivariant sphere-valued map. In particular, we have the following generalization of Theorem \ref{lap.cf.ex}.

\begin{theorem}\label{intro.lap.cf}
Given a closed surface $M$, a finite subgroup $\Gamma \leq \Diff(M)$, and a conformal class $c\subset \Met_\Gamma(M)$, there is a metric $g\in c$, possibly with a finite set of conical singularities, realizing the supremum
$$\Lambda^\Gamma_1(M,c):=\sup\{\bar{\lambda}_1(M,h)\mid h\in c\},$$
provided either $\Lambda^\Gamma_1(M,c)>8\pi$ or $\Gamma$ has no fixed points. Moreover, for some $n\in \mathbb{N}$ there is a linear isometric action of $\Gamma$ on $\mathbb{R}^{n+1}$ and a $\Gamma$-equivariant harmonic map $u\colon (M,g_{})\to \Sph^n$ by first eigenfunctions. 
\end{theorem}
By ``$\Gamma$ has no fixed points'' we mean that there are no points of $M$ that are preserved by {\em all} elements of $\Gamma$. 

Theorem~\ref{intro.lap.cf} could be proved by a variant of the techniques in \cite{Pet1,  KNPP}, but instead we opt for a proof based on a natural min-max construction for $\Gamma$-equivariant harmonic maps generalizing \cite{KSminmax}, which yields additional information necessary for proving analogs of Proposition \ref{intro.handle} and \ref{intro.strip}. In fact, even in the classical case where $\Gamma$ is trivial, we are able to simplify the proofs of the main results of \cite{KSminmax} by streamlining  much of the bubbling analysis.

Similarly, given a compact surface $N$ with boundary $\partial N$ and a finite subgroup $\Gamma\leq \Diff(N)$, a free boundary analog of the techniques in \cite{KSminmax} leads to the following existence result for $\bar{\sigma}_1$-maximizers  within a $\Gamma$-invariant conformal class.

\begin{theorem}\label{intro.stek.cf}
Given a compact surface $N$ with boundary, a finite subgroup $\Gamma\leq \Diff(N)$, and a conformal class $c \subset Met_{\Gamma}(N)$,  there is a metric $g \in c$, possibly with a finite set of conical singularities, realizing
$$\Sigma^\Gamma_1(N,c):=\sup\{\bar{\sigma}_1(N,h)\mid h\in c\},$$
provided either $\Sigma^\Gamma_1(N,c)>2\pi$ or $\Gamma$ has no fixed points on $\partial N$. Moreover,  for some $n\in \mathbb{N}$ there is a linear isometric action of $\Gamma$ on $\mathbb{R}^n$ and a $\Gamma$-equivariant free boundary harmonic map $u\colon (N,g_{})\to \mathbb{B}^n$ by first Steklov eigenfunctions.
\end{theorem}

As in the classical case, to prove existence of metrics realizing the suprema
\begin{align*}
\Lambda_1^{\Gamma}(M) &:=\sup\{\bar{\lambda}_1(M,g)\mid g\in \Met_{\Gamma}(M)\}
\quad
\text{and}\\
\Sigma_1^{\Gamma}(N) &:=\sup\{\bar{\sigma}_1(N,g)\mid g\in \Met_{\Gamma}(N)\},
\end{align*}
it remains to show that a sequence of $\Gamma$-invariant conformal classes $c_n$ for which $\Lambda^{\Gamma}_1(M, c_n)\rightarrow \Lambda_1^{\Gamma}(M)$, or correspondingly for $\Sigma^{\Gamma}_1(M, c_n)$, does not degenerate to the boundary of the relevant moduli space of $\Gamma$-invariant conformal classes. Section \ref{sec:top_degen} describes this degeneration precisely, showing roughly that for $(M, \Gamma)$ as before, the boundary of the moduli space consists of surfaces $M'$ with $\Gamma'\leq \Diff(M')$ arising from pinching off a $\Gamma$-invariant collection of simple closed curves in $M$, a relation which we denote by 
$$(M',\Gamma') \prec (M,\Gamma).$$
In the case with nonempty boundary, a doubling construction described in Section \ref{sec:actions_Sex} associates to each pair $(N,\Gamma)$ a closed surface $\tilde{N}$ with group action $\mathbb{Z}_2\times \Gamma \cong \tilde{\Gamma}\leq \Diff(\tilde{N})$, and we likewise write
$$(N',\Gamma')\prec (N,\Gamma)$$
if $(\tilde{N}',\tilde{\Gamma}')\prec (\tilde{N},\tilde{\Gamma}).$ With this notation in place, we can now state the following generalization of Theorems \ref{pet.glob1} and \ref{pet.glob2}, proved in Section \ref{sec:conf}; see \ref{thm:existence} and \ref{thm:Sexistence} for precise versions.

\begin{theorem}\label{intro.lap.glob}
If $M$ is a closed orientable surface, $\Gamma$ is a finite group of diffeomorphisms, and
\begin{equation}\label{lap.sym.mono}
\Lambda_1^{\Gamma}(M)>\max\{\Lambda_1^{\Gamma'}(M'),8\pi\}
\end{equation}
whenever $(M',\Gamma')\prec (M,\Gamma)$, then there exists $g \in Met_{\Gamma}(M)$ realizing
$$\bar{\lambda}_1(M,g)=\Lambda_1^{\Gamma}(M),$$ 
and a linear isometric action of $\Gamma$ on $\mathbb{R}^{n+1}$ such that $g$ is induced by a $\Gamma$-equivariant branched minimal immersion $M\to \mathbb{S}^n$ by first eigenfunctions.
\end{theorem}

\begin{theorem}\label{intro.stek.glob}
If $N$ is a compact orientable surface with boundary, $\Gamma$ is a finite group of diffeomorphisms, and
\begin{equation}\label{stek.sym.mono}
\Sigma_1^{\Gamma}(N)>\max\{\Sigma_1^{\Gamma'}(N'),2\pi\}
\end{equation}
whenever $(N',\Gamma')\prec (N,\Gamma)$, then there exists $g \in Met_{\Gamma}(N)$ realizing
$$\bar{\sigma}_1(N,g)=\Sigma_1^{\Gamma}(N),$$
and an isometric action of $\Gamma$ on $\mathbb{R}^n$ such that $g$ is induced by a $\Gamma$-equivariant branched free boundary minimal immersion $N\to \mathbb{B}^n$ by first eigenfunctions.
\end{theorem}
In Sections \ref{sec:global_existence} and \ref{sec:global_Sexistence}, we give geometric criteria under which the inequalities \eqref{lap.sym.mono} and \eqref{stek.sym.mono} must hold, generalizing Propositions \ref{intro.handle} and \ref{intro.strip} to the case of nontrivial symmetry groups; see Theorems \ref{thm:general_existence} and \ref{thm:general_Sexistence} for the precise statements.

With this general existence theory for maximizing metrics in place, we turn to the following questions concerning the pairs $(M, \Gamma)$ and $(N, \Gamma)$: for which pairs can one prove \eqref{lap.sym.mono} or \eqref{stek.sym.mono}, and therefore the existence of $\bar{\lambda}_1$- or $\bar{\sigma}_1$-maximizing $\Gamma$-invariant metrics? Among these, for which pairs do the associated minimal surfaces lie in \emph{low-dimensional} spheres or balls, and which of these is necessarily embedded? The construction of the minimal surfaces described in Sections \ref{intro.lap.sec} and \ref{intro.stek.sec} hinges on identifying a large class of examples satisfying all three conditions, which we now briefly describe.

\subsection{Extremal metrics on basic reflection surfaces and applications}
\label{ssbasref}
In Section \ref{S:BRS} we study a class of surfaces with group actions which we call \emph{basic reflection surfaces} (or \emph{BRS}).  In the closed case, $(M, \Gamma)$ is a BRS if $\Gamma$ contains an involution $\tau$ whose fixed-point set $M^\tau$ contains a curve (a \emph{reflection}) with the property that the quotient $M/ \langle \tau \rangle$ has genus zero.  In the case with nonempty boundary, $(N, \Gamma)$ is a BRS if $\Gamma$ contains a reflection $\tau$ such that $N / \langle \tau \rangle $ has genus zero and the image of $\partial N \cup N^\tau_\partial$ in $N / \langle \tau \rangle$ is connected, where $N^\tau_\partial$ is the union of the components of $N^\tau$ meeting $\partial N$.

Each topological type of compact surface---orientable or nonorientable, with or without boundary---is realized as a BRS, see Section~\ref{Tclass}, and each BRS satisfies two important \emph{a priori} estimates,  independent of its topology,  the first an optimal upper bound on the first normalized eigenvalue, and the second an upper bound on the multiplicity of the first eigenvalue.  These bounds imply that any branched minimal immersion by first eigenfunctions into a ball or sphere is necessarily an embedding, with controlled codimension. 

In particular, if $M$ and $N$ are basic reflection surfaces, where $\partial M = \varnothing$ and $\partial N \neq \varnothing$,  Lemma \ref{Levalbdcl} shows that for any $\Gamma$-invariant metric $g$,
\begin{align*}
\lambda_1(M, g)\area (M) < 16 \pi
\quad
\text{and}
\quad
\sigma_1(N,g) \length (\partial N) < 4\pi,
\end{align*}
essentially by applying variants of Hersch's \cite{Hersch} and Weinstock's \cite{Weinstock} arguments to the quotient of the surface by $\langle \tau \rangle$. It follows that any minimal immersion $u\colon M \rightarrow \Sph^n$ by first eigenfunctions has area $< 8\pi$, hence by work of Li-Yau \cite{LiYau}  is an embedding.  Similarly, any free boundary immersion $u\colon  N \rightarrow \B^n$ by first eigenfunctions has area $<2\pi$, hence is an embedding by \cite{Volkmann} .

The multiplicity bounds for the first eigenvalue are proved in Propositions \ref{Lasymcl} and \ref{Lrhoquo2}. %and depend on the orientability character of the surface.
 If $M$ and $N$ are orientable, then the dimension of the first Laplace eigenspace %$\Ecal_{\lambda_1}(M)$
  is at most $4$, and the dimension of the first Steklov eigenspace 
  %$\Ecal_{\sigma_1}(N)$
   is at most $3$.  
  In particular, if $u\colon M \rightarrow \Sph^n$ is a minimal immersion by first eigenfunctions, then $n = 3$, and if $u\colon N \rightarrow \B^n$ is a free boundary minimal immersion by first eigenfunctions, then $n = 3$. Note that for genus zero surfaces with boundary, this multiplicity bound holds without any symmetry assumption \cite{Kokarev, FraserSchoen}. 
 
 Combining these observations with Theorem \ref{intro.lap.glob} shows that if $(M,\Gamma)$ is a closed, oriented BRS for which \eqref{lap.sym.mono} holds, then $\Lambda_1^{\Gamma}(M)$ is achieved by a $\Gamma$-equivariant minimal embedding $M\to \Sph^3$. Likewise, Theorem \ref{intro.stek.glob} shows that if $(N,\Gamma)$ is a compact oriented BRS with boundary and \eqref{stek.sym.mono} holds, then $\Sigma_1^{\Gamma}(N)$ is achieved by a $\Gamma$-equivariant free boundary minimal embedding $N\to \mathbb{B}^3$.
 
 Furthermore, in Section \ref{ssdoub} we show that any such minimal embedding is a doubling of an equatorial sphere $\Sph^2$, or of an equatorial disk $\mathbb{D}^2$.  In the closed case, this is established by using nodal domain arguments to conclude (Lemma \ref{Lgraphsph}) that the restriction of the nearest-point projection to the complement of the fixed-point set is a $2$-sheeted cover, while in the case with boundary, a more delicate Morse-theoretic argument (Lemma \ref{Lbdgraph}) is used to establish the corresponding property.  

Crucially, the preceding multiplicity bounds also make it easier to apply Theorems \ref{thm:general_existence} and \ref{thm:general_Sexistence} 
to prove the inequalities \eqref{lap.sym.mono} and \eqref{stek.sym.mono}. In the closed case, we argue in Section \ref{ss:gamma}  that if \eqref{lap.sym.mono} fails for some degeneration $(M',\Gamma')\prec (M,\Gamma)$, then $(M',\Gamma')$ must be a basic reflection surface of lower genus, which we can assume---proceeding by induction---admits a metric $g_{}$ realizing $\Lambda_1^{\Gamma'}(M')$. The preceding observations together with Theorem  \ref{thm:general_existence} then imply that $g_{}$ is realized by an embedded minimal surface $M'\subset \Sph^3$ whose coordinate functions satisfy \emph{an additional quadratic relation}, and it is shown in Proposition \ref{prop:uniq_map} that this forces $M'$ to be either a great sphere or a Clifford torus. Thus, proving \eqref{lap.sym.mono} in this case reduces to proving 
$$\Lambda_1^{\Gamma}(M)> 4\pi^2 $$%>\bar{\lambda}_1(T_{Cliff})=4\pi^2;$$
for all the group actions appearing in Theorem \ref{Tsph1}, which can be checked directly via comparison with the Lawson surfaces $\xi_{\gamma, 1}$ of genus $\gamma \geq 2$, yielding the existence of $\bar{\lambda}_1$-maximizing metrics with these symmetries, and in particular a proof of Theorem \ref{Tsph1}.

Similarly, for a large class of compact, oriented BRS with nonempty boundary, we argue in Section \ref{sec:global_Sexistence} that an analog of \eqref{stek.k.mono0} continues to hold, and the failure of \eqref{stek.sym.mono} implies the existence of a degeneration $(N',\Gamma')$ realized by a free boundary minimal embedding $N'\to \mathbb{B}^3$ by first eigenfunctions whose boundary $\partial N'$ is the intersection of $\Sph^2$ with an elliptic cylinder. We conjecture that the only such surface $N'$ is the critical catenoid; indeed, such a surface must have two boundary components, but it is not clear \emph{a priori} whether the genus can be positive. Moreover, even when the inequality \eqref{stek.sym.mono} can be reduced to comparison with the critical catenoid, in general we have no analog of the comparison to Lawson surfaces $\xi_{\gamma, 1}$ from the closed setting, and so must resort to more ad-hoc arguments to complete the proof.  We are able to carry out this strategy for many, but not all of the desired group actions, and for this reason the statement of the main existence result, Theorem \ref{thm:group_Sexistence} is more technical than the corresponding theorem in the closed case. 

\subsection{Discussion and open questions}

\subsubsection{Area, varifold limits, and Morse index of the minimal surfaces}

Beyond prescribing the symmetries and topological type, it is natural to seek a finer geometric description of the minimal surfaces constructed in Theorems \ref{Tsph1} and \ref{Tfbms1} above, using these new families to test some well-known conjectures for minimal surfaces in $\Sph^3$ and $\mathbb{B}^3$.

As a first step in this direction, we consider the problem of obtaining improved area estimates for these minimal surfaces, beyond the a priori bounds by $8\pi$ and $2\pi$, respectively. By construction, estimating the areas of these surfaces is equivalent to estimating the maximum value of $\bar{\lambda}_1$ and $\bar{\sigma}_1$ on surface with a prescribed symmetry group, reducing the question to a fairly concrete problem in spectral geometry. An important source of inspiration is the following well-known conjecture--a special case of a conjecture  \cite[Conjecture 8.4]{KusnerWillmore} characterizing the Lawson surfaces $\xi_{\gamma,1}$ as minimizers for the Willmore functional among surfaces of genus $\gamma$.

\begin{conjecture}\label{ag.conj}\cite{KusnerWillmore}
\label{Clawson}
 The minimum area $\Acal_\gamma$ among all minimal surfaces in $\Sph^3$ of genus $\gamma$ is uniquely attained by the Lawson surface $\xi_{\gamma, 1}$.
\end{conjecture}

Of course, the upper bound
$$\Acal_\gamma\leq \Area(\xi_{\gamma,1})=8\pi\left(1-\frac{\log 2}{2\gamma}+O(\gamma^{-2})\right),$$
is immediate, with the expansion on the right-hand side coming from \cite{HHT}, and it is known \cite{KLS} that $\lim_{\gamma\to\infty}\Acal_{\gamma}=8\pi.$ On the other hand, all of the minimal surfaces constructed in Theorem \ref{Tsph1} have area $<8\pi$, and could therefore \emph{ a priori } give examples with area $<\Area(\xi_{\gamma,1})$. In forthcoming work, we rule this out in the large genus limit, thus confirming Conjecture \ref{ag.conj} for all the examples constructed here with genus sufficiently large.

\begin{claim}\label{area.exp}
There exists $\gamma_0\in \mathbb{N}$ such that for every minimal surface $M_{\gamma}\subset \Sph^3$ of genus $\gamma\geq \gamma_0$ constructed in Theorem \ref{Tsph1}, we have $\Area(M_{\gamma})\geq \Area(\xi_{\gamma,1})$. Moreover, there is a constant $C>0$ such that if $M_{\gamma}$ is obtained by $\bar{\lambda}_1$-optimization with respect to a symmetry group $\Gamma\leq \Diff(M_{\gamma})$ not conjugate to a subgroup of $\Isom(\xi_{\gamma,1})$, then 
$$\Area(M_{\gamma})\geq 8\pi-e^{-C\sqrt{\gamma}}.$$
\end{claim}

These lower bounds are obtained by directly constructing test metrics for the associated $\Gamma$-invariant $\bar{\lambda}_1$-optimization problems. For many of these families, we show that the area grows at an even faster rate $\Area(M_{\gamma})\sim 8\pi-e^{-C\gamma}$, which seems to give the largest possible area for basic reflection surfaces embedded by first eigenfunctions.

It is interesting to note that the area expansion for minimal doublings in \cite[Theorem A(v)]{LDG} shows that if $M_\gamma$ is an $\Sph^2$-doubling constructed by  Kapouleas \cite{Kapouleas} or by Kapouleas-McGrath \cite{KapMcG, LDG} with genus $\gamma$, then $\area(M_\gamma) > 8\pi(1-e^{-C\sqrt{\gamma}})$.  Actually, the stronger bound $\area(M_\gamma) > 8\pi ( 1- e^{-C\gamma})$ holds for all of the families except the ``equator-poles" family from \cite{Kapouleas}. Indeed, we expect that there is significant overlap in the large genus limit between the families obtained by $\bar{\lambda}_1$-optimization in Theorem \ref{Tsph1} and some families obtained by gluing methods.

In addition to the area estimates of Claim \ref{area.exp}, we can also characterize the varifold limits for most of the families of minimal surfaces in Theorem \ref{Tsph1} in the large genus limit. A simple compactness argument using the monotonicity formula shows that a varifold limit of such a sequence must be a multiplicity-two great sphere, as for the $\Sph^2$-doublings from \cite{Kapouleas, KapMcG, LDG},  or the disjoint union of two great spheres, as for the Lawson surfaces $\xi_{\gamma, 1}$. For the families of Theorem \ref{Tsph1}, we show that nearly all of them converge to a multiplicity-two sphere, except for those which we expect to coincide with the Lawson surfaces $\xi_{\gamma,1}$. 

\begin{claim}\label{var.lim}
There exists a sequence $\epsilon_{\gamma}$ with $\lim_{\gamma\to\infty}\epsilon_{\gamma}=0$ such that if $M$ is one of the surfaces of genus $\gamma$ obtained in Theorem \ref{Tsph1}, realizing $\Lambda_1^{\Gamma}$ for a group $\Gamma\leq \Diff(M)$ not conjugate to a subgroup of $\Isom(\xi_{\gamma,1})$, then 
$$d_V(M,2\cdot \Sph^2_{eq})\leq \epsilon_{\gamma},$$
where $d_V$ denotes the varifold distance (in the sense of \cite{Pitts}), and $\Sph^2_{eq}\subset \Sph^3$ is the equator separating $M$ into isometric genus zero components.
\end{claim}

More generally, we show in forthcoming work that any family of genus zero free boundary minimal surfaces in the hemisphere $\Sph^3_+$, embedded by first eigenfunctions, with $\gamma+1$ boundary components and area $>8\pi-o(1/\gamma^3)$ must converge to the boundary equator $\Sph^2_{eq}=\partial \Sph^3_+$ as $\gamma\to\infty$. 

Another well-known conjecture asserts that the Lawson surfaces $\xi_{\gamma,1}$ have the least \emph{Morse index} for their genus. 
\begin{conjecture}
\label{Cindex}
If $M$ is a closed minimal surface in $\Sph^3$ of genus $\gamma$, then
\begin{align*}
\mathrm{ind}(M) \geq \mathrm{ind}(\xi_{\gamma, 1}) = 2\gamma+3.
\end{align*} 
\end{conjecture}
\noindent
The fact that $\mathrm{ind}(\xi_{\gamma, 1}) = 2\gamma+3$ was established by Kapouleas-Wiygul \cite{KWindex}. With this conjecture in mind, the following is a natural question.

\begin{question}
What are the Morse indices of the minimal surfaces constructed in Theorem \ref{Tsph1}? 
\end{question}

In a related direction, it would also be interesting to determine the minimal symmetry assumptions needed to characterize the Lawson surfaces $\xi_{\gamma,1}$. Since $\xi_{\gamma,1}$ is embedded by first eigenfunctions \cite{ChoeSoret}, this is essentially equivalent to determining the smallest groups of diffeomorphisms on the closed surface $M$ of genus $\gamma$ for which $\Lambda_1^{\Gamma}(M)=2\Area(\xi_{\gamma,1})$. Work of Kapouleas-Wiygul characterizes the Lawson surfaces by their topology and full symmetry group \cite{KWsym}, but we suspect that they can be characterized at the level of $\mathbb{Z}_2\times \mathbb{Z}_2$ symmetries alone; more precisely, we make the following conjecture. 

\begin{conjecture}
\label{Cmonoz2}
Let $M_{\gamma, a}$ denotes the genus-$\gamma$ minimal surface from Theorem \ref{Tsph1} with $G = \Z_2$ and $\gamma - 2a \geq 0$ doubling holes on the equatorial circle fixed by $\Z_2$. Then
\begin{enumerate}
\item $M_{\gamma, 0}$ is the unique minimal surface with its symmetry group and genus, hence is congruent to $\xi_{\gamma,1}$. 
\item $\area(M_{\gamma, a+1}) \geq \area(M_{\gamma, a})$ for each $a \in \{0, \dots, \lfloor \frac{\gamma}{2}\rfloor\}$.
\end{enumerate}
\end{conjecture}
Note that the monotonicity statement (2) would directly imply Conjecture \ref{ag.conj} for all of the minimal surfaces constructed in Theorem \ref{Tsph1} with $G=\Z_2$.

Most of the preceding questions have natural counterparts for free boundary minimal surfaces in $\mathbb{B}^3$. In particular, with the topological realization problem solved, one can ask about the \emph{lowest-complexity} free boundary minimal surfaces in $\mathbb{B}^3$ of given topological type, in terms of area, Morse index, etc, perhaps identifying candidate minimizers for the Willmore-type energy $\int_{\Sigma}\frac{1}{4}H_{\Sigma}^2+\int_{\partial \Sigma}\kappa_{\partial\Sigma}$ among surfaces in $\mathbb{B}^3$ of given topological type.

\begin{question}
What is the minimum area $\mathcal{A}(\gamma,k)$ of a free boundary minimal surface in $\mathbb{B}^3$ with genus $\gamma$ and $k$ boundary components?
\end{question}
It is well-known \cite{FSadvances} that $\mathcal{A}(\gamma,k)\geq \pi$, with equality realized only by the disk ($\gamma=0$, $k=1$). It is conjectured that $\mathcal{A}(0,2)$ is realized by the critical catenoid, and moreover that $\mathcal{A}(\gamma,k)\geq \mathcal{A}(0,2)$ whenever $\gamma+k\geq 2$. Restricting our attention to the genus zero setting, we propose the following.
\begin{conjecture}
\label{Cknoid}
The following hold. 
\begin{enumerate}
\item For each $k\geq 2$, the minimum area $\Acal(0, k)$ is realized by the $\Z_2\times D_k$-symmetric ``$k$-noid" from Theorem \ref{thm:group_Sexistence}.   
\item For each $k>2$, the ``$k$-noid" from Theorem \ref{thm:group_Sexistence} has area greater than that of the critical catenoid.
\end{enumerate}
\end{conjecture}

Part (2) would significantly simplify the proof of Theorem \ref{thm:group_Sexistence}.  We remark
that $D_k$-symmetric genus zero surfaces with $k$ boundary components have been constructed in~\cite{Zolotareva}
for large $k$ and in~\cite[Theorem 5.1]{FrS} for
any $k$. It is an interesting question whether these
surfaces and the ones from Theorem~\ref{thm:group_Sexistence} coincide (cf.
~\cite[Conjecture 7.8]{Carlotto}). In particular, by construction the surfaces in~\cite{Zolotareva} have areas close to $2\pi$, so that by~\cite{KusnerMcGrath} we have that Conjecture~\ref{Cknoid}.(2) holds for large $k$.

The proof of Theorem \ref{thm:group_Sexistence} would also be greatly simplified if the following conjecture were known to hold (see Proposition \ref{prop:non_uniq} and the discussion just above). 

\begin{conjecture}
Suppose $N$ is an orientable surface with two boundary components, and $\Psi : (N, g) \rightarrow \B^3$ is a free boundary minimal embedding by $\sigma_1(N, g)$-eigenfunctions such that the multiplicity of $\sigma_1(N, g)$ equals $3$.  If there is a free boundary harmonic map $\Xi$ to $\B^3$ by $\sigma_1(N, g)$-eigenfunctions such that  $\Xi \neq A \Psi$ for all $A \in O(3)$, then $(N, g)$ is the critical catenoid.  
\end{conjecture}

Note also that Theorem \ref{Tfbms1} gives the universal upper bound $\mathcal{A}(\gamma,k)<2\pi,$ 
and a geometric measure theoretic argument analogous to that of \cite{KLS} shows this inequality is saturated as $\gamma+ k\to\infty$, that is, that
$$\lim_{\gamma+k\to\infty}\mathcal{A}(\gamma,k)=2\pi.$$ 
For many of the examples of $\B^2$-doublings constructed by $\bar{\sigma}_1$-maximization over basic reflection surfaces, by directly estimating the suprema $\Sigma_1^{\Gamma}$, we can show that the area grows like $2\pi-e^{-c(k+\gamma^2)}$ with respect to the genus $\gamma$ and number $k$ of boundary components.  Analogous to the Lawson surfaces $\xi_{\gamma, 1}$ desingularizing two orthogonal spheres, we expect that the areas of desingularizations of two orthogonal disks (described in the introduction of \cite{KapLi}) approach $2\pi$ more slowly. In addition to estimating the areas of the surfaces constructed in Theorem \ref{Tfbms2}, it would of course be interesting to compute other natural measures of complexity, like the Morse index, and identify those with lowest complexity of a given topology.

Moreover, by improving the existence theory for $\bar{\sigma}_1$-maximizing metrics on a larger class of basic reflection surfaces, results along the following lines seem to be within reach.

\begin{conjecture}
The number of distinct embeddings of the compact orientable surface $N_{\gamma, k}$ with genus $\gamma$ and $k$ boundary components as a free boundary minimal surface in $\B^3$ grows at least linearly with $\gamma+k$. 
\end{conjecture}

\subsubsection{Free boundary minimal surfaces and the realization problem in other settings}

While Theorem \ref{Tfbms1} solves the topological realization problem for free boundary minimal surfaces in $\mathbb{B}^3$, we observe that the $\mathbb{Z}_2\times\mathbb{Z}_2$-symmetric families of minimal surfaces in $\mathbb{S}^3$ constructed in Theorem \ref{Tsph1} resolve the analogous problem for free boundary minimal surfaces in the \emph{hemisphere} $\mathbb{S}^3_+$: for any $b=1,2,\ldots$ and $a=0,1,\ldots,$ the half of $M_{2a+b,a}$ corresponding to the fundamental domain of the second reflection is a free boundary minimal surface in $\mathbb{S}^3_+$ of genus $a$ with $b$ boundary components, giving the following.

\begin{corollary}
The hemisphere $\mathbb{S}^3_+$ contains a free boundary minimal surface of every oriented topological type, with area $<4\pi$.
\end{corollary}

Indeed, it is interesting to note that the problem of maximizing $\bar{\lambda}_1$ among metrics on a closed surface invariant under a reflection $\tau$ is equivalent to the problem of maximizing the minimum $\min\{\lambda_1^{\mathrm{Neu}}(N),\lambda_1^{\mathrm{Dir}}(N)\}\cdot \Area(N)$ of the first Neumann and Dirichlet eigenvalues on a fundamental domain of the reflection. In particular, as a consequence of Theorem \ref{Tsph1} and its proof in the case $G=1$, we have the following.

\begin{corollary}
For any $k\in \mathbb{N}$, the compact oriented surface $N_k$ with genus zero and $k$ boundary components admits a metric realizing the maximum
$$\tilde{\Lambda}_1(N_k):=\max\{\Area(N_k,g)\min\{\lambda_1^{\mathrm{Dir}}(N_k,g),\lambda_1^{\mathrm{Neu}}(N_k,g)\}\mid g\in \Met(N_k)\},$$
induced by a free boundary minimal embedding in $\Sph^3_+$ by first Dirichlet and Neumann eigenfunctions. For the disk $\mathbb{D}$ and annulus $\mathbb{A}$, we have
$$\tilde{\Lambda}_1(\mathbb{D})=4\pi,\text{\hspace{4mm}}\tilde{\Lambda}_1(\mathbb{A})=2\pi^2.$$
\end{corollary}

Recently, in \cite{LimaMenezes2, Medvedev2}, it has been observed that certain other eigenvalue optimization problems give rise to free boundary minimal surfaces in other convex geodesic balls in space forms. By applying the techniques of the present paper to the optimization problems studied in \cite{LimaMenezes2, Medvedev2}, we expect that the realization can be solved for geodesic balls in $\mathbb{S}^3_+$ or $\mathbb{H}^3$. 

\begin{conjecture}
Any geodesic ball in $\mathbb{S}^3_+$ or $\mathbb{H}^3$ admits free boundary minimal surfaces of any prescribed orientable topology.
\end{conjecture}

\subsubsection{Equivariant optimization beyond orientable basic reflection surfaces.}

While our strongest results at present apply to reflection groups on basic reflection surfaces, we anticipate that the equivariant $\bar{\lambda}_1$- and $\bar{\sigma}_1$-optimization techniques developed here can be applied to various other group actions to produce many other families of minimal surfaces in $\Sph^n$ and $\mathbb{B}^n$ with prescribed topology and symmetries.

This article uses three methods for proving embeddedness of minimal immersions by first eigenfunctions.  First, for basic reflection surfaces, the areas of the corresponding immersions must be (Lemma \ref{Levalbdcl}) less than $8\pi$ in the closed case ($2\pi$ in the case with nonempty boundary), hence are embeddings by Li-Yau-type results.  Second, for basic reflection surfaces (but independently from the first argument), nodal domain arguments in the closed case and Morse-theoretic arguments in the case with boundary show the corresponding minimal immersions are doublings of $\Sph^2$ or $\B^2$, hence are embeddings.  Third, for genus zero surfaces with nonempty boundary, work of Fraser-Schoen \cite[Proposition 8.1]{FraserSchoen} shows that all such immersions are necessarily star-shaped, hence are embeddings.  It would be interesting to develop  techniques for proving embeddedness for other classes of surfaces. 

While the applications in this article are for constructions of oriented hypersurfaces, the concept of basic reflection surface includes the nonorientable case. Modulo the existence theory for maximizing metrics in that setting, the results of Section \ref{S:BRS} suggest the following codimension-two realization result  via $\bar{\lambda}_1$- and $\bar{\sigma}_1$-optimization on these nonorientable basic reflection surfaces.

\begin{conjecture}\label{nonor.real}
Every closed nonorientable surface admits a minimal embedding in $\Sph^4$ as an $\Sph^2$-doubling, with area $<8\pi$. Likewise, every compact nonorientable surface with boundary admits a free boundary minimal embedding in $\mathbb{B}^4$ as a $\B^2$-doubling, with area $<2\pi$.
\end{conjecture}

At present, the key challenge lies in proving the non-orientable analog of inequalities \eqref{lap.sym.mono} and \eqref{stek.sym.mono} in this setting. Indeed, even when $\Gamma$ is trivial, we have not yet found an analog of Proposition \ref{intro.handle} for nonorientable surfaces, and new ideas may be needed. 

Finally, we return to the classical problems of maximizing $\bar{\lambda}_1$ among all metrics on a given closed, orientable surface, and maximizing $\bar{\sigma}_1$ among all metrics on a compact, orientable surface with boundary. In view of Propositions \ref{intro.handle} and \ref{intro.strip}, the existence theory in these settings will be complete once the following conjectures are proved.

\begin{conjecture}\label{somepair.lap}
If $M^2\subset \Sph^n$ is a branched minimal immersion by first eigenfunctions for the Laplacian, there exists some pair of points $p,q\in M$ such that for every map $F\colon M\to \Sph^n$ by first eigenfunctions, $F(p)\neq F(q)$.
\end{conjecture}

\begin{conjecture}\label{somepair.stek}
If $N^2\subset\mathbb{B}^n$ is a free boundary branched minimal immersion by first Steklov eigenfunctions such that $\partial N$ has at least two components, there exists a pair of points $p,q$ in distinct components of $\partial N$ such that for every map $F\colon(N,\partial N)\to (\mathbb{B}^n,\Sph^{n-1})$ by first Steklov eigenfunctions, $F(p)\neq F(q)$.
\end{conjecture}

Both conjectures have something of an algebro-geometric flavor, asserting roughly that the minimal surfaces in question cannot lie in the intersection of too large a family of quadric hypersurfaces. We hope to see both proved in the near future, solving the original existence problem for $\bar{\lambda}_1$- and $\bar{\sigma}_1$-maximizing metrics on orientable surfaces of all topological types.

\subsection{Outline of the paper} In Section~\ref{sec:prelim} we collect some preliminary observations about the functionals $\Lambda_1^\Gamma(M)$ and $\Sigma_1^\Gamma(N)$. In particular, we establish characterisation of critical points for eigenvalue functionals in the space of $\Gamma$-symmetric metrics as those induced by equivariant minimal immersions. Sections~\ref{lap.mm} and~\ref{SSprelim} contain the proofs of min-max characterization of  $\Lambda_1^\Gamma(M,[g])$ and $\Sigma_1^\Gamma(N,[g])$ respectively, which, in particular, imply existence of conformal maximizers of Theorems~\ref{intro.lap.cf} and~\ref{intro.stek.cf}. This generalizes the results of~\cite{KSminmax} to Steklov and $\Gamma$-symmetric settings and, moreover, simplifies some of the original arguments of~\cite{KSminmax}. Reflection groups and the geometry of BRS are covered in Section~\ref{S:BRS}. This is where we prove all the geometric properties of the minimal immersions corresponding to eigenvalue maximisers on BRS, such as bounds on the dimension of the sphere, area, and embeddedness. After that, Theorems~\ref{Tsph1},~\ref{Tfbms1} and~\ref{Tfbms2} follow once the existence of maximizers is established. The remainder of the paper, which is the most technical part, is devoted to precisely that task. In Sections~\ref{sec:top_degen} and~\ref{sec:conf} we prove the existence under an additional gap condition, see Theorems~\ref{intro.lap.glob} and~\ref{intro.stek.glob}. Finally, Sections~\ref{sec:global_existence} and~\ref{sec:global_Sexistence} contain a novel method for verifying the gap conditions~\eqref{lap.sym.mono} and~\eqref{intro.stek.glob}, which we then apply to resolve the existence question for most BRS, thus, completing the proofs of our main results.

\subsection*{Acknowledgements} The authors are grateful to Mario Schulz for providing Figures~\ref{fig:fbms1} and~\ref{fig:fbms2} and for remarks on the preliminary version of the manuscript. The authors also thank Jean Lagac\'e for fruitful discussions leading to the proof of Theorem~\ref{thm:gamma+2}, and thank Romain Petrides and Anna Siffert for helpful correspondence. P.M. thanks N. Kapouleas for many discussions about minimal surface doublings. M.K. was partially supported by the NSF grant DMS-2104254 during the initial stages of this project, D.S. was partially supported by the NSF fellowship DMS-2002055, and  P.M. was partially supported by Simons Foundation Collaboration Grant 838990.  R.K. is grateful to NC State University for its hospitality during the Spring 2023 semester, where part of this work was carried out.

%===================================================
%===================================================
%===================================================
%===================================================
%===================================================
%===================================================

%===================================================
%===================================================
%===================================================
%===================================================
%===================================================
%===================================================

\section{Preliminaries and initial observations}
\label{sec:prelim}

In this section, we set some basic notation and definitions concerning isometries of surfaces and their actions on Laplace and Steklov eigenspaces, and describe some basic features of the equivariant optimization problems.

\subsection{Isometries}

\begin{notation}
Given an isometry $\tau$ of $M$, 
let $M^\tau$ denote its fixed-point set.  If $M$ has nonempty boundary $\partial M$, let $M^\tau_\partial \subset M^\tau$ denote the union of the components of $M^\tau$ meeting $\partial M$. 
\end{notation}

\noindent Note \cite{Kobayashi} that $M^\tau$ is a disjoint union of totally geodesic submanifolds of $M$.

\begin{definition}
\label{drefl}
An isometry $\tau$ of $M$ is called a \emph{reflection} if for some $p \in M^\tau$, the differential $d \tau_p$ is a reflection in the Euclidean space $(T_p M, \left. g\right|_p)$ with respect to a hyperplane. 
\end{definition}

\noindent The next results collect some standard facts \cite{Michor} about reflections. 
\begin{lemma}
\label{Lref}
If $\tau$ is a reflection on $M$, then
\begin{enumerate}[label=\emph{(\roman*)}]
\item $\tau$ is an involution, $\tau^2 = \tau \circ \tau =  \Id_M$;
\item $d\tau_p|_{T_p S} = \Id_{T_p S}$ and $d\tau_p |_{T_p S^\perp} = -\Id_{T_p S^\perp}$ for each component $S$ of $M^\tau$ and each $p \in S$;
\item at least one component of $M^\tau$ has codimension $1$; 
\item for any $\rho\in \mathrm{Isom}(M)$, $\rho \tau \rho^{-1}$ is a reflection, and $M^{\rho \tau \rho^{-1}} = \rho(M^\tau)$.
\end{enumerate}
\end{lemma}

\begin{lemma}
\label{Ldis}
If $\tau$ is an isometry on $M$ and $M \setminus M^\tau$ is not connected, then
\begin{enumerate}[label=\emph{(\roman*)}]
\item $\tau$ is a reflection;
\item $M \setminus M^\tau$ has exactly two components, which are exchanged by $\tau$;
\item $M^\tau$ is a disjoint union of codimension-$1$ submanifolds. 
\end{enumerate}
\end{lemma}

\noindent An isometry $\tau$ as in Lemma \ref{Ldis} is called \emph{separating}.  

If $G\leq Isom(M)$ is a subgroup of isometries of $M$, let $M/ G$ denote the quotient and $\pi_G : M \rightarrow M/ G$ denote the quotient map.  We will sometimes write $\pi_G(M)$ instead of $M/G$ for convenience.

\subsection{Eigenfunctions}
We consider two eigenvalue problems on Riemannian manifolds.  If $(M,g)$ is closed, the \emph{Laplace eigenvalue problem} is
\begin{equation}
\label{Elaplace}
\Delta u=\lambda u,
\end{equation}
a function $u$ satisfying \eqref{Elaplace} is a \emph{Laplace eigenfunction}, and $\lambda \in \R$ is its \emph{eigenvalue}.  Note that we are taking the positive Laplacian $\Delta=-d^*d$, so that with respect to the $L^2(M)$ inner product, %Should we write out the integral here???
 $\Delta$ is a self-adjoint differential operator with discrete spectrum
\begin{align*}
0 = \lambda_0 < \lambda_1 \leq \lambda_2 \leq \cdots. 
\end{align*}
Let $\Ecal_{\lambda_i}$ denote the $\Delta$-eigenspace corresponding to $\lambda_i$, and write
$$
\bar{\lambda}_i(M,g):=\Area(M,g)\lambda_i(M,g)
$$
for the scale-invariant normalization of $\lambda_i$ by the area.

If $(M,g)$ has nonempty boundary, the \emph{Steklov eigenvalue problem} is
\begin{equation}
\label{Esteklov}
  \begin{cases} 
      \hfill \Delta u  = 0    \hfill & \text{in}\quad M \\
      \hfill \frac{\partial u}{\partial \eta} = \sigma u \hfill & \text{on}\quad \partial M, \\
  \end{cases}
\end{equation}
where $\eta$ is the unit outward pointing conormal vector field on $\partial M$. 
A function $u$ satisfying \eqref{Esteklov} is called a \emph{Steklov eigenfunction}.  
The eigenvalues of \eqref{Esteklov} are the spectrum of the \emph{Dirichlet-to-Neumann map} 
$\mathcal{D}\colon C^\infty(\partial M)\rightarrow C^{\infty}(\partial M)$ given by
\[ \mathcal{D} u = \frac{\partial \hat{u}}{\partial \eta} \]
where $\hat{u}$ denotes the harmonic extension of $u$ to $M$.  With respect to the $L^2(\partial M)$ inner product, $\mathcal{D}$ is a self-adjoint pseudodifferential operator with discrete spectrum
\[ 0 = \sigma_0 < \sigma_1 \leq \sigma_2 \leq \cdots. \]
Let $\Ecal_{\sigma_i}$ denote the $\mathcal{D}$-eigenspace corresponding to $\sigma_i$, and write
$$\bar{\sigma}_i(M,g):=\mathrm{Length}(\partial M,g)\sigma_i(M,g)$$
for the scale-invariant normalization of $\sigma_i$ by the length of $\partial M$.

\begin{convention}
Because both the situations where $M$ is closed and where $M$ has nonempty boundary will arise, in the later discussion, we often write \emph{eigenfunction} to mean Laplace eigenfunction in the former case, and Steklov eigenfunction in the latter.  Furthermore, $\firsteigen$ will denote the first Laplace eigenspace $\lapone$ in the former case, and the first Steklov eigenspace $\stekone$ in the latter. 
\end{convention}

The \emph{nodal set} of an eigenfunction $u$ is $ \Ncal_u := \{ p \in M  : u(p) = 0\}$ and
a \emph{nodal domain} of $u$ is a connected component of $M \setminus \Ncal_u$.  
 The following Courant-type nodal domain theorem is standard  \cite{Courant, Proc}:
\begin{lemma}
\label{Lcourant}
Each nonzero $u \in \firsteigen$ has exactly two nodal domains. 
\end{lemma}

\subsection{Equivariant optimization}

\label{sec:prelim}

Let $M$ be a compact surface (with or without boundary), let $\Gamma$ be a finite group acting on $M$ by diffeomorphisms via the faithful action $T\colon \Gamma\times M\to M$. Denote by $\Met_T(M)$ the space of Riemannian metrics on $M$ invariant under this action; a standard averaging argument shows that $\Met_T(M)\neq \varnothing$. For any metric $g\in \Met_T(M)$, note also that the action of $T$ preserves the conformal class $[g]$, and by the uniformization theorem, $[g]$ contains a $T$-invariant constant-curvature representative. The latter fact will be particularly useful in Section \ref{sec:conf}, where it is necessary to describe the moduli space of $T$-invariant conformal classes.

Throughout the paper, our central concern is the existence and structure of metrics realizing the suprema
$$\Lambda_1^T(M):=\sup\{\bar{\lambda}_1(M,g)\mid g\in \Met_T(M)\}$$
when $\partial M=\varnothing$, and
$$\Sigma_1^T(M):=\sup\{\bar{\sigma}_1(M,g)\mid g\in \Met_T(M)\}$$
when $\partial M\neq \varnothing$. As a key step in the existence theory, in Sections \ref{lap.mm} and 4 below, we first develop the existence theory for metrics realizing the $\Gamma$-invariant \emph{conformal} suprema
$$\Lambda_1^T(M,[g]):=\sup\{\bar{\lambda}_1(M,h)\mid h\in \Met_T(M)\cap [g]\},$$
$$\Sigma_1^T(M,[g]):=\sup\{\bar{\sigma}_1(M,h)\mid h\in \Met_T(M)\cap [g]\}.$$
Note that when discussing the conformal maximization problem, there is no loss of generality in taking $\Gamma$ to be a subgroup of the diffeomorphism group $\Diff(M)$ with its natural action on $M$, and in this case we write $\Met_{\Gamma}(M)=\Met_T(M)$, $\Lambda_1^{\Gamma}(M,[g])=\Lambda_1^T(M,[g])$, and $\Sigma_1^{\Gamma}(M,[g])=\Sigma_1^T(M,[g])$.

%The setup is as in Daniel's note: $M$ -- closed surface, for now orientable of genus $\gamma\geq 2$ (we have to use $M$, not $\Sigma$, since Steklov eigenvalues will be $\Sigma$), $\Gamma$ -- finite group acting on $M$ via $T\colon \Gamma\times M \to M$. 
 %A standard averaging argument yields the existence of a metric $g$ such that $\Gamma$ acts by isometries of $g$ and, in particular, $\Gamma$ preserves the conformal class $[g]$. Our first goal is to understand the moduli space of all such conformal classes for a fixed action $F$ of $\Gamma$.

%A standard way to describe the space of conformal classes is to choose a canonical representative. Namely, by uniformisation theorem there exists a unique hyperbolic metric $h\in [g]$. Since $\Gamma$ preserves the conformal class $[g]$, uniqueness of $h$ implies that 
%$\Gamma$ acts by isometries of $h$. Conversely, if $h$ is a hyperbolic metric on $M$ such that  all elements of $\Gamma$ preserve $h$, then $\Gamma$ preserves the conformal class $[h]$. 

\subsection{Initial remarks about equivariant optimization.} 

On $\Sph^2$ and $\mathbb{RP}^2$, we observe next that maximization of $\bar{\lambda}_1$ with respect to any symmetry group yields the canonical round metric, as a consequence of the following.

\begin{proposition}[{\cite[Corollary 2.13.2]{Singerman}}]
\label{s2.gp}
Any finite group $\Gamma\subset \aut(\mathbb{S}^2)$ is conjugate to a subgroup of $O(3)$.
\end{proposition}

\begin{corollary}
\label{cor:max_sphere}
For any finite group $\Gamma$ and any action $T$ of $\Gamma$ on $\mathbb{S}^2$ one has
$$
\Lambda_1^T(\mathbb{S}^2) = 8\pi.
$$
\end{corollary}
\begin{proof}
As explained before, without loss of generality, we can assume that $\Gamma$ acts by conformal transformations.
The previous proposition implies that there is a $T$-invariant metric of constant curvature $g$. Any such metric satisfies $\bar\lambda_1(\Sph^2,g) = 8\pi$, so that $\Lambda_1^T(\mathbb{S}^2) \geq 8\pi$. It follows from~\cite{Hersch} that 
\[
\Lambda_1^T(\Sph^2)\leq \Lambda_1(\Sph^2) = 8\pi,
\]
so the converse inequality is true.
\end{proof}

\begin{corollary}
\label{cor:max_RP2}
For any finite group $\Gamma$ and any action $T$ of $\Gamma$ on $\mathbb{RP}^2$ one has
$$
\Lambda_1^T(\mathbb{RP}^2) = 12\pi.
$$
\end{corollary}
\begin{proof}
The action $T$ of $\Gamma$ on $\mathbb{RP}^2$ can be lifted to the action $\wt T$ of $\Gamma$ on $\mathbb{RP}^2$, such that $\wt T$ commutes with the antipodal involution. It is then well-known that any conformal transformation of $\Sph^2$ commuting with the antipodal involution is an isometry, so that $T$ in fact acts by isometries of the round metric on $\mathbb{RP}^2$. The latter fact also follows from~\cite[Corollary 3]{MR}. Once it is established, the rest of the proof is as in Corollary~\ref{cor:max_RP2}.
\end{proof}

Similarly, the standard metric on the disk $\mathbb{D}^2$ maximizes $\bar{\sigma}_1$ under any prescribed symmetry group, as a consequence of the following. 

\begin{proposition}
Any finite group $\Gamma\subset \aut(\mathbb{D}^2)$ is conjugate to a subgroup of $O(2)$.
\end{proposition}
\begin{proof}
This follows in a straightforward way from Proposition \ref{s2.gp} by identifying $\aut(\mathbb{D}^2)$ with the subgroup of $\aut(\Sph^2)$ preserving the closed hemisphere $\Sph^2_+$.
\end{proof}

\begin{corollary}
\label{cor:max_disk}
For any finite group $\Gamma$ and any action $T$ of $\Gamma$ on $\mathbb{D}^2$ one has
$$
\Sigma_1^T(\mathbb{D}^2) = 2\pi.
$$
\end{corollary}
\begin{proof}
Analogous to the proof of Corollary~\ref{cor:max_sphere} using the result of~\cite{Weinstock} that $\Sigma_1(\mathbb{D}^2) = 2\pi$.
\end{proof}

For many group actions $T:\Gamma\times M\to M$, $\Lambda_1(\Sph^2)=8\pi$ and $\Sigma_1(\mathbb{D}^2)=2\pi$ also provide universal lower bounds for the associated maximization problem on any $T$-invariant conformal class.

\begin{proposition}
\label{prop:Lambda_fixedpt}
Suppose that $(M,T)$ has a fixed point, i.e. $M^\Gamma\ne\varnothing$ and let  $\mC$ be a $T$-invariant conformal class. Then
\[
\Lambda_1^T(M,\mC)\geq 8\pi.
\]
\end{proposition}
\begin{proof}
The unconstrained (i.e. without the group action) version of this statement can be found~\cite[Theorem A]{CES}. We provide a slightly different proof using the tools from~\cite{KM}.
 Let $p\in M^\Gamma$ and let $U\ni p$ be a $T$-invariant conformally flat neighbourhood of $p$, which exists because $\mC$ is $T$-invariant. We identify $U$ is a domain in $\R^2$, so that $p$ is identified with the origin. Since $T$ preserves a small circle around $p$, we conclude that $T(\Gamma)$ can be identified with finite subgroup of the isometries of a circle. In particular, the stereographic projection $\pi\colon\Sph^2\setminus\{N\}\to \R^2$ is $\Gamma$-equivariant for the action of $\Gamma$ on $\Sph^2$ leaving $z$-coordinate invariant. Define a metric $g_n$ on $U$ to be the restriction of the pullback $(\pi^{-1}\circ(n\cdot))^*g_{\Sph^2}$, where $(n\cdot)$ is a scaling map $x\mapsto nx$. 
 By~\cite[Lemma 4.6]{KM}, one has $\Lambda^T_1(M,\mC)\geq \bar\lambda^N_1(U, g_n)$, where $\bar\lambda_1^N$ refers to the first non-trivial normalized Neumann eigenvalue. At the same time, since the preimages 
 $\pi^{-1}(n\cdot U)$ cover $\Sph^2\setminus\{N\}$, by~\cite[Lemma 4.1]{KM} one has $\bar\lambda^N_1(M'\setminus D_\eps(P'), g')\to \bar\lambda_1(\Sph^2,g_{\Sph^2}) = 8\pi$.
\end{proof}

\begin{proposition}
\label{prop:Sigma_fixedpt}
Suppose that $(N,T)$ has a fixed point on $\bd N$, i.e. $(\bd N)^\Gamma\ne\varnothing$ and let  $\mC$ be a $T$-invariant conformal class. Then
\[
\Sigma_1^T(M,\mC)\geq 2\pi.
\]
\end{proposition}
\begin{proof}
Analogous to Proposition~\ref{prop:Lambda_fixedpt}, but using the tools from~\cite{Medvedev} as opposed to~\cite{KM}. 
\end{proof}

\subsection{Extremal metrics for $\bar{\lambda}_1$ and $\bar{\sigma}_1$ in $\Met_T(M)$}

Next, we characterize critical metrics for $\bar{\lambda}_k$ or $\bar{\sigma}_k$ metrics on $\Met_T(M)$ as those realized by $\Gamma$-equivariant branched minimal immersions into spheres. The case $\Gamma=\{1\}$ corresponds to well known results of Nadirashvili \cite{Nadirashvili}, El Soufi--Ilias~\cite{ESI} and Fraser--Schoen \cite{FScontemp}; moreover, formally, the \emph{principle of symmetric criticality} suggests that critical points of a $\Gamma$-invariant functional restricted to a space of $\Gamma$-invariant objects are also critical points in an unconstrained sense, so it is unsurprising that critical points of $\bar{\lambda}_k$ and $\bar{\sigma}_k$ on $\Met_T(M)$ are indeed critical points on the space of all metrics $\Met(M)$. Since the functionals $\bar{\lambda}_k$ and $\bar{\sigma}_k$ are, however, only Lipschitz on $\Met(M)$, this requires some verification in practice.

\begin{theorem}
\label{thm:symm_crit_L}
Suppose that $g$ is a $T$-invariant metric, critical for $\bar\lambda_k$ in the class of $T$-invariant metrics on $M$. Then there exists an orthogonal representation $\rho\colon \Gamma\to O(n+1)$ and an equivariant (branched) minimal immersion $\Phi\colon (M,g)\to \Sph^{n}$ to the unit sphere such that the components of $\Phi$ are $\lambda_k(M,g)$-eigenfunctions and $\Phi^*g_{\Sph^n} = \alpha g$ for some $0<\alpha\in\R$.
\end{theorem}
\begin{proof}
The proof follows closely the corresponding arguments for the unconstrained optimization problem originating from~\cite{Nadirashvili}, see e.g.~\cite{ESI,KaMe, FScontemp}. We follow the exposition in~\cite{FScontemp} and only emphasize the features unique to the constrained optimization.

For any symmetric bilinear form $h$ on $M$ we define a quadratic form $Q_h$ on smooth functions as 
\[
Q_h(u) = - \int_M\langle du\otimes du - \frac{1}{2}|\nabla u|_g^2g + \frac{\lambda_k}{2}u^2g, h\rangle\,dv_g.
\]
Let $L^2(S^2(M))^T$ denote the $L^2$-space of $T$-invariant symmetric bilinear forms on $M$, and let $E_k(g)$ denote the $\lambda_k(M,g)$-eigenspace. The following lemma is proved in the same way as~\cite[Lemma 2.3]{FScontemp}. 
\begin{lemma}
For any $\omega\in L^2(S^2(M))^T$ such that $\int\langle\omega, g\rangle_g\,dv_g = 0$ there exists $u\in E_k(g)$ such that $Q_\omega(u) = 0$.
\end{lemma}

The remainder of the proof uses a Hahn-Banach argument similar to the one in~\cite{FScontemp}, only in the space $L^2(S^2(M))^T$ instead of $L^2(S^2(M))$. Scale the metric $g$ so that $\lambda_k(M,g) = 2$ and let $K$ be the convex hull of 
\[
\left\{du\otimes du - \frac{1}{2}|\nabla u|_g^2g + u^2g,\,\,\,u\in E_k(g)\right\}
\]
in $L^2(S^2(M))$. Define $K^T = \pi^T(K)$, where we let $\pi^T\colon L^2(S^2(M))\to L^2(S^2(M))^T$ be an orthogonal projection. Applying the Hahn-Banach separation theorem in $L^2(S^2(M))^T$ to $K^T$ and $g$, the same computations as in~\cite[p. 7]{FScontemp} imply that there exist eigenfunctions $\wt u_1,\ldots, \wt u_k$ such that 
\[
\pi^T\left(\sum_{i=1}^kd\wt u_i\otimes d\wt u_i - \frac{1}{2}|\nabla \wt u_i|_g^2g + \wt u_i^2g\right) = g.
\]
At the same time, observe that 
\[
\pi^T(h) = \frac{1}{|\Gamma|}\sum_{\gamma\in\Gamma}\gamma^*h.
\]
Denote $\wt\Phi = (\wt u_1,\ldots,\wt u_k)\colon (M,g)\to \R^k$ and consider the map $\Phi = (u_1,\ldots, u_{|\Gamma|k})$ given by
 \begin{equation}
 \label{eq:make_it_equiv}
 \Phi(x) = \frac{1}{\sqrt{|\Gamma|}}\left(\wt{\Phi}(\gamma_1\cdot x),\ldots,\wt{\Phi}(\gamma_{|\Gamma|}\cdot x)\right),
 \end{equation}
 where $\gamma_j$ is any enumeration of elements of $\Gamma$. Then $u_i\in E_k(g)$ and one has  
 \[
 \sum_{i=1}^{|\Gamma|k}d u_i\otimes d u_i - \frac{1}{2}|\nabla u_i|_g^2g +  u_i^2g = g,
 \]
 which as in~\cite{FScontemp} implies that $\Phi$ is a branched minimal immersion to a unit sphere. It remains to observe that $\Phi$ defined by~\eqref{eq:make_it_equiv} is equivariant by construction.
\end{proof}

In the same way, one can prove the following conformally constrained version of the theorem. It is not used in the present paper, so we omit the proof; the case $\bar\lambda_1$-conformally maximal metrics also follows directly from the results of Section 3 below.

\begin{theorem}
Suppose that $g$ is a $T$-invariant metric, conformally critical for $\bar\lambda_k$ in the class of $T$-invariant metrics on $M$. Then there exists an orthogonal representation $\rho\colon \Gamma\to O(n+1)$ and an equivariant harmonic map $\Phi\colon (M,g)\to \Sph^{n}$ to the unit sphere such that the components of $u$ are $\lambda_k(M,g)$-eigenfunctions and $|d\Phi|_g^2 = \alpha$ for some $0<\alpha\in\R$, where $h\in[g]$.
\end{theorem}

The following are Steklov counterparts of the previous theorems. 

\begin{theorem}
\label{thm:symm_crit_S}
Suppose that $g$ is a $T$-invariant metric, critical for $\bar\sigma_k$ in the class of $T$-invariant metrics on $N$. Then there exists an orthogonal representation $\rho\colon \Gamma\to O(n+1)$ and an equivariant free boundary (branched) minimal immersion $\Psi\colon (N,g)\to \B^{n+1}$ to the unit ball such that the components of $u$ are $\sigma_k(M,g)$-eigenfunctions and $\left(\Psi|_{\del N}\right)^*g_{\Sph^n} = \alpha g|_{\del N}$ for some $0<\alpha\in\R$.
\end{theorem}
\begin{proof}
Similarly to the proof of Theorem~\ref{thm:symm_crit_L}, a modification of the arguments in~\cite{FScontemp} yields the proof. This time for a pair $(h, h_\del)$, where $h$ ($h_\del$) is a bilinear symmetric form on $N$ (on $\del N$ respectively) one sets
\[
Q_{(h,h_\del)}(u) = - \int_N\langle du\otimes du - \frac{1}{2}|\nabla u|_g^2g, h\rangle\,dv_g - \frac{\sigma_k}{2}\int_{\del N} \langle u^2g|_{\del N}, h_\del\rangle\,ds_g.
\]
Denote by $\mH_T$ the $L^2$ space of $T$-invariant pairs $(h,h_\del)$ and by $E_k(g)$ the $\sigma_k(N,g)$-eigenspace. Then we have the following analogue of~\cite[Lemma 2.6]{FScontemp}.
\begin{lemma}
For any $(h,h_\del)\in\mH^T$  such that $\int \langle g|_{\del N}, h_\del\rangle = 0$ there exists $u\in E_k(g)$ such that $Q_{(h,h_\del)}(u) = 0$.
\end{lemma}
Scale the metric $g$ so that $\sigma_k(M,g) = 2$ and let $K^T$ be the orthogonal projection of the convex hull of 
\[
\left\{\left(du\otimes du - \frac{1}{2}|\nabla u|_g^2g, u^2g|_{\del N}\right),\,\,\,u\in E_k(g)\right\}
\]
to $\mH^T$. Applying Hahn-Banach separation theorem in $\mH^T$ to $K^T$ and $(0,g|_{\del M})$, the same computations as in~\cite[p. 10]{FScontemp} imply that there exist eigenfunctions $\wt u_1,\ldots, \wt u_k$ such that 
\begin{equation*}
\pi^T\left(\sum_{i=1}^kd\wt u_i\otimes d\wt u_i - \frac{1}{2}|\nabla \wt u_i|_g^2g, \sum_{i=1}^k \wt u_i^2g|_{\del N}\right) = (0, g|_{\del N}).
\end{equation*}
Constructing $\Psi$ from $\wt\Psi$ using formula~\eqref{eq:make_it_equiv} completes the proof.\qedhere

\end{proof}

\begin{theorem}
Suppose that $g$ is a $T$-invariant metric, conformally critical for $\bar\sigma_k$ in the class of $T$-invariant metrics on $N$. Then there exists an orthogonal representation $\rho\colon \Gamma\to O(n+1)$ and a free boundary harmonic map $\Psi\colon (N,g)\to \B^{n+1}$ to the unit ball such that the components of $u$ are $\sigma_k(M,g)$-eigenfunctions and 
$|\del_\nu\Psi| = \alpha$ for some $0<\alpha\in\R$, where $\nu$ is an outer unit normal to $\del N$ in the metric $g$.
\end{theorem}

%===================================================
%===================================================
%===================================================
%===================================================
%===================================================
%===================================================

%===================================================
%===================================================
%===================================================
%===================================================
%===================================================
%===================================================

\section{Min-max characterization of $\Lambda_1^T(M,[g])$}\label{lap.mm}

Let $(M,g)$ be a closed surface, and let $\Gamma$ %\leq \Isom(M,g)$ 
be a finite subgroup of the isometry group. Denote by $A_{\Gamma}\cong \mathbb{R}^{|\Gamma|}$ the group algebra
$$A_{\Gamma}:=\bigoplus_{\gamma\in \Gamma}\mathbb{R}[e_\gamma],$$
equipped with the left regular representation defined by requesting that $\rho(\gamma)(e_\sigma)=e_{\gamma\sigma}$. Likewise, for each $m\in \mathbb{N}$, let $\Gamma$ act on $A_{\Gamma}^m\cong \mathbb{R}^{|\Gamma|m}$ via the left regular representation on each factor, and denote by $F_m\subset A_{\Gamma}^m$ the $m$-dimensional fixed point set, spanned by $m$ copies of $\sum\limits_{\gamma\in\Gamma}e_{\gamma}$. 
%$\bar{e}:=\frac{1}{\sqrt{|\Gamma|}}\sum\limits_{\gamma\in\Gamma}e_{\gamma}$.

For each $m$, let $X : = W^{1,2}(M, A^m_\Gamma)$ and define the Hilbert space
$$X_\Gamma: = W^{1,2}_{\Gamma}(M,A_{\Gamma}^m)
:=\{u\in W^{1,2}(M,A_{\Gamma}^m)\mid u\circ \gamma=\rho(\gamma)\cdot u,\,\text{ }\forall \gamma\in \Gamma\}$$
of $\Gamma$-equivariant $W^{1,2}$ maps to $A_{\Gamma}^m \cong \mathbb{R}^{|\Gamma|m}$.

We collect the analytic preliminaries we need in the following Proposition.
\begin{proposition}\label{gl.pre-mm} For each $\epsilon>0$, the Ginzburg-Landau energy
\[
E_{\epsilon}(u):=\int_M\frac{1}{2}|du|^2+\frac{(1-|u|^2)^2}{4\epsilon^2}\,dv_g
\]
defines a $C^2$ functional on the space $X_{\Gamma}$ satisfying the following. %=W^{1,2}_{\Gamma}(M,A_{\Gamma}^m)$ 
\begin{enumerate}[label=\emph{(\roman*)}]
\item $E_{\epsilon}'(u)(v)=\int_M \langle \Delta_g u-\epsilon^{-2}(1-|u|^2)u,v\rangle\,dv_g$.
\item $E_{\epsilon}''(u)(v,v)=\int_M |dv|^2+2\epsilon^{-2}\langle u,v\rangle^2-\epsilon^{-2}(1-|u|^2)|v|^2\,dv_g.$
\item $E_\epsilon$ satisfies the Palais-Smale condition on $X_\Gamma$.
\item For any critical point $u \in X_\Gamma$ of $E_\epsilon$, the map $v\mapsto E''_\epsilon(u) (v, \cdot)$ defines a Fredholm operator $X_\Gamma \rightarrow X_\Gamma^*$.  
\item $u\in X_\Gamma$ is a critical point for $E_\epsilon$ if and only if 
\begin{align}
\label{gl.e}
\Delta_g u + \epsilon^{-2} (1-|u|^2) u = 0.
\end{align}
\item A solution $u$ of \eqref{gl.e} satisfies $|u| \leq 1$, with $|u|=1$ iff $u$ is constant. 
\end{enumerate}
%satisfying a Palais--Smale condition, and such that $v\mapsto E_{\epsilon}''(u)(v,\cdot)$ defines a Fredholm operator $X_{\Gamma}\to X_{\Gamma}^*$ for any critical point $u\in X_{\Gamma}$ of $E_{\epsilon}$. Moreover, $u$ is a critical point for $E_{\epsilon}$ on $X_{\Gamma}$ if and only if
%\begin{equation}
%\label{gl.e}
%\Delta_g u-\frac{(1-|u|^2)}{\epsilon^2}u=0.
%\end{equation}
\end{proposition}
\begin{proof}
As discussed in \cite[Proposition 3.1]{KSminmax}, the energies $E_\epsilon$ define $C^2$ functionals on the larger space $X$, satisfying (i)-(v), but with $X_\Gamma$ replaced with $X$ in (iii)-(v).  It follows immediately that $E_\epsilon$ restricts to a $C^2$ functional on the closed subspace $X_\Gamma$ of $\Gamma$-equivariant maps.  Moreover, if $u \in X_\Gamma$, 
$$
(\Delta_g u)\circ \gamma=\Delta_g (u\circ \gamma)=\Delta_g(\rho(\gamma)\cdot u)=\rho(\gamma)\cdot \Delta_g u
$$
in the weak sense and
$$
((1-|u|^2)u)  \circ \gamma = (1-|u\circ \gamma|^2)u \circ \gamma = (1-|u|^2)\rho(\gamma)\cdot u 
= \rho(\gamma) \cdot (1-|u|^2)u
%\frac{(1-|u|^2)u}{\epsilon^2}\circ \gamma=\frac{(1-|u\circ \gamma|^2)u\circ \gamma}{\epsilon^2}=\frac{(1-|u|^2)\rho(\gamma)\cdot u}{\epsilon^2}=\rho(\gamma)\cdot\frac{(1-|u|^2)u}{\epsilon^2},
$$
so that 
$$
E_{\epsilon}'(u)(v)=0\text{ for }v\in X_{\Gamma}^{\perp}.
$$
In particular, it follows that a sequence $u_j\in X_{\Gamma}$ is Palais-Smale for the restricted functional $E_{\epsilon}|_{X_{\Gamma}}$ if and only if it is Palais-Smale for $E_{\epsilon}$ in $X$, and $u\in X_{\Gamma}$ is a critical point for $E_{\epsilon}|_{X_{\Gamma}}$ if and only if it solves \eqref{gl.e} in the weak sense.  Moreover, if $u\in X_{\Gamma}$ is a critical point for $E_{\epsilon}$, it is likewise easy to check that the restricted Hessian $(E_{\epsilon}|_{X_{\Gamma}})''(u)=E_{\epsilon}''(u)|_{X_{\Gamma}}$ defines a Fredholm linear operator $X_{\Gamma}\to X_{\Gamma}^*$.

Finally, for (vi), it is a straightforward exercise to check that a map $u \in X_\Gamma$ solving \eqref{gl.e} weakly must in fact be smooth.  Note then that 
\[
\Delta_g(1-|u|^2) = 2|du|_g^2 - 2\frac{1-|u|^2}{\epsilon^2}|u|^2\geq -2\frac{1-|u|^2}{\epsilon^2}|u|^2,
\]
so if a minimum of $1-|u|^2$ is negative, then $\Delta_g(1-|u|^2)$ is positive, which is a contradiction. This implies $|u|\leq 1$ and, furthermore,  $|u| = 1$ iff $u$ is a constant map.
\end{proof}

 Given such a solution of \eqref{gl.e}, %equivariant solution $u\in X_{\Gamma}=W^{1,2}_{\Gamma}(M,A_{\Gamma}^m)$ of \eqref{gl.e},
 we define the $\Gamma$-equivariant Morse index $\ind_{\Gamma,E_{\epsilon}}(u)$ to be the Morse index of the restricted functional $E_{\epsilon}|_{X_{\Gamma}}$ at the critical point $u$:
$$
\ind_{\Gamma,E_{\epsilon}}(u):=\max\left\{\dim V\mid V\subset X_{\Gamma}\text{ such that }E_{\epsilon}''(u)|_V\text{ is negative-definite}\right\}.
$$

With Proposition \ref{gl.pre-mm} in place, we are now in a position to produce $\Gamma$-equivariant solutions of \eqref{gl.e} with bounds on $\ind_{\Gamma,E_{\epsilon}}$ via min-max methods. Specifically, we will be interested in the following construction.

Let $\B^m\subset F_m$ denote the closed unit ball in the fixed point set $F_m\subset A_{\Gamma}^m$, and denote by $\mathcal{B}_m(\Gamma)\subset C^0(\B^m,W^{1,2}_{\Gamma}(M,A_{\Gamma}^m))$ the collection of families
$$
\B^m\ni y\mapsto u_y\in W^{1,2}_{\Gamma}(M, A_{\Gamma}^m)
$$
satisfying $u_y\equiv y$ for $y\in \partial \B^m$. %Then, set
%$$
%\mathcal{E}_{m,\epsilon}^{\Gamma}(M,g):=\inf_{(u_y)\in \mathcal{B}_m}\max_{y\in \B^m}E_{\epsilon}(u_y),
%$$

\begin{lemma}
\label{ex.claim.1}\label{mm.ex}
For each $\epsilon>0$, the min-max energy
\begin{align}
\mathcal{E}_{m,\epsilon}^{\Gamma}(M,g):=\inf_{(u_y)\in \mathcal{B}_m}\max_{y\in \B^m}E_{\epsilon}
\end{align} 
is realized as the energy $E_{\epsilon}(u_{\epsilon})$ of a smooth $\Gamma$-equivariant critical point $u_{\epsilon}\in W^{1,2}_{\Gamma}(M, A_{\Gamma}^m)$ for $E_{\epsilon}$, with $\Gamma$-equivariant Morse index
$$
\ind_{\Gamma,E_{\epsilon}}(u_{\epsilon})\leq m.
$$
%$\mathcal{E}^{\Gamma}_{m,\epsilon}$(M,g) is realized as the energy $E_{\epsilon}(u_{\epsilon})$ of a smooth $\Gamma$-equivariant critical point $u_{\epsilon}\in W^{1,2}_{\Gamma}(M, A_{\Gamma}^m)$ for $E_{\epsilon}$, with $\Gamma$-equivariant Morse index
\end{lemma}
\begin{proof} 
This follows by combining Proposition \ref{gl.pre-mm} with results in \cite[Chapter 10]{Ghou93}.
\end{proof}

We turn now to the behavior of these critical points $u_{\epsilon}$ as $\epsilon\to 0$, beginning with the problem of finding upper and--more importantly--sharp lower bounds for the limiting energy
$$
\mathcal{E}_m^{\Gamma}(M,g):=\sup_{\epsilon>0}\mathcal{E}_{m,\epsilon}^{\Gamma}=\lim_{\epsilon\to 0}\mathcal{E}_{m,\epsilon}^{\Gamma}.
$$
A priori, it could be that $\mathcal{E}_m^{\Gamma}=\infty$, but we show next that this is not the case for large $m$, by exhibiting families $y\mapsto u_y^{\epsilon}$ in $\mathcal{B}_m$ for which $E_{\epsilon}(u^{\epsilon}_y)\leq C$ for some $C<\infty$ independent of $\epsilon$.

\begin{lemma}\label{mm.ubd}
For $m$ sufficiently large, $\mathcal{E}_m^{\Gamma}(M,g)<\infty$.
\end{lemma}
\begin{proof}
To begin, note that by \cite[Thm 4.1, p. 315]{Bredon},
 there exists $N\in \mathbb{N}$ and a representation $\rho_0: \Gamma\to O(N-1)$ such that $M$ admits a $\Gamma$-equivariant embedding
$$
u_0\colon M\to \mathbb{R}^{N-1}.
$$
Observe that stereographic projection sends a reflection across a hyperplane in $\R^{N-1}$ to a reflection across a totally geodesic subsphere in $\Sph^{N-1}$. Since $O(N-1)$ is generated by reflections, replacing $u_0$ by its composition 
 with stereographic projection, we may assume (up to changing $\rho_0$) that
$$
u_0(M)\subset \Sph^{N-1}.
$$
In particular, since $A_{\Gamma}$ contains every irreducible representation of $\Gamma$, we see that this representation $\rho_0\colon\Gamma\to O(N)$ is contained in $m$ copies $A_{\Gamma}^m$ of the regular representation for $m\geq m(\rho_0)$ sufficiently large, so we have a $\Gamma$-equivariant embedding
$$
u_0\colon M\to A_{\Gamma}^m\text{ with }u_0(M)\subset \Sph^{m|\Gamma|-1}.
$$
Next, for each vector $a$ of norm $\leq 1$ in the fixed point set $a\in F_m\cap \B^{m|\Gamma|}(0)$--which can be identified with $\B^m\subset \mathbb{R}^m$--consider the conformal dilations $G_a\in \Conf(\Sph^{m|\Gamma|-1})$ given by
$$
G_a(x)=\frac{(1-|a|^2)}{|x+a|^2}(x+a)+a,
$$
and observe that for any $\gamma\in \Gamma$, we have
$$
\rho(\gamma)\cdot G_a(x)=\frac{(1-|a|^2)}{|\rho(\gamma)x+a|^2}(\rho(\gamma)x+a)+a=G_a(\rho(\gamma)x),
$$
since $\rho(\gamma)(x+a)=\rho(\gamma)(x)+a$. Thus, for $a\in F_m$, the conformal dilations $G_a$ %$G_a\colon \Sph^{m|\Gamma|-1}\to \Sph^{m|\Gamma|-1}$
 are $\Gamma$-equivariant, so the composition
$$
u_a:=G_a\circ u_0
$$
gives a conformal $\Gamma$-equivariant map $u_a\colon M\to \Sph^{m|\Gamma|-1}$ for $a\in F_m$ with $|a|<1$, while $u_a\equiv a$ for $a\in F_m$ with $|a|=1$.

Similar to \cite[Proposition 3.3]{KSminmax}, we argue next that the family 
$$
a\ni \B^m\mapsto u_a\in W^{1,2}(M,\Sph^{m|\Gamma|-1})
$$
defined via conformal dilations as above, though not continuous in a $W^{1,2}$ sense, is $L^p$-continuous and satisfies a uniform energy bound, so that a suitable mollification gives rise to the desired families in $\mathcal{B}_m$. 
To this end, first note that the family of conformal dilations $a\mapsto G_a$ is continuous in the space of $C^{\infty}$ maps for $|a|<1$, and as $a\to a_0\in F_m$ with $|a_0|=1$, we have $G_a\to a_0$ in $C^{\infty}_{loc}(\mathbb{S}^{m|\Gamma|-1}\setminus\{-a_0\},\mathbb{S}^{m|\Gamma|-1})$. Thus, since $u_0\colon M\to A_{\Gamma}^m$ is a fixed smooth embedding, we see that the family $a\mapsto u_a$ is continuous in a $C^{\infty}$ sense for $|a|<1$, and as $a\to a_0$ with $|a_0|=1$, we have $u_a\to u_{a_0}$ in $C^{\infty}_{loc}(M\setminus u_0^{-1}\{-a_0\})$, which together with the fact that $|u_a|=1$ pointwise implies that $u_a\to u_{a_0}\equiv a_0$ in $L^p(M,A_{\Gamma}^m)$ for every $p\in [1,\infty)$. To see that
$$
\sup_{a\in \B^m}E(u_a)<\infty,
$$
simply note as in \cite{LiYau} that since $u_a$ is conformal for $|a|<1$, we have
$$
E(u_a)=\Area(u_a(M))\leq \mathcal{W}(u_a(M)),
$$
where $\mathcal{W}$ is the spherical Willmore energy $\mathcal{W}(\Sigma)=\int_{\Sigma}(1+\frac{1}{4}|H_{\Sigma}|^2)$. In particular, since $\mathcal{W}$ is invariant under conformal automorphisms of $\Sph^n$,
$$\mathcal{W}(u_a(M))=\mathcal{W}(G_a(u_0(M)))=\mathcal{W}(u_0(M)),$$
so that
\begin{equation}\label{mm.upper}
\sup_{a\in \B^m}E(u_a)\leq \mathcal{W}(u_0(M))<\infty.
\end{equation}

Now, since the family $\B^m\ni a\mapsto u_a\in W^{1,2}_{\Gamma}(M,A_{\Gamma}^m)$ defined above is continuous in an $L^p$ sense with $u_a\equiv a$ for $a\in F_m\cap \Sph^{m|\Gamma|-1}$, as in \cite[Proposition 3.3]{KSminmax}, mollifying by the heat kernel $K_t(x,y)$ on $M$ for some short time $t>0$ gives a family of maps
$$
u_a^t(x):=\int_M u_a(y)K_t(x,y)dy
$$
such that $\B^m\ni a\mapsto u_a\in C^{\infty}(M,A_{\Gamma}^m)$ is continuous in a $C^{\infty}$ sense, still with the property that $u_a^t\equiv a$ when $|a|=1$. Moreover, since the heat kernel $K_t$ satisfies the isometry invariance $K_t(\gamma x,\gamma y)=K_t(x,y)$ for any $\gamma\in \Gamma$, we see that
\begin{eqnarray*}
u_a^t(\gamma x)&=&\int_Mu_a(y)K_t(\gamma x,y)dy\\
&=&\int_Mu_a(y)K_t(x,\gamma^{-1}y)dy\\
&=&\int_M(u_a\circ \gamma)(\gamma^{-1}y)K_t(x,\gamma^{-1}y)dy\\
&=&\int_M(u_a\circ \gamma)(y')K_t(x,y')dy'\\
&=&\rho(\gamma)u_a^t(x),
\end{eqnarray*}
by the $\Gamma$-equivariance of $u_a$. Hence, $\B^m\ni a\mapsto u_a^t$ defines a continuous family in $W^{1,2}_{\Gamma}(M,A_{\Gamma}^m)$, and in particular a family in $\mathcal{B}_m$. Thus, we see that
$$
\mathcal{E}_{m,\epsilon}^{\Gamma}(M,g)\leq \inf_{t>0}\max_{a\in \B^m}E_{\epsilon}(u_a^t),
$$
and to complete the proof, it suffices to show that
$$
\inf_{t>0}\max_{a\in \B^m}E_{\epsilon}(u_a^t)\leq C<\infty
$$
for some constant $C$ independent of $\epsilon$. 

In fact, for every $\epsilon>0$, we can show that
$$
\inf_{t>0}\max_{a\in \B^m}E_{\epsilon}(u_a^t)\leq \sup_{a\in \B^m}E(u_a),
$$
which together with \eqref{mm.upper} gives the desired bound. Note that by the energy-decreasing property of heat flow, we have $E(u_a^t)\leq E(u_a)$ for every $a\in \B^m$ automatically, so to prove this inequality, it suffices to show that
\begin{equation}\label{fam.w.van}
\liminf_{t\to 0}\max_{a\in \B^m}\int_M(1-|u_a^t|^2)^2=0.
\end{equation}
To prove \eqref{fam.w.van}, simply note that standard identities for the heat equation together with \eqref{mm.upper} give a bound of the form
$$
\|u_a^t-u_a\|_{L^2(M)}^2\leq t\cdot (E(u_a)-E(u_a^t))\leq Ct,
$$
and since $|u_a|\equiv 1$, it follows that
$$
\int_M(|u_a^t|-1)^2\,dv_g\leq\int_M|u_a^t-u_a|^2\,dv_g\leq Ct.
$$
But we also have the pointwise upper bound $|u_a^t|\leq 1$, so that $(1-|u_a^t|^2)=(1-|u_a^t|)(1+|u_a^t|)\leq 2(1-|u_a^t|)$, and therefore the preceding inequality gives 
$$
\int_M(1-|u_a^t|^2)^2\,dv_g\leq 4Ct
$$
for $C<\infty$ independent of $a$, from which \eqref{fam.w.van} follows immediately.
\end{proof}

Next, we record the crucial Hersch-type lower bound for $\mathcal{E}_m^{\Gamma}$, providing an equivariant analog of \cite[Proposition 3.4]{KSminmax}.

\begin{proposition}\label{mm.lower} 
$\mathcal{E}_{m,\epsilon}^{\Gamma}(M,g)\geq \frac{1-\epsilon}{2+\epsilon}\bar{\lambda}_1(M,g)$ for each $\epsilon>0$; consequently 
$$\mathcal{E}_m^{\Gamma}(M,g)\geq \bar{\lambda}_1(M,g).$$
\end{proposition}

\begin{proof}

Fix an arbitrary family $(u_y)\in \mathcal{B}_m$. Defining the orthogonal projection map $P_{F_m}\colon A_{\Gamma}^m\to F_m$ onto the fixed point subspace $F_m$, we claim first that $P_{F_m}(\int_M u_ydv_g)=0$ for some $y\in \B^m$. Indeed, if not, then by definition of $\mathcal{B}_m$, the assignment
$$
f(y):=\frac{P_{F_m}(\int_M u_y\,dv_g)}{|P_{F_m}(\int_M u_y\,dv_g)|}
$$
defines a continuous extension $f\colon \B^m\to \Sph^{m-1}$ of the identity $\Sph^{m-1}\to \Sph^{m-1}$, which cannot exist. On the other hand, since each map $u_y$ is $\Gamma$-equivariant, we see that the integral $\int_Mu_y dv_g$ must lie in $F_m$, which together with the preceding observation implies the existence of some $y\in \B^m$ such that
$$
\int_M u_y\,dv_g=0.
$$

Now we can argue exactly as in \cite[Proposition 3.4]{KSminmax}; since the component functions of $u_y$ are orthogonal to the constants in $L^2_g(M)$, we must have
$$
2E_{\epsilon}(u_y)\geq \int_M |du_y|^2_g\,dv_g\geq \lambda_1(M,g)\int_M|u_y|^2\,dv_g,
$$
and by definition of $E_{\epsilon}$ it's easy to see that
$$
\int_M |u_y|^2=\Area(M,g)-\int_M(1-|u_y|^2)\geq \Area(M)-2\epsilon\sqrt{\Area(M)}E_{\epsilon}(u)^{1/2},
$$
which gives in particular $\int_M |u_y|^2\,dv_g\geq (1-\epsilon)\Area(M,g)-\epsilon E_{\epsilon}(u)$. Combining this with the lower bound for $2E_{\epsilon}(u_y)$ above, we see that
$$
(2+\epsilon)E_{\epsilon}(u_y)\geq (1-\epsilon)\lambda_1(M,g)\Area(M,g),
$$
and therefore
$$
\max_{y\in \B^m}E_{\epsilon}(u_y)\geq \frac{1-\epsilon}{2+\epsilon}\bar{\lambda}_1(M,g).
$$
Since the family $(u_y)$ in $\mathcal{B}_m$ was arbitrary, taking the infimum over all families in $\mathcal{B}_m$ gives the desired estimates.
\end{proof}

\begin{remark}\label{gen.hersch}
The preceding arguments can also be used to show that, for every $\Gamma$-invariant bounded linear operator $T\in (W^{1,2}(M))^*$ and any family $(u_y)\in \mathcal{B}_m$, there exists $y\in \B^m$ such that all components of $u_y$ lie in $\ker(T)$. 
\end{remark}

Moreover, the same simple argument from \cite[Section 3.1]{KSminmax} showing conformal invariance of $\mathcal{E}_m$ in the case of trivial $\Gamma$ continues to hold, so though $\mathcal{E}_{m,\epsilon}^{\Gamma}(M,g)$ depends on the choice of $\Gamma$-invariant metric $g$ for every $\epsilon>0$, the $\epsilon\to 0$ limit $\mathcal{E}_m^{\Gamma}(M,g)$ is invariant under $\Gamma$-invariant conformal changes, which we will emphasize henceforth by writing
$$
\mathcal{E}_m^{\Gamma}(M,[g]):=\mathcal{E}_m^{\Gamma}(M,g),
$$
where $[g]$ denotes all metrics of the form $\tilde{g}=f g$ for $f$ a positive $\Gamma$-invariant function. In particular, 
%writing
%$$\Lambda_1^{\Gamma}(M,[g]):=\sup\{\bar{\lambda}_1(M,g)\mid g\in [g]\},$$
it follows from the preceding proposition that
\begin{equation}\label{mm.lower.1}
2\mathcal{E}_m^{\Gamma}(M,[g])\geq \Lambda_1^{\Gamma}(M,[g])
\end{equation}
for every $m$.

\subsection{Asymptotics for the min-max maps as $\epsilon\to 0$ and $m\to\infty$}
%\hspace{40mm}

Next, building on ideas from \cite{KSminmax} and \cite{KS22}, we use the Morse index bound $\ind_{\Gamma,E_{\epsilon}}(u_{\epsilon})\leq m$ for the critical points produced in Lemma~\ref{ex.claim.1} to obtain an asymptotic upper bound
$$
\lim_{m\to\infty}2\mathcal{E}_m^{\Gamma}(M,[g])\leq \Lambda_1^{\Gamma}(M,[g])
$$
matching~\eqref{mm.lower.1} as $m\to\infty$. This asymptotic version of the desired equality then allows us to streamline the bubbling analysis of \cite{KSminmax}, to conclude that the equality $\mathcal{E}_m^{\Gamma}=\Lambda_1^{\Gamma}$ holds for $m$ sufficiently large before passing to the limit.

\begin{proposition}\label{index.eigen.bd}
Let $u\in W^{1,2}_{\Gamma}(M,A_{\Gamma}^m)$ be a critical point for $E_{\epsilon}$ with 
$$
\ind_{\Gamma,E_{\epsilon}}(u_{\epsilon})\leq m.
$$
Then $g_u:=\frac{(1-|u|^2)}{\epsilon^2}g$ defines a $\Gamma$-invariant metric in $[g]$, and there is a universal constant $C<\infty$ such that for $m>C$, 
$$\lambda_1(g_u): = \lambda_1(M, g_u)  \geq 1-\frac{C}{m}.$$
\end{proposition}

\begin{proof}
Denote by $W_m\subset A_{\Gamma}^m$ the subspace 
$$
W_m:=\mathrm{Span}\{(e_I,0,\ldots,0),\ldots, (0,\ldots,0,e_I)\}
$$
spanned by $m$ copies of $e_I$ (here $I$ is the identity element of $\Gamma$), so that for every unit vector $w\in W_m$, $\{w_{\gamma}=\gamma\cdot w\mid \gamma\in \Gamma\}$ gives an orthonormal basis for an $m$-dimensional subspace on which $\Gamma$ acts by the regular representation. Moreover, one has the orthogonal direct sum decomposition
\[
A^m_\Gamma = \bigoplus_{\gamma\in\Gamma}\gamma\cdot W_m.
\]

Given $\phi\in C^{\infty}(M)$ and a vector $w\in W_m$, define 
$$v_{\phi,w}:=\sum_{\gamma\in \Gamma}(\phi\circ \gamma^{-1})(w_{\gamma}-\langle u,w_{\gamma}\rangle u).$$
Since $u$ is $\Gamma$-equivariant, it's easy to see that $v_{\phi,w}$ is $\Gamma$-equivariant as well. 
In particular, the space of all such $v_{\phi,w}$ for $w\in W_m$ gives a subspace
$$V_{\phi}=\{v_{\phi,w}\mid w\in W_m\}\subset W^{1,2}_{\Gamma}(M,A_{\Gamma}^m)$$
on which we can test $E_{\epsilon}''(u)$ to get lower bounds on the $\Gamma$-equivariant Morse index. For dimension counts, it will be useful to note that
$$v_{\phi,w}=0\text{ if and only if either }\phi=0\text{ or }w=0.$$
To see this, simply note that if $v_{\phi,w}=0$, then 
$$0=\langle v_{\phi,w},u\rangle=(1-|u|^2)\sum_{\gamma\in \Gamma}(\phi\circ \gamma^{-1})\langle u,w_{\gamma}\rangle,$$
and since $u$ is nonconstant, we have by Proposition \ref{gl.pre-mm}(vi) that $|u|<1$, and we deduce that $\sum_{\gamma}(\phi\circ \gamma^{-1})\langle u,w_{\gamma}\rangle=0$, and therefore
$$0=v_{\phi,w}=\sum_{\gamma\in \Gamma}(\phi\circ \gamma^{-1})w_{\gamma}.$$
But if $0\neq w\in W_m$, then the vectors $\{\frac{w_{\gamma}}{|w|}\mid \gamma\in \Gamma\}$ form an orthonormal basis up to scaling, so that the preceding identity would force $\phi\equiv 0$.

The following proposition establishes useful pointwise estimates for the fields $v_{\phi,w}$ and their derivatives. In what follows, for each unit vector $w\in W_m$, denote by $u^w$ the pointwise projection of $u$ onto the subspace generated by the orbit $\Gamma\cdot w$. 

\begin{proposition}
For any $\phi\in C^{\infty}(M)$, vector $w\in W_m$ and any map $u\in W_\Gamma^{1,2}(M,A^m_\Gamma)$ satisfying $|u|\leq 1$ one has the following pointwise estimates
\label{prop:v_lower}
\begin{equation}\label{v.lower}
|v_{\phi,w}|^2\geq |w|^2\sum_{\gamma\in \Gamma}(\phi\circ \gamma)^2(1-2|u^w|^2),
\end{equation}
\[
\frac{|dv_{\phi,w}|^2}{|w|^2}\leq (1+C|u^w|^2)\sum_{\gamma}|d(\phi\circ \gamma)|^2+C\sum_{\gamma}(\phi\circ \gamma)^2(|du^w|^2+|u^w|^2|du|^2),
\]
\[
\langle v_{\phi,w},u\rangle^2\leq (1-|u|^2)^2|w|^2|u^w|^2\sum_{\gamma}(\phi \circ \gamma)^2,
\]
where $C$ is a universal constant.

If, furthermore, $\phi_0$ is a constant function on $M$, then
\begin{equation}\label{pt.cross.2}
\langle v_{\phi,w},u\rangle\langle v_{\phi_0,z},u\rangle\leq (1-|u|^2)^2|u^w||u^z||w||z||\phi_0|\bigg(\sum_{\gamma}(\phi \circ \gamma)^2\bigg)^{1/2}.
\end{equation}
\begin{equation}\label{pt.cross}
|\langle v_{\phi_0,z},v_{\phi,w}\rangle-\phi_0\langle w,z\rangle \sum_{\gamma\in\Gamma}(\phi\circ \gamma)|\leq C|\phi_0|\bigg(\sum_{\gamma}(\phi \circ \gamma)^2\bigg)^{1/2}|u^w||u^z||w||z|,
\end{equation}
\begin{align}
\label{pt.cross.3}
\begin{split}
\frac{1}{|w||z|}&\left|\langle dv_{\phi_0,z},dv_{\phi,w}\rangle - \phi_0\sum_{\gamma,\tau}
\langle u,z_\tau\rangle\langle d(\phi\circ\gamma^{-1}),\langle du,w_\gamma\rangle\rangle\right| 
\leq  \\
C|\phi_0|&\bigg(\sum_{\gamma}|d\phi\circ \gamma|^2\bigg)^{\frac{1}{2}}\left(|u|^2|u^w||du^z| + |u^w||du^z| + |u||du||u^w||u^z|\right) + \\
C|\phi_0|&\bigg(\sum_{\gamma}(\phi\circ \gamma)^2\bigg)^{\frac{1}{2}}\left(|du^w||du^z||u|^2+|du||u|(|u^w||du^z| + |u^z||du^w|) + |du|^2|u^w||u^z|\right),
\end{split}
\end{align}
\end{proposition}

\begin{proof} Throughout the proof we repeatedly use the following observation. 
%If $\{a_\gamma\}_{\gamma\in\Gamma}$ is a sequence of numbers labeled by the elements of $\Gamma$,
 %then various $\ell^p$-norms for such sequences are equivalent. 
% In particular, we repeatedly use the following estimates
Applying the Cauchy-Schwarz inequality to the space formed by sequences $\{a_\gamma\}_{\gamma\in\Gamma}$ of numbers labeled by the elements of $\Gamma$, we obtain
\[
\bigg|\sum_{\gamma\in\Gamma} a_\gamma b_\gamma\bigg|\leq \|a\|_{\ell^2}\|b\|_{\ell^2} = \bigg(\sum_{\gamma\in\Gamma}a_\gamma^2\bigg)^\frac{1}{2}\bigg(\sum_{\gamma\in\Gamma}b_\gamma^2\bigg)^\frac{1}{2},
\]
\[
\bigg|\sum_{\gamma,\tau\in\Gamma}a_\gamma b_\gamma c_\tau d_\tau\bigg| = \bigg|\sum_{\gamma\in\Gamma}a_\gamma b_\gamma \bigg|\bigg|\sum_{\tau\in\Gamma} c_\tau d_\tau\bigg|\leq \|a\|_{\ell^2}\|b\|_{\ell^2}\|c\|_{\ell^2}\|d\|_{\ell^2}.
\]
For the first inequality~\eqref{v.lower} we see that
\[
|v_{\phi,w}|^2=\sum_{\gamma,\tau\in \Gamma}(\phi\circ \gamma^{-1})(\phi\circ \tau^{-1})(|w|^2\delta_{\gamma \tau}-(2-|u|^2)\langle u,w_{\gamma}\rangle \langle u,w_{\tau}\rangle).
\]
Using that $|u|^2\leq 1$ and $\sum_\gamma \langle u,w_\gamma\rangle^2 = |w|^2|u^w|^2$ one has 
\[
(2-|u|^2)\sum_{\gamma,\tau\in \Gamma}(\phi\circ \gamma^{-1})(\phi\circ \tau^{-1})\langle u,w_{\gamma}\rangle \langle u,w_{\tau}\rangle\leq 2\sum_{\gamma\in \Gamma}(\phi\circ \gamma)^2|w|^2|u^w|^2
\]
from which~\eqref{v.lower} follows. Similarly, direct computation shows that
\begin{equation*}
\begin{split}
|d v_{\phi,w}|^2&\leq (1+2|u^w|^2)|w|^2\sum_{\gamma\in \Gamma}|d(\phi\circ \gamma)|^2\\
&+2|w|^2\bigg(\sum_{\gamma\in \Gamma}(\phi \circ \gamma)^2\cdot \sum_{\gamma\in\Gamma}|d(\phi\circ \gamma)|^2\bigg)^{1/2}(2|u^w||du^w|+|u^w|^2|du|)\\
&+|w|^2\bigg(\sum_{\gamma\in \Gamma}(\phi \circ \gamma)^2\bigg)(|du^w|^2+2|u^w||du||du^w|+|u^w|^2|du|^2),
\end{split}
\end{equation*}
and a few applications of the Cauchy--Schwarz inequality give
$$
\frac{|dv_{\phi,w}|^2}{|w|^2}\leq (1+5|u^w|^2)\sum_{\gamma}|d(\phi\circ \gamma)|^2+3\sum_{\gamma}(\phi\circ \gamma)^2(|du^w|^2+|u^w|^2|du|^2)
$$
as required. Moreover, it is easy to see that
$$
\langle v_{\phi,w},u\rangle^2 =(1-|u|^2)^2\sum_{\gamma}(\phi \circ \gamma^{-1})^2\langle u, w_\gamma\rangle^2 \leq (1-|u|^2)^2|u^w|^2|w|^2\sum_{\gamma}(\phi \circ \gamma)^2.
$$
Similar computations give inequalities~\eqref{pt.cross.2},~\eqref{pt.cross} and~\eqref{pt.cross.3}.
%\begin{equation*}
%\left|\langle v_{\phi_0,z},v_{\phi,w}\rangle-\langle w,z\rangle \sum_{\gamma\in\Gamma}(\phi\circ \gamma)\right|\leq C_{\Gamma}|\phi_0|\left(\sum_{\gamma}(\phi \circ \gamma)^2\right)^{1/2}|u^w||u^z||w||z|,
%\end{equation*}
%\begin{equation*}
%\label{pt.cross.2}
%\langle v_{\phi,w},u\rangle\langle v_{\phi_0,z},u\rangle\leq (1-|u|^2)^2|u^w||u^z||\phi_0|\left(\sum_{\gamma}(\phi \circ \gamma)^2\right)^{1/2}.
%\end{equation*}
%\begin{eqnarray*}
%|\langle dv_{\phi_0,z},dv_{\phi,w}\rangle| &\leq & C_{\Gamma}|u^w||u^z||\phi_0|\left(\sum_{\gamma}|d\phi\circ \gamma|^2\right)^{1/2}\\
%&&+C_{\Gamma}|\phi_0|\left(\sum_{\gamma}(\phi\circ \gamma)^2\right)^{1/2}(|du^w||du^z|+|u^w|u^z|),
%\end{eqnarray*}
%and
\end{proof}

At the given critical point $u$ for $E_{\epsilon}$, recall from Proposition \ref{gl.pre-mm}(ii) that
$$
E_{\epsilon}''(u)(v,v)=\int_M |dv|^2+2\epsilon^{-2}\langle u,v\rangle^2-\epsilon^{-2}(1-|u|^2)|v|^2\,dv_g,
$$
so by Proposition \ref{prop:v_lower}, it follows that
\begin{equation*}
\begin{split}
&\frac{1}{|w|^2}E_{\epsilon}''(u)(v_{\phi,w},v_{\phi,w})\leq \sum_{\gamma}\int_M\left(1+C|u^w|^2\right)|d(\phi\circ \gamma)|^2\,dv_g\\
&+C\sum_{\gamma}\int_M(\phi \circ \gamma)^2(|du^w|^2+|u^w|^2|du|^2)\,dv_g\\
&+\sum_{\gamma}\int_M\left(2\epsilon^{-2}(1-|u|^2)^2|u^w|^2
-\frac{(1-|u|^2)}{\epsilon^2}(1-C|u^w|^2)\right)(\phi \circ \gamma)^2\,dv_g.
\end{split}
\end{equation*}
Moreover, since $\Gamma$ acts on $(M,g)$ by isometries and the quantities $|u|, |du|, |u^w|,$ and $|du^w|$ are invariant under precomposition with $\Gamma$ by the equivariance of $u$, we can combine the preceding estimates to see that
\begin{equation}
\label{EEpp}
\begin{split}
&\frac{1}{|w|^2}E_{\epsilon}''(u)(v_{\phi,w},v_{\phi,w})\leq \\
&|\Gamma|\int_M\left(|d\phi|_g^2-\frac{(1-|u|^2)}{\epsilon^2}\phi^2\right)\,dv_g+C|\Gamma|\int_M|u^w|^2|d\phi|_g^2\,dv_g\\
&+C|\Gamma|\int_M\left(|du^w|_g^2+|u^w|^2|du|_g^2+\epsilon^{-2}(1-|u|^2)|u^w|^2\right)\phi^2\,dv_g.
\end{split}
\end{equation}
 Now, define the map
$
J_u\colon W_m\times W^{1,2}(M)\to \mathbb{R}
$
by 
\begin{equation*}
\begin{split}
J_u(w,\phi):&=|w|^2\int_M|d\phi|^2_g|u^w|^2\,dv_g \\
&+|w|^2\int_M \phi^2\left(|u^w|^2|du|^2+|du^w|^2+\frac{(1-|u|^2)}{\epsilon^2}|u^w|^2\right)\,dv_g,
\end{split}
\end{equation*}
and observe that both $w\mapsto J_u(w,\phi)$ and $\phi\mapsto J_u(w,\phi)$ have the structure of quadratic forms. We can then rewrite the estimate \eqref{EEpp} as
\begin{equation*}
E_{\epsilon}''(u)(v_{\phi,w},v_{\phi,w})\leq |w|^2|\Gamma|\int_M\left(|d\phi|_g^2-\frac{(1-|u|^2)}{\epsilon^2}\phi^2\right)dv_g+C|\Gamma|J_u(w,\phi).
\end{equation*}
Now, let $\phi_0$ be a constant function, $\phi_1\in W^{1,2}(M)$, and $w,z\in W_m$ be unit vectors; we would like to estimate $|E_{\epsilon}''(u)(v_{\phi_0,z},v_{\phi_1,w})|$. To that end, first observe that using integration by parts
\begin{equation*}
\begin{split}
&\phi_0\left|\int_M\sum_{\gamma,\tau}\langle u,z^\tau\rangle\left\langle d\left(\phi_1\circ\gamma^{-1}\right), \langle du,w^\gamma\rangle\right\rangle\,dv_g\right| \leq \\
&\phi_0\left|\int_M\sum_{\gamma,\tau}(\phi_1\circ\gamma^{-1})\left(\frac{(1-|u|^2)}{\epsilon^2}\langle u,z^\tau\rangle\langle 
u,w^\gamma\rangle -  \langle \langle du,z^\tau\rangle,
\langle du,w^\gamma\rangle\rangle\right)\,dv_g\right|\leq \\
& \frac{C|\Gamma|}{|z||w|}\sqrt{J_u(z,\phi_0)J_u(w,\phi_1)}.
\end{split}
\end{equation*}
Integrating the right hand sides of~\eqref{pt.cross.2},~\eqref{pt.cross}, and~\eqref{pt.cross.3} then gives that
\begin{equation*}%\label{cross.terms}
|E_{\epsilon}''(u)(v_{\phi_0,z},v_{\phi_1,w})|\leq |\Gamma||\langle w,z\rangle|\left|\langle \phi_0,\phi_1\rangle_{L^2_{g_u}}\right|+C|\Gamma|\sqrt{J_u(z,\phi_0)J_u(w,\phi_1)}.
\end{equation*}

For the remainder of the proof, let's take $\phi_0\equiv 1$, and let $\phi_1\neq 0$ belong to the first nontrivial eigenspace for $\Delta_{g_u}$. Then the estimates above give
\begin{equation}\label{hess.bd.1}
E_{\epsilon}''(u)(v_{\phi_0,z},v_{\phi_0,z})\leq -|\Gamma||z|^2\|\phi_0\|_{L^2_{g_u}}^2+C|\Gamma|J_u(z,\phi_0),
\end{equation}
\begin{equation}\label{hess.bd.2}
E_{\epsilon}''(u)(v_{\phi_1,w},v_{\phi_1,w})\leq |\Gamma||w|^2(\lambda_1(g_u)-1)\|\phi_1\|_{L^2_{g_u}}^2+C|\Gamma| J_u(w,\phi_1),
\end{equation}
and
\begin{equation}\label{hess.bd.3}
|E_{\epsilon}''(u)(v_{\phi_1,w},v_{\phi_0,z})|\leq C|\Gamma|\sqrt{J_u(z,\phi_0)J_u(w,\phi_1)}.
\end{equation}

Next, fixing $\phi=\phi_0$ or $\phi_1$, diagonalizing the quadratic form $W_m\ni w\mapsto J_u(w,\phi)$ gives an orthonormal basis $w_1,\ldots, w_m$ for $W_m$, with respect to which
\begin{equation}\label{j.tr.1}
\sum_{j=1}^mJ_u(w_j,\phi)\leq\int_M|d\phi|^2+\phi^2\left(2|du|^2+\frac{(1-|u|^2)}{\epsilon^2}\right)\,dv_g.
\end{equation}

Now, note that since $\Delta u=-\mathrm{div}(du)=\frac{(1-|u|^2)}{\epsilon^2}u,$ we have
\begin{equation*}
\begin{split}
\int_M \phi^2|du|^2_g\,dv_g &= \int_M \phi^2\frac{(1-|u|^2)}{\epsilon^2}|u|^2 - \langle 2\phi d\phi,\langle u,du\rangle\rangle\,dv_g\\
&\leq \int_M \phi^2\frac{(1-|u|^2)}{\epsilon^2} + \frac{1}{2} \phi^2|du|^2_g + 2|d\phi|_g^2\,dv_g.
\end{split}
\end{equation*}
As a result,
\[
\int_M \phi^2|du|^2_g\,dv_g\leq 2 \int_M \phi^2\frac{(1-|u|^2)}{\epsilon^2} + 2|d\phi|_g^2\,dv_g,
\]
%$$div(\langle u,du\rangle)=|du|^2-\frac{(1-|u|^2)}{\epsilon^2}|u|^2,$$
%and therefore we can write
%$$|du|^2+\frac{(1-|u|^2)^2}{\epsilon^2}=div(\langle u,du\rangle)+\frac{(1-|u|^2)}{\epsilon^2};$$
%in particular
%\begin{eqnarray*}
%\int_M\phi^2(2|du|^2+2\frac{(1-|u|^2)^2}{\epsilon^2})&=&\int_M \phi^2\frac{1-|u|^2}{\epsilon^2}+\phi^2div(\langle u,du\rangle)\\
%&=&\int_M(\phi^2\frac{1-|u|^2}{\epsilon^2}-\langle 2\phi d\phi,\langle u,du\rangle)\\
%&\leq &\int_M(\phi^2\frac{1-|u|^2}{\epsilon^2}+|d\phi|^2|u|^2+\phi^2|du|^2,
%\end{eqnarray*}
%so that
%$$\int \phi^2(|du|^2+2\frac{(1-|u|^2)^2}{\epsilon^2})\leq \int \phi^2\frac{1-|u|^2}{\epsilon^2}+|d\phi|^2.$$
so that \eqref{j.tr.1} implies
$$
\sum_{j=1}^mJ_u(w_j,\phi)\leq \int_M\left(9|d\phi|^2+6\phi^2\frac{1-|u|^2}{\epsilon^2}\right)\,dv_g\leq (9\lambda_1(g_u)+5)\|\phi\|_{L^2_{g_u}}^2,
$$
and the set $S_{\delta}(\phi)$ of $j\in\{1,\ldots,m\}$ for which $J_u(w_j,\phi)\geq \delta \|\phi\|_{L^2_{g_u}}^2$ must have size at most
$$
|S_{\delta}(\phi)|\leq \frac{9\lambda_1(g_u)+5}{\delta}.
$$
As a consequence, the space
$$
Z_{\delta}(\phi):=\mathrm{Span}\{w_j\mid j\notin S_{\delta}(\phi)\}\subset W_m,
$$
on which 
$$J_u(w,\phi)\leq \delta |w|^2\|\phi\|_{L^2_{g_u}}^2,$$
must have dimension
$$\dim Z_{\delta}(\phi)\geq m-|S_{\delta}(\phi)|\geq m-\frac{9\lambda_1(g_u)+5}{\delta}.$$
Writing
$$V_{\delta}(\phi):=\{v_{\phi,w}\mid w\in Z_{\delta}(\phi)\}\subset W^{1,2}_{\Gamma}(M,A_{\Gamma}^m),$$
it follows from the injectivity of the map 
$w\mapsto v_{\phi,w}$
that
$$\dim (V_{\delta}(\phi))\geq m-\frac{9\lambda_1(g_u)+5}{\delta}.$$

Now, defining
$$
V_{\delta}:=V_{\delta}(\phi_0)+V_{\delta}(\phi_1)\subset W^{1,2}_{\Gamma}(M,A_{\Gamma}^m),
$$
it follows that
$$
\dim(V_{\delta})\geq 2m-\frac{18\lambda_1(g_u)+10}{\delta}-\dim(V_{\delta}(\phi_0)\cap V_{\delta}(\phi_1)).
$$
Moreover, for any $v_{\phi_1,w}\in V_{\delta}(\phi_1)$ and $v_{\phi_0,z}\in V_{\delta}(\phi_0)$, then it follows from \eqref{v.lower} and the bound $J_u(w,\phi)\leq \delta |w|^2\|\phi\|_{L^2_{g_u}}^2$ for $w\in Z_{\delta}(\phi)$ that
$$\|v_{\phi_1,w}\|_{L^2_{g_u}}^2\geq |\Gamma| |w|^2\|\phi_1\|_{L^2_{g_u}}^2(1-C\delta)$$
and
$$\|v_{\phi_0,z}\|_{L^2_{g_u}}^2\geq |\Gamma||z|^2\|\phi_0\|_{L^2_{g_u}}^2(1-C\delta),$$
while it follows from \eqref{pt.cross} that
$$\langle v_{\phi_0,z},v_{\phi_1,w}\rangle_{L^2_{g_u}}\leq C|\Gamma||z||w|\delta\|\phi_0\|_{L^2_{g_u}}\|\phi_1\|_{L^2_{g_u}}.$$
If $v=v_{\phi_0,w}=v_{\phi_1,z}\in V_{\delta}(\phi_0)\cap V_{\delta}(\phi_1),$ then combining the three preceding estimates for $\|v\|_{L^2_{g_u}}^2$ gives an estimate of the form
$$(1-C\delta)\leq C\delta,$$
which clearly cannot hold if $\delta<\frac{1}{2C}$.

Thus, if $\delta<\frac{1}{2C}$, it follows that
$$
\dim(V_{\delta})\geq 2m-\frac{18\lambda_1(g_u)+10}{\delta}.
$$
On the other hand, for every $v=v_{\phi_0,z}+v_{\phi_1,w}\in V_{\delta}$, we have
\begin{equation*}
\begin{split}
E_{\epsilon}''(u)(v,v)\leq &-|\Gamma||z|^2\|\phi_0\|_{L^2_{g_u}}^2+C|\Gamma|J_u(z,\phi_0)
+|\Gamma|(\lambda_1(g_u)-1)|w|^2\|\phi_1\|_{L^2_{g_u}}^2\\ &+C|\Gamma|J_u(w,\phi_0) 
+2C|\Gamma|\sqrt{J_u(z,\phi_0)J_u(w,\phi_1)} \\
\leq &(-|\Gamma|+C|\Gamma|\delta)|z|^2\|\phi_0\|_{L^2_{g_u}}^2
+|\Gamma|(\lambda_1(g_u)-1+C\delta)|w|^2\|\phi_1\|_{L^2_{g_u}}^2\\
&+2C|\Gamma|\delta|w||z|\|\phi_0\|_{L^2_{g_u}}\|\phi_1\|_{L^2_{g_u}}.
\end{split}
\end{equation*}

Finally, note that if $\delta<c$ is sufficiently small and $\lambda_1(g_u)-1+C\delta<-\delta,$ then an application of the Cauchy-Schwarz inequality in the estimate above gives
$$E_{\epsilon}''(u)(v,v)<0,$$
while
$$\dim(V_{\delta})\geq 2m-\frac{18\lambda_1(g_u)+10}{\delta}\geq 2m-\frac{30}{\delta}.$$
But since $\ind_{\Gamma,E_{\epsilon}}(u)\leq m$, this cannot hold if $\frac{30}{\delta}\leq m$. Thus, for $\delta$ in the range
$$\frac{30}{m}\leq\delta<c,$$
we must have
$$\lambda_1(g_u)-1+C\delta \geq -\delta.$$
In other words, if $m$ is large enough that $c>\frac{10}{m}$, we must have
$$\lambda_1(g_u)-1\geq -(1+C)\frac{10}{m},$$
giving an estimate of the desired form.
\end{proof}

Since critical points $u$ of $E_{\epsilon}$ satisfy
$$
\mathrm{div}(\langle u,du\rangle)=|du|^2-\frac{(1-|u|^2)}{\epsilon^2}|u|^2,
$$
their energy can be bounded above by
\begin{equation*}
\begin{split}
E_{\epsilon}(u)&=\int_M \frac{1}{2}|du|^2+\frac{(1-|u|^2)^2}{4\epsilon^2}\,dv_g
=\int_M\frac{(1-|u|^2)|u|^2}{2\epsilon^2}+\frac{(1-|u|^2)^2}{4\epsilon^2}\,dv_g\\
&\leq \int_M\frac{(1-|u|^2)}{2\epsilon^2}\,dv_g=\frac{1}{2}\Area(g_u).
\end{split}
\end{equation*}
Thus, for the min-max critical points $u_{\epsilon}$ given by Lemma \ref{mm.ex}, combining this estimate with Proposition \ref{index.eigen.bd} gives 
\begin{equation}\label{lam.lim}
\liminf_{\epsilon\to 0}\bar{\lambda}_1(M,g_{u_{\epsilon}})\geq \lim_{\epsilon\to 0}\left(1-\frac{C}{m}\right)2\mathcal{E}_{m,\epsilon}^{\Gamma}(M,g)\geq \left(1-\frac{C}{m}\right)\Lambda_1^{\Gamma}(M,[g]).
\end{equation}
Assuming $\Lambda_1^{\Gamma}(M,[g])>8\pi$, we can use this estimate as in \cite{Kok.var, Pet1} to rule out energy concentration for the min-max critical points as $\epsilon\to 0$ and $m\to\infty$.

As a first ingredient, it is useful to recall a variant of the small energy regularity theorem for critical points of $E_{\epsilon}$ on $W^{1,2}(M,A_{\Gamma}^m)$. The following observation is well known to experts (cf. \cite{CS89}), but since the precise statement below does not seem to appear in the literature, we include a brief proof for the sake of completeness.

\begin{lemma}\label{gl.ereg} There is a constant $\eta(M,g)>0$ independent of $m$ and $\epsilon$ such that for any critical point $u\colon M\to A_{\Gamma}^m$ for $E_{\epsilon}(u)=\int_M\frac{1}{2}|du|^2+\frac{(1-|u|^2)^2}{\epsilon^2}$, if
$$
\int_{B_{\delta}(p)}\frac{1-|u|^2}{\epsilon^2}\,dv_g\leq \eta,
$$
on some ball $B_{\delta}(p)\subset M$ of radius $\delta<\min\{\inj(M,g),1\}$, then
$$
\delta^2\left\|\frac{1-|u|^2}{\epsilon^2}\right\|_{L^{\infty}(B_{\delta/2}(p))}\leq 1.
$$
\end{lemma}
\begin{proof}
To begin, we observe as in \cite[Section 4.1]{S19}--where the Laplacian is taken with the opposite sign convention--that if $C(g)$ is an upper bound $|K_g|\leq C$ for the Gauss curvature $K_g$ of $(M,g)$, then for the function $w: = |du|^2-(1+C\epsilon^2)\frac{1-|u|^2}{\epsilon^2}$ the equation $\Delta_g u=\frac{1-|u|^2}{\epsilon^2}u$ together with the Bochner identity gives
$$
\Delta_g w\leq -\epsilon^4|u|^2w,
$$
and, consequently, $w$ can not have positive maxima, so
$$
|du|_g^2\leq (1+C\epsilon^2)\frac{1-|u|^2}{\epsilon^2}.
$$
In particular, it follows that
\begin{eqnarray*}
\frac{1}{2}\Delta_g(1-|u|^2)&=&|du|_g^2-\frac{(1-|u|^2)|u|^2}{\epsilon^2}\\
&\leq & \frac{1-|u|^2}{\epsilon^2}+C(1-|u|^2)-\frac{(1-|u|^2)|u|^2}{\epsilon^2}\\
&=&\frac{(1-|u|^2)^2}{\epsilon^2}+C(1-|u|^2),
\end{eqnarray*}
so multiplying by $\epsilon^{-2}$ and writing $f=\frac{1-|u|^2}{\epsilon^2}$, we arrive at an estimate of the form
$$
\frac{1}{2}\Delta_g f\leq f^2+Cf.
$$
It then follows from standard techniques \cite[Remark 5.4]{KS22} that there exists $\eta(M,g)>0$ such that
whenever $\int_{B_{\delta}(p)}f\leq \eta$ for $\delta<\min\{\inj(M,g),1\}$, then
$$
\delta^2\|f\|_{L^{\infty}(B_{\delta/2}(p))}\leq 1.
$$
This completes the proof.
\end{proof}

Next, we show that Proposition \ref{index.eigen.bd}, together with Lemma \ref{mm.ex} and the assumption $\Lambda_1^{\Gamma}(M,[g])>8\pi$ forces the hypotheses of Lemma \ref{gl.ereg} to hold on all small balls for $m$ sufficiently large and $\epsilon$ small, forcing smooth subsequential convergence $u_{m,\epsilon}\to u_m$ of the min-max critical points to a harmonic map $u_m\colon(M,[g])\to \Sph^{m|\Gamma|-1}$ realizing $\mathcal{E}_m^{\Gamma}(M,[g])$. In what follows, we define the $\Gamma$-equivariant Morse index $\ind_{E,\Gamma}(u)$ of a $\Gamma$-equivariant harmonic map $u\colon (M,[g])\to \Sph^{m|\Gamma|-1}\subset A_{\Gamma}^m$ as the maximal dimension of a subspace 
$$
V\subset \{v\in W^{1,2}_{\Gamma}(M,A_{\Gamma}^m)\mid \langle v,u\rangle \equiv 0\}
$$
on which 
$$
v\mapsto \int_M(|dv|_g^2-|du|_g^2|v|^2)\,dv_g
$$
is negative definite.  Below, we say $\Gamma$ has \emph{empty fixed point set} if $\Gamma p \neq \{p \}$ for any $p \in M$. 

\begin{proposition}\label{hm.ex}
\label{fixpt.rk1}
Suppose either $\Lambda_1^{\Gamma}(M,[g])>8\pi$ or $\Gamma$ has empty fixed point set.  For each sufficiently large $m$, 
 % and with $C$ as in the conclusion of Lemma \ref{index.eigen.bd}, choose $m$ large enough that
%$$
%\left(1-\frac{C}{m}\right)\Lambda_1^{\Gamma}(M,[g])>8\pi.
%$$
as $\epsilon\to 0$, a subsequence of the critical points $u_{m,\epsilon}$ converges smoothly to a $\Gamma$-equivariant harmonic map
$$
u_m\colon (M,[g])\to \Sph^{m|\Gamma|-1}\subset A_{\Gamma}^m
$$
satisfying
$$
E(u_m)=\mathcal{E}_m^{\Gamma}(M,[g]),\hspace{10mm} \ind_{E,\Gamma}(u_m)\leq m,
$$
and
$$
|du_m|_g^2\leq C'
$$
for some constant $C'=C'(M,g,\Gamma)$ independent of $m$.
\end{proposition}
\begin{proof}
We first complete the proof assuming $\Lambda_1^{\Gamma}(M,[g])>8\pi$, and leave the modifications to the case where $\Gamma$ has empty fixed point set to the end of the proof.
For each $\epsilon>0$, let
$$
\delta_{\epsilon}:=\inf\left\{\delta>0\left|\,\,\, \int_{B_{\delta}(p)}\frac{1-|u_{m,\epsilon}|^2}{\epsilon^2}\,dv_g>\eta\right.\right\}
$$
be the smallest scale at which the hypothesis of Proposition \ref{index.eigen.bd} fails for $u_{m,\epsilon}$. We first claim that
\begin{equation}\label{no.bub}
\delta_{\epsilon}\geq \delta_0>0
\end{equation}
for some $\delta_0$ independent of $\epsilon$.

To see this, let $p=p_{\epsilon}\in M$ be a point at which the map $u=u_{m,\epsilon}$ on the ball of radius $\delta=\delta_{\epsilon}$ satisfies
$$
\int_{B_{\delta}(p)}\frac{1-|u|^2}{\epsilon^2}\,dv_g\geq \eta,
$$
and recall that by the variational characterization of $\lambda_1(g_u)$, we must have either
\begin{equation}\label{dir.case.1}
\lambda_1(g_u)\int_M \phi^2\frac{1-|u|^2}{\epsilon^2}dv_g\leq \int_M|d\phi|^2\text{ for all }\phi\in W_0^{1,2}(B_{\sqrt{\delta}}(p))
\end{equation}
or
\begin{equation}\label{dir.case.2}
\lambda_1(g_u)\int_M \phi^2\frac{1-|u|^2}{\epsilon^2}dv_g\leq \int_M|d\phi|^2\text{ for all }\phi\in W_0^{1,2}(M\setminus B_{\sqrt{\delta}}(p)),
\end{equation}
where $\lambda_1(g_u)\geq 1-\frac{C}{m}>\frac{1}{2}$. If $\delta<1$ and \eqref{dir.case.1} holds, then one can use logarithmic cutoffs to find a function $\phi\in W_0^{1,2}(B_{\sqrt{\delta}}(p))$ such that $\phi\equiv 1$ on $B_{\delta}(p)$ 
while 
$$\int_M|d\phi|^2\leq \frac{C'}{|\log\delta|},$$
which together with \eqref{dir.case.1} implies
$$\frac{\eta}{2}\leq \frac{1}{2}\int_M\phi^2\frac{1-|u|^2}{\epsilon^2}\,dv_g\leq \frac{C'}{|\log\delta|},$$
and consequently $\delta=\delta_{m,\epsilon}\geq \min\{1, e^{-2C'/\eta}\}$. Thus, to prove \eqref{no.bub}, it suffices to consider what happens when \eqref{dir.case.2} holds. 

If \eqref{dir.case.2} holds, then similarly we can use logarithmic cutoffs to find $0\leq \phi\in W_0^{1,2}(M\setminus B_{\sqrt{\delta}}(p))$ such that $\phi\equiv 1$ on $B_{\delta^{1/4}}(p)$ while $\int_M|d\phi|^2\leq \frac{C'}{|\log\delta|}$, and \eqref{dir.case.2} yields a bound of the form
\begin{equation}\label{comp.small}
\lambda_1(g_u)\int_{M\setminus B_{\delta^{1/4}}(p)}\frac{1-|u|^2}{\epsilon^2}\,dv_g\leq \lambda_1(g_u)\int_M\phi^2\frac{1-|u|^2}{\epsilon^2}\,dv_g\leq \frac{C'}{|\log\delta|}.
\end{equation}
On the other hand, if $\delta<\left(\inj(M,g)\right)^5$, then the disk $B_{\delta^{1/5}}(p)\subset M$ can be conformally identified with a domain in $\Sph^2$, and choosing another cutoff function with $\psi\in C_c^{\infty}(B_{\delta^{1/5}}(p))$ with $\psi\equiv 1$ on $B_{\delta^{1/4}}(p)$ and $\int_M |d\psi|^2\leq \frac{C'}{|\log\delta|}$, we can apply the Hersch trick as in \cite[Lemma 3.1]{Kok.var} to deduce the existence of a conformal map $F=(F_1,F_2,F_3)\colon B_{\delta^{1/4}}(p)\to \mathbb{S}^2\subset \mathbb{R}^3$ of energy $E(F)\leq 4\pi$ and satisfying the balancing condition
$$
\int_M\frac{1-|u|^2}{\epsilon^2} \psi F_i \,dv_g=0\text{ for }i=1,2,3.
$$
Thus, applying the variational characterization of $\lambda_1(g_u)$ with test functions $\psi F_i$ gives
\begin{eqnarray*}
\lambda_1(\Delta_{g_u})\int_M\frac{1-|u|^2}{\epsilon^2}\psi^2 F_i^2\,dv_g&\leq & \int_M|d(\psi F_i)|^2\,dv_g\leq \int_M \psi^2|dF_i|^2\,dv_g\\
&&+(2\|d\psi\|_{L^2}\|dF_i\|_{L^2}+\|d\psi\|_{L^2}^2),
\end{eqnarray*}
and summing over $i=1,2,3$ gives
$$
\lambda_1(g_u)\int_M\frac{1-|u|^2}{\epsilon^2}\psi^2\,dv_g\leq 8\pi+\frac{C'}{\sqrt{|\log \delta|}}.
$$

Since $\psi \equiv 1$ on $B_{\delta^{1/4}(p)}$, we can combine this with \eqref{comp.small} to obtain an estimate of the form
$$
\bar{\lambda}_1(g_{u_{m,\epsilon}})=\lambda_1({u_{m,\epsilon}})\int_M\frac{1-|u_{m,\epsilon}|^2}{\epsilon^2}\,dv_g\leq 8\pi+\frac{C''}{\sqrt{|\log \delta_{\epsilon}|}}.
$$
Taking $\epsilon\to 0$, it then follows from \eqref{lam.lim} that 
$$
\left(1-\frac{C}{m}\right)\Lambda_1^{\Gamma}(M,[g])\leq 8\pi+\liminf_{\epsilon\to 0}\frac{C''}{|\log \delta_{\epsilon}|^{1/2}},
$$
and by choosing $m$ large enough, the assumption $\Lambda^\Gamma_1(M, [g])>8\pi$ implies the left-hand side is $>8\pi$, implying the positive lower bound \eqref{no.bub} for $\delta_{\epsilon}$.

In fact, the same argument shows that the scales
$$
\delta_{m,\epsilon}:=\inf\left\{\delta>0\left|\,\,\, \int_{B_{\delta}(p)}\frac{1-|u_{m,\epsilon}|^2}{\epsilon^2}>\eta\right.\right\}
$$
have a universal lower bound
$$\delta_{m,\epsilon}\geq \delta_0>0$$
independent of both $\epsilon$ and $m$. Together with Lemma \ref{gl.ereg}, this implies that, for $m$ large enough, the maps $u_{m,\epsilon}$ satisfy a uniform bound of the form
\begin{equation}\label{linfty.bd}
\frac{1-|u_{m,\epsilon}|^2}{\epsilon^2}\leq C'
\end{equation}
independent of $m$ and $\epsilon$.

With the uniform control \eqref{linfty.bd} established, it follows from the analysis of \cite{LW02} that, along a subsequence $\epsilon\to 0$, the maps $u_{m,\epsilon}$ converge smoothly to a harmonic map
$$
u_m\colon (M,g)\to \Sph^{m|\Gamma|-1}\subset A_{\Gamma}^m,
$$
with energy
\begin{eqnarray*}
E(u_m)&=&\frac{1}{2}\int_M|du_m|^2=\frac{1}{2}\lim_{\epsilon\to 0}\int_M|du_{m,\epsilon}|^2\\
&=&\lim_{\epsilon\to 0}\left(\mathcal{E}_{m,\epsilon}^{\Gamma}(M,[g])-\int_M\frac{(1-|u|^2)^2}{\epsilon^2}dv_g\right),
\end{eqnarray*}
and since $\frac{(1-|u|^2)^2}{\epsilon^2}\leq \frac{(C')^2\epsilon^4}{\epsilon^2}\leq (C')^2\epsilon^2$ by \eqref{linfty.bd}, it follows that
$$
E(u_m)=\lim_{\epsilon\to 0}\mathcal{E}_{m,\epsilon}^{\Gamma}(M,[g])=\mathcal{E}_m^{\Gamma}(M,[g]).
$$
For the lower semi-continuity
$$
\ind_{E,\Gamma}(u)\leq \ind_{E_{\epsilon},\Gamma}(u_{\epsilon})
$$
of $\Gamma$-equivariant index, one can argue exactly as in \cite[Lemma 3.6]{KSminmax}, and in fact the proof is simpler in the present situation, since the convergence $u_{m,\epsilon}\to u_m$ is smooth.

Finally, note that the limit harmonic map $u_m\colon M\to \Sph^{m|\Gamma|-1}$ satisfies $\Delta_g u_m=|du_m|_g^2u_m$, and since the convergence $u_{m,\epsilon}\to u_m$ is smooth, we have
$$
|du_m|_g^2=\langle \Delta_g u_m,u_m\rangle=\lim_{\epsilon\to 0}\langle \Delta_g u_{m,\epsilon},u_{m,\epsilon}\rangle=\lim_{\epsilon\to 0}\frac{(1-|u_{m,\epsilon}|^2)|u_{m,\epsilon}|^2}{\epsilon^2},
$$
which together with \eqref{linfty.bd} gives a uniform gradient bound
$$
|du_m|_g^2\leq C'
$$
independent of $m$.

Finally, we consider the case where $\Gamma$ has empty fixed point set.  In this case, it is easy to see that the conclusions of the proposition hold as soon as $m$ is large enough such that $\lambda_1(g_u)\geq \frac{1}{2}$. To see this, take $\delta_0(N,\Gamma)>0$ small enough such that the orbit of any point under $\Gamma$ contains elements of distance $>2\sqrt{\delta_0}$ apart, and note that if \eqref{dir.case.1} fails at the point $p=p_{\epsilon}$, then by the equivariance of the map $u=u_{m,\epsilon}$, it must fail at every point $q\in\Gamma p_{\epsilon}$ in the orbit of $p_{\epsilon}$ under $\Gamma$ as well. If $\delta_{\epsilon}<\delta_0$ and $q\in \Gamma p_{\epsilon}\setminus\{p_{\epsilon}\}$, then we must have $B_{\sqrt{\delta_{\epsilon}}}(q_{\epsilon})\subset M\setminus B_{\sqrt{\delta_{\epsilon}}}(p)$, so the failure of \eqref{dir.case.1} implies the failure of \eqref{dir.case.2} as well; however, we know that at least one of these cases must hold, and therefore $\delta_{\epsilon}\geq \delta_0(N,\Gamma)>0$. The rest of the argument then follows exactly as above.
\end{proof}

\subsection{Stabilization and realization of $\Lambda_1^{\Gamma}(M, [g])$.}\label{closed.stab}

By a simplified version of the arguments in \cite[Section 3]{KSminmax} and \cite[Section 3.2]{KS22}, assuming that $\Lambda_1^{\Gamma}(M,[g])>8\pi$ or that $\Gamma$ has empty fixed point set, we can now conclude that the maps $u_m$ stabilize as $m$ gets large to a map by first Laplace eigenfunctions for a metric realizing $\Lambda_1^{\Gamma}(M,[g])$. As a first step, as in \cite{KSminmax, KS22}, we show that the maps $u_m$ take values in a equatorial subsphere of fixed dimension as $m\to\infty$. This can be deduced from the boundedness of the min-max energies $\mathcal{E}_m^{\Gamma}(M,[g])$ by appealing to \cite[Theorem 3.12]{KSminmax}, but we give here a different argument based on the estimates obtained in the preceding subsection, which also provides a much simpler path to the results of \cite{KSminmax} in the case of the trivial group $\Gamma$.

\begin{lemma}\label{closed.stab.lem}
If $\Lambda_1^{\Gamma}(M,[g])>8\pi$ or $\Gamma$ has empty fixed point set, there exists $N(M,[g])\in \mathbb{N}$ such that the harmonic maps $u_m\colon M\to \Sph^{m|\Gamma|-1}$ of Proposition \ref{hm.ex} take values in a totally geodesic subsphere of dimension $N$, i.e. $u_m(M)\subset \Sph^N\subset \Sph^{m|\Gamma|-1}$.
\end{lemma}
\begin{proof}
%By Proposition \ref{hm.ex}, for $m$ sufficiently large, the maps $u_m$ have a uniform gradient bound
%$$
%\|du_m\|_{L^{\infty}(M,g)}^2\leq C
%$$
%independent of $m$. 
As in \cite[Section 3]{KS22},  observe that the Bochner identity for sphere-valued harmonic maps gives
$$
-\Delta_g\frac{1}{2}|du_m|^2\geq |d|du_m||^2-\|\mathrm{Ric}_g\|_{L^{\infty}}|du_m|^2-|du_m|^4,
$$
so integrating and applying the $L^{\infty}$ gradient bound from Proposition \ref{hm.ex} gives the $W^{1,2}$ bound
$$\||du_m|\|_{W^{1,2}}^2\leq C'(M,g,\Gamma).$$
We now argue as in the proof of \cite[Proposition 3.8]{KS22}. If the desired conclusion failed, then the space $\mathcal{C}_m\subset W^{1,2}(M)$ spanned by the coordinate functions of $u_m$ would have unbounded dimension as $m\to\infty$; in particular, since the harmonic map equation gives
$$
\Delta_g u_m=|du_m|_g^2u_m,
$$
it follows that, along a subsequence $m\to\infty$, the eigenvalues
$$
\lambda_k(|du_m|_g^2 g):=\inf_{\substack{V\subset W^{1,2}(M),\\ \dim(V)=k+1}}\max_{\phi\in V\setminus\{0\}}\frac{\int_M|d\phi|^2\,dv_g}{\int_M|du_m|^2\phi^2\,dv_g} 
$$
satisfy
\begin{equation}\label{eigen.small}
\lim_{m\to\infty}\lambda_k(|du_m|_g^2 g )\leq 1,
\end{equation}
since $\mathcal{C}_m$ gives a space of eigenfunctions with eigenvalue $1$. On the the other hand, since $|du_m|_g$ is bounded in $W^{1,2}(M)$, we can find a nonnegative function $0\leq \rho\in W^{1,2}(M)$ and a further subsequence such that
$$
|du_m|_g^2\to \rho\text{ in }L^p
$$
for every $p\in [1,\infty)$. It is then straightforward to check \cite[Prop. 5.1]{GKL} that $\lambda_k(|du_m|_g^2 g)\to \lambda_k(\rho g)$ for every $k\in \mathbb{N}$, which together with \eqref{eigen.small} implies that 
$$\lambda_k(\rho)\in [0,1]\text{ for every }k\in \mathbb{N}.$$
On the other hand, since $\rho\in W^{1,2}(M)\subset L^p(M)$, it is easy to see that the spectrum
$\{\lambda_k(\rho)\}$
is discrete, with $\lim_{k\to\infty}\lambda_k(\rho)=\infty$, contradicting the containment $\lambda_k(\rho)\in [0,1]$. Thus, $\dim(\mathcal{C}_m)$ must have an upper bound $N(M,[g])+1$ independent of $m$, which implies the desired conclusion. % that $u_m(M)$ lies inside a geodesic subsphere $\Sph^N\subset \Sph^{m|\Gamma|-1}$ of dimension $N(M,[g])$, as desired.
\end{proof}

Next, by combining the preceding lemma with a $\Gamma$-equivariant analog of \cite[Proposition 3.11]{KarRP2}, we arrive at the desired identification of $2\mathcal{E}_m^{\Gamma}$ and $\Lambda_1^{\Gamma}$ for $m$ sufficiently large.

\begin{theorem}\label{lap.mm.char}
Let $u_m\colon (M,g)\to \Sph^{m|\Gamma|-1}$ be the harmonic maps of Proposition \ref{hm.ex} realizing $E(u_m)=\mathcal{E}_m^{\Gamma}(M,[g])$. Suppose $\Lambda_1^{\Gamma}(M,[g])>8\pi$ or that $\Gamma$ has empty fixed point set. Then for $m$ sufficiently large, $g_m:=|du_m|_g^2g$ maximizes $\bar{\lambda}_1$ among $\Gamma$-invariant metrics in $[g]$, that is
$$
\Lambda_1^{\Gamma}(M,[g])=\bar{\lambda}_1(M,g_m)=2\mathcal{E}^{\Gamma}_m(M,[g]).
$$
\end{theorem}
\begin{proof}
We already know that
$$
\Area(M,g_m)=\frac{1}{2}\int_M|du_m|_g^2dv_g=\mathcal{E}_m(M,[g])\geq \frac{1}{2}\Lambda_1^{\Gamma}(M,[g]),
$$
so to prove that $\bar{\lambda}_1(M,g_m)=\frac{1}{2}\Lambda_1^{\Gamma}(M,[g])$ for $m$ large, it suffices to show that $\lambda_1(g_m)=2$,  that
$$
\int_M( |d\phi|^2-|du_m|_g^2\phi^2)dv_g\geq 0
$$
for all $\phi\in C^{\infty}(M)$ with $\int_M|du_m|_g^2\phi dv_g=0$. 

To this end, we mimick the proof of \cite[Proposition 3.11]{KSminmax} in our equivariant setting. Let $Y_m\subset C^{\infty}(M)$ be the subspace spanned by eigenfunctions  for $\Delta_g-|du_m|_g^2$ with negative eigenvalues, and consider the linear map
$$
T_m\colon Y_m\to W^{1,2}(M, A_{\Gamma})
$$
given by
$$
T_m(\phi):=\sum_{\sigma\in \Gamma}(\phi \circ \sigma^{-1})e_{\sigma}.
$$
Note that
\begin{equation*}
\begin{split}
\gamma\cdot T_m(\phi) &=\sum_{\sigma\in \Gamma}(\phi\circ \sigma^{-1})e_{\gamma\sigma}
=\sum_{\tau\in\Gamma}(\phi\circ (\tau^{-1}\gamma))e_{\tau} \\
&=\sum_{\tau\in \Gamma}(\phi\circ \tau^{-1})\circ \gamma e_{\tau}
=T_m(\phi)\circ \gamma,
\end{split}
\end{equation*}
so in fact $T_m$ gives a map $T_m\colon Y_m\to W^{1,2}_{\Gamma}(M,A_{\Gamma})$; moreover, this map is clearly injective, since we can recover $\phi$ as $\langle T_m(\phi), e_{I}\rangle$. Now, appealing to Lemma \ref{closed.stab.lem}, let $V_{N+1}\subset A_{\Gamma}^m$ be an $(N+1)$-dimensional subspace such that $u_m(M)\subset V_{N+1}$; then the orthogonal complement
$$
V_{N+1}^{\perp}\subset A_{\Gamma}^m
$$
contains a direct sum of at least $m-(N+1)$ copies of $A_{\Gamma}$, and therefore we can find a direct sum of $m-(N+1)$ copies of $T_m(Y_m)$ inside 
$$
W^{1,2}_{\Gamma}(M, V_{N+1}^{\perp})\subset \mathcal{V}_{\Gamma}(u_m):=\left\{v\in W^{1,2}_{\Gamma}(M, A_{\Gamma}^m)\mid \langle v,u_m\rangle \equiv 0\right\}.
$$
Moreover, it is clear that
\begin{eqnarray*}
\int_M |d(T_m(\phi))|^2-|du_m|^2|T_m(\phi)|^2&=&\sum_{\sigma\in \Gamma}\int_M |d(\phi\circ \sigma^{-1})|^2-|du_m|^2(\phi\circ \sigma^{-1})^2\\
&=&|\Gamma|\int_M|d\phi|^2-|du_m|^2\phi^2<0
\end{eqnarray*}
for all $0\neq \phi\in Y_m$, so these $m-(N+1)$ copies of $W_m$ give rise to an $(m-(N+1))\cdot \dim(Y_m)$-dimensional subspace of $\mathcal{V}_{\Gamma}(u_m)$ along which $E''(u)$ is negative definite. In particular, we deduce that
$$
m\geq \ind_{\Gamma}(u_m)\geq (m-(N+1))\dim(W_m),
$$
so for $m>2N+3$, we must have $\dim(Y_m)=1$. 

In other words, the Schr\"odinger operator $\Delta_g-|du_m|_g^2$ has index one, which is equivalent to the desired statement $\lambda_1(g_m)=2$, completing the proof.
\end{proof}

In particular, we arrive at the following existence result for the symmetric conformal $\bar{\lambda}_1$-maxization problem.

\begin{theorem}\label{lap.conf.ex}
Let $(M, g)$ be a closed surface.  For any finite subgroup $\Gamma\leq \Isom(M,g)$ such that either $\Lambda_1^{\Gamma}(M,[g])>8\pi$ or $\Gamma$ has an empty fixed point set, there exists $m\in \mathbb{N}$ and a $\Gamma$-equivariant harmonic map
$$
u\colon M\to \mathbb{S}^{m|\Gamma|-1}\subset A_{\Gamma}^m
$$
such that the $\Gamma$-invariant conformal metric $\tilde{g}=|du|_g^2g$ realizes
$$
\bar{\lambda}_1(M,\tilde{g})=\Lambda_1^{\Gamma}(M,[g]),
$$
and the coordinates of $u$ are first eigenfunctions for $\Delta_{\tilde{g}}$.
\end{theorem}

\section{Min-max characterization of $\Sigma_1^T(N,[g])$}
\label{SSprelim}
Let $(N,g)$ be a compact surface with boundary, and let $\Gamma$ be a discrete subgroup of the isometry group.  For $m \geq 3$, define $A^m_\Gamma$, $X: = W^{1, 2}_\Gamma(N, A^m_\Gamma)$, and $X_\Gamma : = W^{1, 2}_\Gamma(N, A^m_\Gamma)$ analogously as in the closed case. 

We now consider energies $F_\epsilon$ previously studied by Millot and Sire \cite{MS13} and check that these functionals satisfy an analog of Proposition \ref{gl.pre-mm}, allowing us to produce critical points with bounded Morse index via min-max methods.
\begin{proposition}\label{fb.ps}
For each $\epsilon> 0$, the Ginzburg-Landau energy 
$$
F_{\epsilon}(u):=\int_N\frac{1}{2}|du|_g^2\,dv_g+\int_{\partial N}\frac{(1-|u|^2)^2}{4\epsilon}\,ds_g,
$$
defines a $C^2$ functional on the space $X_\Gamma$ satisfying the following.
\begin{enumerate}[label=\emph{(\roman*)}]
\item  $F_{\epsilon}'(u)(v)=\int_N \langle du,dv\rangle\,dv_g-\epsilon^{-1}\int_{\partial N}(1-|u|^2)\langle u,v\rangle\,ds_g$.
\item $F_{\epsilon}''(u)(v,v)=\int_N|dv|^2\,dv_g+\epsilon^{-1}\int_{\partial N}\left(2\langle u,v\rangle^2-(1-|u|^2)|v|^2\right)\,ds_g$.
\item $F_\epsilon$ satisfies the Palais-Smale condition on $X_\Gamma$. 
\item For any critical point $u \in X_\Gamma$ of $F_\epsilon$, the map $v \mapsto F''_\epsilon(u)(v, \cdot)$ defines a Fredholm operator $X_\Gamma \rightarrow X^*_\Gamma$. 
\item $u \in X_\Gamma$ is a critical point for $F_\epsilon$ if and only if
	\begin{align*}
	\Delta u = 0 \text{ on $N$ and  } \frac{\partial u}{\partial \nu} = \frac{1-|u|^2}{\epsilon} u \text{ on $\partial N$}.
	\end{align*}
\end{enumerate}
\end{proposition}
\begin{proof}
That $F_\epsilon$ defines a $C^2$ functional on $X$ and on $X_\Gamma$ satisfying (i)-(ii) is straightforward.  Since $F_\epsilon$ is invariant under postcomposition with isometries of $A^m_\Gamma$ and precomposition with isometries of $(N, g)$, it suffices to check items (iii)-(v) on $X$.  We next check the Palais-Smale condition. For fixed $\epsilon>0$, let $u_j\in W^{1,2}(N,\mathbb{B}^n)$ satisfy $F_{\epsilon}(u_j)\leq C$ and $\|F_{\epsilon}'(u_j)\|_{(W^{1,2})^*}=\delta_j\to 0$. We wish to show that a subsequence converges strongly. After passing to a subsequence, clearly we can assume $u_j\rightharpoonup u$ weakly in $W^{1,2}$ and $u_j\to u$ strongly in $L^p(\partial N)$ for every $p<\infty$. In particular, it follows that
$$
\lim_{j\to\infty}\int_{\partial N}\frac{1-|u_j|^2}{\epsilon}\langle u_j,u_j-u\rangle\,ds_g =0
$$
and
$$
\lim_{j\to\infty}| F_{\epsilon}'(u_j)(u_j-u)|\leq \lim_{j\to\infty} \delta_j\|u_j-u\|_{W^{1,2}}=0.
$$
Subtracting these two equations, we see that
\begin{equation*}
\begin{split}
&\lim_{j\to\infty}\int_N( |du_j|^2-|du|^2)\,dv_g=\lim_{j\to\infty}\int_N\langle du_j,du_j-du\rangle\,dv_g= \\
&\lim_{j\to\infty}\left(F_{\epsilon}'(u_j)(u_j-u)+\int_{\partial N}\frac{1-|u_j|^2}{\epsilon}\langle u_j,u_j-u\rangle\,ds_g\right)
=0,
\end{split}
\end{equation*}
so the convergence $u_j\to u$ is indeed strong in $W^{1,2}$, confirming the Palais-Smale condition. Using (ii), it is likewise straightforward to see that the self-adjoint operator $F_{\epsilon}''(u)$ is Fredholm, since
$$
F_{\epsilon}''(u)(v,v)\geq \int_N|dv|^2\,dv_g-\frac{1}{\epsilon}\|v\|_{L^2(\partial N)}^2\geq \|v\|_{W^{1,2}(N)}^2-C_{\epsilon}\|v\|_{L^2(\partial N)}^2
$$
and the trace embedding $W^{1,2}(N)\to L^2(\partial N)$ is compact.
\end{proof}

Next, define $\mathcal{B}_m(\Gamma)\subset C^0(\B^m,W_{\Gamma}^{1,2}(N,A_{\Gamma}^m))$ just as in Section \ref{lap.mm}, and set
$$
\mathcal{F}_{m,\epsilon}^{\Gamma}(N,g):=\inf_{(u_y)\in \mathcal{B}_m}\max_{y\in \B^m}F_{\epsilon}(u_y)$$
and
$$
\mathcal{F}_m^{\Gamma}(N,g):=\sup_{\epsilon>0}\mathcal{F}_{m,\epsilon}^{\Gamma}=\lim_{\epsilon\to 0}\mathcal{F}_{m,\epsilon}^{\Gamma}.
$$
As in the closed case, from the conformal invariance of the Dirichlet energy, it is straightforward to check that $\mathcal{F}_m^{\Gamma}(N,g)=\mathcal{F}_m^{\Gamma}(N,[g])$ is a conformal invariant within the class of $\Gamma$-invariant metrics. Next, we establish analogs of Lemma \ref{mm.ubd} and Proposition \ref{mm.lower}, showing that $\mathcal{F}_m^{\Gamma}(N,[g])$ is finite, and bounded below by the length-normalized first Steklov eigenvalue for any $\Gamma$-invariant metric conformal to $g$.

\begin{lemma}
For $m$ sufficiently large, $\mathcal{F}_m^{\Gamma}(N,[g])<\infty$.
\end{lemma}

\begin{proof}
Arguing exactly as in the proof Lemma \ref{mm.ubd}, for $m$ sufficiently large, we can find a $\Gamma$-equivariant embedding
$$
u_0\colon N\to \Sph^{m|\Gamma|-1}\subset A_{\Gamma}^m,
$$
and for each unit vector $a\in F_m\cap \B_1^{m|\Gamma|}(0)$ in the fixed set $F_m\subset A_{\Gamma}^m$ of $\Gamma$, we postcompose with the conformal dilation
$$
G_a(x)=\frac{(1-|a|^2)}{|x+a|^2}(x+a)+a
$$
to obtain a family of $\Gamma$-equivariant maps
$$
u_a=G_a\circ u_0\colon N\to \Sph^{m|\Gamma|-1},
$$
which we extend to a family on the closed unit ball $F_m\cap \overline{\B}$--which can be identified with $\overline{\B}^m\subset \mathbb{R}^m$--by setting
$$
u_a\equiv a\text{ for }a\in \Sph^{m|\Gamma|-1}\cap F_m.
$$
Note that in the setting with boundary, the condition that $u_0$ maps the interior of $N$ to $\Sph^{m|\Gamma|-1}$ is unnecessary, and the upper bound on $\mathcal{F}_m^{\Gamma}(N,[g])$ obtained from this family will not be sharp. 

Arguing as in the proof of Lemma \ref{mm.ubd}, we see that the family $a\mapsto u_a$ is continuous in $L^p(\partial N)$ for every $p<\infty$, with a uniform upper bound
$$
\sup_{a\in \B^m}E(u_a)\leq \mathcal{W}(u_0(N))
$$
on the Dirichlet energy by the Willmore energy of the initial map $u_0$, and mollifying with the heat kernel $K_t(x,y)$ for a short time yields a continuous family
$$
\B^m\cong F_m\cap \overline{\B}_1^{m|\Gamma|}(0)\ni a\mapsto u_a\in W^{1,2}_{\Gamma}(N,A_{\Gamma}^m)
$$
with 
$$
E(u^t_a)\leq E(u_a)\leq \mathcal{W}(u_0)
$$
and 
$$
\|u_a^t-u_a\|_{L^2(N)}^2\leq Ct,
$$
which belongs to the collection $\mathcal{B}_m$, since $u_a\equiv a$ for $a\in \Sph^{m|\Gamma|-1}\cap F_m$.
Further, standard estimates for the trace embedding $W^{1,2}(N)\to L^2(\partial N)$ give
$$
\|f\|_{L^2(\partial N)}^2\leq C(N,g)\left[\|f\|_{L^2(N)}^2+\|f\|_{L^2(N)}\|df\|_{L^2(N)}\right],
$$
which together with the preceding estimates in the case $f=u_a^t-u_a$ gives
$$
\|u_a^t-u_a\|_{L^2(\partial N)}^2\leq C(N,g)t^{1/2}
$$
for $t\in (0,1)$. In particular, since $|u_a|\equiv 1$, it follows that
$$
\int_{\partial N}(|u_a^t|^2-1)\,ds_g\leq \int_{\partial N}|u_a^t-u_a|^2\,ds_g\leq C(N,g)\sqrt{t},
$$
which together with the uniform bound on the Dirichlet energy $E(u_a^t)$ implies
$$
\inf_{t>0}\max_{a\in \B^m}F_{\epsilon}(u_a^t)\leq \mathcal{W}(u_0)
$$
for every $\epsilon>0$, giving the desired finite upper bound
$$
\mathcal{F}_m^{\Gamma}(N,[g])\leq \mathcal{W}(u_0)<\infty
$$
for $m$ large.
\end{proof}

Analogous to Proposition \ref{mm.lower} in the closed setting, we also have the following sharp lower bound.

\begin{lemma}\label{fb.lbd}
There is a constant $C(N,g)$ such that, for every $\epsilon>0$, 
$$2\mathcal{F}_{m,\epsilon}^{\Gamma}(N,g)\geq (1-C\epsilon)\bar{\sigma}_1(N,g),$$
and consequently
$$2\mathcal{F}_m^{\Gamma}(N,[g])\geq \bar{\sigma}_1(N,g).$$
In particular,
$$
2\mathcal{F}_m^{\Gamma}(N,[g])\geq \Sigma_1^{\Gamma}(N,[g]).
$$
\end{lemma}
\begin{proof}
Fix an arbitrary family $(u_y)$ in $\mathcal{B}_m$. Just as in the proof of Proposition \ref{mm.lower}, defining the orthogonal projection map $P_{F_m}\colon A_{\Gamma}^m\to F_m$ onto the fixed point set of $\Gamma$, it follows from the definition of $\mathcal{B}_m$ that the map
$$
f\colon \B^m\to F_m\cong \mathbb{R}^m
$$
given by
$$
f(y):=\frac{\int_{\partial N}u_y ds_g}{\int_{\partial N}1 ds_g}
$$
restricts to the identity on the boundary $\Sph^{m-1}=\partial \B^m\cong \partial (\B^{m|\Gamma|}\cap F_m)$, and therefore there is at least one $y\in \B^m$ for which $P_{F_m}(\int_{\partial N}u_y ds_g)=0$. Since the $\Gamma$-equivariance of $u_y\in W^{1,2}_{\Gamma}(N,A_{\Gamma}^m)$ forces $\int_{\partial N}u_y ds_g\in F_m$, this forces
$$
\int_{\partial N}u_yds_g=0.
$$
For this $y\in \B^m$, the variational characterization of $\sigma_1(N,g)$ yields the inequality
\begin{eqnarray*}
\int_N|du_y|^2\,dv_g&\geq& \sigma_1(N,g)\int_{\partial N}|u_y|^2\,ds_g\\
&=&\bar{\sigma}_1(N,g)-\sigma_1(N,g)\int_{\partial N}(1-|u_y|^2)ds_g,
\end{eqnarray*}
which together with the defintion of $F_{\epsilon}$ and a simple application of the Cauchy-Schwarz inequality gives
$$
(1-C\epsilon)\bar{\sigma}_1(N,g)\leq 2 F_{\epsilon}(u_y)\leq 2\max_{y\in \B^m}F_{\epsilon}(u_y)
$$
for a suitable constant $C=C(N,g)$. Taking the infimum over all families $(u_y)$ in $\B^m$ gives the desired inequality.
%$$
%(1-C\epsilon)\bar{\sigma}_1(N,g)\leq 2\mathcal{F}_{m,\epsilon}^{\Gamma}(N,g).
%$$
\end{proof}

Finally, combining Claim \ref{fb.ps} with the results of \cite[Chapter 10]{Ghou93}, we have the following preliminary existence result for fixed $m\in \mathbb{N}$ and $\epsilon>0$.

\begin{lemma}\label{fb.gl.ex}
For every $\epsilon>0$, $\mathcal{F}_{m,\epsilon}^{\Gamma}$ is realized as the energy $F_{\epsilon}(u_{m,\epsilon})$ of a $\Gamma$-equivariant critical point $u_{m,\epsilon}\in W^{1,2}_{\Gamma}(N,A_{\Gamma}^m)$ of $\Gamma$-equivariant Morse index
$$
\ind_{F_{\epsilon},\Gamma}(u_{m,\epsilon})\leq m.
$$
\end{lemma}

As discussed in \cite[Theorem 3.2]{MS13} in the Euclidean setting, smoothness of critical points for $F_{\epsilon}$ can be deduced, for instance, from arguments of \cite{CSM05}; it is straightforward to adapt these and other estimates from \cite{MS13} to general surfaces $N$, e.g. by covering a neighborhood of $\partial N$ with half-disks conformally equivalent to the standard one $D_1^+$, where critical points $u$ of $F_{\epsilon}$ on $N$ can be identified with solutions of
$$\Delta u=0\text{ on }D_1^+,\text{ }\frac{\partial u}{\partial\nu}=\rho \frac{1-|u|^2}{\epsilon}u \text{ on }[-1,1]\times\{0\},$$
where $\frac{1}{2}\leq \rho\leq 2$ is a fixed smooth conformal weight.

In the next subsections, under the assumption $\Sigma_1^{\Gamma}(N,[g])>2\pi$, we argue as in the closed case that for $m$ sufficiently large these maps converge as $\epsilon\to 0$ to $\Gamma$-equivariant free boundary harmonic maps $u_m\colon(N,\partial N)\to (\B^{m|\Gamma|},\Sph^{m|\Gamma|-1})$ with uniform gradient bounds, and show that the metrics $g_m=|du_m|_g^2g$ realize $\Sigma_1^{\Gamma}(N,[g])$ for $m$ sufficiently large.

\subsection{Asymptotics for the %free-boundary
 min-max maps as $\epsilon\to 0$ and $m\to\infty$.}

Similar to the closed case, the first key step in the asymptotic analysis of the maps $u_{m,\epsilon}$ for large $m$ is the following lemma, showing that the index upper bound $\ind_{\Gamma,F_{\epsilon}}(u_{m,\epsilon})\leq m$ forces the first Steklov eigenvalue of a naturally associated metric to approach $1$ as $m$ becomes large.

\begin{lemma}\label{mm.stek.lower}
Let $u\in W^{1,2}_{\Gamma}(N,A_{\Gamma}^m)$ be a critical point of $F_{\epsilon}$ on $(N,g)$ with 
$$
\ind_{\Gamma,F_{\epsilon}}(u)\leq m.
$$
Then $g_u:=\frac{(1-|u|^2)^2}{\epsilon^2}g$ defines a $\Gamma$-invariant metric in $[g]$, and there is a universal constant $C<\infty$ such that for $m>C$, the first Steklov eigenvalue $\sigma_1(N,g_u)$ satisfies a lower bound of the form
$$\sigma_1(N,g_u)\geq 1-\frac{C}{m}.$$
\end{lemma}

\begin{proof}
The proof is similar to that of Proposition \ref{index.eigen.bd}. We use the same vector fields $v_{\phi,w}$ and the same pointwise lower bounds from Proposition~\ref{prop:v_lower}. Similar to the closed case, we will obtain the lower bound on the first Steklov eigenvalue $\sigma_1(N,g_u)$ for the metric $g_u=\frac{(1-|u|^2)^2}{\epsilon^2}g$ by estimating the second variation $F_{\epsilon}''(u)$ on the space $V_{\phi_0,w}+V_{\phi_1,w}$ where $\phi_0=1$ and $\phi_1$ is a first Steklov eigenfunction for $(N,g_u)$. The main practical difference is in the equation satisfied by $u$, so we need to adapt those parts of the proof that rely on it.

% modulo some obvious modifications. As in the proof of Lemma \ref{index.eigen.bd}, denote by $W_m\subset A_{\Gamma}^m$ the $m$-dimensional subspace spanned by $(e_I,0,\ldots,0),\ldots, (0,\ldots,0,e_I)$, so that for any unit vector $w\in W_m$, the orbit 
%$$\Gamma\cdot w=\{w_{\gamma}=\gamma\cdot w\mid \gamma\in \Gamma\}$$
%gives an orthonormal basis for a subspace on which $\Gamma$ acts by the regular representation. And for $\phi\in W^{1,2}(N)$ and $w\in W_m$, define 
%$$v_{\phi,w}=\sum_{\gamma\in \Gamma}(\phi\circ \gamma^{-1})(w_{\gamma}-\langle u,w_{\gamma}\rangle u)\in W^{1,2}_{\Gamma}(N,A_{\Gamma}^m)$$
%and
%$$V_{\phi}:=\{v_{\phi,w}\mid w\in W_m\}\subset W^{1,2}_{\Gamma}(N,A_{\Gamma}^m)$$
%as in the proof of Lemma \ref{index.eigen.bd} as well. 
%Many of the basic properties of the fields $v_{\phi,w}$ described in the proof of Lemma \ref{index.eigen.bd} are clearly unaffected by the presence of a nonempty boundary $\partial N\neq \varnothing$ or the particular equations satisfied by the $\Gamma$-equivariant map $u$ , and therefore continue to hold verbatim in this setting. In particular, it is still the case that
%$$v_{\phi,w}=0\text{ if and only if }\phi=0\text{ or }w=0,$$
%and the various pointwise estimates--e.g., \eqref{v.lower}-\eqref{pt.cross.2}--for $v_{\phi,w}$ and $dv_{\phi,w}$ remain true for the $F_{\epsilon}$-critical map $u\in W^{1,2}_{\Gamma}(N,A_{\Gamma}^m)$.

Recalling Proposition \ref{fb.ps}(ii) and appealing to the pointwise estimates of Proposition~\ref{prop:v_lower}, we see that the second variation $F_{\epsilon}''(u)$ in the direction $v_{\phi,w}$ for each unit vector $w\in W_m$ satisfies
\begin{equation*}
\begin{split}
&F_{\epsilon}''(u)(v_{\phi,w},v_{\phi,w})=\int_N|dv_{\phi,w}|^2+\int_{\partial N}\left(2\frac{\langle u,v_{\phi,w}\rangle^2}{\epsilon}-\frac{1-|u|^2}{\epsilon}|v_{\phi,w}|^2\right) \\
&\leq  \int_N(1+C|u^w|^2)\sum_{\gamma\in\Gamma}|d(\phi\circ \gamma)|^2
+C\int_N\sum_{\gamma\in \Gamma}(\phi\circ \gamma)^2(|du^w|^2+|u^w|^2|du|^2) \\
&+\int_{\partial N}\frac{2}{\epsilon}(1-|u|^2)^2|u^w|^2\sum_{\gamma\in \Gamma}(\phi\circ \gamma)^2
-\int_{\partial N}\frac{1-|u|^2}{\epsilon}(1-C|u^w|^2)\sum_{\gamma\in \Gamma}(\phi\circ \gamma)^2.
\end{split}
\end{equation*}
 In view of the equivariance of $u$ and $u^w$ under the action of $\Gamma\leq \Isom(N)$, it follows that
\begin{equation*}
\begin{split}
&F_{\epsilon}''(u)(v_{\phi,w},v_{\phi,w})\leq |\Gamma|\left(\int_N|d\phi|^2-\int_{\partial N}\frac{1-|u|^2}{\epsilon}\phi^2\right)\\
&+C|\Gamma|\int_N|u^w|^2|d\phi|^2+(|du^w|^2+|u^w|^2|du|^2)\phi^2
+C|\Gamma|\int_{\partial N}|u^w|^2\frac{1-|u|^2}{\epsilon}\phi^2.
\end{split}
\end{equation*}
Defining $J_u\colon W_m\times W^{1,2}(N)\to \mathbb{R}$ by
\begin{equation*}
\begin{split}
J_u(w,\phi)=&|w|^2\int_N|u^w|^2|d\phi|^2+(|du^w|^2+|u^w|^2|du|^2)\phi^2\,dv_g\\
+&|w|^2\int_{\partial N}|u^w|^2\frac{1-|u|^2}{\epsilon}\phi^2\,ds_g,
\end{split}
\end{equation*}
we see that both $w\mapsto J_u(w,\phi)$ and $\phi\mapsto J_u(w,\phi)$ define nonnegative quadratic forms, and we can recast the preceding estimate as
$$
F_{\epsilon}''(u)(v_{\phi,w},v_{\phi,w})\leq |\Gamma||w|^2\left(\int_N|d\phi|_g^2\,dv_g-\int_{\partial N}\phi^2ds_{g_u}\right)+C|\Gamma|J_u(w,\phi).
$$
As before, to estimate the mixed terms we first have 
\begin{equation*}
\begin{split}
&\phi_0\left|\int_N\sum_{\gamma,\tau}\langle u,z^\tau\rangle\left\langle d\left(\phi_1\circ\gamma^{-1}\right), \langle du,w^\gamma\rangle\right\rangle\,dv_g\right| \leq \\
&\phi_0\left|\int_N\sum_{\gamma,\tau}(\phi_1\circ\gamma^{-1})\langle \langle du,z^\tau\rangle,
\langle du,w^\gamma\rangle\rangle\,dv_g\right|+ \\
&\phi_0\left|\int_{\bd N}\sum_{\gamma,\tau}(\phi_1\circ\gamma^{-1})\frac{(1-|u|^2)}{\eps}\langle u,z^\tau\rangle\langle 
u,w^\gamma\rangle\,ds_g\right|\leq
 \frac{C|\Gamma|}{|z||w|}\sqrt{J_u(z,\phi_0)J_u(w,\phi_1)}.
\end{split}
\end{equation*}
Integrating the right hand sides of ~\eqref{pt.cross.2},~\eqref{pt.cross}, and~\eqref{pt.cross.3} then gives that
$$
\left|F_{\epsilon}''(u)(v_{\phi_0,w},v_{\phi_1,z})-|\Gamma|\int_{\partial N}\langle w,z\rangle\phi_1 ds_{g_u}\right|\leq C|\Gamma|\sqrt{J_u(w,\phi_0)J_u(z,\phi_1)}.
$$

In particular, letting $\phi_0\equiv 1$, and $\phi_1$ be a first Steklov eigenfunction for $(N,g_u)$. Then the computations above yield
\begin{equation}\label{fj.1}
F_{\epsilon}''(u)(v_{\phi_0,w},v_{\phi_0,w})\leq -|\Gamma||w|^2\|\phi_0\|_{L^2_{g_u}(\partial N)}^2+C|\Gamma|J_u(w,\phi_0),
\end{equation}
\begin{equation}\label{fj.2}
F_{\epsilon}''(u)(v_{\phi_1,w},v_{\phi_1,w})\leq (\sigma_1(g_u)-1)|\Gamma||w|^2\|\phi_1\|_{L^2_{g_u}(\partial N)}^2+C|\Gamma|J_u(w,\phi_1),
\end{equation}
and
\begin{equation}\label{fj.3}
|F_{\epsilon}''(u)(v_{\phi_0,w},v_{\phi_1,z})|\leq C|\Gamma|\sqrt{J_u(w,\phi_0)J_u(z,\phi_1)}.
\end{equation}

For any $\phi\in C^{\infty}(N)$, diagonalizing the quadratic form $W_m\in w\mapsto J_u(w,\phi)$ gives an orthonormal basis $w_1,\ldots,w_m$ for $W_m$, with respect to which we can compute the trace
$$
\sum_{j=1}^mJ_u(w_j,\phi)\leq\int_N\left(|u|^2|d\phi|_g^2+2|du|_g^2\phi^2\right)\,dv_g+\int_{\partial N}\frac{1-|u|^2}{\epsilon}\phi^2\,ds_g
$$
of the quadratic form $w\mapsto J_u(w,\phi)$. Now, since $u$ is a critical point of $F_{\epsilon}$, recall from Proposition \ref{fb.ps} that it solves the equations
$$\Delta u=0\text{ on }N\text{ and }\frac{\partial u}{\partial\nu}=\frac{1-|u|^2}{\epsilon}u,$$
and therefore we can estimate
\begin{equation*}
\begin{split}
\int_N \phi^2|du|^2_g\,dv_g &= \int_{\bd N} \phi^2\frac{1-|u|^2}{\epsilon}|u|^2\,ds_g - \int_N\langle 2\phi d\phi,\langle u,du\rangle\rangle\,dv_g\\
&\leq \int_{\bd N} \phi^2\frac{1-|u|^2}{\epsilon}\,ds_g + \int_N\frac{1}{2} \phi^2|du|^2_g + 2|d\phi|_g^2\,dv_g
\end{split}
\end{equation*}
%
%
%\begin{eqnarray*}
%\int_N2|du|^2\phi^2&=&\int_Ndiv(\phi^2d|u|^2)-\langle d(\phi^2),d|u|^2\rangle\\
%&=&\int_{\partial N}2\phi^2\langle \frac{\partial u}{\partial \nu},u\rangle-\int_N4\phi|u|\langle d\phi,d|u|\rangle\\
%&=&\int_{\partial N}\frac{2(1-|u|^2)}{\epsilon}|u|^2\phi^2-4\int_N\phi|u|\langle d\phi,d|u|\\
%&\leq &\int_{\partial N}2\phi^2ds_{g_u}+4\int_N|d\phi|^2+\int_N|du|^2\phi^2,
%\end{eqnarray*}
%where the last inequality follows from a simple application of Cauchy-Schwarz. 
In particular, we deduce that 
$$
\int_N|du|^2\phi^2\,dv_g\leq 2\int_{\partial N}\phi^2ds_{g_u}+4\int_N|d\phi|^2_g\,dv_g,
$$
and applying this in the preceding formula for the trace of $w\mapsto J_u(w,\phi)$, we obtain the estimate
$$\sum_{j=1}^mJ_u(w_j,\phi)\leq 9\int_N|d\phi|^2+5\int_{\partial N}\phi^2 ds_{g_u}.$$
Assuming without loss of generality that $\sigma_1(N,g_u)\leq 1$ and letting $\phi=\phi_0$ or $\phi_1$, it follows that
\begin{equation}
\sum_{j=1}^mJ_u(w_j,\phi)\leq 14\int_{\partial N}\phi^2ds_{g_u}.
\end{equation}

In particular, for any $\delta>0$, it follows that the set $S_{\delta}(\phi)\subset\{1,\ldots,m\}$ of indices for which $J_u(w_j,\phi)\geq \delta \|\phi\|_{L^2_{g_u}(\partial N)}^2$ must have size at most
$$
|S_{\delta}(\phi)|\leq \frac{8}{\delta},
$$
and therefore the space
$$
Z_{\delta}(\phi):=\mathrm{Span}\{w_j\mid j\notin S_{\delta}(\phi)\}\subset W_m
$$
on which $J_u(w,\phi)\leq \delta |w|^2\|\phi\|_{L^2_{g_u}(\partial N)}^2$ must have complementary dimension
$$
\dim Z_{\delta}(\phi)\geq m-|S_{\delta}(\phi)|\geq m-\frac{14}{\delta}.
$$
Once again setting
$$
V_{\delta}(\phi):=\{v_{\phi,w}\mid w\in Z_{\delta}(\phi)\}\subset W_{\Gamma}^{1,2}(N,A_{\Gamma}^m),
$$
it follows from the injectivity of $w\mapsto v_{\phi,w}$ that
$$
\dim(V_{\delta}(\phi))\geq m-\frac{14}{\delta},
$$
and letting
$$V_{\delta}:=V_{\delta}(\phi_0)+V_{\delta}(\phi_1),$$
we therefore have
$$\dim(V_{\delta})\geq 2m-\frac{30}{\delta}-\dim(V_{\delta}(\phi_0)\cap V_{\delta}(\phi_1)).$$
Moreover, if $v=v_{\phi_0,w}=v_{\phi_1,z}\in V_{\delta}(\phi_0)\cap V_{\delta}(\phi_1),$ where $w\in Z_{\delta}(\phi_0)$ and $z\in Z_{\delta}(\phi_1)$, then by \eqref{v.lower} and the definition of $J_u$, we have
\begin{eqnarray*}
\|v_{\phi_0,w}\|_{L^2(\partial N, g_u)}^2& \geq & |\Gamma||w|^2\|\phi_0\|_{L^2_{g_u}(\partial N)}^2-C|\Gamma|J_u(w,\phi_0)\\
&\geq& |\Gamma||w|^2\|\phi_0\|_{L^2_{g_u}(\partial N)}^2-C|\Gamma|\delta |w|^2\|\phi_0\|_{L^2_{g_u}(\partial N)}^2,
\end{eqnarray*}
and similarly
\begin{eqnarray*}
\|v_{\phi_1,z}\|_{L^2(\partial N, g_u)}^2& \geq & |\Gamma||z|^2\|\phi_1\|_{L^2_{g_u}(\partial N)}^2-C|\Gamma|J_u(z,\phi_1)\\
&\geq& |\Gamma||z|^2\|\phi_1\|_{L^2_{g_u}(\partial N)}^2-C|\Gamma|\delta |z|^2\|\phi_1\|_{L^2_{g_u}(\partial N)}^2.
\end{eqnarray*}
On the other hand, by \eqref{pt.cross}, and the orthogonality of $\phi_0,\phi_1$ in $L^2_{g_u}(\partial N)$, we see that
\begin{eqnarray*}
\langle v_{\phi_0,w},v_{\phi_1,z}\rangle_{L^2_{g_u}(\partial N)}&\leq& C|\Gamma|\sqrt{J_u(w,\phi_0)J_u(z,\phi_1)}\\
&\leq & 8C|\Gamma|\delta|w||z|\|\phi_0\|_{L^2_{g_u}(\partial N)}\|\phi_1\|_{L^2_{g_u}(\partial N)},
\end{eqnarray*}
and since $v=v_{\phi_0,w}=v_{\phi_1,z}$, the left-hand sides of the preceding three inequalities are equal, so we can combine these three estimates to obtain an estimate of the form
$$|w|^2\|\phi_0\|_{L^2_{g_u}(\partial N)}^2+|z|^2\|\phi_1\|_{L^2_{g_u}(\partial N)}^2\leq C'\delta(|w|^2\|\phi_0\|_{L^2_{g_u}(\partial N)}^2+|z|^2\|\phi_1\|_{L^2_{g_u}(\partial N)}^2),$$
which clearly cannot hold for $\delta<\frac{1}{C'}$ unless $v=0$. Thus, for $0<\delta<\frac{1}{C'}$, we must have
$$V_{\delta}(\phi_0)\cap V_{\delta}(\phi_1)=0,$$
and therefore
$$\dim(V_{\delta})\geq 2m-\frac{30}{\delta}.$$
Taking $m$ large enough that $\delta_m=\frac{30}{m}<\frac{1}{C'}$, it follows in particular that
\begin{equation}\label{dim.big}
\dim(V_{\delta_m})>m.
\end{equation}

Finally, it follows from \eqref{dim.big} and the index bound $\ind_{F_{\epsilon},\Gamma}(u)\leq m$ that there exists $v=v_{\phi_0,w}+v_{\phi_1,z}\in V_{\delta}$ for which
$$F_{\epsilon}''(u)(v,v)>0,$$
while by \eqref{fj.1}-\eqref{fj.3} and the definition of $V_{\delta_m}$, we see that
\begin{eqnarray*}
F_{\epsilon}''(u)(v,v)&\leq &-|\Gamma||w|^2\|\phi_0\|_{L^2_{g_u}(\partial N)}^2+(\sigma_1(N,g_u)-1)|\Gamma||z|^2\|\phi_1\|_{L^2_{g_u}}\\
&&+C'|\Gamma|(J_u(w,\phi_0)+J_u(z,\phi_1))\\
&\leq &|\Gamma|(C'\delta_m-1)|w|^2\|\phi_0\|_{L^2_{g_u}(\partial N)}^2\\
&&+|\Gamma|(C'\delta_m+[\sigma_1(N,g_u)-1])|z|^2\|\phi_1\|_{L^2_{g_u}}.
\end{eqnarray*}
Clearly $C'\delta_m=C'\cdot\frac{30}{m}<|\Gamma|$ for $m$ sufficiently large, so combining the two inequalities, we see that we must have
$$C'\cdot\frac{30}{m}=C'\delta_m\geq 1-\sigma_1(N,g_u),$$
which we can rearrange to arrive at an estimate of the desired form
\begin{equation*}
\sigma_1(N,g_u)\geq 1-\frac{30C'}{m}.\qedhere
\end{equation*}
\end{proof}

Just as Proposition \ref{index.eigen.bd} allowed us to rule out energy concentration for the min-max critical points of $E_{\epsilon}$ on the closed surface $M$, Lemma \ref{mm.stek.lower} will enable us to rule out energy concentration for the maps $u_{m,\epsilon}$ at the boundary $\partial N$, assuming $\Sigma_1^{\Gamma}(N,[g])>2\pi$. As a first step, we have the following.

\begin{lemma}\label{no.dir.conc}
Assume either $\Sigma_1^{\Gamma}(N,[g])>2\pi$ or $\Gamma$ has no fixed points on $\partial N$. Then there exists $r_0(N,g)>0$, $m_0(N,g)\in \mathbb{N}$ such that if $m\geq m_0$, $\epsilon<\epsilon_0(N,g,m)$ is sufficiently small, and $u=u_{m,\epsilon}\in W^{1,2}_{\Gamma}(N,A_{\Gamma}^m)$ is the min-max critical point of $F_{\epsilon}$ realizing $\mathcal{F}_{m,\epsilon}^{\Gamma}(N,g)$
then on any geodesic (half-)disk $D_r(p)\subset N$ with $r<r_0$, we have
$$
\frac{1}{2}\int_{\partial N}\phi^2\frac{(1-|u|^2)}{\epsilon}ds_g\leq \int_N|d\phi|^2\,dv_g
$$
for all $\phi\in C_c^{\infty}(D_r(p))$.
\end{lemma}
\begin{proof}
We first complete the proof in the case where $\Sigma_1^{\Gamma}(N,[g])>2\pi$.  
Since $\ind_{\Gamma,F_{\epsilon}}(u_{m,\epsilon})\leq m$, for $m>C$, Lemma \ref{mm.stek.lower} gives
$$\sigma_1(N,g_{u})\geq 1-\frac{C}{m}\geq \frac{1}{2},$$
and by the variational characterization of $\sigma_1(N,g_u)$, for any $D_r(p)\subset N$, we have either
\begin{equation}\label{no.eigen.conc}
\frac{1}{2}\int_{\partial N}\phi^2 \frac{(1-|u|^2)}{\epsilon}ds_g\leq \int_N|d\phi|^2\text{ for all }\phi\in W_0^{1,2}(D_r(p))
\end{equation}
or 
\begin{equation}\label{eigen.conc}
\frac{1}{2}\int_{\partial N}\phi^2 \frac{(1-|u|^2)}{\epsilon}ds_g\leq \int_N|d\phi|^2\text{ for all }\phi\in W_0^{1,2}(N\setminus D_r(p)).
\end{equation}
We wish to show that \eqref{no.eigen.conc} holds for small $r$ and small $\epsilon$. Suppose, to the contrary, that \eqref{no.eigen.conc} fails on some small (half-)disk $D_r(p)$; then \eqref{eigen.conc} holds, and taking $\phi\in W_0^{1,2}(N\setminus D_r(p))$ to be a logarithmic cutoff function such that $\phi\equiv 1$ on $N\setminus D_{\sqrt{r}}(p)$ and $\int_N |d\phi|^2\leq \frac{C}{|\log r|}$, we deduce that
$$
\int_{\partial N\setminus D_{\sqrt{r}}(p)}\frac{(1-|u|^2)}{\epsilon}\,ds_g\leq \frac{C}{|\log r|}.
$$
In particular, since 
\begin{align*}
\int_{\partial N}\frac{(1-|u|^2)}{\epsilon}\,ds_g
&=\int_{\partial N}\frac{(1-|u|^2)^2}{\epsilon}+\frac{1}{2}\left\langle u,\frac{\partial u}{\partial\nu}\right\rangle\,ds_g
\\
&=\int_{\partial N}\frac{(1-|u|^2)^2}{\epsilon}\,ds_g+\int_N|du|^2,
\end{align*}
it follows that
$$
\int_{\partial N\cap D_{\sqrt{r}}(p)}\frac{1-|u|^2}{\epsilon}\,ds_g\geq F_{\epsilon}(u)-\frac{C}{|\log r|}=\mathcal{F}_{m,\epsilon}^{\Gamma}(N,g)-\frac{C}{|\log r|}.
$$
Now, fix another log cutoff function $\psi\in C_c^{\infty}(D_{r^{1/4}}(p))$ such that $\psi\equiv 1$ on $D_{\sqrt{r}}(p)$ and $\int |d\psi|^2\leq \frac{C'}{|\log r|}.$ For $r<r_1(N)$ small enough, we can apply the Weinstock analog of \cite[Lemma 3.1]{Kok.var} to find a conformal map 
$$F=(F_1,F_2): (D_{r^{1/4}}(p),\partial N\cap D_{r^{1/4}}(p))\to (D_1(0),\partial D_1(0))$$
of energy $E(F)\leq \pi$ such that
$$\int_{\partial N}\frac{(1-|u|^2)}{\epsilon}\psi F_i ds_g=0\text{ for }i=1,2.$$
It follows that 
\begin{eqnarray*}
\int_N|d(\psi F_i)|^2&\geq& \sigma_1(N,g_u)\int_{\partial N}\frac{(1-|u|^2)}{\epsilon}\psi^2F_i^2 ds_g\\
&\geq& \left(1-\frac{C}{m}\right)\int_{\partial N}\psi^2F_i^2\frac{(1-|u|^2)}{\epsilon}ds_g,
\end{eqnarray*}
and summing over $i=1,2$ gives
\begin{eqnarray*}
\int_N|d(\psi F)|^2&\geq& \left(1-\frac{C}{m}\right)\int_{\partial N}\psi^2\frac{(1-|u|^2)}{\epsilon}ds_g\\
&\geq& \left(1-\frac{C}{m}\right)\left(2\mathcal{F}_{m,\epsilon}^{\Gamma}(N,g)-\frac{C}{|\log r|}\right).
\end{eqnarray*}
On the other hand, since $F$ maps $D_{r^{1/4}}(p)$ conformally onto a subset of $D_1(0)$, we see that
\begin{eqnarray*}
\int_N|d(\psi F)|^2&\leq & \int_N\psi^2|dF|^2+2\langle \psi dF, Fd\psi\rangle+|F|^2|d\psi|^2\\
&\leq & 2\pi+2\sqrt{2\pi}\|d\psi\|_{L^2(N)}+\|d\psi\|_{L^2(N)}^2\\
&\leq &2\pi+\frac{C}{\sqrt{|\log r|}},
\end{eqnarray*}
and combining this with the previous estimate gives
$$\left(1-\frac{C}{m}\right)2\mathcal{F}_{m,\epsilon}^{\Gamma}(N,g)\leq 2\pi +\frac{C}{\sqrt{|\log r|}}.$$

In particular, letting
$$r_{m,\epsilon}:=\max\{s\mid \eqref{no.eigen.conc}\text{ holds for all }r\leq s\},$$
it follows that
$$\left(1-\frac{C}{m}\right)2\mathcal{F}_{m,\epsilon}^{\Gamma}(N,g)\leq 2\pi+\frac{C}{\sqrt{|\log r_{m,\epsilon}|}},$$
and taking $\epsilon\to 0$ and applying Lemma \ref{fb.lbd}, we deduce that
$$\left(1-\frac{C}{m}\right)\Sigma_1^{\Gamma}(N,[g])\leq 2\pi+\frac{C}{\sqrt{|\log r_m|}},$$
where $r_m:=\liminf_{\epsilon\to 0}r_{m,\epsilon}.$ But by assumption, $\Sigma_1^{\Gamma}(N,[g])>2\pi$, and it follows that
$$\liminf_{m\to\infty}r_m>0.$$
Taking $r_0(N,g):=\frac{1}{2}\liminf_{m\to\infty}r_m$, we arrive at the desired conclusion.

By a straightforward analog of the argument at the end of Proposition \ref{fixpt.rk1}, we see that the  conclusion holds if $\Gamma$ has no fixed points on the boundary $\partial N$.
\end{proof}

Next, we wish to conclude that, for $m$ sufficiently large, as $\epsilon\to 0$, the min-max critical points $u_{m,\epsilon}\in W^{1,2}_{\Gamma}(N,A_{\Gamma}^m)$ converge smoothly to a free boundary harmonic map $u_m\colon (N,\partial N)\to (\B^{|\Gamma|m},\Sph^{|\Gamma|m-1})$ with $\Gamma$-equivariant Morse index $\leq m$ as a critical point for the energy functional. To this end, we record the following small-energy regularity result from \cite{MS13}.

\begin{proposition}[{\cite[Theorem 5.1]{MS13}}]
\label{fbgl.ereg}
There exists $C(m,N,g)<\infty$, \newline $\eta(m,N,g)>0$ and $r_0(N,g)>0$ such that if $u\in W^{1,2}_{\Gamma}(N,g)$ is a critical point of $F_{\epsilon}$ on $(N,g)$, and for some $p\in \partial N$ and $r<r_0(N)$,
\begin{equation}
\int_{D_r(p)}|du|^2\,dv_g+\int_{D_r(p)\cap \partial N}\frac{(1-|u|^2)^2}{4\epsilon}\,ds_g\leq \eta,
\end{equation}
then 
$$\|du\|_{L^{\infty}(D_{r/4}(p))}^2+\epsilon^{-2}\|(1-|u|^2)^2\|_{L^{\infty}(\partial N\cap D_{r/4}(p))}\leq \frac{C}{r^2}\eta.$$
\end{proposition}

\begin{remark}
Though \cite[Thm 5.1]{MS13} is stated for solutions on the half-disk $D_1^+(0)\subset \mathbb{R}^2_+$, the proof is easily adapted to the case of general surfaces with boundary; in particular, on small enough geodesic (half-)disks $D_r(p)$ for $p\in\partial N$, after identifying $D_r(p)$ conformally with the half-disk $D_1^+(0)$, critical points of $F_{\epsilon}$ correspond to solutions of
$$\Delta v=0\text{ on }D_1^+(0),\text{ }\frac{\partial v}{\partial \nu}=\rho \frac{(1-|v|^2)}{\epsilon}v\text{ on }[-1,1]\times\{0\},$$
with $\rho$ a conformal weight that can be taken arbitrarily close to $1$ in a $C^k$ sense.
\end{remark}

Combining Proposition \ref{fbgl.ereg} with Lemma \ref{no.dir.conc}, it is not difficult to show that, for $m$ sufficiently large, the min-max critical points $u_{m,\epsilon}\colon N\to A_{\Gamma}^m$ of $F_{\epsilon}$ converge as $\epsilon\to 0$ to free boundary harmonic maps $u_m\colon (N,\partial N)\to (\B^{m|\Gamma|},\Sph^{m|\Gamma|-1})$ satisfying the desired index control and quantative non-concentration estimates.

\begin{proposition}\label{fbhm.ex}
Suppose either $\Sigma_1^{\Gamma}(N,[g])>2\pi$ or  $\Gamma$ has no fixed points on $\bd N$. For $m$ sufficiently large, there exist $\Gamma$-equivariant free boundary harmonic maps 
$$
u_m\colon (N,\partial N)\to (\B^{m|\Gamma|},\Sph^{m|\Gamma|-1})
$$
of energy 
$$
E(u_m)=\mathcal{F}_m^{\Gamma}(N,[g]),
$$
satisfying the $\Gamma$-equivariant Morse index bound
$$
\ind_{E,\Gamma}(u_m)\leq m,
$$
and the non-concentration estimate
\begin{equation}\label{fbhm.no.conc}
\int_{D_r(p)}|du_m|^2+\int_{D_r(p)\cap \partial N}\left|\frac{\partial u_m}{\partial\nu}\right|\,ds_g\leq \frac{C}{|\log r|}
\end{equation}
for all $p \in \partial N$ and $r\in (0,r_0(N,g))$, for $C(N,[g])$ independent of $m$.
\end{proposition}

\begin{proof}
First, by Lemma \ref{no.dir.conc}, for $m$ sufficiently large and $\epsilon<\epsilon_0(N,g,m)$ sufficiently small, we know that for some $r_0(N,g)>0$ and all $p\in \partial N$, we have
$$
\frac{1}{2}\int_{\partial N}\phi^2\frac{(1-|u_{m,\epsilon}|^2)}{\epsilon}ds_g\leq \int_N|d\phi|^2\,dv_g,
$$
for all $\phi\in C_c^{\infty}(D_{r_0}(p))$. Moreover, writing
\begin{equation*}
\begin{split}
\int_{\partial N}\phi^2\frac{(1-|u_{m,\epsilon}|^2)}{\epsilon}\,ds_g&\geq \int_{\partial N}\phi^2\left\langle u_{m,\epsilon},\frac{\partial u_{m,\epsilon}}{\partial \nu}\right\rangle\,ds_g\\
=\int_N\mathrm{div}(\phi^2 \langle u_{m,\epsilon},du_{m,\epsilon}\rangle)\,dv_g
&=\int_N\phi^2|du_{m,\epsilon}|^2+2\langle u_{m,\epsilon} d\phi, \phi du_{m,\epsilon}\rangle\,dv_g\\
&\geq \frac{1}{2}\int_N\phi^2|du_{m,\epsilon}|^2\,dv_g-4\int_N|d\phi|^2\,dv_g,
\end{split}
\end{equation*}
we deduce that
$$
\frac{1}{2}\int_N\phi^2|du_{m,\epsilon}|^2\,dv_g\leq 6\int_N|d\phi|^2\,dv_g
$$
as well. Letting $\phi\in C_c^{\infty}(D_{r_0}(p))$ be a log cutoff function with $\phi\equiv 1$ on $D_{\delta}(p)$ and $\int |d\phi|^2\leq \frac{C}{\log(r_0/\delta)}$, we then see that
\begin{equation}
\label{m.e.no.conc}
\int_{D_{\delta}(p)}|du_{m,\epsilon}|^2\,dv_g+\int_{D_{\delta}(p)\cap \partial N}\frac{(1-|u_{m,\epsilon}|^2)}{\epsilon}\,ds_g\leq \frac{C}{\log(r_0/\delta)}
\end{equation}
for all $\delta<r_0(N,g)$. In particular, taking $r$ sufficiently small, it then follows from Proposition \ref{fbgl.ereg} that $u_{m,\epsilon}$ satisfies
$$\|u_{m,\epsilon}\|_{C^1(N)}+\left\|\frac{\partial u_{m,\epsilon}}{\partial \nu}\right\|_{C^0(\partial N)}\leq C(N,g,m)$$
for all $\epsilon<\epsilon_0(N,g,m)$. 

Passing to a subsequence as $\epsilon\to 0$, we can therefore find a map $u_m\colon N\to \B^{m|\Gamma|}$ such that 
$$
u_{m,\epsilon}\to u_m\text{ in }C^0(N)\text{ and weakly in }W^{1,2}(N).
$$
Clearly, the limit $u_m\colon N\to \B^{m|\Gamma|}$ must be $\Gamma$-equivariant and harmonic on $N$. Moreover, by the harmonicity of $u_{m,\epsilon}$ and $u_m$, we see that
\begin{eqnarray*}
\int_N|du_{m,\epsilon}|^2-|du_m|^2&=&\int_{\partial N}\left\langle u_{m,\epsilon},\frac{\partial u_{m,\epsilon}}{\partial\nu}\right\rangle-\left\langle u_m,\frac{\partial u_m}{\partial\nu}\right\rangle\,ds_g\\
&=&\int_{\partial N}\left\langle u_{m,\epsilon}-u_m,\frac{\partial u_{m,\epsilon}}{\partial\nu}\right\rangle+\left\langle u_m,\frac{\partial u_{m,\epsilon}}{\partial\nu}-\frac{\partial u_m}{\partial\nu}\right\rangle\\
&\leq &C\cdot \|u_{m,\epsilon}-u_m\|_{C^0}+\int_N\langle du,d u_{m,\epsilon}-du_m\rangle;
\end{eqnarray*}
since $u_{m,\epsilon}\to u_m$ in $C^0$ and weakly in $W^{1,2}$, it follows that the latter term vanishes as $\epsilon\to 0$, so that
$$u_{m,\epsilon}\to u_m\text{ strongly in }W^{1,2}$$
as well. In particular, since
$$
\int_{\partial N}\frac{(1-|u_{m,\epsilon}|^2)^2}{\epsilon}\,ds_g\to 0
$$
as $\epsilon\to 0$, it follows that
$$E(u_m)=\lim_{\epsilon\to 0}\int_N\frac{1}{2}|du_{m,\epsilon}|^2=\lim_{\epsilon\to 0}F_{\epsilon}(u_{m,\epsilon})=\mathcal{F}_m^{\Gamma}(N,[g]).$$

Next, note that if $v(x)\cdot u_m(x)\equiv 0$ on $\partial N$, then 
$$\|v(x)\cdot u_{m,\epsilon}(x)\|_{C^0(\partial N)}\to 0$$
as $\epsilon\to 0$, and consequently
$$
\int_{\partial N} \left\langle v,\frac{\partial u_m}{\partial\nu}\right\rangle=\lim_{\epsilon\to 0}\int_{\partial N}\left\langle v,\frac{\partial u_{m,\epsilon}}{\partial \nu}\right\rangle=\lim_{\epsilon\to 0}\int_{\partial N}\langle v,u\rangle\frac{(1-|u|^2)}{\epsilon}=0.
$$
Thus, in addition to being harmonic in the interior of $N$, $u_m$ must satisfy
$$
\frac{\partial u_m}{\partial\nu}=\left|\frac{\partial u_m}{\partial\nu}\right|u_m,
$$
so that $u_m$ is indeed a free boundary harmonic map $(N,\partial N)\to (\B^{m|\Gamma|},\Sph^{m|\Gamma|-1})$. 

To check that $\ind_{\Gamma}(u_m)\leq m$, let 
$$
\mathcal{V}(u_m):=\left\{v\in W^{1,2}_{\Gamma}(N,A_{\Gamma}^m)\cap C^{\infty}\mid \langle v,u_m\rangle \equiv 0\text{ on }\partial N\right\};
$$
we need to show that any subspace $V\subset \mathcal{V}(u_m)$ on which the quadratic form
\begin{equation*}
\begin{split}
E''(u_m)(v,v)&=\int_N|dv|^2\,dv_g-\int_{\partial N}\left|\frac{\partial u_m}{\partial \nu}\right||v|^2\,ds_g\\
&=\int_N|dv|^2-|du_m|^2|v|^2-\frac{1}{2}\left\langle d\left(|u_m|^2\right),d\left(|v|^2\right)\right\rangle\,dv_g
\end{split}
\end{equation*}
is negative-definite has dimension $\leq m$. To this end, for each $\epsilon>0$, consider the map
$$T_{\epsilon}:\mathcal{V}(u_m)\to W^{1,2}_{\Gamma}(N,A_{\Gamma}^m)$$
defined by 
$$T_{\epsilon}v=v_{\epsilon}:=v-\langle v,u_{m,\epsilon}-u_m\rangle u_{m,\epsilon}.$$
Then for any $v\in \mathcal{V}(u_m)$,
\begin{equation*}
\begin{split}
&F_{\epsilon}''(u_{m,\epsilon})(v_{\epsilon},v_{\epsilon})=\int_N|dv_{\epsilon}|^2\,dv_g+\int_{\partial N}\left(2\frac{\langle u_{m,\epsilon},v_{\epsilon}\rangle^2}{\epsilon}-\frac{(1-|u_{m,\epsilon}|^2)}{\epsilon}|v_{m,\epsilon}|^2\right)\,ds_g\\
&=\int_N|dv_{\epsilon}|^2\,dv_g+\int_{\partial N}2\frac{(1-|u_{m,\epsilon}|^2)^2\langle v,u_{m,\epsilon}\rangle^2}{\epsilon}\,ds_g
-\int_{\partial N}\frac{(1-|u_{m,\epsilon}|^2)}{\epsilon}|v_{m,\epsilon}|^2\,ds_g\\
&\leq \int_N|dv_{\epsilon}|^2\,dv_g-\int_{\partial N} \left\langle \frac{\partial u_{m,\epsilon}}{\partial \nu},u_{m,\epsilon}\right\rangle|v_{m,\epsilon}|^2\,ds_g
+C\|v\|_{C^0}\int_{\partial N}\frac{(1-|u_{m,\epsilon}|^2)^2}{\epsilon}\,ds_g\\
&\leq \int_N|dv_{\epsilon}|^2-|du_{m,\epsilon}|^2|v_{m,\epsilon}|^2-\frac{1}{2}\left\langle d|u_{m,\epsilon}|^2,d|v_{m,\epsilon}|^2\right\rangle\,dv_g \\
&+C\|v\|_{C^0}\int_{\partial N}\frac{(1-|u_{m,\epsilon}|^2)^2}{\epsilon}.
\end{split}
\end{equation*}
Now, since $u_{m,\epsilon}\to u_m$ in $C^0\cap W^{1,2}$, it's clear that, for any fixed $v\in \mathcal{V}(u_m)$, $v_{\epsilon}\to v$ in $W^{1,2}\cap C^0$ as well, and together with the fact that $\int_{\partial N}\frac{(1-|u_{m,\epsilon}|^2)^2}{\epsilon}\to 0$ as $\epsilon\to 0$, it follows from the preceding computations that
$$\liminf_{\epsilon\to 0}F_{\epsilon}''(u_{m,\epsilon})(v_{\epsilon},v_{\epsilon})\leq E''(u_m)(v,v).$$
Moreover, since
$$\|T_{\epsilon}v-v\|_{C^0}\leq \|v\|_{C^0}\|u_{m,\epsilon}-u_m\|_{C^0}\leq \frac{1}{2}\|v\|_{C^0}$$
for $\epsilon$ sufficiently small, $T_{\epsilon}$ must be injective for $\epsilon$ sufficiently small. Hence, if $V\subset \mathcal{V}(u_m)$ is a subspace such that
$$E''(u_m)(v,v)\leq -\delta \|v\|_{C^0}^2\text{ for all }v\in V,$$
then $T_{\epsilon}(V)\subset W^{1,2}_{\Gamma}(N,A_{\Gamma}^m)$ is a space of the same dimension such that
$$F_{\epsilon}''(u_{m,\epsilon})(v_{\epsilon},v_{\epsilon})\leq -\frac{\delta}{2}\|v\|_{C^0}^2\leq -\frac{\delta}{4}\|v_{\epsilon}\|_{C^0}^2$$
for all $v_{\epsilon}\in T_{\epsilon}(V)$. In other words, $F_{\epsilon}''(u_{m,\epsilon})$ must be negative definite on $T_{\epsilon}V$, and therefore
$$\dim(V)=\dim(T_{\epsilon}V)\leq m,$$
giving the desired index estimate for $u_m$. 

Finally, passing the estimate \eqref{m.e.no.conc} to the limit $\epsilon\to 0$ gives the desired non-concentration estimate
\begin{equation}
\int_{D_{\delta}(p)}|du_m|^2\,dv_g+\int_{D_{\delta}(p)\cap \partial N}\left|\frac{\partial u_m}{\partial\nu}\right|\,ds_g\leq \frac{C}{\log(r_0/\delta)}
\end{equation}
for all $\delta\in (0,r_0(N,g))$.
\end{proof}

Before showing in the next section that the maps $u_m$ stabilize to a Steklov eigenmap for a metric realizing $\Sigma_1^{\Gamma}(N,[g])$, we need one more analytic ingredient: the observation that the constants in the small-energy regularity theorem for free boundary harmonic maps $(N,\partial N)\to (\B^n,\Sph^{n-1})$ is independent of $n$. Recall that since the free boundary harmonic condition is conformally invariant, it suffices to consider the case of half-disks in $\mathbb{R}^2$.

\begin{lemma}[cf. \cite{Schev06, DR11, LP15}]
\label{fb.ereg}
Let $u\colon (D_1^+(0),I_1(0))\to (\B^n,\Sph^{n-1})$ solve 
$$
\Delta u=0\text{ on }D_1^+(0)=D_1(0)\cap(\mathbb{R}\times [0,\infty)),
$$
$$
\frac{\partial u}{\partial \nu}=\left|\frac{\partial u}{\partial \nu}\right|u\text{ on }I_1(0)=[-1,1]\times\{0\}.
$$
There exist $\eta>0$ and $C<\infty$ independent of $n$ such that if
$$
\int_{I_1(0)}\left|\frac{\partial u}{\partial\nu}\right|\,dx+\int_{D_1^+(0)}|du|^2\,dxdy\leq \eta,
$$
then $|du(x)|^2\leq C\int_{D_1^+(0)}|du|^2$ for all $x\in D_{1/2}^+(0)$.
\end{lemma}

\begin{proof}
For the convenience of the reader, we opt for a simple, self-contained proof via Bochner identities modeled on \cite{CS89}, taking advantage of the a priori regularity established in, e.g., \cite{Schev06, DR11, LP15}. Introducing the function
$$
\psi(x)=(1-|x|)^2|du(x)|^2,
$$
which achieves its max at some point $x_0\in D_1^+(0)$, it suffices to show that
$$
\max \psi = (1-|x_0|)^2|du(x_0)|^2\leq C\int_{D_1^+(0)}|du|^2\,dxdy.
$$
To this end, let $\sigma_0:=\frac{1-|x_0|}{2}$ and set $\delta_0:=\mathrm{dist}(x_0,I_1(0))$. If $\sigma_0\leq 10\delta_0$, then $D_{\sigma_0/10}(x_0)\subset D_1^+(0)$, and the subharmonicity of $|du|^2$ gives
$$
|du(x_0)|^2\leq \frac{C}{\sigma_0^2}\int_{D_{\sigma_0/10}(x_0)}|du|^2\,dxdy\leq \frac{C}{\sigma_0^2}\int_{D_1^+(0)}|du|^2\,dxdy,
$$
from which the desired estimate follows. Thus, we can assume without loss of generality that $\delta_0<\frac{\sigma_0}{10}$. Subharmonicity of $|du|^2$ on the interior of $D_1^+(0)$ then gives 
$$
\delta_0^2|du(x_0)|^2\leq \int_{D_{\delta_0}(x_0)}|du|^2\,dxdy;
$$
moreover, for any $x\in B_{\sigma_0}(x_0)$, we have

$$
4\sigma_0^2|du|^2(x_0)\geq (1-|x|)^2|du(x)|^2\geq (1-|x_0|-|x-x_0|)^2|du(x)|^2\geq \sigma_0^2|du(x)|^2,
$$
so that 
$$|du|^2\leq 4|du(x_0)|^2\text{ on }B_{\sigma_0}(x_0).$$
Next, the Bochner identity gives $-\Delta \frac{1}{2}|du|^2=|\mathrm{Hess}(u)|^2$, while on the boundary $I_1(0)$, we have
\begin{equation*}
\begin{split}
\frac{\partial}{\partial \nu}\frac{1}{2}\left(|du|^2\right)=\mathrm{Hess}(u)(\nabla u,\nu)
&=\left\langle \frac{\partial u}{\partial \nu},\frac{\partial^2u}{\partial \nu^2}\right\rangle+\left\langle \frac{\partial u}{\partial \tau},\frac{\partial^2 u}{\partial \tau\partial\nu}\right\rangle  \\
\left(\text{using }\Delta u=0\text{ and }\frac{\partial u}{\partial\nu}=\left|\frac{\partial u}{\partial\nu}\right|u\right)&=-\left\langle \left|\frac{\partial u}{\partial\nu}\right|u,\frac{\partial ^2u}{\partial \tau^2}\right\rangle+\left\langle \frac{\partial u}{\partial \tau},\frac{\partial}{\partial \tau}\left(\left|\frac{\partial u}{\partial\nu}\right|u\right)\right\rangle\\
\left(\text{using }\left\langle\frac{\partial u}{\partial \tau},u\right\rangle=0\text{ repeatedly }\right)&=2\left|\frac{\partial u}{\partial\nu}\right|\left|\frac{\partial u}{\partial\tau}\right|^2.
\end{split}
\end{equation*}
Now, let $y_0\in I_1(0)$ be such that $|y_0-x_0|=\delta_0$, and for each radius $s<\frac{9\sigma_0}{10}$, so that $D_s^+(y_0)\subset D_{\sigma_0}(x_0)$, we compute
\begin{equation*}
\begin{split}
\frac{d}{ds}&\left(\frac{1}{s}\int_{\partial^+D_s^+(y_0)}\frac{|du|^2}{2}\right)=\frac{1}{s}\int_{D_s^+(y_0)}|\mathrm{Hess}(u)|^2\,dxdy-\frac{1}{s}\int_{I_s(y_0)}2\left|\frac{\partial u}{\partial\nu}\right|\left|\frac{\partial u}{\partial \tau}\right|^2\,dx \\
\geq& -\frac{2}{s}\cdot 4|du(x_0)|^{\frac{5}{2}}\int_{I_s(y_0)}\left|\frac{\partial u}{\partial\nu}\right|^{\frac{1}{2}}\,dx
\geq -\frac{C}{s}|du(x_0)|^{\frac{5}{2}}\cdot\sqrt{s}\left(\int_{I_s(y_0)}\left|\frac{\partial u}{\partial\nu}\right|\,dx\right)^{1/2} \\
\geq& -C\sqrt{\eta}\cdot |du(x_0)|^{\frac{5}{2}}s^{-\frac{1}{2}}.
\end{split}
\end{equation*}
Now, integrating the above over $[t,s]\subset [0,\sigma_0/2]$ gives
$$
\frac{1}{t}\int_{\partial^+D^+_t(y_0)}|du|^2 \leq \frac{1}{s}\int_{\partial^+D^+_s(y_0)}|du|^2+C\sqrt{\eta}|du(x_0)|^{5/2}s^{1/2}.
$$
Multiplying by $s$ and integrating the right hand side over $s\in [\sigma/2,\sigma]\subset [t,\sigma_0/2]$ then gives
$$
\frac{\sigma^2}{2t}\int_{\partial^+D^+_t(y_0)}|du|^2\leq C\int_{D^+_{\sigma}(y_0)}|du|^2\,dxdy+C\sigma^{5/2}\sqrt{\eta}|du(x_0)|^{5/2},
$$
and after multiplying by $t$ and integrating over $t\in [0,2\delta_0]$, we find
$$
\sigma^2\int_{D_{2\delta_0}^+(y_0)}|du|^2\leq C\delta_0^2\int_{D_1^+(0)}|du|^2\,dxdy+C(\sigma |du(x_0)|)^{5/2}\delta_0^2.
$$
In particular, since $|du(x_0)|^2\leq \frac{C}{\delta_0^2}\int_{D_{\delta_0}(x_0)}|du|^2\leq \frac{C}{\delta_0^2}\int_{D^+_{2\delta_0}(y_0)}|du|^2,$ it follows that
\begin{equation}
\label{near.ctr}
(\sigma|du(x_0)|)^2\leq C\int_{D_1^+(0)}|du|^2\,dxdy+C\sqrt{\eta}(\sigma|du(x_0)|)^{5/2}
\end{equation}
for all $\sigma \in [2\delta_0,\sigma_0/2]$, and in fact for all $\sigma \in [0,\sigma_0/2]$, since we already know that $\delta_0^2|du(x_0)|^2$ is bounded above by the energy. To conclude, let 
$$
\sigma^2:=2C\frac{\int_{D_1^+(0)}|du|^2\,dxdy}{|du(x_0)|^2}.
$$
If $\sigma \geq \sigma_0/2$, then we get the estimate $\sigma_0^2|du(x_0)|^2\leq 8C\int_{D_1^+(0)}|du|^2$, as desired; so suppose instead that $\sigma\in [0,\sigma_0/2]$. Then we can use this value of $\sigma$ in \eqref{near.ctr} to get
\begin{eqnarray*}
2C\int_{D_1^+(0)}|du|^2\,dxdy&\leq& C\int_{D_1^+(0)}|du|^2\,dxdy+C\sqrt{\eta}\cdot \left(2C\int_{D_1^+(0)}|du|^2\,dxdy\right)^{5/4}\\
&\leq& C\int_{D_1^+(0)}|du|^2\,dxdy+2^{5/4}C^{9/4}\eta^{3/4} \int_{D_1^+(0)}|du|^2\,dxdy.
\end{eqnarray*}
Finally, taking $\eta$ small enough so that $2^{5/4}C^{9/4}\eta^{3/4} < C$, we arrive at a contradiction, completing the proof.\qedhere

\end{proof}

\subsection{Stabilization and realization of $\Sigma_1^{\Gamma}(N,[g])$}
%\hspace{40mm}
Now, we derive an analog of Lemma \ref{closed.stab.lem}, from which we will be able to deduce the desired equality $\Sigma_1^{\Gamma}(N,[g])=2\mathcal{F}_m^{\Gamma}(N,[g])$ for $m$ sufficiently large. Throughout, we continue to assume that either $\Sigma_1^{\Gamma}(N,[g])>2\pi$ or elements of $\Gamma$ have no fixed points on $\bd N$.

\begin{lemma}
\label{stek.stab}
There exists $K(N,[g])\in \mathbb{N}$ such that the free boundary harmonic maps $u_m: (N,\partial N)\to (\B^{m|\Gamma|},\Sph^{m|\Gamma|-1})$ of Proposition \ref{fbhm.ex} take values in an equatorial ball $\B^K\subset \B^{m|\Gamma|}$ of fixed dimension $K$.
\end{lemma}
\begin{proof}
Similar to Lemma \ref{closed.stab.lem}, note that if the desired conclusion failed, it would follow that the space %$\mathrm{Span}\left\{u_m^1,\ldots,u_m^{m|\Gamma|}\right\}$
 spanned by the coordinate functions $u_m^i$ of $u$ would have unbounded dimension as $m\to\infty$, and since the coordinate functions solve the Steklov problem
$$
\Delta u_m^j=0\text{ in }N;\text{ }\frac{\partial u_m^j}{\partial \nu}=\left|\frac{\partial u_m}{\partial \nu}\right|u_m^j\text{ on }\partial N,
$$
this would imply that the eigenvalues
$$
\lambda_k\left(N,[g],\left|\frac{\partial u_m}{\partial\nu}\right|ds_g\right):=\inf_{\dim(V)=k+1}\max\left\{\left.\frac{\int_N|d\phi|^2}{\int_{\partial N}|\frac{\partial u_m}{\partial \nu}|\phi^2}\right|\,\,\, \phi \in V\right\}
$$
satisfy
\begin{equation}
\label{eigen.trap}
\lim_{\ell\to\infty}\lambda_k\left(N,[g],\left|\frac{\partial u_{m_{\ell}}}{\partial\nu}\right|ds_g\right)\leq 1
\end{equation}
for each $k\in \mathbb{N}$ along a subsequence $m_{\ell}$. For simplicity of notation, we assume henceforth that this subsequence $m_{\ell}$ coincides with the full sequence $m$.

On the other hand, by \eqref{fbhm.no.conc} and Lemma \ref{fb.ereg}, we see that the maps $u_m$ satisfy a uniform gradient bound
$$\|du_m\|_{L^{\infty}(N,g)}\leq C$$
independent of $m$. Moreover, on any small half-disk $D_{r_0}(p)\subset N$ with $p\in \partial N$, identifying $D_{r_0}(p)$ conformally with $(D_1^+(0),g_0)\subset \mathbb{R}^2_+$, we see that $f_m^2:=|du_m|_{g_0}^2=\rho^2|du_m|_g^2$ satisfies
$$
\Delta_{g_0} \frac{1}{2}f_m^2=-|\mathrm{Hess}_{g_0}(u_m)|_{g_0}^2\text{ on }N
$$
and
$$
\left|\frac{\partial}{\partial\nu_{g_0}}\frac{1}{2}f_m^2\right|\leq 2f_m^4\leq C^4.
$$
Integrating against a test function $\phi\in C_c^{\infty}(D_1^+)$ then gives
\begin{equation*}
\begin{split}
&\int\phi^2|df_m|_{g_0}^2\,dv_{g_0}\leq \int\phi^2|\mathrm{Hess}(u_m)|_{g_0}^2\,dv_{g_0}
=-\int\phi^2\Delta_{g_0} \frac{1}{2}f_m^2\,dv_{g_0} \\
&=-\int 2\langle \phi d\phi, f_mdf_m\rangle_{g_0}\,dv_{g_0}
\leq \int 2f_m^2|d\phi|_{g_0}^2\,dv_{g_0}+\frac{1}{2}\phi^2|df_m|_{g_0}^2\,dv_{g_0}.
\end{split}
\end{equation*}
Fixing $\phi$ such that $|d\phi|_{g_0}\leq 3$ and $\phi \equiv 1$ on $D_{1/2}^+(0)$ gives
$$
\int_{D_{1/2}^+}|df_m|^2_{g_0}\,dv_{g_0}\leq \int \phi^2|df_m|_{g_0}^2\,dv_{g_0}\leq 4\int f_m^2|d\phi|^2\leq C\int_{D_{r_0}(p)}|du_m|_g^2dv_g.
$$
In particular, since $|du_m|_g$ differs from $|du_m|_{g_0}$ by a fixed conformal change, it follows that $|du_m|_g$ is bounded in $W^{1,2}$ norm in a neighborhood of $\partial N$, and by the harmonicity of $u_m$, this bound extends to the rest of $N$, giving
$$\||du_m|_g\|_{W^{1,2}(N,g)}\leq C.$$

By compactness of the trace embedding $W^{1,2}(N)\to L^p(\partial N)$ for all $p\in [1,\infty)$ in dimension two, it follows that there exists some $0\leq f\in L^{\infty}(\partial N)$ such that, after passing to a further subsequence,
$$
\left|\frac{\partial u_m}{\partial \nu}\right|\to f\text{ in }L^p(\partial N)
$$
for every $p\in [1,\infty)$. In particular, by boundedness of the trace embedding $W^{1,q}(N)\to L^q(\partial N)$, it follows that the measures 
$$
\mu_m:=\left|\frac{\partial u_m}{\partial \nu}\right|ds_g\in C^0(N)^*
$$
converge to $\mu:=fds_g$ in $(W^{1,q}(N))^*$ for every $q\in (1,\infty]$. In particular, since this holds for $q\in (1,2)$, we can apply \cite[Prop 4.11]{GKL} to conclude convergence of the eigenvalues
$$
\lambda_k\left(N,[g],\left|\frac{\partial u_m}{\partial\nu}\right|ds_g\right)\to \lambda_k(N,[g],fds_g)
$$
for every $k\in \mathbb{N}$, which together with \eqref{eigen.trap} implies
$$\lambda_k(N,[g],fds_g)\leq 1$$
for every $k\in \mathbb{N}$. On the other hand, since $f\in L^{\infty}(\partial N)$, it's clear that the map $W^{1,2}(N)\to L^2(\partial N, fds_g)$ is compact, and therefore $\lambda_k(N,[g],fds_g)\to +\infty$ as $k\to\infty$, giving the desired contradiction.
\end{proof}

Finally, we can argue as in Section \ref{closed.stab} to deduce that $\Sigma_1^{\Gamma}(N,[g])=2\mathcal{F}_m^{\Gamma}(N,[g])$, with $\left|\frac{\partial u_m}{\partial\nu}\right|$ giving the $\bar{\sigma}_1$-maximizing conformal factor on the boundary.

\begin{theorem}\label{stek.mm.char}
Let $u_m\colon (N,\partial N)\to (\B^{m|\Gamma|},\Sph^{m|\Gamma|-1})$ be the free boundary harmonic maps of Proposition \ref{fbhm.ex}, realizing $E(u_m)=\mathcal{F}_m^{\Gamma}(N,[g])$. For $m$ sufficiently large, metrics of the form $g_m=f g$, where $f\in C^{\infty}(N)$ is a positive, $\Gamma$-invariant extension of $\left|\frac{\partial u_m}{\partial \nu}\right|$, maximize $\bar{\sigma}_1$ among $\Gamma$-invariant metrics in $[g]$; that is,
$$\Sigma_1^{\Gamma}(N,[g])=\bar{\sigma}_1(N,g_m)=2\mathcal{F}_m^{\Gamma}(N,[g]).$$
\end{theorem}

\begin{proof}
Since $\Delta u_m=0$ in $N$ and $\frac{\partial u_m}{\partial \nu}=\left|\frac{\partial u_m}{\partial \nu}\right|u$ on $\partial N$, for any metric $g_m=f g$ where $f|_{\partial N}=\left|\frac{\partial u_m}{\partial\nu}\right|$, we have
\begin{equation*}
\begin{split}
\int_{\partial N}ds_{g_m}=\int_{\partial N}\left\langle \frac{\partial u_m}{\partial \nu},u_m\right\rangle ds_g
=\int_N|du_m|_g^2\,dv_g
=2\mathcal{F}_m^{\Gamma}(N,[g]),
\end{split}
\end{equation*}
so that
$$\bar{\sigma}_1(N,g_m)\geq 2\sigma_1(N,g_m)\mathcal{F}_m^{\Gamma}(N,[g]),$$
and by Lemma \ref{fb.lbd}, it follows that
$$\bar{\sigma}_1(N,g_m)\geq \sigma_1(N,g_m) \Sigma_1^{\Gamma}(N,[g]).$$
Thus, all that remains is to show that
$$\sigma_1(N,g_m)=1$$
for $m$ sufficiently large. To this end, let $Y_m\subset C^{\infty}(\partial N)$ be a space on which the quadratic form
$$
Q_m(\phi,\phi):=\int_N|d\hat{\phi}|_{g}^2\,dv_{g}-\int_{\partial N}\left|\frac{\partial u_m}{\partial\nu}\right|\phi^2ds_{g_m}
$$
is negative-definite, where $\hat{\phi}$ denotes the harmonic extension of $\phi$ to $N$, and consider the injective linear map
$$
T_m\colon Y_m\to W^{1,2}(N,A_{\Gamma})
$$
given by
$$
T_m(\phi):=\sum_{\sigma\in \Gamma}\left(\hat{\phi} \circ \sigma^{-1}\right)e_{\sigma}.
$$
We will show now that $\dim(Y_m)=1$, which then implies that $\sigma_1(N,g_m)=1$, completing the proof.

By Lemma \ref{stek.stab}, there exists some $K(N,[g])\in \mathbb{N}$ such that, without loss of generality,
$$
u_m(N)\subset \B^K\times\{0\}\subset \B^{m|\Gamma|}
$$
for all $m\in \mathbb{N}$, and therefore we have
$$
u_m(x) \perp \{0\}\times A_{\Gamma}^{m-K}
$$
pointwise. For each $1\leq i\leq m-K$, denote by 
$$
T_m^i\colon Y_m\to W^{1,2}(N,A_{\Gamma}^m)
$$
the maps given by composing $T_m\colon Y_m\to W^{1,2}(N,A_{\Gamma})$ with the inclusion of $A_{\Gamma}\to A_{\Gamma}^m$ in the $(K+i)$th summand. Since $v=T_m^i(\phi)$ is $\Gamma$-equivariant and pointwise orthogonal to $u$, we can test it against the second variation $E''(u_m)$ to obtain
\begin{eqnarray*}
E''(u_m)(v,v)&=&\int_N|dv|_g^2\,dv_g-\int_{\partial N}\left|\frac{\partial u_m}{\partial\nu}\right||v|^2ds_g\\
&=&\int_N|d(T_m(\phi))|_g^2\,dv_g-\int_{\partial N}\left|\frac{\partial u_m}{\partial\nu}\right||T_m(\phi)|^2ds_g\\
&=&|\Gamma|\cdot\left(\int_N|d\hat{\phi}|_g^2\,dv_g-\int_{\partial N}\left|\frac{\partial u_m}{\partial\nu}\right|\phi^2ds_g\right)<0
\end{eqnarray*}
whenever $\phi\neq 0$, by definition of $Y_m$. Thus, the second variation $E''(u_m)$ is negative definite on the space
$$T_m^1(W_m)\oplus\cdots \oplus T_m^{m-K}(Y_m)\subset W^{1,2}_{\Gamma}(N,A_{\Gamma}^m),$$
which has dimension $(m-K)\dim(Y_m)$, so that
$$(m-K)\dim(Y_m)\leq \ind_{E,\Gamma}(u_m)\leq m.$$
For $m>2K+1$, this clearly forces $\dim(Y_m)=1$, and hence $\sigma_1(N,g_m)=1$, completing the proof.
\end{proof}

In particular, summarizing, we have the following existence result.

\begin{theorem}\label{stek.conf.ex}
For any finite subgroup $\Gamma\leq \Isom (N,g)$ such that either $\Sigma_1^{\Gamma}(N,[g])>2\pi$ or $\Gamma$ has no fixed points on $\bd N$, there exists $m\in \mathbb{N}$ and a $\Gamma$-equivariant free boundary harmonic map
$$u\colon (N,\partial N)\to (\mathbb{B}^{m|\Gamma|},\mathbb{S}^{m|\Gamma|-1})$$
such that any $\Gamma$-invariant conformal metric of the form $\tilde{g}=\rho^2g$ where $\rho|_{\partial N}=\left|\frac{\partial u}{\partial \nu}\right|$ realizes
$$\bar{\sigma}_1(N,\tilde{g})=\Sigma_1^{\Gamma}(N,[g])$$
and the coordinates of $u$ are first Steklov eigenfunctions for $\tilde{g}$.
\end{theorem}

%===================================================
%===================================================
%===================================================
%===================================================
%===================================================
%===================================================

%===================================================
%===================================================
%===================================================
%===================================================
%===================================================
%===================================================

\section{Basic reflection surfaces}
\label{S:BRS} 

\subsection{Isometries and functions}
Let $G$ be a group of isometries on $M$.  Then $G$ has a natural right action on the space of all real-valued functions on $M$, given by 
\begin{align*}
(u, \rho) \mapsto u \circ \rho = \rho^* u. 
\end{align*}
A function $u$ is \emph{$G$-invariant} if $\rho^* u = u$ for each $\rho \in G$.  If $X$ is a space of functions on $M$, denote by $X^G \subset X$ the subspace of $G$-invariant functions.

Let $\rho$ be an involutive isometry.   A function $u$ is \emph{even} (\emph{odd}) under $\rho$ if $\rho^* u = u$ ($\rho^*u = - u$), and a set of functions is \emph{$\rho$-even} (\emph{$\rho$-odd}) if each element is even (odd) under $\rho$.  Finally, note that if $X$ is a space of functions and $\rho^* X \subset X$, then $X$ admits a direct sum decomposition $X = \Acal_\rho(X) \oplus \Scal_\rho(X)$ into even and odd parts
\begin{align*}
 \Acal_\rho(X) := \{u \in X: \rho^* u =-u\} 
\quad
\text{and}
\quad
\Scal_\rho(X) := \{ u \in X : \rho^*u  = u\}.
\end{align*}

\begin{lemma}
\label{Lasymconj}
If $X$ is a space of functions on $M$ with  $\omega^* X \subset X$ for each isometry $\omega$ of $M$, if $\rho$ and $\tau$ are isometries, and if $\rho$ is involutive, then
\begin{enumerate}[label=\emph{(\roman*)}]
\item $\tau^{-1} \rho \tau$ is an involutive isometry, 
\item $\tau^* \Acal_{\rho}( X)  = \Acal_{\tau^{-1} \rho \tau} (X)$, and
\item $\tau^* \Scal_{\rho}( X)  = \Scal_{\tau^{-1} \rho \tau} (X)$.
\end{enumerate}
\end{lemma}
\begin{proof}
Straightforward.
\end{proof}

\subsection{Definitions and basic properties}

\begin{definition}
\label{dref}
If $M$ is a compact, connected Riemannian surface and $\tau$ is a reflection on $M$, the pair $(M, \tau)$ is called a \emph{reflection surface}.  
For a reflection surface $(M,\tau)$, denote by $\pi : M \rightarrow M / \langle \tau \rangle$ the canonical projection. 
\end{definition}

Reflection surfaces $(M, \tau_M)$ and $(N, \tau_{N})$ are \emph{equivariantly homeomorphic} if there exists a homeomorphism $f: M \rightarrow N$ with $f \circ \tau_M = \tau_N \circ f$. 

The following lemma goes back to Klein \cite{Klein}; a proof in modern notation follows from work of Dugger (see Corollary 3.2 and Proposition 2.2 in \cite{Dugger}).
\begin{lemma}
\label{Lsigconno}
If $(M, \tau)$ is a reflection surface and  $\pi(M)$ is orientable, 
then $M$ is orientable if and only if $M \setminus M^\tau$ is not connected. 
\end{lemma}

\begin{definition}%[Basic reflection surfaces]
\label{dbasref}
A reflection surface $(M, \tau)$ is called \emph{basic} if
\begin{enumerate}[label={(\roman*)}]
\item $\pi(M)$ is an orientable surface of genus zero, and
\item $\pi(\partial M \cup M^\tau_\partial)$ is connected.
\end{enumerate}
\end{definition}
\noindent Note that (ii) holds vacuously if $\partial M = \varnothing$.

\begin{example}
For $M$ a flat rhombic torus and $\tau$ the reflection about one of the diagonals, the reflection surface $(M, \tau)$  is \emph{not} basic, because $M \setminus M^\tau$ is connected, so $\pi(M)$ is nonorientable by Lemma \ref{Lsigconno}. 
\end{example}

\subsection{Classification}

\begin{definition}
\label{dtop}
The \emph{topological type} of a compact, connected surface $M$ is $(\gamma, k, \epsilon)$, where $\gamma$ is the maximum number of pairwise disjoint, embedded loops in $M$ which can be removed without disconnecting $M$, $k$ is the number of boundary components, and $\epsilon = +$ if $M$ is orientable and $\epsilon=-$ otherwise. 
\end{definition}
\noindent A surface with topological type $(\gamma, k, +)$  (type $(\gamma, k, -)$), is homeomorphic to the connect sum of $\gamma$ tori ($\gamma$ projective planes) with $k$ points removed. 

Although  compact surfaces are classified up to homeomorphism by their topological type, more information is needed to classify reflection surfaces up to equivariant homeomorphism.  Such a classification can be found in \cite{Scherrer, Dugger, Bujalance} and depends on additional invariants related to the fixed-point sets of the involution.

\begin{definition}
Let $(M, \tau)$ be a reflection surface.  An embedded circle in $M^\tau$ is called an \emph{oval}.  A union $\Ccal$ of connected components of $M^\tau_\partial$ is called a \emph{chain} if for each component $B$ of $\partial M$, either $\Ccal \cap B = \varnothing$ or $\Ccal \cap B$ consists of two distinct points.
\end{definition}

\noindent By \cite{Kobayashi}, the fixed-point set $M^\tau$ of a reflection surface $(M, \tau)$ is the disjoint union of its isolated fixed-points, its ovals, and its chains. 

An embedded circle in a surface is called \emph{twisted} (\emph{untwisted}) if every sufficiently small neighborhood is nonorientable (orientable); thus each oval in $M^\tau$ is twisted or untwisted.  A chain $\Ccal$ is called \emph{twisted} (\emph{untwisted}) as follows: the closed surface $\hat{M}$ obtained from $M$ by capping off each boundary component with a disk can be made into a reflection surface in an obvious way, and $\Ccal$ is twisted (untwisted) if the  circle $\hat{\Ccal} \subset \hat{M}^\tau$ containing $\Ccal$ is twisted (untwisted).

\begin{definition}
\label{dspecies}
The \emph{species} associated to a reflection surface $(M, \tau)$ is
\begin{align*}
[\gamma, k, \epsilon : F, C_-, C_+, T_-, T_+]
\end{align*} 
where $(\gamma, k, \epsilon)$ is its topological type and $F, C_-, C_+, T_-, T_+$ are the number of isolated fixed-points, twisted ovals, untwisted ovals, twisted chains, and untwisted chains, respectively. 
Let $C=C_- + C_+$ be the number of ovals and $T= T_-+T_+$ be the number of chains.  Finally,  let $W$ be the number of fixed-point arcs (components of $M^\tau_\partial$).
\end{definition}

\begin{theorem}[Classification of basic reflection surfaces]
\label{Tclass}

The species
\begin{align*}
[\gamma, k, \epsilon : F, C_-, C_+, T_-, T_+]
\end{align*}
of a basic reflection surface $(M, \tau)$ satisfies the following:
\begin{enumerate}[label=\emph{(\roman*)}]
\item $C+W\geq 1$ and $T \in \{0, 1\}$.
\item  $T = 0$ only if $k \in \{0, 2\}$, and $k = 0$ only if $T =0$.
\item If $\epsilon = +$, then $F=C_-=T_-=0$ and $C_+ +T_+ = \gamma+1$. 
\item If $\epsilon = -$, then $F + 2(C+ T) = \gamma+ 2$.  
\end{enumerate}
Moreover, each species consistent with (i)-(iv) is realized by a basic reflection surface which is unique up to equivariant homeomorphism.  
\end{theorem}
\begin{proof}
Item (i) is immediate from the definitions and Lemma \ref{Lref}(iii).  For (ii), if $T = 0$, then $M^\tau_\partial = \varnothing$, and Definition \ref{dbasref} asserts that $\pi(\partial M)$ is connected.  Since $\tau$ is an involution, and $\tau$ has no fixed-points on $\partial M$, it follows that either $\partial M = \varnothing$ or $\partial M$ has two components which are exchanged by $\tau$, so $k=2$.  This proves (ii).

For (iii)-(iv), first suppose $\partial M = \varnothing$.  In this case, (iii) follows from  \cite[Theorem 1.7]{Dugger} and the definitions, so suppose $\epsilon=-$.  Note that $C> 0$ by (i) above, and that $\pi(M)$ is orientable by Definition \ref{dbasref}.  It follows that $(M, \tau )$ is as in either the third or fourth bullet point of Dugger's classification (see \cite[p. 926-927]{Dugger}) and the fact from Definition \ref{dbasref} that $\pi(M)$ has genus zero furthermore rules out the latter case.

Further, because $\pi(M)$ has genus zero, it follows from the third bullet point in Dugger's classification that $F+2C = \gamma+2$, completing the proof of (iv) in the closed case.  The corresponding existence part also follows from Dugger's classification. 

Now suppose $\partial M \neq \varnothing$.  Let $\hat{M}$ be the closed surface obtained from $M$ by capping off each boundary component with a disk.  $\hat{M}$ can be made into a reflection surface in an obvious way, and it is easy to see $(\hat{M}, \tau)$ is also basic.  The remaining conclusions are now easy to check by comparing with the results already proved in the closed case.  
\end{proof}

\begin{cor}
\label{CEuler}
If $(M,\tau)$ is a basic reflection surface, then
\begin{align*}
\chi(M) = 4 - 2C - 2T - F - k,
\end{align*}
where $\chi(M)$ is the Euler number and $C, T, F, k$ are as in Definition \ref{dspecies}.
\end{cor}
\begin{proof}
This follows from Theorem \ref{Tclass}, using that $\chi(M) = 2- 2\gamma - k$ in the orientable case, and that $\chi(M) = 2- \gamma - k$ in the nonorientable case. 
\end{proof}

\begin{remark}
Note that Theorem \ref{Tclass} implies the following:
\begin{itemize}
\item Up to equivariant homeomorphism, there is a unique basic reflection surface with topological type $(\gamma, 0, +)$.
\item A surface with topological type $(\gamma, k, -)$ and $\gamma \geq 2$ may be realized by at least two basic reflection surfaces with distinct species. 
\end{itemize}
\end{remark}

\begin{remark}
\label{Rbas}
Basic reflection surfaces realizing each species from Theorem \ref{Tclass} may be constructed by the following surgery procedure:
\begin{enumerate}
\item For $(\gamma, k, \epsilon) = (\gamma, 0, +)$, double an orientable genus $0$ surface with $\gamma+1$ boundary components along its boundary.  See Figure \ref{Fdoub}.
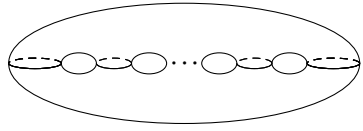
\begin{figure}[h]
\centering
\def\ra{4}
\begin{tikzpicture}[scale=.7,transform shape]
\draw [draw=black] (0, 0) ellipse ({\ra-\ra/6} and {\ra/3.5});
\foreach \j in {1, 2}
{
\foreach \i in {-1, 1}
{
\draw [draw=black] ({\i*(\j*\ra/3)-\i*\ra/6}, 0 ) ellipse ({\ra/12 and \ra/20}); 
\draw [draw=black, densely dashed] ({(5*\ra/6}, 0) arc (0:180:{\ra/8} and {\ra/40});%outline of dotted 
\draw [draw=black] ({(5*\ra/6}, 0) arc (0:-180:{\ra/8} and {\ra/40});%outline of dotted 
\draw [draw=black, densely dashed] ({\i*(\ra/2)-\i*\ra/6+\ra/12}, 0) arc (0:180:{\ra/12} and {\ra/40});
\draw [draw=black] ({\i*(\ra/2)-\i*\ra/6+\ra/12}, 0) arc (0:-180:{\ra/12} and {\ra/40});%outline of dotted 
\draw [draw=black, densely dashed] ({(-5*\ra/6+\ra/4}, 0) arc (0:180:{\ra/8} and {\ra/40});%outline of dotted 
\draw [draw=black] ({(-5*\ra/6+\ra/4}, 0) arc (0:-180:{\ra/8} and {\ra/40});%outline of dotted 
}
}
\node at (\ra/90, 0) {$\boldsymbol{\cdots}$}; 
\end{tikzpicture}  
        \caption{A basic reflection surface with topological type $(\gamma, 0, +)$.}
        \label{Fdoub}
\end{figure}
\item For $(\gamma, k, \epsilon) = (\gamma, 0, -)$, start with a genus $n: = (\gamma-C_- -2C_+)/2$ surface with an involution $\tau$ fixing $2n+2$ isolated points. Replace neighborhoods of $C_-$ of these points with crosscaps in an obvious way to extend $\tau$.  Finally, remove $C_+$ pairs of $\tau$-invariant disks, and identify the boundaries of these disks via $\tau$.  See Figure \ref{Fhandlebody}.
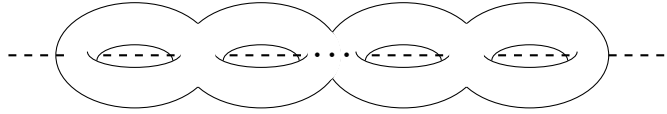
\begin{figure}[h]
\begin{tikzpicture}[scale=.7]
                \begin{scope}[xscale=.6]
         \draw[thick, dashed] (-4, 0)--(-2, 0);
         \draw[thick, dashed] (-1, 0)--(2, 0);
         \draw[thick, dashed] (3, 0)--(6, 0);
         \draw[thick, dashed] (7.5, 0)--(10, 0);
         \draw[thick, dashed] (11.5, 0)--(14, 0);
         \draw[thick, dashed] (15, 0)--(17, 0);
          \draw (0,0) ellipse (2.5cm and 1cm);
          \draw (-1.5,.07) arc (180:360:1.5cm and .3cm);
          \draw (1.2,-.1) arc (0:180:1.2cm and .3cm);
          \draw (4,0) ellipse (2.5cm and 1cm);
          \draw (2.5,.07) arc (180:360:1.5cm and .3cm);
          \draw (5.2,-.1) arc (0:180:1.2cm and .3cm);
          \fill[color=white](2,.6) to [bend right=50] (1.45,0) to [bend right = 50] (2,-.6) to [bend right=50] (2.55,.0) to [bend right=50] (2,.6);
          \begin{scope}[xshift=8.5cm]
          \draw (0,0) ellipse (2.5cm and 1cm);
          \draw (-1.5,.07) arc (180:360:1.5cm and .3cm);
          \draw (1.2,-.1) arc (0:180:1.2cm and .3cm);
          \draw (4,0) ellipse (2.5cm and 1cm);
          \draw (2.5,.07) arc (180:360:1.5cm and .3cm);
          \draw (5.2,-.1) arc (0:180:1.2cm and .3cm);
          \fill[color=white](2,.6) to [bend right=50] (1.45,0) to [bend right = 50] (2,-.6) to [bend right=50] (2.55,.0) to [bend right=50] (2,.6);
          \fill[color=white](-2.5,0) ellipse (.53cm and .5cm);
          \end{scope}
          \end{scope}
          \node at (3.8, 0) {$\boldsymbol{\cdots}$}; 
          \end{tikzpicture}
          \caption{An involution $\tau$ on a genus $n$ surface fixing $2n+2$ points.  As depicted, $\tau$ is the half-turn about the dashed line.}
                   \label{Fhandlebody}
          \end{figure}

\item For $(\gamma, k, \epsilon)=(\gamma, k, \pm)$ and $T, k > 0$, remove $k$ $\tau$-symmetric disks meeting an oval of a reflection surface $(M, \tau)$ as in as in (1), if $\epsilon = +$, or as in (2), if $\epsilon = -$.
\item For $(\gamma, k, \epsilon)=(\gamma, 2, \pm)$ and $T= 0$, remove $2$ $\tau$-symmetric disks disjoint from $M^\tau$ for a reflection surface $(M, \tau)$ as in (1), if $\epsilon = +$, or as in (2), if $\epsilon = -$.
\end{enumerate}
\end{remark}

\subsection{Multiplicity bounds for the first eigenvalue}

Let $(M, \tau)$ be a reflection surface, and recall the notation introduced in Section 2.2.  Note that the involutive isometry $\tau$ induces a linear involutive isometry $\tau^* : \Ecal \rightarrow \Ecal$ given by $\tau^* u = u \circ \tau$, hence the first eigenspace $\Ecal$ admits an orthogonal direct sum decomposition $\Ecal = \Acal_\tau(\Ecal) \oplus \Scal_\tau(\Ecal)$ into odd and even parts.  In order to bound the dimension of $\Ecal$, it therefore suffices to consider $\Acal_\tau(\Ecal)$ and $\Scal_\tau(\Ecal)$ separately.

\begin{notation}
\label{Nevenodd}
When the role of $\tau$ is clear, we adopt a simplified notation and write $\Ecal^-$ for $\Acal_\tau(\Ecal)$ and $\Ecal^+$ for $\Scal_\tau( \Ecal)$. 
\end{notation} 

\begin{prop}
\label{Lasymcl}
\label{Poddnon1}
\label{Poddnon2}
If $(M, \tau)$ is a reflection surface, then
\begin{enumerate}[label=\emph{(\roman*)}]
\item $M^\tau \subset \Ncal_u$ for each $\tau$-odd eigenfunction $u$. 
\end{enumerate}
Moreover, if $M \setminus M^\tau$ is not connected, then
\begin{enumerate}[label=\emph{(\roman*)}]
\item[\emph{(ii)}] $M^\tau = \Ncal_u$ for each $\tau$-odd eigenfunction $u$ with two nodal domains; 
\item[\emph{(iii)}] $\dim \Ecal^- \leq 1$.
\end{enumerate}
Finally, if $(M, \tau)$ is basic and $M \setminus M^\tau$ has genus zero, then $\dim \Ecal^- \leq 2$. 
\end{prop}
\begin{proof}
Item (i) is immediate.  For (ii), observe that $M \setminus M^\tau$ has exactly two components by Lemma \ref{Ldis}(ii); however, if $M \setminus \Ncal_u$ has two components, it then follows from (i) that $M^\tau = \Ncal_u$. 

For the dimension bound (iii), fix $p \in M \setminus M^\tau$, consider the linear map 
\begin{align*}
T: \Ecal^- \rightarrow \R, 
\quad
Tu = u(p),
\end{align*}
and note by (ii) that $T$ has trivial kernel, so $\dim \Ecal^- \leq \dim \R = 1$.

For the last assertion, suppose $M \setminus M^\tau$ has genus zero.  By Definition \ref{dbasref}, Lemma \ref{Lsigconno}, and part (iii) above, it suffices to consider the case where $M$ is nonorientable, and $M \setminus M^\tau$ is connected.  The rest of the proof is split into two cases, depending on the number $F$ of isolated fixed-points of $\tau$. 

\emph{Case 1: $F \geq 1$.} Fix an isolated point  $p \in M^\tau$, consider the linear map
\begin{align*}
T: \Ecal^- \rightarrow T^*_p M, 
\quad
Tu = du_p,
\end{align*}
and note for any $u \in \Ecal^-$  that $u$ vanishes at $p$.  If $T$ were not injective, there would exist nonzero $u\in \Ecal^-$ for which at least two nodal lines pass through $p$; because $M \setminus M^\tau$ has genus zero, a standard connectivity argument now gives a contradiction. 

\emph{Case 2: $F =0$.} 
By Theorem \ref{Tclass}(iv), it follows that $M$ has a twisted oval $O$.  Fix distinct points $p, q \in O$ and consider the linear map 
\begin{align*}
T: \Ecal^- \rightarrow T^*_p M \times T^*_q M, 
\quad
Tu = (du_p, du_q).
\end{align*}
For any $u \in \Ecal^-$, $u$ vanishes on $O$, so the image $\mathrm{Im}\,  T$ lies in a $2$-dimensional subspace of $T^*_p M \times T^*_q M$.  To complete the proof, it suffices to prove $T$ is injective, for then $\dim \Ecal^- \leq \dim \mathrm{Im}\,  T \leq 2$. 

Because $O$ is twisted, its complement in any sufficiently small neighborhood is homeomorphic to a punctured disk.   Consider the space $M^\cut$ obtained from $M \setminus M^\tau$ by compactifying the end corresponding to $O$ in $M \setminus M^\tau$ by adding a single point $o$, and let $\iota: M \setminus M^\tau \rightarrow M^\cut$ be the inclusion map.  Because $M \setminus M^\tau$ is an orientable, genus zero surface, so is $M^\cut$. 

 If $T$ were not injective, there would exist a nonzero $u \in \Ecal^-$ containing nodal lines transverse to $M^\tau$ at $p$ and $q$, respectively.  Then $\Ncal^\cut_u : = \{ o\} \cup \iota(\Ncal_u \setminus M^\tau)$ would contain a ``cross" consisting of at least four arcs meeting at $o$.  Denoting by $U_1, U_2$ the nodal domains of $u$, it is easy to see the connected components of $M^\cut \setminus \Ncal^\cut_u$ are $\iota(U_1)$ and $\iota(U_2)$.  Because $M^\cut$ has genus zero, applying same connectivity argument in Case 1 results in a contradiction. 
\end{proof}

We next bound the dimension of $\Ecal^+$.  First, we introduce some notation. 
\begin{definition}
\label{dbullet}
Let $\pi_{\bullet} : \pi(M) \rightarrow \pi_{\bullet}(\pi(M))$ be the quotient obtained from $\pi(M)$ by collapsing each component of $\pi(M^\tau)$ to a point.  For any subset $S \subset M$, denote by $S^\bullet$ its image under $\pi_{\bullet} \circ \pi$.
\end{definition}

\begin{lemma}
\label{Lrhoquo}
If $(M, \tau)$ is a reflection surface and 
 $u$ is a $\tau$-even eigenfunction with nodal domains $U_1$ and $U_2$, then 
\begin{enumerate}[label=\emph{(\roman*)}]
\item the  components of $\pi(M) \setminus \pi(\Ncal_u)$ are $\pi(U_1)$ and $\pi(U_2)$;
\item the  components of $M^\bullet \setminus \Ncal_u^{\bullet}$ are $U^\bullet_1 \setminus \Ncal_u^{\bullet}$ and $U^\bullet_2 \setminus \Ncal_u^\bullet$.
\end{enumerate}
\end{lemma}
\begin{proof}
We have $\pi(M) \setminus \pi(\Ncal_u)= \cup_{i=1}^2\pi(U_i)$.  Since $u$ is even, $\tau(U_i) = U_i$ and so $\pi(U_1) \cap \pi(U_2) = \varnothing$.  On the other hand, since each $U_i$ is a connected open set, so is its image under the continuous open map $\pi$.  Item (i) follows from these observations. 

Next, we have $
M^\bullet \setminus \Ncal_u^{\bullet} = \cup_{i=1}^2 ( U^\bullet_i \setminus \Ncal_u^{\bullet}). 
$
The two sets in this union are clearly open and disjoint, so it remains to show that each one is connected.  

Consider a connected component $S_\alpha$ of $M^\tau$, and fix a tubular neighborhood $V_\alpha$ of $S_\alpha$ in $U_i$.  We claim that $\pi(V_\alpha \setminus S_\alpha)$ is connected.  If $S_\alpha$ has codimension at least $2$ in $M$, then $V_\alpha \setminus S_\alpha$ is path-connected, so $\pi(V_\alpha \setminus S_\alpha)$ is connected.  If, on the other hand, $S_\alpha$ has codimension $1$ in $M$, then using Lemma \ref{Lref}(ii) with the fact that $u$ is even, it follows that $\pi$ identifies points on opposite sides of the two (locally defined) sides of $S_\alpha$ in $M$, so $\pi(V_\alpha \setminus S_\alpha)$ is connected. 

Next, we claim that $\pi(U_i \setminus S_\alpha)$ is connected.  Let $W_\alpha \subset U_i$ be the largest open set containing $V_\alpha$ such that $\pi(W_\alpha\setminus S_\alpha)$ is connected, and suppose that $\pi(U_i \setminus S_\alpha) = \pi(W_\alpha \setminus S_\alpha) \cup V$ for some open set $V$ satisfying $\pi(W_\alpha \setminus S_\alpha) \cap V = \varnothing$.  Then $\pi(U_i ) = \pi(W_\alpha) \cup V$, and $\pi(W_\alpha)$ and $V$ are disjoint open sets.  Since $\pi(U_i)$ is connected, it follows that $W_\alpha = U_i$ and $V = \varnothing$, so $\pi(U_i \setminus S_\alpha)$ is connected.  

By arguing inductively, it follows that $\pi(U_i \setminus S)$ is connected, where $S \subset M^\tau$ is the union of the components of $M^\tau$ which meet $\Ncal_u$ nontrivially.  Consequently,  $( U_i \setminus S)^\bullet$ is connected.
Since $U^\bullet_i \setminus \Ncal^\bullet_u = ( U_i \setminus S)^\bullet$, the proof is complete. 
\end{proof}

\begin{prop}
\label{Lrhoquo2}
\label{Lrhoquo3}
If $(M, \tau)$ is a basic reflection surface, then
\begin{enumerate}[label=\emph{(\roman*)}]
\item If $u$ is a $\tau$-even eigenfunction with two nodal domains, 
	\begin{enumerate}[label=\emph{(\alph*)}]
	\item the intersection number of $\Ncal_u$ with each oval is either $0$ or $2$;
	\item  $\pi(\Ncal_u \cap ( \partial M \cup M^\tau_\partial))$ consists of two points if
		$\partial M \neq \varnothing$; 
\item  $du_p \neq 0$ for each $p \in \Ncal_u \setminus M^\tau$; 
\end{enumerate}
\item  $\dim \Ecal^+ \leq 3$ if $\partial M = \varnothing$, and $\dim \Ecal^+ \leq 2$ if $\partial M \neq \varnothing$.
\end{enumerate}
\end{prop}
\begin{proof}
\emph{Case 1:} $M$ is a closed surface.
For $u$ is as in (i), note first that $\pi_{\bullet} \circ \pi|_{M \setminus M^\tau}$ is a local homeomorphism.  Therefore, since $\Ncal_u$ is a union of piecewise $C^1$ arcs, so is $\Ncal^\bullet_u$.

For each oval $O \subset M^\tau$,  $u$ changes sign across each nodal arc meeting $O$, so there are an even number of such arcs, hence an even number of arcs emanate from $O^\bullet$ in $\Ncal^\bullet_u$. 

\emph{Claim:} $\Ncal^\bullet_u$ is an embedded circle in $M^\bullet$.  To prove the claim, note by the preceding discussion, we can find an embedded, piecewise $C^1$ circle $S \subset \Ncal^\bullet_u$.   Because $(M, \tau)$ is basic, $M^\bullet$ is homeomorphic to a $2$-sphere, so  $M^\bullet \setminus S$ is the disjoint union of two domains.  On the other hand, $M^\bullet \setminus \Ncal_u^{\bullet}$ has two components by Lemma \ref{Lrhoquo}(ii), so $\Ncal^\bullet_u = S$.  This proves the claim, and implies (i).

For the dimension bound (ii), fix $p \in M \setminus M^\tau$ and consider the linear map \begin{align}
\label{EevenT}
T: \Ecal^+ \rightarrow \R\times T^*_p M, 
\quad
T u = (u(p), d u_p). 
\end{align}
By (i), $T$ is injective, so $\dim \Ecal^+ \leq \dim ( \R \times T^*_p M) = 3$.  This completes the proof in this case. 

\emph{Case 2:} $M$ has nonempty boundary.  Because $(M, \tau)$ is basic, Definition \ref{dbasref} implies $M^\bullet$ is homeomorphic to a $2$-disk.  

For $u$ as in (i),  recall that $u$ is $L^2(\partial M)$-orthogonal to the constant functions, that is, that $\int_{\partial M} u = 0$.  By using this in combination with the facts that $\partial M^\bullet$ is a circle and that $u$ is even, it follows that $\Ncal_u^{\bullet} \cap \partial M^\bullet \neq \varnothing$.  Moreover, by using that $M^\bullet$ is a disk and arguing as in Case 1, it follows that an even number of arcs in $\Ncal_u^{\bullet}$ meet $\partial M^\bullet$. 

By the preceding, we can find an embedded arc $S \subset \Ncal^\bullet_u$ joining two distinct points of $\partial M^\bullet$.  Because $M^\bullet$ is a disk,  $M^\bullet \setminus S$ is the disjoint union of two domains.  On the other hand, $M^\bullet \setminus \Ncal^\bullet_u$ has two components by Lemma \ref{Lrhoquo}(ii), so $\Ncal^\bullet_u = S$.  This implies (i).

For the dimension bound (ii), fix $p \in \partial M \setminus M^\tau$ and define $T$ as in \eqref{EevenT}. 
By combining (i) with the fact that $\pi_{\bullet} \circ \pi|_{M \setminus M^\tau}$ is a local diffeomorphism, $T$ is injective, so $\dim \Ecal^+\leq \dim (\mathrm{Im}\, T)$.  Finally,  the Steklov eigenvalue condition gives the linear relation $\partial_\eta u(p) = \sigma_1 u(p)$, so the image $\mathrm{Im} \, T$ of $T$ lies in a $2$-dimensional subspace of $\R\times T^*_p M$.  
\end{proof}

\begin{remark}
The conclusion of Proposition \ref{Lrhoquo3}(ii) need not hold without the assumption from Definition \ref{dbasref} that $\pi(\partial M \cup M^\tau_\partial)$ is connected.
\end{remark}

\subsection{Critical points for even Steklov eigenfunctions}

In Lemma \ref{Lbdgraph} below, we give sufficient criteria for all of the critical points of a $\tau$-even Steklov eigenfunction to lie on the fixed-point set $M^\tau$.  The proof of Lemma \ref{Lbdgraph} relies on the following Morse-type inequality, which is proved using work of Hoffman-Mart{\'i}n-White \cite{White}. 

\begin{lemma}[Morse inequality for Steklov eigenfunctions]
\label{Lmorse}
Suppose $M$ is a compact surface with boundary and $u$ is a Steklov eigenfunction such that
\begin{enumerate}[label=\emph{(\alph*)}]
\item  $u|_{\partial M}$  has finitely many local minima, and
\item $\partial M \cap \{u = t\}$ is finite for each $t \in \R$. 
\end{enumerate}
Let $\Nserif$ be the number of interior critical points of $u$ and $B^{\leq 0}_{\min}$ ($B^{\leq 0}_{\max}$) be the set of points where $u|_{\partial M}$ achieves nonpositive minima (maxima).  Then
\begin{align*}
\Nserif + \chi(M) \leq |B^{\leq 0}_{\min}| - |B^{\leq 0}_{\max}|. 
\end{align*}
Equality holds if and only if each (interior or boundary) saddle point has multiplicity one. 
\end{lemma}
\begin{proof}
Since $u$ is a harmonic function, it is clear that the restriction of $u$ to the interior of $M$ is a Rad{\'o} function in the sense of \cite[Definition 5]{White}.   Moreover, the assumptions of the Lemma and \cite[Theorem 46]{White} imply that $u$ is a Rad{\'o} function on all of $M$.  
The hypotheses of \cite[Theorem 26]{White}  apply, and it follows that
\begin{align}
\label{Emorsefb}
\Nserif \leq |B| - |A| - \chi(M),
\end{align}
where $B$ is the set of local minima of $u|_{\partial M}$ and $A$ is the set of boundary saddle points. 

Next, recall that $\partial_{\eta} u = \sigma u$ on $\partial M$.
It follows that if $p \in B$ and $u(p) >0$, then $p \in A$; similarly if $u|_{\partial M}$ attains a local maximum at $q \in \partial M$ and $u(q) >0$, then $q \notin A$. The desired inequality now follows from combining these observations with \eqref{Emorsefb}.  Finally, the assertion regarding the equality case follows from the corresponding result in \cite[Theorem 26]{White}.
\end{proof}

\begin{lemma}
\label{Lbdgraph}
If $(M, \tau)$ is a basic reflection surface with boundary
and $u$ is an $\tau$-even Steklov eigenfunction such that
\begin{enumerate}[label=\emph{(\alph*)}]
\item  $u$ has exactly two nodal domains,
\item   $u|_{\partial M}$  has finitely many local minima and
\item $\partial M \cap \{u = t\}$ is finite for each $t \in \R$, 
\end{enumerate}
then each interior critical point of $u$ lies on $M^\tau$ and is a multiplicity-$1$ saddle; moreover $u$ has exactly two critical points on each oval in $M^\tau$. 
\end{lemma}
\begin{proof}
Given $u$ as in the statement of the lemma, let $n_\partial $ be the number of points in $\pi( \Ncal_u \cap \partial M)$ and $n_\tau$  be the number in $\pi( \Ncal_u \cap M^\tau_\partial)$.  Finally, write the number $\Nserif$ of interior critical points of $u$ as
\begin{align*}
\Nserif =  \Nserif_{M \setminus M^\tau}+  \Nserif^{\out}_{M^\tau}+ \Nserif^{\inn}_{M^\tau},
\end{align*}
with  $\Nserif_{M \setminus M^\tau}$, $\Nserif^{\out}_{M^\tau}$, and $\Nserif^{\inn}_{M^\tau}$  the number occurring on $M \setminus M^\tau$,  $M^\tau_\partial$, and $M^\tau \setminus M^\tau_\partial$, respectively.  We claim that
%\vspace{-1cm}
\begin{enumerate}[label={(\roman*)}]
\item $n_\partial + n_\tau = 2$,
\item $\Nserif^{\inn}_{M^\tau} \geq 2 C+ F$,
\item $\Nserif^{\out}_{M^\tau} + n_\tau \geq W$, and
\item $n_\partial = |B^{\leq 0}_{\min}| - |B^{\leq 0}_{\max}|$,
\end{enumerate}
where $C, F, W$ are respectively the number of ovals, isolated fixed-points, and fixed-point intervals for $M$ (components of $M^\tau_\partial$), as in Definition \ref{dspecies}.

Item (i) follows from Proposition \ref{Lrhoquo3}(i).  Next, $M^\tau \setminus M^\tau_\partial$ consists of $C$ ovals and $F$ isolated fixed-points. One one hand, a smooth function on a circle must have at least two critical points; on the other, each isolated fixed-point is easily seen to be a critical point since $u$ is $\tau$-even.  Item (ii) follows from these observations.  

For (iii), observe first from the symmetry and the facts that $\int_{\partial M} u = 0$ and $u$ has two nodal domains that $u$ has at most one zero on each component of $M^\tau_\partial$.  Hence there are $W-n_\tau$ components of $M^\tau_\partial$ where $u$ has a sign; to complete the proof of (iii), it suffices to prove $u$ has at least one critical point on each such component.  Now fix a component $S$ of $M^\tau_\partial$ and suppose that $u|_S$ has a sign.  By the Steklov condition, the outward normal derivative $\partial_\eta u =\sigma u$ has the same sign on each point of $\partial S$, so $u|_S$ has a critical point $q \in S$.  By the $\tau$-symmetry, $q$ is also a critical point for $u$, and (iii) follows.  

For (iv), note that $n_\partial \in \{ 0, 1, 2\}$ by (i).  The proof is split up into cases.

\emph{Case 1:} $n_\partial = 0$.  Then for each component $S$ of $\partial M$, $S \cap \Ncal_u = \varnothing$ and $|B^{\leq 0}_{\min}| - |B^{\leq 0}_{\max}| = 0$ since the number of maxima and minima of $u|_S$ on each circle $S$ are equal. 

\emph{Case 2:} $n_\partial = 1$.  Then a single nodal line $\ell$ meets $\partial M$, so $\tau(\ell) = \ell$ and    $\ell$ bounds a segment $S$ on a component of $\partial M$ with $u|_S \leq 0$.  The two local extrema of $u|_S$ nearest to $\partial S$ are then minima, so $|B^{\leq 0}_{\min} \cap S| - |B^{\leq 0}_{\max}\cap S| = 1$ because the local extrema on $S$ alternate.  The proof is completed by arguing as in Case 1 for the remaining components of $\partial M$.

\emph{Case 3:} $n_\partial =2$.  We first consider the case where $W = 0$ and hence $k = 2$.  By the symmetry, $\Ncal_u$ meets one of the boundary components at two points, and the conclusion easily follows using the symmetry.  
Next suppose $W> 0$, and fix a nodal line $\ell$ meeting $\partial M$.  Since $n_\partial = 2$, $W > 0$, and  $u$ has exactly two nodal domains,  the endpoints of $\ell$ lie in different components of $\partial M$, and $\tau(\ell)$ is the other nodal line meeting $\partial M$.  Together, $\ell$ and $\tau(\ell)$ bound two segments on different components of $\partial M$ where $u$ restricts to be nonpositive.  It follows by arguing as in Case 1 and Case 2 that  $|B^{\leq 0}_{\min}| - |B^{\leq 0}_{\max}| =2$.

By combining (i)-(iv) with Lemma \ref{Lmorse} and Corollary \ref{CEuler}, 
\begin{equation*}
\begin{split}
n_\partial &=  |B^{\leq 0}_{\min}| - |B^{\leq 0}_{\max}| 
%&\geq \Nserif + \chi(M)\\
\geq  \Nserif_{M \setminus M^\tau}+  \Nserif^{\out}_{M^\tau}+\Nserif^{\inn}_{M^\tau} + \chi(M) \\
%+ n_\partial +n_\tau- 2 \gamma - k\\
&\geq \Nserif_{M \setminus M^\tau} + n_\partial + W + 2 - 2T - k
= \Nserif_{M \setminus M^\tau} + n_\partial.
\end{split}
\end{equation*}
The final equality uses that $W + 2 - 2T - k = 0$, which follows since (recalling Theorem \ref{Tclass}) $W = 0$ implies $T = 0$ and $k= 2$, while $W> 0$ implies $T = 1$ and $W=k$.

Thus  $\Nserif_{M \setminus M^\tau} =0$.  Since equality must hold in each inequality above, Lemma \ref{Lmorse} implies each critical point is a multiplicity-$1$ saddle; furthermore $\Nserif^{\inn}_{M^\tau} = 2 C+ F$, which implies the last assertion.
\end{proof}

\subsection{Bounds on the first normalized eigenvalue}

\begin{lemma}
\label{Levalbdcl}
\label{Levalbdbd}
If $(M, \tau)$ is a basic reflection surface, then
\begin{enumerate}[label=\emph{(\roman*)}]
\item  $\lambda_1(M) |M| < 16\pi$, if $\partial M = \varnothing$.
\item $\sigma_1(M) |\partial M| < 4\pi$, if $\partial M \neq \varnothing$.
\end{enumerate}
\end{lemma}
\begin{proof}
%For (i), suppose $(M, \tau)$ is a closed basic reflection surface. 
The assumptions,  \cite[Fact 5]{LiYau}, and \cite[Theorem 1]{LiYau} imply that
\begin{align*}
\lambda^+_1 | \pi(M)| < 8\pi, 
\end{align*}
where $\lambda_1^+$ is the first eigenvalue of the laplacian, with Neumann boundary conditions, on $\pi(M)$, for $M$ as in (i).  Since $|\pi(M)| = |M|/2$, it follows that 
\begin{align*}
\left( \inf \frac {\int_{M} |du|^2}{\int_{M} u^2} \right) |M| < 16\pi,
\end{align*}
where the infimum is over all $\tau$-even Lipschitz functions $u$ with average zero,  $\int_M u = 0$.  Since $\lambda_1$ is the infimum of $\int_M |du|^2/ \int_M u^2$ over \emph{all} Lipschitz functions with average zero, the conclusion follows. 

For (ii), let $(M, \tau)$ be a basic reflection surface with nonempty boundary, and consider the mixed Steklov-Neumann problem
\begin{align}
\label{Emixed}
\begin{cases}
\Delta u = 0 \hfill & \text{on} \quad \pi( M)\\
\frac{\partial u}{\partial \nu} = \sigma u \hfill & \text{on} \quad \pi(\partial M) \\
\frac{\partial u}{\partial \nu} = 0 \hfill & \text{on} \quad  \pi(M^\tau).
\end{cases}
\end{align}
It is straightforward to see that the first eigenvalue $\sigma^+_1$ of \eqref{Emixed} satisfies 
\begin{align}
\label{Evarmix} 
\sigma^+_1 = \inf \frac{\int_{\pi(M)} | du |^2}{\int_{\pi(\partial M)} u^2},
\end{align}
where the infimum is over all Lipschitz functions $u$ with $\int_{\pi(\partial M)} u =0$. 

By Definition \ref{dbasref}, $\pi(M)$ is a genus zero surface with $\pi(\partial M)$ a proper subset of a component $S$ of $\partial \pi(M)$.  Consequently, there is a conformal map $\varphi : \pi(M) \rightarrow \B^2$ with the property that $\varphi|_{S} : S \rightarrow \partial \B^2$ is a homeomorphism.  By a well-known balancing argument \cite{Weinstock, Hersch}, we can use the conformal  automorphisms of the disk $\B^2$ and use the coordinate functions of $\varphi$ as test functions in \eqref{Emixed} to conclude that $\sigma_1^+| \pi(\partial M)| < 2\pi$.  Note the strict inequality holds because $\varphi (\pi (\partial M))$ is a proper subset of $\partial \B^2$. 
\begin{comment} %%%%%% DETAILS OF HERSCH ARGUMENT COMMENTED OUT TO SAVE SPACE
to assume without loss of generality that 
\begin{align*}
\int_{\pi(\partial M)} \varphi^i ds = 0,
\end{align*}
for $i=1, 2$, where $ds$ is the arclength element with respect to the metric.  Then each $\varphi^i$ can be used as a test function in the variational characterization for $\sigma^+_1$ from \eqref{Evarmix}:  
\begin{align*}
\sigma^+_1 \int_{\pi(\partial M)} (\varphi^i)^2 ds 
\leq 
\int_{\pi(M)} | d \varphi^i |^2. 
\end{align*} 
Since $\varphi$ is conformal and maps $\pi(M)$ into $\B^2$, 
\begin{align*}
\sum_{i=1}^2 \int_{\pi(M)} | d \varphi^i|^2 = 2 |\varphi ( \pi(M))| < 2\pi.
\end{align*}
On the other hand, since $\varphi ( \pi( \partial M)) \subset \partial \B^2$, 
\begin{align*}
\sum_{i=1}^2 \int_{\pi(\partial M)} ( \varphi^i)^2 ds 
= \int_{\pi( \partial M)} ds 
= |\pi(\partial M)|.
\end{align*}
Combining the preceding facts proves \eqref{Emixed1}.  
\end{comment}

Finally, since $|\pi(\partial M)| = |\partial M|/2$, the preceding implies 
\begin{align*}
\left( \inf \frac{\int_M |du|^2}{\int_{\partial M} u^2}\right) |\partial M| < 4\pi,
\end{align*}
where the infimum is over all $\tau$-even Lipschitz functions $u$ with $\int_{\partial M}u = 0$.  Since $\sigma_1$ is the infimum of $\int_{M} |du|^2/ \int_{\partial M} u^2$ over \emph{all} Lipschitz functions $u$ satisfying $\int_{\partial M}u =0$, the conclusion follows. 
\end{proof}

\begin{remark}
The conclusion of Lemma \ref{Levalbdbd}(ii) need not hold without the assumption from Definition \ref{dbasref} that $\pi(\partial M \cup M^\tau_\partial)$ is connected.  In particular, 
for a genus zero surface $M$ with four boundary components and tetrahedral symmetry studied in \cite{KOO}, item (i) of Definition \ref{dbasref} (with $\tau$ one of the reflections in the tetrahedral group) holds, but $\pi(\partial M \cup M^\tau_\partial)$ is \emph{not} connected, and $\sigma_1 |\partial M| \approx 13.67 > 4\pi$.  
\end{remark}

\subsection{Reflection groups and chambers}
We recall some classical facts \cite{Davis} concerning groups generated by reflections on surfaces. 

Let $M$ be a connected surface.  A \emph{topological reflection} is a smooth involution $\tau: M \rightarrow M$ whose fixed-point set $M^\tau$ separates $M$.  If $\Gamma$ is a finite group acting properly, smoothly, and effectively on $M$ and $\Gamma$ is generated by topological reflections,  $\Gamma$ is called a \emph{reflection group} on $M$. 
%Let $\Gamma$ be a reflection group on $M$.  
Given $x \in M$, denote by $R(x)$ the set of all (topological) reflections in $\Gamma$ fixing $x$.  A point $x \in M$ is called \emph{nonsingular} if $R(x)=\varnothing$.  A \emph{chamber} of $\Gamma$ on $M$ is the closure of a connected component of nonsingular points.

Let $\Omega$ be a chamber, and denote by $V$ the set of reflections $v$ such that $\{v\} = R(x)$ for some $x \in \Omega$.  If $v \in V$, the set $\Omega^v : = M^v \cap \Omega$ is called a \emph{panel} of $\Omega$, and $V$ is the set of \emph{reflections through the panels of} $\Omega$.  The pair $(\Gamma, V)$ is called a \emph{reflection system}.

Reflection systems are related to Coxeter systems through Lemma \ref{Lrefcox} below, which is essentially contained in \cite{Davis} (see in particular \cite[Theorem 4.1]{Davis}); in order to state it precisely, we recall some more notation. 

 A \emph{Coxeter system} is a pair $(\Gamma, V)$ consisting of a group $\Gamma$ with a set of generators $V$ such that the relations $v^2 = 1$  ($v \in V$) and $(vw)^{m(v, w)} = 1$ ($v, w \in V$ and $m(v, w) \leq \infty$) form a presentation for $\Gamma$.

Let $\Omega$ be a surface with corners, and $S$ the set of corners.  An \emph{edge} of $\partial \Omega$ is the closure of a connected component of $\partial \Omega \setminus S$.  A surface with corners is called \emph{nice} if each corner is contained in two edges.

 \begin{lemma}
 \label{Lrefcox}
If $(\Gamma, V)$ is a reflection system on a closed surface $M$, then 
\begin{enumerate}[label=\emph{(\roman*)}]
\item $(\Gamma, V)$ is a Coxeter system. 
\item $\Gamma$ acts freely and transitively on the set of chambers. 
\item Each chamber $\Omega$ is a fundamental domain: for $\pi : M \rightarrow M/ \Gamma$ the quotient, $\pi|_\Omega$ is a homeomorphism. 
\item Each chamber $\Omega$ is a nice connected surface with corners equipped with a collection $(\Omega^v)_{v \in V}$ of panels, each panel being a pairwise disjoint union of edges, and each edge contained in exactly one panel. 
\end{enumerate}
\end{lemma}

\begin{comment}
\begin{definition}
\label{dfund}
Given data $(\Gamma, \Omega)$ consisting of 
\begin{enumerate}
\item A finite Coxeter group $\Gamma$ with generators $\gamma_1, \dots, \gamma_n$ satisfying the relations $(\gamma_i \gamma_j)^{k_{ij}}= e$, where $k_{ii} = 1$ and $2 \leq k_{ij} < \infty$ for $i\neq j$; and
\item A compact connected orientable surface $\Omega$ with corners equipped with a collection $(\Omega^v)_{v \in V}$ of \emph{panels} satisfying
	\begin{itemize}
	\item Each panel is a pairwise disjoint union of edges.
	\item Each edge is contained in exactly one panel,
	\end{itemize} 
%\item A connected orientable surface $\Omega$ and a decomposition $\partial \Omega = \cup_{i=1}^n \Omega^{\gamma_i}$ with each $\Omega^{\gamma_i}$ a nonempty finite collection of $1$-manifolds for which each $p \in \Omega^{\gamma_i} \cap \Omega^{\gamma_j}$ is a boundary point of $\Omega^{\gamma_i}$ and of $\Omega^{\gamma_j}$, 
\end{enumerate}
let $M =M(\Gamma, \Omega) : = \Gamma \times \Omega / \sim$ be the quotient by the relation $\sim$ generated by requesting for each $i=1, \dots, n$ and each $p \in \Omega^{\gamma_i}$ that $(\gamma, p) \sim (\gamma', p)$ if $\gamma = \gamma' \gamma_i$. 
\end{definition}
\end{comment}

Conversely, a surface with a reflection system may be obtained from a Coxeter system and an appropriate fundamental domain, as follows. 
%A Coxeter system can be made into a reflection system on a surface by supplying an appropriate fundamental domain, as follows. 

\begin{lemma}
\label{LMomega}
Given data $(\Gamma, V, \Omega)$ consisting of 
\begin{enumerate}[label=\emph{(\alph*)}]
\item A Coxeter system $(\Gamma, V)$ with $\Gamma$ a finite group; and
\item A nice connected surface with corners $\Omega$ equipped with a collection $(\Omega^v)_{v \in V}$ of \emph{panels}, each panel being a pairwise disjoint union of edges, and each edge contained in exactly one panel,  
\end{enumerate}
the quotient space $M = M(\Gamma, \Omega)$ of $\Gamma \times \Omega$ by the relation $\sim$ defined by $(g, x) \sim (h, y) \iff x = y$ and $g^{-1} h \in \langle v\in V : x \in \Omega^v \rangle$ satisfies:
\begin{enumerate}[label=\emph{(\roman*)}]
\item $M$ is a connected, smooth surface, orientable if and only if $\Omega$ is; 
\item $(\Gamma, V)$ is a reflection system on $M$;
\item The image in $M$ of each $\{\gamma\} \times \Omega, \gamma \in \Gamma$ is a chamber. 
\end{enumerate}
\end{lemma}
\begin{proof}
Classical, see for example \cite{Davis}, in particular \cite[Theorem 15.2]{Davis}. 
\end{proof}

\subsection{Reflection groups on basic reflection surfaces} 
\label{Ss:fund1} 
\label{sec:or_Z2}

\begin{comment}
\begin{definition}
\label{dMa}
Given $a \in \N$, 
\begin{enumerate}
\item $\Gamma = \Z_2 = \{e, \tau\}$, and
\item $\Omega = \Omega(a)$ a genus zero surface with $a$ boundary components and $\Omega^\tau = \partial \Omega$,
\end{enumerate}
let $M = M(\Gamma, \Omega)$ be as in Lemma \ref{LMomega}.  The fundamental domain $\Omega(a)$ is said to be of \emph{type} $a$, and   
a surface with a $\Z_2$ action is said to be of \emph{type} $a$ if it is equivariantly homeomorphic to an $M(\Gamma, \Omega(a))$. 
\end{definition}

If $M$ is as in Definition \ref{dMa}, then $(M, \tau)$ is a closed, orientable, genus $a-1$ basic reflection reflection surface as in Definition \ref{dbasref}.  Conversely, it follows from the classification Theorem \ref{Tclass} that each closed orientable basic reflection surface is equivariantly homeomorphic to an $(M, \tau)$ as in \ref{dMa}.

For the purposes of studying equivariant degenerations, the abbreviation $M = M(a)$ will mean $M$ is as in Definition \ref{dMa} with type $a$.
%A fundamental domain of is of \emph{type} $a$ if it is homeomorphic to $\Omega(a)$. Similarly, we denote by $M(a)$ the corresponding closed surface with the action of $\mathbb{Z}_2$ and we say that a surface with an action of $\mathbb{Z}_2$ is of \emph{type} $a$ if it is topologically equivalent to $M(a)$. Of course, a surface of type $M(a)$ has fundamental domain of type $\Omega(a)$. The surface $M(a)$ is connected, orientable and of genus $a-1$.

\subsection{Dihedral or Platonic symmetry}
\end{comment}

\label{Ss:fund2}
We now study reflection groups of the form $\Gamma = \langle \tau \rangle \times G$ on closed orientable basic reflection surfaces $(M, \tau)$.  After characterizing all such surfaces by properties of their fundamental chambers, we enumerate the possible groups $G$ and parametrize the collection of such surfaces by their \emph{type} $\bunder$, introduced in Definition \ref{dplatclass}.

\begin{prop}
\label{Agamma}
Let $(M, \tau)$ be a closed orientable basic reflection surface.  If $\Gamma = \langle \tau \rangle \times G = \Z_2 \times G$ is a reflection group on $M$,   $\Omega$ is a chamber, and $(\Gamma, V)$ is the corresponding reflection system, then
\begin{enumerate}[label=\emph{(\roman*)}]
\item $V = \{\tau\} \cup V_G$ for a set $V_G$ of generators for $G$. 
\item $G$ is isomorphic to a subgroup of $O(3)$ generated by reflections. 
\item $\Omega$ has genus $0$, and $\cup_{v \in V_G} \Omega^v$ lies in a single component of $\partial \Omega$.
\item For each $v \in V_G$, any two points in $\Omega^{v}$ can be joined by a path in $\Omega^{v} \cup \Omega^\tau$. 
\end{enumerate}
Conversely, if  $(\Gamma, V, \Omega)$ is as in Lemma \ref{LMomega}, $\Gamma = \langle \tau \rangle \times G = \Z_2 \times G$, and items (i)-(iv) above hold, then $(M(\Gamma, \Omega), \tau)$ is a closed orientable basic reflection surface and $\Gamma$ is a reflection group on $M$. 
\end{prop}
\noindent The component of $\partial \Omega$ in (iii) is denoted by $\partial \Omegaout$.

\begin{proof}
Let $(M, \tau)$ be as in the first part of the Proposition.  Item (i) follows easily from the definitions.  For (ii), let $\pi : M \rightarrow M/ \langle \tau \rangle$ be the quotient, and note from the assumptions that $G$ acts by reflections on the genus zero surface $\pi(M)$.  Therefore $G$ is isomorphic to one of the classical reflection groups acting on the sphere \cite[page 53]{Conway}, proving (ii).  Items (iii)-(iv) follow easily from (ii) and the definitions, by considering the action of $G$ on $\pi(M)$. 

Conversely, if  $(\Gamma, V, \Omega)$ are as in the second part of the Proposition, items (ii)-(iv) imply $M/ \langle \tau \rangle$ has genus zero, hence $(M, \tau)$ is basic.  The remaining conclusions follow from Lemma \ref{LMomega}.
\end{proof}

%We now study closed orientable basic reflection surfaces $(M, \tau)$ whose isometry group has a subgroup of the form $\Gamma = \langle \tau \rangle \times G= \Z_2\times G$ for  $G$  a finite group generated by reflections on $M$.  In this case, $G$ acts by reflections on the genus zero surface $\pi(M)$, hence is isomorphic to one of the classical reflection groups acting on the sphere, see for example \cite[page 53]{Conway}. 

%After first introducing some notation and properties concerning the corresponding reflection groups $G$, we use Lemma \ref{LMomega} to construct all such surfaces with the properties above.

The classification of the groups $G$ which arise in Proposition \ref{Agamma} is well-known \cite[page 53]{Conway}; to fix notation, we enumerate the groups as follows. 
\begin{definition}[Symmetry groups]
\label{dplat}
Given
\begin{align*}
(k_{12}, k_{13}, k_{23}) \in \{ (2, 3, 3), (2, 3, 4), (2, 3, 5)\} \cup \{ (2, 2, m) : m \in \N, m \geq 2\},
\end{align*}
 define the group $*k_{12}k_{13}k_{23}$ with generators $\{\rho_i\}_{i=1}^3$ satisfying the relations 
\begin{align*}
\rho_1^2 = \rho^2_2 = \rho^2_3 = (\rho_1 \rho_2)^{k_{12}}  = (\rho_1 \rho_3)^{k_{13}} = (\rho_2 \rho_3)^{k_{23}} = \Id.
\end{align*}
For $k \geq 2$, define the group $*kk$ with generators $\{\rho_i\}_{i=1}^2$ and relations
\begin{align*}
\rho_1^2 = \rho_2^2 = (\rho_1\rho_2)^k = \Id.
\end{align*}
Finally, let $1*$ be the group $\Z_2$ of order two, with generator $\rho_1$, and let $\{\Id\}$ be the trivial group. 
\end{definition}

A group $*k_{12}k_{13}k_{23}$ is called a \emph{platonic} symmetry group, and the groups  $*233, *234, *235$, and $*22m$ are called respectively \emph{tetrahedral}, \emph{octahedral}, \emph{icosahedral}, and \emph{prismatic dihedral}  groups.  The group $*kk$ is called the $k$-fold \emph{dihedral group}, and is also denoted by $D_k$. 

Each group $G$ in \ref{dplat} acts in a standard way \cite[page 53]{Conway} on $\R^3$, with the generators acting by reflections; the fixed-point sets of all reflections in $G$ induce a tesselation of $\Sph^2$ by fundamental regions whose edges are geodesic segments.  For a platonic group $G = *k_{12} k_{13} k_{23}$, each fundamental region is a triangle with vertex angles $\pi/k_{12}, \pi/k_{13}$, and $\pi/k_{23}$, while for a dihedral group $G = *kk$, each fundamental region is a digon with each vertex angle equal to $\pi/k$.  Finally, for $G = 1* = \Z_2$, each fundamental region is a hemisphere.

\begin{definition}
\label{dplatclass}
The \emph{type} $\bunder = \bunder(M)$ associated  to $M$ as in Proposition \ref{Agamma} is the formal  linear combination 
 \begin{align*}
 \bunder = f + \sum_{i} e_i \rho_i + \sum_{i< j} v_{ij} \rho_i \rho_j
 \end{align*}
 where $f, e_i$, and $v_{ij}$ are the number of components $S$ of $M^\tau$ meeting $\Omega$ satisfying respectively $\Stab_\Gamma(S) = \langle \tau\rangle$, $\Stab_\Gamma(S) = \langle \tau, \rho_i \rangle$, and $\Stab_\Gamma(S) = \langle \tau, \rho_i, \rho_j\rangle$.  The indices in the sums are understood to run over the indices corresponding to the generators for $G$.  
\end{definition}

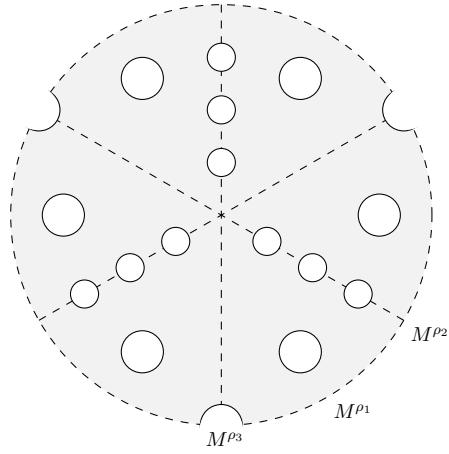
\begin{figure}[h]
\begin{tikzpicture}[scale=.7,transform shape]
\centering
\def\ra{4} %size for the radius
\fill [fill = light-gray] (0, 0) circle (\ra); % the filled in circle
\draw [dashed] (0, 0) circle (\ra); % the outer dashed circle
\foreach \i in {0, 1, 2}
{
\draw [dashed] ({\ra*sin( \i * 120)}, {\ra*cos(\i * 120)})--({-\ra*sin( \i * 120)}, {-\ra*cos(\i * 120)}); % the dashed diameters
	\foreach \j in {0, 1, 2}
	{
	\draw [fill =white] ({ (\ra/4+ \j*\ra/4)*sin(\i*120)}, {(\ra/4+ \j*\ra/4)*cos(\i*120)}) circle (\ra/15); % 3 removed circles along the diameter
	}
}

\foreach \i in {0, 1, 2}
{
\fill [fill= white] ({\ra*sin(60+\i*120)}, {\ra*cos(60+\i*120)}) circle (\ra/10); %the outer three circles
\draw [black] ({\ra*sin(60+\i*120)}, {\ra*cos(60+\i*120)}) circle (\ra/10); %the outer three circles
\fill [fill= white] ({\ra*1.02*sin(60+\i*120)}, {\ra*1.02*cos(60+\i*120)}) circle (\ra*1.04/10); %the outer three circles
}

\foreach \i in {0, 1, 2, 3, 4, 5}
{
\draw [fill =white] ({ .75*\ra*sin(30 + \i*60) }, { .75*\ra*cos(30 + \i*60)}) circle (\ra/10); %interior removed circles
}
%finally draw labels
\draw[] ({\ra*cos(271)}, {\ra*sin(271)}) node[below]{$M^{\rho_3}$};
\draw[] ({\ra*cos(300)}, {\ra*sin(300)}) node[below right]{$M^{\rho_1}$};
\draw[] ({\ra*cos(330)}, {\ra*sin(330)}) node[below right]{$M^{\rho_2}$};

\end{tikzpicture}
\caption{Schematic for one quarter of a surface $M(\Gamma, \Omega)$ with $\Gamma = \Z_2 \times *223$ and configuration $1+ 3\rho_2+\rho_1\rho_3$, depicting six fundamental chambers for $\Gamma$.  %The entire surface is obtained by doubling the depicted portion along $M^{\rho_1}$.
}
\label{Fplat}
\end{figure}

\begin{remark}
The type $\bunder$ can be equivalently described as follows: $f$ is the number of components of $\partial \Omega$ other than $\partial \Omegaout$, $e_i$ is the number of edges of $\Omega^{\tau}$ which meet only $\Omega^{\rho_i}$ along their endpoints, and $v_{ij}$ is the number of edges of $\Omega^\tau$ which meet $\Omega^{\rho_i} \cap \Omega^{\rho_j}$ along their endpoints.  See Figure \ref{Fplat}.
 \end{remark}

\begin{notation}
\label{NMgb}
Clearly, the decomposition $\partial \Omega = \cup_{i=1}^n \Omega^{\rho_i}$ into panels can be reconstructed from $\bunder$, and we sometimes write $\Omega = \Omega(\ub)$ and  $M = M(G, \bunder)$ to denote the dependence of $M$ on $(G, \bunder)$.  In the case where $G$ is trivial, the type $\bunder$ is a positive integer $a$, and we adopt the simplified notation $M(a)$. 
\end{notation}

The genus of $M(G, \bunder)$ is recorded below for the reader's convenience.

\begin{lemma}%[Properties of {$M(\Gamma, \Omega)$}]
\label{Lbcplat}
\phantom{ab}
%as in \ref{dfund} and $(\Gamma, \Omega)$ as in \ref{Agamma}, 
\begin{enumerate}[label=\emph{(\roman*)}]
\item $*k_{12}k_{13}k_{23}$ has order $4/( \sum_{i < j} \frac{1}{k_{ij}}-1)$, and $*kk$ has order $2k$. 
\item The subgroup $\langle \rho_i , \rho_j \rangle  \leq *k_{12}k_{13}k_{23}$ has order $2k_{ij}$. 
\item The genus of $M(G, \bunder)$ is $|G| \big(f+ \frac{1}{2}\sum_i e_i + \frac{1}{2} \sum_{i< j} \frac{v_{ij}}{k_{ij}}\big) - 1$,
with the same convention on sums as in Definition \ref{dplatclass}.
\end{enumerate}
\end{lemma}
\begin{proof}
Items (i) and (ii) are standard, and (iii) follows from Proposition \ref{Agamma} and the definitions.  
%Item (i) follows from Lemma \ref{LMomega}, the properties in Assumption \ref{Agamma}, and the definitions, and item (ii) follows from Lemma \ref{Lorbstab} and Definition \ref{dplatclass}.  For (iii), that $\tau$ is separating was already proved in (i), so let $\rho_i$ be one of the generators for $G$.  Because $\rho_i$ commutes with $\tau$, $\rho_i$ acts as a reflection on the genus zero quotient $\pi(M)$, hence is separating on $\pi(M)$.  Since $M$ is obtained from doubling $\pi(M)$ along its boundary, $\rho_i$ is separating on $M$.
\end{proof}

The following two examples are used to demonstrate Notation \ref{NMgb} and will be used in the proof of Theorem \ref{thm:group_existence}.

\begin{example}
\label{EMGtori}
We classify the $M(G, \bunder)$ from Proposition \ref{Agamma} with genus one:  using the definitions and Lemma \ref{Lbcplat}, all such  examples are
$M(2)$, $M(\Z_2, 1)$, $M(D_k, 2\rho_1 \rho_2)$, $M(D_2, \rho_1) = M(D_2, \rho_2)$, $M(\Z_2 \times D_k, \rho_2\rho_3)$, and $M(\Z_2 \times D_2, \rho_1 \rho_3) = M(\Z_2 \times D_2, \rho_1 \rho_2)$. 
\end{example}

\begin{example}
The Lawson surfaces $\xi_{g, 1}$ are examples of various $M(G, \bunder)$ from Proposition \ref{Agamma}, as follows:
\begin{itemize}
\item $\xi_{2, 1}$ is an $M(3)$, an $M(\Z_2, 1+ \rho_1)$, and an  $M(D_2, \rho_1+\rho_1\rho_2)$. % Note also that  $M(D_2, \rho_1+\rho_1\rho_2)=M(D_2, \rho_2 + \rho_1 \rho_2)$.
\item $\xi_{3, 1}$ is an $M(D_2,1)$. 
\item $\xi_{k-1, 1}$ is an $M(D_k, \rho_1)$ and an $M(\Z_2\times D_k, \rho_1 \rho_2)$. 
\end{itemize}
\end{example}

The following lemma is used in Section~\ref{sec:global_existence}.

\begin{lemma}
\label{lem:2basic}
Let $M$ be a closed orientable surface of genus $\gamma>1$. Suppose that $\tau\ne \rho_1$ are two reflections on $M$ such that $(M,\tau)$ and $(M,\rho_1)$ are both basic reflection surfaces. Then $\langle \tau,\rho_1\rangle = \Z_2\times\Z_2 $ and $M$ with this action of  $\Z_2\times\Z_2$ is of type $M(\Z_2,(\gamma+1)\rho_1)$.
\end{lemma}
\begin{proof}
As is explained in Section~\ref{sec:prelim}, one can always find a hyperbolic metric $h$ such that $\tau,\rho_1$ are both isometries of $(M,h)$. Since the isometry group of hyperbolic surfaces is finite, there exists $k>1$ such that $(\tau\rho_1)^k = e$ so that $\langle \tau,\rho_1\rangle = D_k$. Then $M=M(D_k,\Omega)$ as in Lemma~\ref{Agamma}. Since $(M,\tau)$ is basic, we conclude that $\Omega$, which can be identified with a subset of $M/\langle\tau\rangle$, has genus zero. Suppose that $\Omega$ has more than one boundary component. Then either there exists a connected component $S$ of $\bd\Omega$ properly contained in $\Omega^\tau$ (or $\Omega^{\rho_1}$), or two points $p,q$ in $\Omega^\tau$ (or $\Omega^{\rho_1}$) lie on different connected components. Indeed, otherwise $\bd \Omega$ has exactly two components, one in $\Omega^\tau$, one in $\Omega^{\rho_1}$ and, hence, $M(D_k,\Omega)$ is a torus, a contradiction. 
If $S$ is a connected component of $\bd \Omega$ such that $S\varsubsetneq \Omega^\tau$, then $\Omega\cup(\tau\Omega)\setminus S$ is connected. Since $\Omega\cup(\tau\Omega)$ can be identified with a subset of $M/\langle\rho_1\rangle$, we obtain a contradiction with the fact that the latter has genus $0$. Similarly, if $p,q$ are in $\Omega^\tau$ and lie on different connected components of $\bd\Omega$, we repeat the same argument with $S = I\cup\tau(I)$, where $I\subset\overline\Omega$ is a segment connecting $p$ and $q$.

Thus, we have proved that $\Omega$ is diffeomorphic to a disk. Let $m$ be the number of connected components of $\Omega^\tau$, which coincides with the number of connected components of $\Omega^{\rho_1}$. Note that $m>1$ as otherwise, $M$ is a sphere.
On one hand, using Euler's formula we compute $\chi(M) = 2m-2km+2k$. On the other hand, the number of connected components of $M^\tau$ equals $\gamma+1$ and does not exceed $m$, so one has $4-2m\geq 2m-2km+2k$ or $m(k-2)\geq k-2$. Using $m>1,k>1$, we conclude that $k=2$ and then it is easy to see that $m=\gamma+1$, completing the proof.
\end{proof}

\subsection{Surfaces with boundary} 
\label{sec:actions_Sex}
For each surface $N$ with boundary one can construct a closed surface $\wt N$ obtained by gluing two copies of $N$ along the identity map of the boundary $\partial N$. The surface $\wt N$ is naturally endowed with a reflection $\iota$ interchanging the two copies of $N$, so that the fixed set $\wt N^{\iota}$ is the image of the boundary $\partial N$. This way $\iota$ is a separating reflection of $\wt N$ and the factor space can be naturally identified with $N$. We call $\wt N$ with this action of $\mathbb{Z}_2 = \langle\iota\rangle$ a {\em double} of $N$. Conversely, any closed surface with a separating reflection is a double of some surface with boundary.

A proper (that is, boundary-preserving) action $T\colon \Gamma\times N\to N$ of a group $\Gamma$ on $N$ induces an action $\wt T$ of $\wt \Gamma = \Gamma\times \mathbb{Z}_2$ on $\wt N$. Conversely, if $\wt\Gamma$ acts on a closed surface $M$ with the nontrivial element of $\mathbb{Z}_2$ acting by a separating reflection $\iota$, then $\Gamma$ acts properly on the surface with boundary $N=M/ \langle \iota \rangle$. It is easy to see that this bijection preserves topological equivalence of surfaces with group actions. 

\begin{example}
If the action $T$ is trivial, then we recover topological degenerations of surfaces with boundary. For example, the surface $M(a)$ from Section~\ref{sec:or_Z2} can be naturally identified with $\wt{N(a)}$, where $N(a)$ is a surface of genus $0$ with $a$ boundary components.
\end{example}

\begin{example}
The closed surfaces $M(G,\ub)$ can be viewed as doubles of $N_\tau(G,\ub)$, where the subscript $\tau$ indicates that we choose $\iota=\tau$ to be the doubling involution. Surfaces $N_\tau(G,\ub)$
have genus $0$ and have number of boundary components equal to the genus of $M(G,\ub)$ plus $1$. Examples of surfaces of this type are proper surfaces of genus $0$ in $\mathbb{B}^3$ with boundary components arranged according to the standard action of $G$ on $\mathbb{S}^2 = \bd \mathbb{B}^3$.
\end{example}

\begin{example}
If the group $G$ splits off a $\Z_2$ factor, then it is possible to view $M(G,\ub)$ as a double in a different way by taking $\iota$ to be the non-trivial element of that factor (provided it is a separating reflection). This happens when $G = \Z_2, D_2$ or $\Z_2\times D_k$ in which case we can set $\iota = \rho_1$ without loss of generality. The corresponding surface is denoted as $N_{\rho_1}(G,\ub)$.

For $G =\mathbb{Z}_2$ the surface $N_{\rho_1}(\Z_2,f+e_1\rho_1)$ has $e_1$ boundary components and genus $f$ if $e_1>0$, or $2$ boundary components and genus $f-1$ otherwise. In particular, for different values of $\{f,e\}$ these surfaces cover all possible topological types of surfaces with boundary.
For $G = D_2$ the surface $N_{\rho_1}(D_2, f+e_1\rho_1+e_2\rho_2 + v_{12}\rho_1\rho_2)$ has $2e_1+v_{12}$ boundary components and genus $2f+e_2$ if $e_1+v_{12}>0$, or $2$ boundary components and genus $2f+e_2-1$ otherwise. Finally, for $G = \Z_2\times D_k$ the surface $N_{\rho_1}(G,\ub)$ has $k(2e_1+v_{12}+v_{13})$ boundary components and genus $k(2f+e_2+e_3)+v_{23}$ if $2e_1+v_{12}+v_{13}>0$, or 2 boundary components and genus $k(2f+e_2+e_3)+v_{23}-1$ otherwise. Examples of surfaces of this type are proper surfaces in $\mathbb{B}^3$ that look like two parallel copies of the disk connected by tubes and half-tubes that are arranged in $G/\Z_2$-invariant way on $\mathbb{D}^2$ for the standard action of $G/\Z_2$ on $\mathbb{D}^2$.

\end{example}

%\times \mathbb{Z}_2$, then $\wt \Gamma = \mathbb{Z}_2\times \mathbb{Z}_2\times\mathbb{Z}_2$ and, since $\rho_1$ and $\rho_2$ are interchangeable, closed surfaces defined in Section~\ref{sec:or_Z2Z2Z2} can be viewed as doubles in two different ways.
%
%If $\iota = \tau$, then $M_{a,b,c,d} = \wt{N}_{\tau,a,b,c,d}$, where $N_{\tau,a,b,c,d}$ is a surface of genus $0$ with $4a+2b+2c+d$ boundary components and the group $\mathbb{Z}_2\times \mathbb{Z}_2 = \{e,\rho_1\}\times \{e,\rho_2\}$ acts by .....
%
%If $\iota = \rho_1$, then $M_{a,b,c,d} = \wt{N}_{\rho_1,a,b,c,d}$, where $N_{\rho_1,a,b,c,d}$ is a surface of genus $2a+c$ with $2b+d$ boundary components if either $b\ne 0$ or $d\ne 0$, and of genus $2a+c-1$ with $2$ boundary components if $b=d=0$. The group $\mathbb{Z}_2\times \mathbb{Z}_2 = \{e,\tau\}\times \{e,\rho_2\}$ acts as...
%\end{example}
%
%\begin{example}
%Consider $\Gamma = D_n$, $n\ne 2$, then $\wt \Gamma = \mathbb{Z}_2\times D_n$ and for $n\ne 2$ there is only one $\mathbb{Z}_2$ factor in $\wt{\Gamma}$. Thus, setting $\iota = \tau$ yields $M^n_{a,b,c,d} = \wt{N}^n_{a,b,c,d}$, where $N^n_{a,b,c,d}$ is a surface of genus $0$ with $n(2a+b+c)+d$ boundary components. The group $D_n$ acts by...
%\end{example}

\subsection{Minimal immersions of basic reflection surfaces}
\label{ssdoub}
Here we study isometric minimal immersions of basic reflection surfaces $(M, \tau)$ by first eigenfunctions; in particular minimal immersions into the sphere $\Sph^n$, and  free boundary minimal immersions into the ball $\B^n$.  

We prove that such immersions are necessarily embeddings; more strongly, such immersions are minimal \emph{doublings}, either of a $2$-sphere $\Sph^2 \subset \Sph^n$ or of a $2$-ball $\B^2\subset \B^n$, in the sense of \cite[Definition 1.1]{LDG}.  To be precise, we recall the definition from \cite{LDG} before proceeding.

\begin{definition}
\label{ddoub}
Given a Riemannian manifold $N$ and a surface $\Sigma \subset N$, a (surface) \emph{doubling} of $\Sigma$ in $N$ is a smooth surface $M$ in $N$ satisfying
\begin{enumerate}[label={(\roman*)}]
\item The nearest-point projection $\pi : M \rightarrow \Sigma$ is well-defined;
\item $\pi(M)$ is a disjoint union  $\Sigma_1 \cup \Sigma_2$, with $\Sigma_1$ a disjoint union of curves and isolated points, $\Sigma_2\subset \Sigma$  a domain, and moreover
\item $\pi|_{\pi^{-1}(\Sigma_1)} : \pi^{-1}(\Sigma_1) \rightarrow \Sigma_1$ is a diffeomorphism; and
\item $\pi|_{\pi^{-1}(\Sigma_2)} : \pi^{-1}(\Sigma_2) \rightarrow \Sigma_2$ is a smooth $2$-sheeted covering map.
\end{enumerate}
The doubling is called \emph{minimal} if $M$ is minimal, and called \emph{symmetric} if $\Sigma$ is fixed pointwise and $M$ is fixed setwise by an involutive isometry of $N$.
\end{definition}

Actually, Definition \ref{ddoub} modifies \cite[Definition 1.1]{LDG} slightly in order to allow doublings in codimension bigger than one.  For nonorientable $(M, \tau)$, the corresponding doublings of $\Sph^2$ or of $\B^2$ have codimension at least two. 

In each case, we show that $M$ has area less than twice the area of the surface it doubles, less than $2|\Sph^2| = 8\pi$ in the $\Sph^n$ case, and less than $2|\B^2| = 2\pi$ in the $\B^n$ case. 

An important difference between the $\Sph^n$ and $\B^n$ settings is the space of ambient coordinate functions on $\Sph^n$ is $(n+1)$-dimensional, while on $\B^n$ it is only $n$-dimensional. 
For this reason, the proofs in the $\B^n$ case are more involved, and in particular rely on the Morse-inequality type results Lemma \ref{Lmorse} and Lemma \ref{Lbdgraph}.

Before proceeding, we introduce some notation.
Given a vector subspace $V\subset \R^n$, denote by $V^\perp$ the orthogonal complement of $V$ in $\R^n$, and define the reflection $\Rcapunder_V : \R^n \rightarrow \R^n$ with respect to $V$, by
\begin{align*}
\Rcapunder_V : = \Pi_V - \Pi_{V^\perp},
\end{align*}
where $\Pi_V$ and $\Pi_{V^\perp}$ are the orthogonal projections of $\R^n$ onto $V$ and $V^\perp$ respectively.  

Define also the nearest point projection $\Pi^\Sph_V : \R^n \setminus V^\perp \rightarrow \Sph(V)$ by
\begin{align*}
\Pi^\Sph_V = \Pi^\Sph \circ \Pi_V, 
\quad
\text{where}
\quad
\Pi^\Sph (p ) = p / |p|,
\end{align*}
and we denote $\Sph(V): = \Sph^{n-1} \cap V$.

\begin{lemma}
\label{Lgraphsph}
If $\iota: M \rightarrow \Sph^n$ is a full branched minimal immersion with
\begin{enumerate}[label=\emph{(\alph*)}]
\item $(M, \tau)$ is a closed basic reflection surface with genus $>0$, and
\item $\iota$ is a (branched) isometric immersion by first eigenfunctions,
\end{enumerate}
then the following hold:
\begin{enumerate}[label=\emph{(\roman*)}]
\item $n \geq 3$, with $n = 3$ if and only if $M$ is orientable. 
\item $\tau = \iota^* \Rcapunder_P$ for some $3$-dimensional subspace $P \subset \R^{n+1}$; 
\item For each oval $O \subset M^\tau$, $\iota(O) \subset \Sph(P)$ is embedded and strictly convex; 
\item $M$ is a symmetric minimal doubling of $\Sph(P)$ in  $\Sph^n$ as in \ref{ddoub};
\item The area of $M$ is strictly less than $8\pi$.
\end{enumerate}
In particular, $\iota$ is an embedding.  Finally, if $M$ is $\langle \tau \rangle \times G$-invariant as in \ref{Agamma},  then the $\langle \tau \rangle \times G$-action  induced by $\iota$ is conjugate in $O(4)$ to the standard product action.
\end{lemma}
\begin{proof}
By the assumptions, Propositions  \ref{Lasymcl} and \ref{Lrhoquo2}(ii) imply that
\begin{align*}
n+1 \leq \dim \Ecal_2 = \dim \Ecal_2^- + \dim \Ecal_2^+ \leq \dim \Ecal_2^-+ 3, 
\end{align*}
and $n+1 \leq 4$ if $M$ is orientable.  By the assumptions, Lemma \ref{Levalbdbd} implies $|M|< 8\pi$, proving (v); thus, work of Li-Yau \cite[Corollary 10]{LiYau} implies $\iota : M \rightarrow \Sph^n$ is an embedding.  Combined with the preceding dimension bounds, it follows that $n \geq 3$, with $n > 3$ only if $M$ is nonorientable.  Finally, Alexander duality implies that no closed nonorientable surface may be embedded into $\Sph^3$, concluding the proof of (i).

Item (ii) follows from the proof of (i), because $\iota$ is an isometric embedding. 

For convenience, in the remainder of the proof we identify $M$ with its image in $\Sph^n$ under $\iota$.  For (iii), note from (ii) that $\Ecal_2^+$ is spanned by the restrictions to $M$ of the nonzero coordinate functions on the $2$-sphere $\Sph(P)$.  It follows from this and Proposition \ref{Lrhoquo2}(i) that for each oval $O \subset M^\tau$, each great circle in $\Sph(P)$ has intersection number either $0$ or $2$ with $O$.  This implies (iii). 
Next, because $\Ecal_2^+$ is $3$-dimensional and the coordinate functions in $\Ecal_2^+$ all vanish on $P^\perp$,  it follows from  Proposition \ref{Lrhoquo2}(i) that
\begin{align} 
\label{Etanvert}
(T_p M + \mathrm{span}(p) ) \cap P^\perp = \{0\}
\quad
\forall p \in M \setminus M^\tau.
\end{align}
In particular, $M \cap P^\perp = \varnothing$, proving the nearest-point projection $\Pi^\Sph_P$ restricts to a local homeomorphism on $M \setminus M^\tau$.

For ease of notation, now set $\pi : = \Pi^{\Sph}_P|_M$, and in this proof denote by $\pi_G: M \rightarrow M / \langle \tau\rangle$ the quotient projection.  For any $p \in M$, the differential $d\pi_p : T_p M \rightarrow T_{\pi(p)} \Sph(P)$ satisfies
\begin{align*}
d \pi_p = \frac{\Pi_{T_{\pi(p)} \Sph^n} \circ\Pi_P }{|\Pi_P(p)|}.
\end{align*}
In particular, $\ker d \pi_p = T_p M \cap  \Pi_{T_p \Sph^n} P^\perp$, so $d\pi_p$ is an isomorphism for each $p \in M \setminus M^\tau$ by \eqref{Etanvert}.  Thus, $\pi|_{M \setminus M^\tau}$ is a local diffeomorphism by the inverse function theorem.

Since $\pi$ is constant on the fibers of $\pi_G$, there is a unique continuous map $\varphi : \pi_G(M) \rightarrow \Sph(P)$ satisfying $\varphi \circ \pi_G = \pi$.  By Definition \ref{dbasref}, $\pi_G(M)$ may be identified with a subset of $\Sph^2$, with each component $S\subset \partial \pi_G(M)$ bounding a disk $D_S \subset \Sph^2 \setminus \pi_G(M)$.  For any such component $S$, item (iii) shows that $\varphi (S)$ is an embedded circle bounding a strongly convex disk $D_{\varphi(S)} \subset \Sph(P)$.  By minimality of $M$, the image under $\varphi$ of a small enough neighborhood of $S$ in $\pi_G(M)$ is disjoint from the interior of $D_{\varphi(S)}$.  From this and the fact that $\pi|_{M \setminus M^\tau}$ is a local diffeomorphism, $\varphi$ can be extended to a local homeomorphism $\Phi : \Sph^2 \rightarrow \Sph(P)$ with the property that $\Phi(D_S) = D_{\varphi(S)}$ for each component $S$ of $\partial \pi_G(M)$. 
Since $\Sph(P)$ is simply connected and $\Sph^2$ is connected, the local homeomorphism $\Phi$ is actually a global homeomorphism, so $\varphi = \Phi |_{\pi_G(M)}$ is a homeomorphism onto its image.  Since $\pi_G|_{M \setminus M^\tau}$ is a $2$-sheeted covering map, so is $\varphi \circ \pi_G|_{M \setminus M^\tau} = \pi|_{M \setminus M^\tau}$, proving (iv).

Finally, suppose $n=3$ and $M$ is as in Proposition \ref{Agamma}.  As shown above, $ \R^4 \cong \Ecal_{\lambda_1} (M, g)$ decomposes as
\begin{align*}
\R^4 = \Ecal^-_{\lambda_1} \oplus \Ecal^+_{\lambda_1}
\quad
\text{with} 
\quad
\dim \Ecal^-_{\lambda_1} = 1,
\quad
\dim \Ecal^+_{\lambda_1} = 3, 
\end{align*}
where $\Ecal^-_{\lambda_1} : = \Acal_{\tau}(\Ecal_{\lambda_1})$ and $\Ecal^+_{\lambda_1} : = \Scal_{\tau} (\Ecal_{\lambda_1})$. 
We may therefore identify $\Ecal^+_{\lambda_1}$ with $\R^3 \subset \R^4$ and identify $\{\Id\} \times G \cong G$ with a subgroup of $O(3)$.  

Let $\rho_i$ be a generator for $G$.  From the $\Z_2 \times G$-invariance,  $\R^3 \cong \Ecal^+_{\lambda_1}$ admits an orthogonal direct sum decomposition
\begin{align*}
\R^3 \cong \Ecal^+_{\lambda_1} = \Acal_{\rho_i} ( \Ecal^+_{\lambda_1}) \oplus \Scal_{\rho_i}(\Ecal^+_{\lambda_1}).
\end{align*}
From Proposition \ref{Agamma}, $\rho_i$ is a separating reflection, so from \ref{Lasymcl}(iii) and arguing as in the proof of Lemma \ref{Lrhoquo3}, it follows that $\dim \Acal_{\rho_i} ( \Ecal^+_{\lambda_1}) = 1$ and $\dim \Scal_{\rho_i}(\Ecal^+_{\lambda_1}) = 2$.  In particular, each $\rho_i$ is a reflection through a $2$-dimensional subspace of $\R^3$.  This completes the proof when $G = 1*$, so suppose $G$ has at least two generators, and fix distinct generators $\rho_i, \rho_j$.  %There are two cases. 

%Since $(M, \tau)$ is basic and is $\Z_2 \times G$-invariant, each fixed-point set $M^{\rho_i}$ separates $M$ for $i=1, 2, 3$, so Proposition \ref{Lasymcl}(iii) implies $\dim \Acal_{\rho_i}(\Ecal^+_{\lambda_1}) \leq 1$.  On the other hand, fixing $p \in M^{\rho_i}$ and arguing as in the proof of of Lemma \ref{Lrhoquo3} shows that the linear map $T: \Scal_{\rho_i} (\Ecal^+_{\lambda_1}) \rightarrow \R \times T^*_p M$ given by $T u = (u(p), du_p)$ is injective, and because of the $\rho_i$-symmetry, $\dim \Scal_{\rho_i} ( \Ecal^+_{\lambda_1}) \leq 2$.  

%Because $\R^3 = \Acal_{\rho_i}(\Ecal^+_{\lambda_1}) \oplus \Scal_{\rho_i}(\Ecal^+_{\lambda_1})$, it follows that each of the preceding inequalities is an equality.  In particular, each $\rho_i$ acts on $O(3)$ by reflection through a $2$-dimensional subspace of $\R^3$. 

Because $\rho_i, \rho_j$ are reflections in $2$-planes meeting along a line $\ell$, the product $\rho_i \rho_j$ is a rotation fixing $\ell$ with order $k_{ij}$ (recall Definition \ref{dplat}).  Now let $\gamma$ be an embedded loop in $M$ which is stabilized by $\langle \rho_i, \rho_j\rangle$; such a loop clearly exists, for $\gamma$ can be taken to be a small loop around a point $p \in  M^{\rho_i} \cap M^{\rho_j}$ if the latter is nonempty, while if $M^{\rho_i} \cap M^{\rho_j} = \varnothing$, there is an oval $O \subset M^{\tau}$ with $\mathrm{Stab}_G(O) = \langle \rho_i, \rho_j\rangle$, and $\gamma$ can be taken to be an embedding of $O$. 

 Because $\iota$ is an embedding, $\iota ( \gamma)$ winds once around $\ell \subset \R^3$, hence the planes in $\R^3$ fixed by $\rho_i$ and $\rho_j$ meet at angle $\pi/ k_{ij}$.  Thus, each $\rho_i$ is a reflection through a $2$-dimensional subspace, and the planes corresponding to each pair $\rho_i, \rho_j$ of generators meet at angle $\pi/ k_{ij}$, so the action of $G$ is conjugate in $O(3)$ to the standard action of $G$ on $\R^3$. 
\end{proof}

\begin{remark}
In combination with results \cite{ChoeSoret, KusnerMcGrath} proving that $\lambda_1 =2$, Lemma \ref{Lgraphsph}(v) shows that the following minimal surfaces in $\Sph^3$ have area below $8\pi$:
\begin{itemize}
\item The Karcher-Pinkall-Sterling doublings of $\Sph^2$ \cite{KPS}. 
\item Doublings of $\Sph^2$ constructed by Kapouleas and McGrath \cite{Kapouleas, KapMcG, LDG}.  
\item The Lawson surfaces $\xi_{\gamma, 1}$ \cite{Lawson}.
\end{itemize}
This result for the doublings from \cite{KPS} appears to be new.  
\end{remark}

\begin{cor}
\label{Prealcl}
Let $(M, T)$ be a closed surface with group action of the type described in Proposition \ref{Agamma}.  If there exists a $T$-invariant $\bar{\lambda}_1^T$-maximal metric $g$ on $M$, then there exists an isometric minimal embedding $\iota: (M,g) \rightarrow \Sph^3$ by first eigenfunctions satisfying the conclusions of Lemma \ref{Lgraphsph}.
\end{cor}
\begin{proof}
The existence of an isometric branched minimal immersion $\iota: M \rightarrow \Sph^n$ by first eigenfunctions follows from Theorem~\ref{thm:symm_crit_L}.  The fact that $n = 3$ and the remaining conclusions follow from Lemma \ref{Lgraphsph}.
\end{proof}

\begin{lemma}
\label{Lgraphbd}
If $\iota : M \rightarrow \B^n$ is a full branched minimal immersion with free boundary such that
\begin{enumerate}[label=\emph{(\alph*)}]
\item $(M, \tau)$ is a basic reflection surface which is not a disk, and
\item $\iota$ is a (branched) isometric immersion by first eigenfunctions,
\end{enumerate}
then the following hold: 
\begin{enumerate}[label=\emph{(\roman*)}]
\item $n \geq 3$, with $n=3$ if and only if $M$ is orientable;
\item $\tau = \iota^* \Rcapunder_P$ for some $2$-dimensional subspace $P \subset \R^{n}$; 
\item Each component of $M^\tau$ is an embedded convex curve in $P$; 
\item The radial projection $p \mapsto p/ | p|$ restricts to a homeomorphism from $\Pi_P(\partial M \cup M^\tau_\partial)$ to the unit circle $\partial \B^2 \subset P$;
\item $M$ is a symmetric doubling of $\B^2=P\cap \B^n$ as in \ref{ddoub};
\item The area of $M$ is strictly less than $2\pi$.
\end{enumerate}
In particular, $\iota$ is an embedding.
\end{lemma}
\begin{proof}

By the assumptions, Propositions \ref{Lasymcl} and \ref{Lrhoquo3}(ii) imply that 
\begin{align*}
n \leq \dim \Ecal_1 = \dim \Ecal_1^- + \dim \Ecal_1^+ \leq \dim \Ecal_1^- + 2, 
\end{align*}
and $n \leq 3$ if $M$ is orientable.  By the assumptions, Lemma \ref{Levalbdbd} implies $|M| < 2\pi$, proving (vi); thus work of Volkmann \cite[Corollary 4.4]{Volkmann} implies $\iota : M \rightarrow \B^n$ is an embedding.  Combined with the preceding dimension bounds, it follows that $n \geq 3$ with $n >3$, only if $M$ is nonorientable.  Finally, Alexander duality implies that no nonorientable surface with boundary may be properly embedded in the ball $\B^3$, concluding the proof of (i).  

Item (ii) follows from the proof of (i), because $\iota$ is an isometric embedding.   In particular, note that $\Ecal_1^+$ is spanned by the restrictions to $M$ of the nonzero coordinate functions on $P$.

For convenience, in the remainder of the proof we identify $M$ with its image in $\B^n$ under $\iota$.
The proofs of parts (iii) and (iv) involve the application of Lemma \ref{Lbdgraph} to a nonzero $u \in \Ecal_1^+$, so we first check that its hypotheses (a)-(c) hold: $(M, \tau)$ is basic by assumption, \ref{Lbdgraph}(a) holds for $u \in \Ecal_1^+$ by assumption (b) above, and \ref{Lbdgraph}(b)-(c) hold because  the free boundary minimal surface $M \subset \B^n$ is analytic up to its boundary by \cite[Theorem 1]{Gruter}.   

For (iii), fix $p \in M^\tau \setminus \partial M$.  By the symmetry and because $\dim \Ecal_1^+ = 2$, $p$ is a critical point for some nonzero $u \in \Ecal_1^+$. By Lemma \ref{Lbdgraph}, this critical point is a multiplicity-$1$ saddle.  In particular, the principal curvature of $M^\tau$ at $p$ is not zero.  The second conclusion of Lemma \ref{Lbdgraph} implies each component of $M^\tau$ is embedded, and (iii) follows. 

For (iv), fix a $1$-dimensional linear subspace $\ell \subset P$, and let $u \in \Ecal_1^+$ be the restriction to $M$ of a $\Rcapunder_P$-even coordinate function vanishing on $\ell$.  Proposition \ref{Lrhoquo3}(i) then implies that $\ell$ meets $\Pi_P(\partial M \cap M^\tau_\partial)$ in exactly two points.  More  strongly, we claim each of the two rays from the origin in $\ell \cap P$ meets $\Pi_P (\partial M \cap M^\tau_\partial)$.  To see this, let $u' \in \Ecal_1^+$ be the restriction to $M$ of a $\Rcapunder_P$-even linear function vanishing on the line $\ell^\perp \subset P$ orthogonal to $\ell$.  If the claim were not true, $u'$ would attain an interior critical point at a point of $M \setminus M^\tau$, contradicting Lemma \ref{Lbdgraph}.  Thus, each ray from the origin meets $\Pi_P(\partial M \cap M^\tau_\partial)$ exactly once, and (iv) follows. 

For (v) we first show $\pi : = \Pi_P|_\Sigma$ restricts to a local diffeomorphism on $M \setminus M^\tau$.  By Lemma \ref{Lbdgraph}, the restriction to $M$ of each coordinate function on $P$ has no interior critical points on $M \setminus M^\tau$, so $\pi$ restricts to a local diffeomorphism on $\mathrm{int}(M) \setminus M^\tau$. 

Next we show $\pi$ is a diffeomorphism up to the boundary: fix $p \in \partial M \setminus P$, and let $u \in \Ecal_1^+$ be a nonzero function vanishing at $p$.  By Proposition \ref{Lrhoquo3}(i), $du_p  \neq 0$.  On the other hand, $T_p M \subset \R^n$ is a linear subspace containing $p$ by the free boundary condition; combining with the preceding reveals that $T_p M \cap P^\perp = \{0\}$.  Therefore $d\pi_p : T_p M \rightarrow T_{\pi(p)} P$ is surjective; combined with the preceding, it follows that $\pi$ restricts to a local diffeomorphism on $M \setminus M^\tau$.  

Since $\pi$ is constant on the fibers of $\pi_G$, there exists a unique continuous map $\varphi : \pi_G(M) \rightarrow \B^2$ satisfying $\varphi \circ \pi_G = \pi$.  Because $(M, \tau)$ is a basic reflection surface (recall Definition \ref{dbasref}), $\pi_G( M)$ may be identified with a subset of the unit disk $\B^2$ obtained by removing a finite collection of disjoint subdisks.  Fix a component $S \subset \partial \pi_G(M)$ disjoint from $\pi_G(\partial M)$.  By (ii), $\varphi(S)$ is an embedded, strictly convex circle bounding a disk $D_{\varphi(S)} \subset \B^2$.  By the minimality, the image under $\varphi$ of a neighborhood of $S$ in $\pi_G(M)$ is disjoint from the interior of $D_{\varphi(S)}$.  From this and the fact that $\pi|_{M \setminus \Sigma^\tau}$ is a local diffeomorphism, $\varphi$ can be extended to a local homeomorphism $\Phi : \B^2 \rightarrow \B^2$ with the property that $\Phi(D_S) = D_{\varphi(S)}$ for each component $S$ of $\partial \pi_G(M)$ disjoint from $\pi_G(\partial M)$. 

Since $\B^2$ is simply connected, the local homeomorphism $\Phi$ is actually a global homeomorphism, so $\varphi = \Phi|_{\pi_G(\Sigma)}$ is a homeomorphism onto its image.  Since $\pi_G|_{M \setminus M^\tau}$ is a $2$-sheeted covering map, then so is $\varphi \circ \pi_G|_{M \setminus M^\tau} = \pi|_{M \setminus M^\tau}$, proving (v).
\end{proof}

\begin{remark}
In combination with results \cite{McGrath, KusnerMcGrath} proving that $\sigma_1 = 1$, Lemma \ref{Lgraphbd}(v) shows that the following free boundary minimal surfaces in $\B^3$ have area below $2\pi$:
\begin{itemize}
\item Doublings of the disk constructed in \cite{Zolotareva}.
\item Desingularizations of two disks announced in \cite{KapLi}.
\end{itemize}
\end{remark}

As in the closed case, the existence of a $\bar{\sigma}_1^T$-maximizing metric on one of the surfaces defined earlier leads to a free boundary minimal embedding into $\B^3$. 
\begin{cor}
\label{Prealbd}
Let $(M, T)$ be a closed surface with group action as defined in \ref{sec:actions_Sex}.  If there exists a $T$-invariant $\bar{\sigma}_1^T$-maximal metric $g$ on $M$, then there exists an isometric free boundary minimal embedding $\iota: (M,g) \rightarrow \B^3$ by first eigenfunctions satisfying the conclusions of Lemma \ref{Lgraphbd}.
\end{cor}
\begin{proof}
Analogous to the proof of Corollary \ref{Prealcl}%; this time the embeddedness follows from Lemma \ref{Levalbdbd} and \cite[Corollary 4.4]{Volkmann}. 
\end{proof}

%===================================================
%===================================================
%===================================================
%===================================================
%===================================================
%===================================================

%===================================================
%===================================================
%===================================================
%===================================================
%===================================================
%===================================================

\section{Topological degenerations} 

\subsection{Equivariant surgery and degenerations}
\label{sec:top_degen}

Let $M$ be a closed surface with an action $T:\Gamma\times M\to M$ of the group $\Gamma$, and let $\fC = \cup c_i\subset M$ be a $\Gamma$-invariant union of disjoint non-contractible simple closed curves in $M$. The action of $\Gamma$ can be extended to the boundary of $M\setminus \fC$, and we denote by $M'_\fC$ the surface obtained by contracting each boundary component of  $M\setminus \fC$ to a point, and by $T'_\fC:\Gamma\times M'_\fC\to M'_\fC $ the corresponding action of $\Gamma$ on $M'_\fC$. We say that $(M'_\fC, T'_\fC)$ is obtained from $(M,T)$ by collapsing $\fC$, and we denote by $P_\fC\subset M'_\fC$ the $T'_{\fC}$-invariant set of points corresponding to the contracted boundary components. 

Equivalently, $(M'_\fC,T'_{\fC})$ can be obtained by an equivariant surgery along a neighbourhood of $\fC$. Indeed, let $\fU = \cup_i \fU_i$ be a $T$-invariant neighbourhood of $\fC$, where the sets $\fU_i$ are disjoint tubular neighborhoods of the curves $c_i$'. If the normal bundle to $c_i$ is orientable, then $\fU_i\approx\mathbb{S}^1\times \mathbb{B}^1$, so that $\del \fU_i$ is homeomorphic to two copies of $\mathbb{S}^1$. The surgery along $\fU_i$ then amounts to gluing two copies of a disk to $M\setminus \fU_i$ along the common boundary.
%, and, similarly, we denote by $\fU_i'$ a surface (homeomorphic to the sphere) obtained by gluing two disks to $\fU_i$. 
If the normal bundle to $c_i$ is not orientable, then $\fU_i$ is homeomorphic to the M\"obius strip, so that $\del \fU_i$ is a single circle. Surgery along $\fU_i$ in this case amounts to gluing a single copy of a disk to $M\setminus \fU_i$.
%, and this time $\fU_i'$ is homeomorphic to the projective space. 
Since $\fU$ is $T$-invariant, one can arrange all the gluings to be equivariant, so that the restriction of $T$ to $M\setminus \fU$ and $\fU$ respectively extends to the glued disks as well. In this interpretation, the role of $P_\fC$ is played by the centers of disks glued to $M\setminus \fU$.
% and, furthermore, one can also identify the complimentary set $P'_{\fC}\subset \cup_i \fU_i'$ of centers of the disks attached to $\fU$. 
 %This description makes it evident that $(M,T)$ can be reconstructed from $(M'_{\fC}, T'_{\fC})$ by the inverse surgery, i.e. by connecting neighbourhoods of points in $P_{\fC}$ to neighbourhoods of points in $P'_{\fC}$ by cylindrical tubes. 
% is homeomorphic to a surface produced by the surgery along an invariant tubular neighbourhood of $\fC$ and that such a surface admits an action of $\Gamma$ topologically equivalent to $T'_\fC$. For that reason we sometimes refer to this construction as ``equivariant surgery along $\fC$''. 

%THE SETS $\fU'$ DON'T SEEM TO BE NEEDED.
 
\begin{definition}
\label{def:degeneration}
Let $M$, $M'$ be closed connected surfaces on which the group $\Gamma$ acts by $T,T'$ respectively. We say that $(M',T')$ is a topological degeneration of $(M,T)$, and write
$(M',T')\prec (M,T)$, if there exists a non-empty $\Gamma$-invariant union of non-contractible disjoint simple closed curves $\fC\subset M$ such that $(M',T')$ is topologically equivalent to a $\Gamma$-invariant connected component $M'_{\fC,0}$ of $(M'_\fC,T'_\fC)$.
\end{definition}

\begin{definition}
We say that $(M',T')$ is an {\em elementary degeneration} of $(M,T)$ if $(M',T')\prec (M,T)$ and there does not exist $(M'',T'')$ such that $(M',T')\prec (M'',T'')\prec (M,T)$.
\end{definition}

\subsection{Collapsed sets and the inverse surgery.}
\label{sec:inverse_surgery}

If $(M',T')\prec (M,T)$, then there exists a non-empty $T'$-invariant collection of points $P'\subset M'$, corresponding to the points of $P_{\fC}\cap M'_{\fC,0}$, which we refer to as a {\em collapsed set}. Furthermore, the set $P'$ can be partitioned as $P'=\sqcup P'_i$ in the following way. Let $M_0\subset M$ be a connected component of $M\setminus \fU$ corresponding to $M'_{\fC,0}$. We say that $p'\sim q'$ for $p',q'\in P'$ if the corresponding boundary components of $M_0$ can be joined by a path in $M\setminus M_0$. The partition $P'=\sqcup P'_i$ is then formed by the equivalence classes. We remark that given $(M',T')\prec (M,T)$ the set $P'$ and its partition $P'=\sqcup P'_i$ is not unique as different choices of $\fC$ could correspond to topologically different collapsed sets.

If $(M',T')\prec (M,T)$ is a topological degeneration, then it is possible to construct $(M,T)$ from $(M',T')$ using equivariant surgery along the neighbourhoods of a collapsed set. We make use of a particular procedure described below. Let
$P'$ be a collapsed set for $(M',T')\prec (M,T)$, $P' =\sqcup P'_i$ be the corresponding partition, $\fC = \cup c_i$ be the collection of curves given by the Definition~\ref{def:degeneration}, and $M_0$ be the connected component of $M\setminus \fU$ corresponding to $M'_{\fC,0}$. For each connected component $M_i$ of $M\setminus M_0$, let $M''_i$ be a closed surface obtained by gluing a disk along each boundary component of $M''_i$ and let $P_i''$ be the collection of centers of those disks. Define a (possibly disconnected) surface $M'' = \sqcup M_i''$ and the collection $P'' = \sqcup P_i''\subset M''$, with the gluings done in such a way that the restriction of $T$ to $M\setminus M_0 = \cup M_i$ extends to an action $T''$ of $\Gamma$ on $M''$ for which $P''$ is $T''$-invariant. 

Each boundary component of $M_0$ is also a boundary component of some $M_i\subset M\setminus M_0$, and by identifying boundary components of $M_0$ and $M_i$ with points in $P'$ and $P_i''\subset P''$, respectively, this yields a natural assignment $\alpha\colon P'\to P''$. By definition of the equivalence relation on $P'$ one has that $p'\sim q'$ iff $p''$ and $q''$ belong to the same connected component of $M''$. This is the reason we use the same subscripts $i$ for $P_i'$ and $M_i''$, and, up to relabeling, we have that $p'\in P_i'$ iff $\alpha(p')\in M_i''$. Finally, we observe that performing equivariant surgery along $\Gamma$-invariant neighborhoods of $\{p',\alpha(p')\}_{p'\in P'}$ results in a surface with an action of $\Gamma$ of the same topological type as $(M,T)$. 

To summarize our construction, if $(M',T')\prec (M,T)$ with collapsed set $P'$ and the partition $P' =\sqcup P'_i$, then there exist connected closed surfaces $M_i''$ and collections of points $P_i''\subset M_i''$ together with bijections $\alpha_i\colon P'_i\to P_i''$ such that
\begin{itemize}
\item The (possibly disconnected) surface $M'' = \sqcup M_i''$ and the collection $P'' = \sqcup P_i''\subset M''$ admit an action $T''$ of $\Gamma$ such that such that $P''$ is preserved and  $\alpha = \sqcup \alpha_i$ is equivariant.
\item Performing equivariant surgery at pairs $\{p',\alpha(p')\}$ for all $p\in P'$ on $(M',T')\sqcup (M'',T'')$ leads to a surface with a group action topologically equivalent to $(M,T)$.
\end{itemize}

\begin{example}
In the absence of the group action, these notions reduce to the topological degenerations studied e.g. in~\cite{Pet1} in the context of classical eigenvalue optimization. If $M$ is an orientable surface of genus $\gamma$, then $M'\prec M$ iff $M'$ is orientable of genus $\gamma'<\gamma$. In particular, $M'\prec M$ is an elementary degeneration iff $\gamma' = \gamma-1$. For such elementary degenerations there are two possibilities for the collapsed set $P'$. If $\fC$ is a single separating curve such that one of the components of $M\setminus \fU$ has genus $1$, then $P'$ is a single point. To recover $M'$ from $M$ one then takes $M''$ to be a torus with $P''$ consisting of a single point and connects $p'\in P'$ to $p''\in P''$ with a cylindrical tube, which is the same as forming a connected sum of $M'$ with $M''$.
If $\fC$ is a single non-separating curve, then $P'$ is any two-point set in $M'$. To recover $M'$ from $M$ one takes $M''$ to be a sphere, $P''\subset M''$ to be any two points, and then connects $p'\in P'$ to $p''\in P''$ and $q'\in P'$ to $q''\in P''$ with two cylindrical tubes. Both points in $P'$ belong to the same equivalence class.
\end{example}

%Without loss of generality, we may assume that for any $\fU_i$ one has $\del \fU_i\cap \del M_0\ne\varnothing$, otherwise, we may remove such $c_i$ form $\fC$ without changing $M'_{\fC,0}$. This way 

\subsection{Examples of topological degeneration} 
\label{sec:actions_ex}
In this section we classify topological degeneration for the types of group actions considered in Section~\ref{Ss:fund2}. Our main tool is the following proposition describing the intersection of $\fC$ with the fundamental domain.

\begin{proposition}
\label{prop:geo_FundDomain}
Let $M = M(\Gamma, \Omega)$ be as in Lemma \ref{LMomega} and $\fC$ be an embedded, $\Gamma$-invariant collection of disjoint simple closed curves consisting of a single $\Gamma$-orbit.  Then the intersection $\fC \cap \Omega$ is one of the following:
\begin{enumerate}
\item a closed curve in $\mathrm{int}(\Omega)$; 
\item a connected component of $\Omega^{\g_i}$ for some generator $\g_i$ of $\Gamma$;
\item a segment in $\mathrm{int}(\Omega)$ connecting two boundary points of $\Omega$, none of which are end-points of segments in $\Omega^{\g_j}$. 
\end{enumerate}
\end{proposition}
%Peter removed the contractible assumption here... it doesn't appear to  get used in this proof, only later. 
\begin{proof}

%I CHANGED THE PROOF SLIGHTLY, THE OLD ONE IS IN THE COMMENTS.

Without loss of generality, we may assume $\fC = \Gamma(c)$ for $c=c_0$ a simple closed curve meeting $\Omega$.  Let $\fU_0$ be a $T$-invariant tubular neighborhood of $c$,  $\Gamma'$ be the subgroup of $\Gamma$ generated by those $\g_i$ that preserve $\fU_0$. Each $\g_i\in \Gamma'$ is a separating reflection, hence, their restriction to $\fU_0$ is a separating reflection as well. Let 
 $\fU_0' = \cup_{\gamma_i \in \Gamma'} \fU_0^{\gamma_i}$ the union of the fixed-point sets of the $\gamma_i$'s on $\fU_0$. By the classification of reflections on a cylinder, we observe that $\Gamma'$ acts transitively on connected components of $\fU_0\setminus\fU_0'$. Since $\partial \Omega \subset \cup_{\gamma_i \in \Gamma} M^{\g_i}$, one has that  $\mathrm{int}(\Omega) \cap \fU_0$ is a union of  connected components of $\fU_0\setminus\fU_0'$. The transitivity of the action of $\Gamma'$ on connected components together with the fact that by construction $\mathrm{int}(\Omega)\cap \mathrm{int}(\g(\Omega)) = \varnothing$ for $\g\ne e$, implies that $\mathrm{int}(\Omega) \cap \fU_0$ is, in fact, a single connected component of $\fU_0\setminus\fU_0'$. Furthermore, $\fU_0 \subset \cup_{\gamma' \in \Gamma'}\gamma'(\Omega)$ yields that if $\gamma \in \Gamma$ and $c$ meets $\gamma(\Omega)$, then $\gamma \in \Gamma'$. Therefore, if $\gamma \in \Gamma \setminus \Gamma'$, then $\gamma(c) \cap \Omega = \gamma( c \cap \gamma^{-1} (\Omega) ) = \emptyset$ since $\gamma \notin \Gamma'$. 
 
%Without loss of generality, we may assume $\fC = \Gamma(c)$ for $c$ a simple closed curve meeting $\Omega$.  Let $C$ be a collar neighborhood of $c$,  $\Gamma'$ be the subgroup of $\Gamma$ generated by those $\g_i$ that preserve $C$, and $C' = \cup_{\gamma_i \in \Gamma'} C^{\gamma_i}$ the union of the fixed-point sets of the $\gamma_i$'s on $C$.
%
%Note that $\partial \Omega \subset \cup_{\gamma_i \in \Gamma} M^{\g_i}$, so $\partial \Omega \cap C \subset C'$; since $\mathrm{int}(\Omega)$ is connected and disjoint from each $M^{\gamma_i}$, note also $\mathrm{int}(\Omega) \cap C$ is a connected component of  $C \setminus C'$.  Further, since each $\gamma_i \in \Gamma'$ is a reflection, $\gamma_i$ is a reflection on $C$,  $\Gamma'$ acts transitively on the connected components of $C\setminus C'$, and $C \subset \cup_{\gamma' \in \Gamma'}\gamma'(\Omega)$.  In particular, if $\gamma \in \Gamma$ and $c$ meets $\gamma(\Omega)$, then $\gamma \in \Gamma'$.  Therefore, if $\gamma \in \Gamma \setminus \Gamma'$, then $\gamma(c) \cap \Omega = \gamma( c \cap \gamma^{-1} (\Omega) ) = \emptyset$ since $\gamma \notin \Gamma'$. 

There are now three cases.

{\bf Case 1.} $\Gamma'$ is trivial.  Then $\partial \Omega \cap \fU_0 \subset \fU_0' = \varnothing$, so $\fU_0 \subset \mathrm{int}(\Omega)$. 

{\bf Case 2.} $c$ is fixed pointwise by one of the generators $\gamma_i$ of $\Gamma'$.  Then $c \cap \Omega \subset \partial \Omega$ and therefore must be a connected component of $\Omega^{\gamma_i}$. 

{\bf Case 3.} $\Gamma'$ is nontrivial, but no generator of $\Gamma'$ fixes $c$ pointwise.  Then $c \cap \partial \Omega$ is exactly two points, none of which is the end-point of a segment in one of the $\Omega^{\gamma_i}$. 
\end{proof}

\begin{proposition}
If $(M',T')$ is such that $(M',T')\prec M(a)$, then $(M',T')$ is of type $a'$ for some $a'<a$. In particular, $(M',T')$ is an elementary degeneration of $M(a)$ if and only if it is of type $M(a-1)$.
\end{proposition} 
\begin{proof}
It is sufficient to consider the case when 
$\fC$ is a single orbit of a noncontractible simple closed curve $c$ such that $c\cap \Omega \ne\varnothing$. We consider different cases according to Proposition~\ref{prop:geo_FundDomain}.

{\bf Case 1. }  $c\subset \mathrm{int}(\Omega)$. Since $c$ is not contractible, it encircles at least one boundary component of $\Omega(a)$. Thus, collapsing $c$ leads to two fundamental domains of types $a'$ and $a''$, where $a',a''>0$ and $a'+a'' = a$. 

{\bf Case 2.} The intersection is a connected component of $\Omega^{\rho}$, hence, a connected component of $\partial\Omega$. Thus, collapsing it leads to a fundamental domain of type $a-1$.

{\bf Case 3.} The intersection is a segment $I$ connecting two points $p,q\in \partial \Omega$.

{\bf Case 3.(a).} Points $p,q$ lie on the same boundary component. Since $c$ is not contractible, the segment $I$ together with a segment of the boundary connecting $p,q$ encloses at least one boundary component of $\Omega(a)$. Thus, collapsing $I$ leads to two fundamental domains of types $a'$ and $a''$, where $a',a''>1$ and $a'+a'' = a+1$. 

{\bf Case 3.(b).}  Points $p,q$ lie on different boundary components. Then $\Omega\setminus I$ is connected, hence, collapsing $I$ leads to a fundamental domain of type $a-1$. 
\end{proof}

Let us discuss the possible collapsing sets arising in the relation $M(a-1)\prec M(a)$. Analyzing the proof, one observes that there are at least two different possibilities for the collapsing set: either a pair of points interchanged by $\rho$ or a pair of points on the {\em same} connected component of $M(a-1)^\rho$. 

\begin{proposition}
\label{prop:bunder_degen}
If $(M',T')$ is such that $(M',T')\prec M(G, \bunder)$, then $(M',T')$ is of type $\ub'$ for some $\ub'$. The only elementary degenerations of $M(G,\ub)$ are $M(G,\ub-\underline v)$, where $\underline v$ could be equal to  $1 - \rho_i$, or $\rho_i - \rho_i \rho_j$, or $\rho_i \rho_j$.
\end{proposition}
\begin{proof}
It is sufficient to consider the case when $\fC$ is a single orbit of a noncontractible simple closed curve $c$. We consider different cases according to Proposition~\ref{prop:geo_FundDomain}.

{\bf Case 1.} $c\subset \mathrm{int}(\Omega)$.  Since $c$ is not contractible, one of the connected components of $\Omega \setminus c$ does not contain the outer boundary component $\partial \Omegaout$, but does contain at least one boundary component in $\Omega^\tau$. This connected component is disjoint from each $\Omega^{\rho_i}$ for $i \in \{1, \dots, n\}$, $n=2,3$, hence, $\rho_i$ does not preserve it.  Therefore, the corresponding closed manifold with the action of $\Gamma$ is not connected and can not be a degeneration of $M(\Gamma,\ub)$ according to Definition~\ref{def:degeneration}. 
 As a result, collapsing $\fC$ yields only one $\Gamma$-invariant component and it has fundamental domain of type $\ub -k$, for some $k\in \mathbb{N}$.

{\bf Case 2.(a)} The intersection is contained in $\partial\Omega$, but is disjoint from $\partial\Omegaout$. Then the intersection is one of the remaining boundary components, hence, collapsing it leads to a fundamental domain of type $\ub - 1$.

{\bf Case 2.(b)} The intersection is contained in $\partial \Omegaout$, hence is one of the segments.  Then, collapsing it leads to a fundamental domain of type $\bunder - \uv$, where depending on the segment $\uv$ could be equal to $\rho_i$ if $e_i>0$, or $\rho_i \rho_j$ if $v_{ij}>0$, or $\rho_i-\rho_i \rho_j$ if $e_i>0$ and $v_{ij}=0$.

{\bf Case 3.} The intersection is a segment $I$ connecting two points $p,q\in \partial \Omega$ such that at least one of $p,q$ does not lie on $\partial \Omegaout$.

{\bf Case 3.(a).} Points $p,q$ lie on the same boundary component, which is not $\partial \Omegaout$, so $p,q\in\Omega^\tau$. This case similar to Case 1. Since $c$ is not contractible, one of the connected components of $\Omega\setminus I$ does not contain $\partial \Omegaout$, therefore, the corresponding closed surface with the action of $\Gamma$ is not connected. As a result, collapsing $\fC$ yields only one $\Gamma$-invariant connected component and its fundamental domain is of type $\ub-k$ for some $k\in \mathbb{N}$.

{\bf Case 3.(b).}  Points $p,q$ lie on different boundary components, neither of which is $\partial \Omegaout$. Then $\Omega\setminus I$ is connected, hence, collapsing $I$ leads to a fundamental domain of type $\ub-1$.

{\bf Case 3.(c).}  Points $p,q$ lie on different boundary components, only one of which (say, $p$) is on the outer boundary component. Then $\Omega\setminus I$ is connected, hence, collapsing $I$ leads to a fundamental domain of type $\ub-1$ if $p\in\Omega^\tau$, or of type $\ub - 1+\rho_i$ if $p\in\Omega^{\rho_i}$.

{\bf Case 4.} Points $p,q$ lie on $\partial \Omegaout$. 
Then $\Omega \setminus I$ has two connected components $\Omega_1$ and $\Omega_2$. One of these components, which we may assume to be $\Omega_2$, meets $\Omega^{\rho_i}$ for $i \in \{1, \dots n\}$, for otherwise, collapsing $\fC$ does not result in any $\Gamma$-invariant connected components.  Finally, note that if $n=3$, then $\Omega_1$ can meet at most two of the $\Omega^{\rho_i}$.  There are therefore two cases. 

{\bf Case 4.(a).} $\Omega_1\cap \Omega^{\rho_j}\ne\varnothing$ for only one $j \in \{1, \dots, n\}$. Since $c$ is not contractible,  $\Omega_1$ must contain a connected component of $\Omega^\tau$ or $\Omega^{\rho_j}$. As a result, collapsing $\fC$ leads to a single $\Gamma$-invariant connected component with fundamental domain  of type $\ub- k - l \rho_j$ for some $k,l\in\mathbb{N}_0$ with at least one of $k,l$ positive.

{\bf Case 4.(b).} $\Omega_1\cap \Omega^{\rho_i}\ne\varnothing$ and $\Omega_1 \cap \Omega^{\rho_j} \ne \varnothing$ for  distinct generators $\rho_i, \rho_j$ of $G$.  There are different cases depending on the location of points $p,q$. %Since $\Omega_1\cap \Omega^{\rho_i}=\varnothing$, the part of the outer boundary component containing $\Omega^{\rho_i}$ does not enter our considerations, so that $e_i$,$v_{ij}$ components of $\ub$ do not change. %The cases below are essentially equivalent to  Case 4.(b) in the proof of Proposition~\ref{prop:z2z2z2_degen}.  
%Let $j_1,j_2$ be such that $\{i,j_1,j_2\} = \{1,2,3\}$.

\begin{itemize}
\item If $p,q\in\Omega^{\tau}$, then collapsing $I$ leads to a fundamental domain of type $\ub-k-l_i\rho_{i} - l_j\rho_{j}$, where $k\in \mathbb{N}_0$, $l_i,l_j\in\mathbb{N}$.
%$b_1'+b_2' = b$, $c_1'+c_2'=c$, $d_j'\geq 1$ and $d_1'+d_2' = d+2$. Furthermore, since $\Omega_1'$ and $\Omega_2'$ both intersect $\Omega^{\rho_1}$ and $\Omega^{\rho_2}$ one has $b_j',c_j'\geq 1$, hence, $b_j'<b$, $c_j'<c$.
\item If $p\in \Omega^{\tau}$, $q\in \Omega^{\rho_{j}}$, then collapsing $I$ leads to a fundamental domain of type $\ub'$ with $v_{ij}'=1$, $e_{i}'<e_{i}$, $e_{j}'\leq e_{j}$, $f'\leq f$ and the remaining coordinates of $\ub'$ agree with those of  $\ub$.
%then $b_1'+b_2' = b$, $b_j'\geq 1$, $d_1'+d_2' = d+2$, $d_j'\geq 1$, $c_1'+c_2' = c+1$ and $c_j'\geq 1$. Hence, $b_j'<b$, $c_j'\leq c$ and $d_j'\leq d+1$.
\item  If $p\in \Omega^{\tau}$, $q\in \Omega^{\rho_{i}}$, then it is the same as in the previous case with $i$ and $j$ interchanged.
\item If $p\in \Omega^{\rho_i}$, $q\in\Omega^{\rho_j}$, then collapsing $I$ leads to a fundamental domain of type $\ub-k-l_i \rho_{i} - l_j\rho_{j}- m \rho_{i}\rho_{j}$, where $k,l_i,l_j\in \mathbb{N}_0$, $m \in \{0,1\}$ and at least one of those numbers is positive.\qedhere
\end{itemize}
\end{proof}

We conclude this section by discussing some choices of collapsed sets for the elementary degenerations described in Proposition~\ref{prop:bunder_degen}. We do not list all possible choices, but rather focus on those that are used later in the text. Set $\ub' = \ub-\uv$.

If $\uv = 1-\rho_i$, then $e_i'>0$ and the collapsed set $P'=\sqcup P_k'$ can be chosen in the following way. Set one equivalence class to be $P_1' = \{p',q'\}\in \Omega^{\rho_i}\cap\Omega^\tau$, where $p',q'$ lie on the same connected component of $\Omega^\tau$. Such points exist because $e_i'>0$. The remaining equivalence classes are formed by $\Gamma$-orbits of $P_1'$.

If $\uv = \rho_i-\rho_i\rho_j$, then $v'_{ij}>0$. Set one equivalence class $P_1'$ to be $\langle \rho_i,\rho_j\rangle$-orbit of a point $p\in S\cap \Omega^{\rho_j}$, where $S$ is a connected component of $M^\tau$ with $\Stab_\Gamma(S) = \langle \rho_i,\rho_j\rangle$ (which exists because $v'_{ij}>0$). The remaining equivalence classes of $P'$ are formed by $\Gamma$-orbits of $P_1'$.

If $\uv = \rho_i\rho_j$, then there exists a point $p'\in\Omega^{\rho_i}\cap \Omega^{\rho_j}$. Set $P_1' = \{p',\tau (p')\}$ to be one equivalence class, the remaining ones are $\Gamma$-orbits of $P_1'$.

\subsection{Degenerations on surfaces with boundary}
\label{sec:top_degen_boundary}
Recall the construction of the double $(\wt N,\wt T)$ from Section~\ref{sec:actions_Sex}  of a surface with boundary $(N,T)$. It established a bijection between $\Gamma$-group actions on surfaces with boundary and $\Z_2\times\Gamma = \wt\Gamma$-group actions on closed surfaces with the property that the nontrivial element $\iota$ of $\mathbb{Z}_2$ acts by a separating reflection. Let $(M^*,T^*)\prec (\wt N,\wt T)$ be a degeneration, then considering the $\iota$-fundamental domain on $M^*$, Proposition~\ref{prop:geo_FundDomain} implies that $\iota$ still acts by a separating involution on $M^*$. Therefore, $(M^*,T^*)$ is a double of some surface $(N',T')$ and it is natural to write $(N',T')\prec(N,T)$. To summarize, we say that $(N',T')\prec(N,T)$ iff $(\wt{N'},\wt{T'})\prec (\wt N,\wt T)$, so that the results of the previous section also describe degenerations of surfaces $N_\tau(G,\ub)$ and $N_{\rho_1}(G,\ub)$ considered in Section~\ref{sec:actions_Sex}.

We finish the section by discussing the collapsed set of a degeneration $(N',T')\prec (N,T)$. Let $(\wt{N'},\wt{T'})\prec (\wt N,\wt T)$ be the corresponding degeneration of the doubles and denote by $\wt P' = \sqcup \wt P'_j$ a collapsed set and its partition. Recall that $\wt P'$ is $\wt T'$-invariant whereas the equivalence classes $\wt P_j'$ are permuted by elements of $\wt \Gamma$. Thus, viewing $(N',T')$ as the factor space $\wt N'/\wt T'(\iota)$, we can define $P'$ as $\wt P'/\wt T'(\iota)$ and the equivalence classes $P_i'$ to be the images of $\wt P_j'$ under the factorization. Note that there are two different types of $P_i'$ depending on whether their preimage in $\wt P'$ consists of one or two equivalence classes. We define $ P'^b$ to be the union of those $P_i'$ whose preimages are a single $\iota$-invariant equivalence class in $\wt P'$. 

Finally, we define $P'^\iota:= P'^b\setminus \bd N'$. This set (or, more precisely, it non-emptiness) plays an important role in the existence theory of Section~\ref{sec:global_Sexistence}.

We can also interpret the set $P'^\iota$ in the language of the inverse surgery construction of Section~\ref{sec:inverse_surgery}. Let $M_j''$ be connected closed surfaces that are attached to $\wt P'_j$ to form $(\wt N,\wt T)$ from $(\wt{N'},\wt{T'})$. If $\wt P_j'$ is $\iota$-invariant and projects onto $P_i'\subset P'^b$, then $M_j''$ is also $\iota$-invariant and, hence, the surface $N_i'' = M_j''/T''(\iota)$ that is glued to $P_i'$ has non-empty boundary.  If $\wt P_j'$ is not $\iota$-invariant and projects onto $P_i'\subset P'^b$, then $N_i''$ can be identified with $M_j''$ and is a closed surface. Altogether, the inverse surgery construction for $(N',T')\prec (N,T)$ can be described using the same notation as the inverse surgery construction for closed surfaces: there exist connected surfaces $N_i''$, sets $P_i''\subset N_i''$, maps $\alpha_i\colon P_i'\to P_i''$, and an action $T''$ of $\Gamma$ on $N'' = \sqcup N_i''$, which satisfy the same conditions as in Section~\ref{sec:inverse_surgery}. The only differences are in the interactions of this data with the boundary. First, $P_i'\subset P'^b$ iff $N_i''$ has non-empty boundary. And second, if $p'\in \bd N'$, then $p'\in P'^b$ and $\alpha(p')\in \bd N''$, so that $p'$ is glued to $\alpha(p')$ by a half-cylinder (or a ribbon) as opposed to the usual cylinder.

\section{Conformal degenerations}
\label{sec:conf}
As before, let $M$ be a closed surface and $T\colon \Gamma\times M\to M$ be a fixed action of the group $\Gamma$ on $M$. Additionally, let $\mC$ be a conformal class of metrics on $M$ such that $T$ acts by conformal transformations. Then there exists a metric $h\in \mC$ of constant curvature $\pm 1$ or $0$ (and unit area in the latter case) such that $T$ acts by $h$-isometries. Since for $M=\mathbb{S}^2,\mathbb{RP}^2$ the solution to the equivariant eigenvalue optimization problem is completely described by Corollaries~\ref{cor:max_sphere} and~\ref{cor:max_RP2}, for the remainder of this section we assume $M\ne \mathbb{S}^2$ or $\mathbb{RP}^2$. Then, by the uniqueness part of uniformization theorem, the metric $h$ is also unique, so we can identify the space of $T$-invariant conformal classes with its constant curvature representatives. In the following $h$ always refers to this canonical representative of the conformal class. 

\subsection{Limits of conformal classes} 
\label{sec:cc_limits}
We define the moduli space $\mM^T(M)$ as the set of equivalence classes of $T$-invariant metrics of constant curvature (of unit area if the curvature is $0$) on $M$, where $h_1\sim h_2$ if there exists a diffeomorphism $F$ of $M$ commuting with $T$ such that $F^*h_1=h_2$. Equivalently, $\mM^T(M)$ is the set of equivalence classes of $T$-invariant conformal classes on $M$. We denote by $\bar h\in\mM^T(M)$ the equivalence class of the metric $h$. We say that a sequence 
$\{\bar h_n\}_{n=1}^\infty\subset \mM^T(M)$ converges to $\bar h$ if there exists a sequence of $T$-invariant diffeomorphisms $F_n$ such that $F_n^*h_n$ converges to $h$ in the $C^\infty$-topology.

In the absence of group actions there is a well-known compactness theory for the moduli space $\mM(M)$, see e.g.~\cite{Hummel}. Below we review the adaptation of this theory to the equivariant setting presented in~\cite{BSS}.

\begin{remark}
\label{rmk:app_BSS}
In paper~\cite{BSS} the authors consider orientable hyperbolic surfaces and provide details of the proof only for  orientation preserving actions $T$. Their arguments, however, can be adapted with minor modifications to the case of the torus, non-orientable surfaces, and actions $T$ that are allowed to reverse orientation (the latter is also remarked by the authors). We outline the proof of~\cite{BSS} as well as necessary adaptations in Appendix~\ref{app:moduli_space}.
\end{remark}

\begin{theorem}[Equivariant Mahler compactness theorem~\cite{BSS}]
\label{thm:Mahler}
 For any $\varepsilon>0$ the set
$$
\mM_{\varepsilon}^T(M):=\{\bar h\in \mM_\gamma^T,\,\,\mathrm{inj}(h)\geq\varepsilon \}
$$
is sequentially compact.
\end{theorem}

To describe the compactness properties of the full space $\mM^T(M)$ we restrict ourselves to the surfaces of negative Euler characteristic, and defer the detailed statements for other surfaces to the appendix. Recall the following facts from hyperbolic geometry.

The first fact is the collar lemma, see e.g.~\cite[Chapter IV, Proposition 4.2]{Hummel} or ~\cite[Lemma 4.2]{Mzhu} for the exact version we use below, stating that any disjoint collection $\{c_i\}$ of simple closed geodesics have disjoint tubular neighbourhoods $\fU_i\supset c_i$, usually referred to as {\em collars}, such that each $\fU_i$ is conformally equivalent to a flat cylinder $[-\mu_i,\mu_i]\times \Sph^1$, where $\mu_i$ only depends on $\ell_h(c_i)$ and $\mu_i\to\infty$ as $\ell_h(c_i)\to 0$. As a corollary, on a closed hyperbolic surface of genus $\gamma$ any collection of pairwise disjoint simple closed geodesics has at most $3\gamma-3$ elements. Furthermore, as a corollary of a thick-thin decomposition~\cite[Theorem 4.1.6]{Buser} one has the following lemma.

\begin{lemma}
\label{lemma:short_geo}
There exists a constant $\varepsilon_0>0$ such that for any hyperbolic metric $h$ on $M$ any two simple closed geodesics $c_1,c_2$ of length $\ell_h(c_i)<\varepsilon_0$ are disjoint.
\end{lemma}

Consider now a sequence $\{\bar h_n\}_{n=1}^\infty\subset \mM_\gamma^T$ with $\mathrm{inj}(h_n)\to 0$. Up to a choice of a subsequence, we may assume that $c_{n,1},\ldots,c_{n,m}$, $m\leq 3\gamma-3$ are  simple closed geodesics on $(M,h_n)$ such that $\ell_{h_n}(c_{i,n})\to 0$ and the lengths of all other simple closed geodesics on $(M,h_n)$ are uniformly bounded away from zero. Furthermore, we may assume that $\ell_{h_n}(c_{n,i})<\varepsilon_0$, where $\varepsilon_0$ is the constant from Lemma~\ref{lemma:short_geo}.
% Denote by $C_{n,i}$ the collars of these 
Since $\Gamma$ acts by isometries, for any $\g\in\Gamma$ one has $\ell_{h_n}(\g(c_{n,i})) = \ell_{h_n}(c_{n,i})\leq \varepsilon_0$ and, therefore, for each $n$ the group $\Gamma$ permutes curves $c_{n,i}$ and their collars. In particular, 
one has that $\fC_n = \cup_i c_{n,i}$ is a $\Gamma$-invariant collection of disjoint non-contractible simple closed curves, as in Section~\ref{sec:top_degen}. As we have seen before, 
$\Gamma$ acts on $M_n:=M\setminus \fC_n$ by $h_n$-isometries. Note that $(M_n,h_n)$ is a (possibly disconnected) hyperbolic surface with geodesic boundary. Up to a choice of a subsequence, all $M_n$ have the same topological type and the actions of $\Gamma$ are topologically the same, i.e. there exist $\Gamma$-equivariant diffeomorphisms $F_{n,m}\colon M_n\to M_m$. The latter can be seen by constructing $\Gamma$-invariant triangulations on $M_n$ with uniformly bounded number of vertices (as is done in~\cite{BSS} and reviewed in Appendix~\ref{app:moduli_space}) and extracting a subsequence with the same combinatorial structure and the same action of $\Gamma$ on the discrete set of vertices of the triangulation (see Appendix~\ref{app:moduli_space} for details). 

Since all $(M_n, T|_{M_n})$ are topologically equivalent, the topological type of $(\widehat M_\infty,\widehat T_\infty):=(M'_{\fC_n}, T'_{\fC_n})$ does not depend on $n$.
 Finally, let $M_\infty = \widehat M_\infty\setminus P_{\fC_n}$.

\begin{theorem}[Equivariant Baily-Mumford compactification~\cite{BSS}]
\label{them_EquivDM}
Up to a choice of a subsequence, there exist $\Gamma$-equivariant diffeomorphisms $\Psi_n\colon M_\infty\to \mathrm{int} (M_n)$ such that the sequence $\Psi_n^*h_n$ of hyperbolic metric converges in $C^\infty_{loc}$ to the complete $\Gamma$-invariant hyperbolic metric $h_\infty$ on $M_\infty$. 

Furthermore, there exists a $\Gamma$-invariant metric of locally constant curvature $\widehat h_\infty$ in $\widehat M_\infty$ such that its restriction to $M_\infty$ is conformal to $h_\infty$, and any $\widehat T_\infty$-invariant conformal class on $\widehat M_\infty$ can be realised as $[\widehat h_\infty]$ for some degenerating sequence of hyperbolic metrics. 
\end{theorem}

\begin{remark}
A version of this theorem for surfaces of vanishing Euler characteristic is stated in Appendix~\ref{app:moduli_space}. The only difference is that the word ``hyperbolic'' is replaced by ``constant curvature''. 
%Theorem~\ref{them_EquivDM} is not explicitly stated in~\cite{BSS} and they only provide details for $\Gamma$ acting by orientation-preserving isometries. We might want to provide elaborate on these things (in appendix?).  
\end{remark}

%The last statement of the theorem means that there is well-defined $\Gamma$-invariant conformal class $[\widehat h_\infty]$ on $\widehat M_\infty$ and we say that $(\widehat M_\infty,[\widehat h_\infty])$ is the limit of the sequence $(M,[h_n])$. 
%
%In the language of Section~\ref{sec:actions_ex} we see that for $\fC_n = \bigcup_{i=1}^m c_{n,i}$ the surface $M'_{\fC_n}$ is homeomorphic to $\hat M_\infty$ and the action $T'_{\fC_n}$ is topologically equivalent to the action of $\Gamma$ defined above. 

We conclude this section with a Riemannian version of the inverse surgery of Section~\ref{sec:inverse_surgery}, which can also be viewed as an explicit construction of the sequence of conformal classes $\mC_n$ on $M$ converging to a given class $\widehat\mC_\infty = [\widehat h_\infty]$ on $\widehat M_\infty$. This particular version is used later in Sections~\ref{sec:global_existence} and~\ref{sec:global_Sexistence}.

In the notation of Section~\ref{sec:inverse_surgery} let $(M',T')\prec (M,T)$, let $\mC'$ be a $T'$-invariant conformal class and $g'\in\mC'$ is a $T'$-invariant metric. Let also $g''$ be any $T''$-invariant metric on $M''$. We construct the metric $g_{\eps, L}$ on $M$ in the following way. For each $p\in P'\sqcup P''$ we can find an open neighbourhood $p\in U_p$ and a flat metric $g_p$ defined on $U_p$ such that 
\begin{enumerate}
\item $g_p$ is conformal to $g'$ if $p\in P'$, and conformal to $g''$ if $p\in P''$;
\item for all $\g\in\Gamma$ one has $U_{\gamma\cdot p} = \gamma\cdot U_p$ and $g_{\gamma\cdot p} = (\gamma\cdot)_*g_p$;
\item for some constants $c,C>0$  on $U_p$ one has
\[
cg'\leq g_p\leq Cg',\text{ if } p\in P';\qquad cg''\leq g_p\leq Cg'',\text{ if } p\in P''.
\] 
\end{enumerate}
Let $D_\eps(p)\subset U_p$ denote the disk of radius $\eps$ around $p$ in the metric $g_p$, which is defined for sufficiently small $\eps>0$. For all $p\in P'\sqcup P''$ we cut out $D_\eps(p)$ and glue $S_\eps(p):=\partial D_\eps(p)$ to $S_\eps(\alpha(p))$  using a flat cylinder $\mathbb{S}^1_\eps\times[0,L\eps]$, which we denote by $C_{\eps, L}(p)$. Furthermore, the isometric action of $\Gamma$ on $S_{\epsilon}(p)\sqcup S_{\epsilon}(\alpha(p))$ extends uniquely to an action by isometries on $C_{\epsilon,L}(p)$, so that there is an action $T$ of $\Gamma$ on the resulting surface $M$ that restricts to $T'$ on $M\cap M'$ (and $T''$ on $M\cap M''$, respectively), for which the resulting metric $g_{\eps,L}$ is $T$-invariant. The metric $g_{\eps,L}$ is bounded, but not necessarily continuous. Nevertheless, using cut-off functions it is easy to construct a function $\omega_{\eps,L}$ such that $e^{2\omega_{\eps,L}}g_{\eps,L}$ is smooth and, hence, $\mC_{\eps,L}:=[g_{\eps,L}] = [e^{2\omega_{\eps,L}}g_{\eps,L}]$ is a well-defined conformal class of metrics on $M$. Along a sequence $\eps_n\to 0$, $L_n\to \infty$ we recover the original conformal classes on $M'$, $M''$.
%
%====
%
%DEGENERATION ON THE TORUS (and non-orientable things) MIGHT BE BETTER HERE

\subsection{Convergence properties of $\Lambda_1^T(M,\mC)$}
\label{sec:LT1_convergence}
 In this section we study the behavior of conformal maxima $\Lambda_1^T(M,\mC)$. Namely, we are concerned with the convergence properties of $\Lambda_1^T(M,\mC_n)$ for a sequence of $T$-invariant conformal classes $\mC_n$.

Let us first consider the case of a precompact sequence $\mC_n$, i.e. such that canonical representatives $h_n\in \mC_n$ have a uniform lower bound on their injectivity radius. Then by Theorem~\ref{thm:Mahler}, there exists a subsequence, which we continue to denote $\mC_n$, such that $\mC_n\to \mC$, i.e. $h_n\to h$ in $C^\infty$, where $h\in\mC$ is the canonical representative. Since for any $\omega\in C^\infty(M)$ this also implies $e^{2\omega}h_n\to e^{2\omega}h$, we can conclude that $\Lambda_1^T(M,\mC_n)\to \Lambda_1^T(M,\mC)$ as $n\to \infty$, see e.g.~\cite[Proposition 4.2]{KM}.

Assume now that the sequence $\mC_n$ is such that the corresponding sequence of canonical representatives $h_n\in\mC_n$ degenerates. Then by Theorem~\ref{them_EquivDM}, up to a subsequence, $\mC_n$ converges to a conformal class $\wh\mC_\infty$ on $(\wh M_\infty, \wh T_\infty)$. The next theorem describes the behavior of $\Lambda_1^T(M,\mC_n)$ as $n\to\infty$. It is an equivariant version of the results~\cite{Pet1, KM} for orientable and non-orientable surfaces respectively.

\begin{theorem}
\label{thm:DM_limit}
Let $\mC_n$ be a degenerate sequence of $T$-invariant conformal classes on $M$ such  that $(M,T,\mC_n)$ converges to $(\widehat M_\infty, \wh T_\infty, \wh \mC_\infty)$ in the sense of Theorem~\ref{them_EquivDM}. Then one has 
\begin{equation}
\label{eq:DM_limit1}
\limsup \Lambda_1^T (M,\mC_n)\leq \max\left\{\Lambda_1^{\wh T_\infty}(\widehat M_\infty,\wh\mC_\infty), 8\pi\epsilon_{or},12\pi\epsilon_{nor}  \right\},
\end{equation}
where $\epsilon_{or} = 1$ if, along a subsequence, one of the collapsing two-sided geodesics is $T$-invariant, $\epsilon_{or} = 0$ otherwise, and, similarly, $\epsilon_{nor} = 1$ if if for an infinite sequence of $n$'s at least one of the collapsing one-sided geodesics is $T$-invariant, $\epsilon_{nor} = 0$ otherwise. Note that if $M$ is orientable, then $\epsilon_{nor} = 0$.
\end{theorem}
%\begin{remark}
%If $\Gamma$ has a fixed point on $\widehat M_\infty$, then $\Lambda_1^\Gamma(\widehat M_\infty,[\widehat h_\infty])\geq 8\pi$ by Proposition~\ref{prop:Lambda_fixedpt}, so in that situation one can remove $8\pi$ from the r.h.s. of~\eqref{eq:DM_limit1}.
%\end{remark}
%\begin{remark}
%Recall that $\widehat M_\infty$ could be disconnected and $\Gamma$ could act non-trivially on the set of its connected components. For each orbit the action could be reduced to the action on a connected component by a subgroup of $\Gamma$. Then the r.h.s. equals the maximum of the values of equivariant $\Lambda_1$ over all possible orbits.  
%\end{remark}
\begin{proof}
First we observe that if the action $T$ on $M$ has fixed points, then either there is an invariant collapsing geodesic (if the fixed point is on a collapsing geodesic) or the action of $\wh T_\infty$ on  $\wh M_\infty$ has fixed points (if the fixed point is not on a collapsing geodesic).
By the existence of equivariant maximizers of Theorem~\ref{lap.conf.ex}, either $\Lambda_1^T (M,\mC_n) = 8\pi$ for large enough $n$ (and then~\eqref{eq:DM_limit1} vacuously holds) or, up to a subsequence, there exist $T$-invariant metrics $g_n$ attaining $\Lambda_1^T(M,\mC_n)$. If there are no fixed points, then the maximal metric always exists. Thus, without loss of generality, we can assume that there exist $T$-invariant metrics $g_n\in\mC_n$ of unit area such that $\bar\lambda_1(M,g_n) = \Lambda_1^T(M,\mC_n)$ and the corresponding $T$-equivariant harmonic maps $\Phi_n\colon (M,\mC_n)\to\mathbb{S}^k$, where we can assume that $k$ is independent of $n$ by multiplicity bounds, see e.g.~\cite{Besson}.

The rest of the proof is analogous to that in~\cite[Section 4]{Pet1} (and~\cite[Section 5.2]{KM} for the non-orientable case), where the classical eigenvalue optimization problem is treated. Below we provide the main steps and outline the differences.

{\bf Step 1.} Let $\fU_{n,i}$ denote the collars of $c_{n,i}$ in $(M,h_n)$. In~\cite[Claim 11]{Pet1} it is proven that $\bar\lambda_1(M,g_n)\to 0$ unless (up to a subsequence) the area of $g_n$ is concentrated either in one of the collars or in one of the components of $M\setminus\fU_n$, where $\fU_n:=\cup_i\fU_{n,i}$. This statement can be directly applied to the sequence $\{g_n\}$. Note that since $g_n$ is $T$-invariant, the connected set $g_n$ is concentrating on has to be $T$-invariant as well.

{\bf Step 2.} Suppose that $g_n$ concentrates on one of the collars $\fU_{n,i}$. One then has $\bar\lambda_1(M,g_n)\to 0$ unless $\fU_{n,i}$ is preserved by $\Gamma$. Since $\Gamma$ acts by isometries on $(M,h_n)$ and $c_{n,i}$ is the unique simple closed geodesic on $\fU_{n,i}$, one has that $c_{n,i}$ is $T$-invariant. If $c_{n,i}$ are two-sided, then $\epsilon_{or} = 1$ and~\cite[Claim 12]{Pet1} yields $\limsup\bar\lambda_1(M,g_n)\leq 8\pi$. If $c_{n,i}$ are one-sided, then $\epsilon_{nor} = 1$ and the proof~\cite[Section 5.2]{KM} gives $\limsup\bar\lambda_1(M,g_n)\leq 12\pi$.

{\bf Step 3.} Suppose now that $g_n$ concentrates on a $T$-invariant component $M_n$ of $M\setminus\fU_n$. The proof of this step closely follows~\cite[Claim 13]{Pet1}. 

Let us denote by $X$ the corresponding component of $M_\infty$ containing $\Psi_n^{-1}(M_n)$ and denote $h_n' =\Psi_n^*h_n$ and $\Phi_n' = \Phi_n\circ\Psi_n|_X$, so that $\Phi_n'\colon(X,h_n')\to \mathbb{S}^k$ is a sequence of harmonic maps with $h_n'\to h_\infty$ in $C^\infty_{loc}$. By the results of~\cite{Mzhu} one has that up to a choice of a subsequence $\Phi_n'$ bubble converges to a harmonic map $\Phi\colon (X,h_\infty)\to \mathbb{S}^k$, i.e. there exists a collection of points $x_1,\ldots, x_N\in X$ such that $\Phi_n'\to \Phi$ in $C^\infty_{loc}(X\setminus\{x_1,\ldots,x_N\})$ and 
\begin{equation}
\label{eq:bubble_conv}
E_{h_n'}(\Phi_n')\to E_{h_\infty}(\Phi) + \sum_{i=1}^N E_i,
\end{equation}
where $E_i$ are are energies of bubbles at $x_i$. By the classical argument~\cite[Lemma 2.1]{Kok.var}, we have $\lambda_1(M,g_n)\to 0$ unless there is exactly one non-zero term in the r.h.s. of~\eqref{eq:bubble_conv}. Suppose that it is one of the $E_i$. Since the measures $|d\Phi_n'|^2_{h_n'}dv_{h_n'}$ are $T$-invariant, the points $x_i$ are permuted by $\Gamma$ and the bubble energies $E_i$ are preserved by $\Gamma$. Therefore, if exactly one of $E_i$ is non-zero, then $x_i$ is fixed by $\wh T_\infty$, hence, by Proposition~\ref{prop:Lambda_fixedpt} $\Lambda_1^{\wh T_\infty}(\widehat M_\infty,\mC_\infty)\geq 8\pi$. At the same time, since $g_n$ concentrates in a neighborhood of $x_i$, ~\cite[Lemma 3.1]{Kok.var} implies that the l.h.s. of~\eqref{eq:DM_limit1} does not exceed $8\pi$.

Suppose now that all $E_i$ vanish, i.e. $N=0$ and $\Phi_n'\to \Phi$ in $C^\infty_{loc}(X)$. Let $\widehat X$ be a connected component of $\widehat M_\infty$ containing $X$. By conformal invariance of energy, $\Phi$ has finite energy with respect to $\widehat h_\infty$ and, hence, by the removable singularity theorem, extends to a smooth harmonic map $\widehat\Phi\colon (\widehat X,\widehat h_\infty)\to\mathbb{S}^k$ of energy $E_{\widehat h_\infty}(\widehat\Phi) = \frac{1}{2}\limsup \Lambda_1^T(M,[h_n])$ and, furthermore $|d\Phi'_n|^2_{h_n'}\,dv_{h_n'}\rightharpoonup^*|d\widehat \Phi|^2_{\widehat h_\infty}dv_{\widehat h_\infty}$
as measures on $\widehat X$. Upper semi-continuity of eigenvalues~\cite[Proposition 1.1]{Kok.var} then implies that 
\[
\limsup_{n\to\infty}\bar\lambda_1(M,g_n)\leq \bar\lambda_1(\widehat X, |d\hat\Phi|^2_{\widehat h_\infty}\widehat h_\infty)\leq \Lambda_1^{\wh T_\infty}(\widehat X, \wh\mC_\infty)\leq \Lambda_1^{\wh T_\infty}(\widehat M_\infty, \wh\mC_\infty).\qedhere
\]
%It remains to observe that $\Lambda_1^\Gamma(\widehat M_\infty,[\widehat h_\infty])$ equals the maximal value of $\Lambda_1^\Gamma(\widehat Y, [\widehat h_\infty])$, where $\widehat Y$ ranges over all possible $\Gamma$-invariant connected components of $\widehat M_\infty$. 
\end{proof}

%THIS DEFINITION SHOULD NOT HAVE NON-CONTRACTIBLE CLAUSE, NON-CONTRACTIBILITY SHOULD BE A REMARK. 
%
%\begin{definition}
%Let $M$, $M'$ be two surfaces, and $T$, $T'$ be the actions of $\Gamma$ on the $M$, $M'$ respectively. We say that 
%$(M',T')\prec (M,T)$ if there exists a $\Gamma$-invariant collection of disjoint non-contractible simple closed curves in $M$ such that $M'$ is one of $\Gamma$-invariant connected components obtained after collapsing the curves in the sense of Section~\ref{sec:cc_limits}.
%\end{definition}

\begin{theorem}
\label{thm:existence}
For a surface with a finite group action $(M,T)$ set $\delta_{or} = 1$ if it has a $T$-invariant simple closed $2$-sided curve, $\delta_{or} = 0$ otherwise, and set $\delta_{nor} = 1$ if it has a $T$-invariant simple closed $1$-sided curve, $\delta_{nor} = 0$ otherwise. Then one has
\begin{equation}
\label{ineq:MM'}
\Lambda_1^T(M)\geq\max_{(M',T')\prec (M,T)} \left\{\Lambda_1^{T'}(M'), 8\pi\delta_{or}, 12 \pi \delta_{nor}\right\}.
\end{equation}

 Furthermore, if ~\eqref{ineq:MM'} is strict, then
 there exists a maximal metric achieving $\Lambda_1^T(M)$, smooth up to possibly finitely many conical singularities.
\end{theorem}
\begin{proof}
The idea of the proof is similar to that of Proposition~\ref{prop:Lambda_fixedpt}. For the classical eigenvalue optimization problem (with trivial group action) the inequality~\eqref{ineq:MM'} is well-known, with one of the first versions appearing in~\cite[Theorem 3]{CES}. The proof in the equivariant setting can be easily adapted; we use the ideas from~\cite{KM} to do so. Let $(M',T')\prec (M,T)$ and $g'$ be a $T'$-invariant metric on $M'$ such that $\bar\lambda_1(M',g')\geq \Lambda_1^{T'}(M')-\delta$. We use the Riemannian inverse surgery construction described at the end of Section~\ref{sec:cc_limits} and consider the metric $g_{\eps,L}$ on $M$. By~\cite[Lemma 4.6]{KM}, one has $\Lambda^T_1(M,[g_{\eps,L}])\geq \bar\lambda^N_1(M'\setminus D_\eps(P'), g')$, where $\lambda_1^N$ refers to the first non-trivial Neumann eigenvalue. At the same time, by~\cite[Lemma 4.1]{KM} one has $\bar\lambda^N_1(M'\setminus D_\eps(P'), g')\to \bar\lambda_1(M',g')$. Hence, for small enough $\eps>0$ one has $\Lambda_1^T(M)\geq \Lambda_1^{T'}(M')-2\delta$. A similar argument can be used to show  $\Lambda_1^T(M)\geq\max\{8\pi\delta_{or},12\pi\delta_{nor}\}$. Indeed, let $c$ be a $T$-invariant simple closed curve and let $\fU$ be its $T$-invariant tubular neighbourhood. Consider a $T$-invariant metric $g_\eps$ on $\fU$ that makes it isometric to $\mathbb{S}^2\setminus (D_{\eps}(N)\cup D_\eps(S))$, where $N,S$ are north and south poles (or to $\mathbb{RP}^2\setminus D_{\eps}(p)$ if $c$ is one sided) and extend it arbitrarily to a $T$-invariant metric $\wt g_\eps$ on $M$. Then the same arguments as before give
\[
\Lambda^T_1(M,[\wt g_\eps])\geq \bar\lambda^N_1(\fU,g_\eps)\to \Lambda_1(M_0),
\]
where $M_0 = \mathbb{S}^2$ if $c$ is two-sided and $M_0=\mathbb{RP}^2$ otherwise.

The second part of the theorem is the usual consequence of the convergence properties of $\Lambda_1^T(M,\mC)$. Indeed, let $\mC_n$ be a sequence of conformal classes such that $\Lambda_1^T(M,\mC_n)\to\Lambda_1^T(M,\mC)$. Assume first that this sequence is precompact, so that, up to a choice of a subsequence $\mC_n\to\mC$, by the discussion at the beginning of this section one has $\Lambda_1^T(M,\mC_n)\to\Lambda_1^T(M,\mC) = \Lambda_1^T(M)$. It remains to refer to Theorem \ref{lap.conf.ex} to conclude that there exists a conformally maximal metric achieving $\Lambda_1^T(M,\mC)$. Indeed, If $(M,T)$ has no fixed points, then the conformally maximal metric always exists. If $(M,T)$ has a fixed point, then a small circle around it is a simple closed $T$-invariant metric, hence, $\delta_{or} =1$ and strictness of~\eqref{ineq:MM'} implies $\Lambda_1^T(M,\mC)>8\pi$, which is sufficient for the existence of the conformally maximal metric.

In the case that $\mC_n$ degenerates, we have the upper bound~\eqref{eq:DM_limit1}. Note that $\delta_{or} = 0$ implies $\epsilon_{or}=0$, and similarly for $\delta_{nor}$ and $\epsilon_{nor}$. Furthermore, $\Lambda_1^{\wh T_\infty}(\wh M_\infty,\wh \mC_\infty)=\Lambda_1^{\wh T_\infty}(M',\wh \mC_\infty)$ for some $\wh T_\infty$-invariant connected component $M'$ of $\wh M_\infty$ and by construction $(M',\wh T_\infty)\prec (M,T)$. It is now easy to see that~\eqref{eq:DM_limit1} is in contradiction with the strict version of~\eqref{ineq:MM'}.\qedhere
\end{proof}

\subsection{Convergence properties of $\Sigma_1^T(N,\mC)$}

In this section we discuss the analogs of the results in Section~\ref{sec:LT1_convergence} for Steklov eigenvalues. Let $(N,T)$ be a manifold with boundary endowed with an action $T:\Gamma\times N\to N$ of the finite group $\Gamma$. Recall the construction of a double $(\wt N,\wt T)$ from Section~\ref{sec:actions_Sex}. Since the case $N=\mathbb{D}^2$ is covered by Corollary~\ref{cor:max_disk}, we assume henceforth that $N\neq \mathbb{D}^2$, so that for any $T$-invariant conformal class $\mC$ on $N\neq\mathbb{D}^2$, the uniformization theorem gives a unique representative with constant curvature and geodesic boundary (and fixed area if $\chi(N) = 0$). Doubling this metric gives a smooth $\wt T$-invariant metric of constant curvature and, conversely, the restriction of any $\wt T$-invariant constant curvature metric on $\wt N$ to $N$ gives a constant curvature metric on $N$ with geodesic boundary.
Thus, there is a bijection between $T$-invariant conformal classes on $N$ and $\wt T$-invariant conformal classes on $\wt N$. As a result, we can use the results of Section~\ref{sec:cc_limits} to introduce the topology on the moduli space of conformal classes and talk about degenerating sequences $\mC_n$ and their limits.

\begin{theorem}
\label{thm:SDM_limit}
Let $\mC_n$ be a degenerate sequence of $T$-invariant conformal classes on a surface with boundary $N$ such that $(N,T,\mC_n)$ converges to $(\wh N_\infty,\wh T_\infty, \wh\mC_\infty)$, meaning that their doubles converge in the sense of Theorem~\ref{them_EquivDM}. Then one has
\begin{equation}
\label{ineq:SDM_limit}
\limsup\Sigma_1^T(N,\mC_n)\leq \max\left\{\Sigma_1^{\wh T_\infty}(\wh{N_\infty},\wh \mC_\infty), 2\pi\wt\epsilon_{or}\right\}, 
\end{equation}
where $\wt\epsilon_{or} = 1$ if, along a subsequence, at least one of the collapsing geodesics in $\wt N$ is $\wt T$-invariant, or, equivalently there is a $T$-invariant collapsing segment in $N$ connecting two points of the boundary, and $\wt\eps_{or} = 0$ otherwise.
\end{theorem}

\begin{proof}
For the classical eigenvalue problem without group action, this result is proven in~\cite{PetridesS} (and~\cite{Medvedev} for the non-orientable case). The modifications required for the equivariant case are analogous to those needed for Theorem~\ref{thm:DM_limit}. Namely, the latter proof is based on Proposition~\ref{prop:Lambda_fixedpt} and the papers~\cite{Besson, Pet1, KM, Mzhu}, while the former uses the analogous results in Steklov settings: Proposition~\ref{prop:Sigma_fixedpt} and the papers~\cite{Kokarev, PetridesS, Medvedev, LP15}.

We leave the details to the reader. Let us, however, address some of the differences in the formulation of Theorem~\ref{thm:SDM_limit} compared to Theorem~\ref{thm:DM_limit}. First, closed $T$-invariant geodesics in $N$ do not affect the r.h.s. of~\eqref{ineq:SDM_limit}. This is because such geodesics do not intersect the boundary, and therefore the boundary measure (which is the measure appearing in the denominator of the Rayleigh quotient for the Steklov problem) can not concentrate in their collars. Second, it makes no difference whether the $\wt T$-invariant collapsing geodesics in $\wt N$ are $1$ or $2$-sided, since the collar neighborhood of the corresponding segment in $N$ is always diffeomorphic to a disk. 
\end{proof}

\begin{theorem}
\label{thm:Sexistence}
For a surface with boundary $N$ endowed with a group action $T$ set $\wt\delta_{or} = 1$ if its double $(\wt N,\wt T)$ contains a $\wt T$-invariant simple closed curve, and $\wt\delta_{or} = 0$ otherwise. Then one has 
\begin{equation}
\label{ineq:NN'}
\Sigma_1^T(N)\geq \max_{(N',T')\prec (N,T)}\left\{\Sigma_1^{T'}(N'), 2\pi\wt\delta_{or}\right\}.
\end{equation}

Furthermore, if the inequality~\eqref{ineq:NN'} is strict, then there exists a smooth maximal metric achieving $\Sigma_1^T(N)$.
\end{theorem}   
\begin{proof}
Again, the proof is analogous to the proof of Theorem~\ref{thm:existence}, with the statements from~\cite{KM} replaced by their counterparts in~\cite{Medvedev}.
\end{proof}

%===================================================
%===================================================
%===================================================
%===================================================
%===================================================
%===================================================

%===================================================
%===================================================
%===================================================
%===================================================
%===================================================
%===================================================

\section{Existence of maximizers on closed surfaces}

\label{sec:global_existence}

In the present section we provide a novel tool for proving inequalities of the type $\Lambda_1^T(M)>\Lambda^{T'}_1(M')$ whenever $(M',T')\prec (M,T)$ (recall Definition \ref{def:degeneration}). It is then used to prove the existence of metrics achieving $\Lambda_1^T(M)$ for many topological types of actions $(M,T)$. To formulate the main technical result, recall that to a given degeneration $(M',T')\prec (M,T)$ we associate a collapsed set $P'\subset M'$ together with its partition $P' = \sqcup P'_i$, and that for the same degeneration there could be different choices of $P',P_i'$;  see Section~\ref{sec:inverse_surgery}.

\begin{theorem}
\label{thm:general_existence}
Suppose that $(M',T')\prec (M,T)$. Let $\mC'$ be a $T'$-invariant conformal class on $M'$. Then 
$$
\Lambda_1^T(M)>\Lambda^{T'}_1(M',\mC')
$$ 
provided that there exists a collapsed set $P'\subset M'$ with partition $P'=\sqcup P'_i$ and a $T'$-invariant $\bar\lambda_1^{T'}$-maximal metric $g'\in\mC'$ such that for any $T'$-equivariant map $\Phi\colon (M',g')\to\mathbb{S}^n$ by first eigenfunctions, the image $\Phi(P_i)$ contains more than one point for some $i$.
\end{theorem}

We postpone the proof of this theorem to the end of the section, and first explore its applications 
to the existence of maximizers for $\Lambda_1^T(M)$.

\subsection{Partial existence without symmetries} Theorem~\ref{thm:general_existence} yields new results even for the trivial group, that is, for the classical eigenvalue optimisation problem. Let $M$ be an oriented surface of genus $\gamma\geq 0$ and denote $\Lambda_1(\gamma):=\Lambda_1(M)$.

\begin{corollary}
\label{cor:classical_existence}
Let $M$ be a surface of genus $\gamma$. Suppose there is a  $\bar\lambda_1$-maximal metric $g'$ on $M$ and two points $p,q\in M$ such that for any map $\Phi\colon (M,g')\to\mathbb{S}^n$ by first eigenfunctions one has $\Phi(p)\ne \Phi(q)$. Then
$$
\Lambda_1(\gamma+1)>\Lambda_1(\gamma),
$$
and, in particular, there exists a maximal metric achieving $\Lambda_1(\gamma+1)$.
\end{corollary} 

\begin{proof}
Choose $\mC'$ to be the conformal class of the $\bar\lambda_1$-maximal metric $g'$.
The surface of genus $\gamma+1$ can be obtained from a surface of genus $\gamma$ by gluing a cylinder at two arbitrary points $p,q\in M$. For such degeneration $P=\{p,q\}$ and $p\sim q$, hence if $\Phi(p)\ne\Phi(q)$ the image of the equivalence class contains more than one point. Thus, the conditions of Theorem~\ref{thm:general_existence} are satisfied, and the conclusion follows.
\end{proof}

\begin{corollary}
\label{cor:gamma2}
For any $\gamma\geq 0$ one has
\[
\Lambda_1(\gamma+2)>\Lambda_1(\gamma).
\]
In particular, for all $\gamma\geq 0$ either $\Lambda_1(\gamma)$ or $\Lambda_1(\gamma+1)$ is achieved by a smooth metric (up to a finite number of conical singularities).
\end{corollary}
\begin{proof}
We show both conclusions at the same time by induction on $\gamma$. If $\gamma=0$, then it is known \cite{Hersch, Nadirashvili} that $\Lambda_1(1)>\Lambda_1(0)$ and they are both achieved by a smooth metric. Moreover, $\Lambda_1(2)=16\pi >\Lambda_1(0)=8\pi$ by~\cite{NaySh}.

Suppose that the statement is known for $\gamma$. If $\Lambda_1(\gamma+2)>\Lambda_1(\gamma+1)$, then $\Lambda_1(\gamma+2)$ is achieved by a smooth metric (up to a finite number of conical singularities) and $\Lambda_1(\gamma+3)\geq \Lambda_1(\gamma+2)>\Lambda_1(\gamma+1)$, so in this case the corollary is proved.

If $\Lambda_1(\gamma+2)=\Lambda_1(\gamma+1)$, then $\Lambda_1(\gamma+1)>\Lambda_1(\gamma)$ and $\Lambda_1(\gamma+1)$ is achieved by a smooth metric (up to a finite number of conical singularities). It remains to show that $\Lambda_1(\gamma+3)>\Lambda_1(\gamma+1)$. 

Arguing by contradiction, assume that $\Lambda_1(\gamma+3)=\Lambda_1(\gamma+1)$.  Let $M'$ be an orientable surface of genus $\gamma+1$ and let $g'$ be a $\bar\lambda_1$-maximal metric on $M'$, which is smooth up to a finite number of conical singularities. Choose $p\in M'$ to be a smooth point of $g'$ and consider normal geodesic coordinates centered at $p$ associated to an orthonormal basis $\{\xi,\eta\}\in T_pM'$. 
%Set $M'\ni q_n = t_n\xi(= \exp_p(t_n\xi))$, $M'\ni r_n = t_n\eta( = \exp_p(t_n\eta))$, where $t_n>0$ is any sequence tending to $0$. 
Fix a sequence $\{t_n\}$ tending to $0$ and set $q_n = \exp_p(t_n \xi), r_n = \exp_p(t_n \eta)$. We now construct a surface $M_n$ of genus $\gamma+3$ by connecting $M'$ to a sphere along 3 cylinders based at $p,q_n,r_n$. 

For the degeneration $M'\prec M_n$ one has $P = \{p, q_n, r_n\}$ with all points lying in the same equivalence class. As we are assuming that  $\Lambda_1(\gamma+3)=\Lambda_1(\gamma+1)$, Theorem~\ref{thm:general_existence} implies that there exists a sphere-valued map $\Phi_n$ by $\lambda_1(M',g')$-eigenfunctions such that $\Phi_n(p) = \Phi_n(q_n) = \Phi_n(r_n)$. Since any such map satisfies $\|\Phi_n\|_{L^2} = \area(M',g')$ and the $\bar\lambda_1(M',g')$-eigenspace is finite-dimensional, one can choose a subsequence converging in $C^\infty$ to $\Phi$, which is also a sphere-valued map by $\lambda_1(M',g')$-eigenfunctions.
Since $\Phi_n(p) = \Phi_n(q_n)$, it follows from the definitions and the $C^1$ convergence $\Phi_n\to\Phi$ that $d_p\Phi(\xi)=\lim_{n\rightarrow \infty} d_p \Phi_n(\xi) = 0$, and similarly $d_p\Phi(\eta)=\lim_{n \rightarrow \infty } d_p \Phi_n (\eta) = 0$, so $d_p \Phi = 0$.  At the same time, $|d \Phi|_g^2\equiv \lambda_1(M',g')$, which is impossible at $p$ if it is a smooth point of $g'$. This contradiction completes the proof. 
% Since $\Phi_n(p) = \Phi_n(q_n)$ by intermediate value theorem $d_{s_n\xi}\Phi_n(\xi) = 0$ and, similarly, $d_{s'_n\eta}\Phi_n(\eta) = 0$ for $0\leq s_n,s'_n\leq t_n$. Passing to the limit as $n\to \infty$ one has $d_p\Phi(\xi) = d_p\Phi(\eta) = 0$, i.e. $d_p\Phi = 0$. 
\end{proof}

\subsection{Existence of group-invariant maximal metrics}
\label{ss:gamma}
Let us now consider surfaces with group actions introduced in Section~\ref{Ss:fund1}. %~\ref{sec:actions_ex}. 
We first record two auxiliary statements, that are used in the proof of existence. 

\begin{lemma}
\label{lemma:collapsed_sets}
Let $(M,T)$ be one of the surfaces with group actions considered in Section~\ref{Ss:fund1}.%\ref{sec:actions_ex}. 
Then for any elementary degeneration $(M',T')\prec (M,T)$ there exists a choice of a collapsed set $P$ and its partition $P=\sqcup P_i$, such that all $P_i$ have at least two elements. 
\end{lemma}
\begin{proof}
Follows from the discussion of collapsed sets at the end of Section~\ref{sec:actions_ex}.
\end{proof}
\begin{proposition}
\label{prop:uniq_map}
Let $M$ be an orientable surface of genus $>1$.
Suppose that $\Phi\colon (M,g)\to\mathbb{S}^3$ is a minimal embedding by $\lambda_1(M,g)$-eigenfunctions such that the multiplicity of $\lambda_1(M,g)$ equals $4$. If $\Psi$ is another harmonic map to $\mathbb{S}^n$ for any $n$ by   
$\lambda_1(M,g)$-eigenfunctions, then $n=3$ and $\Psi = A\Phi$ for some $A\in O(4)$. 
\end{proposition}
\begin{proof}
%Let $M, \Phi$, and $\Psi$ be as in the Proposition. 
%Let $\Psi\colon (M,g)\to\mathbb{S}^n$ be another harmonic map by $\lambda_1(M,g)$-eigenfunctions. 
Since $\Phi$ is an embedding and $M$ is not diffeomorphic to the sphere, the components of $\Phi$ are linearly independent and, hence, span the $\bar\lambda_1(M,g)$-eigenspace. As a result, the components of $\Psi$ are linear combinations of components of $\Phi$, and since $|\Psi|^2 \equiv 1$,  the image $\Phi(M)$ lies on the intersection of $\mathbb{S}^3$ with a level set $Q(x)=1$ of a non-negative definite quadratic form $Q$ on $\mathbb{R}^4$. If $Q(x) = |x|^2$, then $\Psi = A\Phi$ for some $A\in O(4)$. Otherwise, the intersection $\mathbb{S}^3\cap \{Q(x)=1\}$ is at most $2$-dimensional, and $\Phi(M)$ must be a connected component of this intersection. At the same time, we claim that such an intersection has genus at most $1$, a contradiction.

It remains to prove the claim. By subtracting $|x|^2=1$ from $Q(x)$ we may assume that $\Phi(M)\subset \Sph^3\cap\{Q'(x)=0\}$, where 
\[
Q'(x) = a_1x_1^2 + a_2x_2^2 + a_3x_3^2 + a_4x_4^2.
\]
Suppose that at least one of the $a_i$, say $a_4$, vanishes. Then $p =(0,0,0, 1)\in \Phi(M)$ and the set $L_\eps = \Phi(M)\cap \{x_4=1-\eps\}$ for $\eps>0$ small enough is homeomorphic to a link of $p$ in $\Phi(M)$. Since $\Phi$ is an embedding, one has $L_\eps \approx \Sph^1$. At the same time, substituting $x_4 = 1-\eps$ into equations for $\Sph^3$ and $Q'(x)$ gives that $L_\eps$ is an intersection of a $2$-sphere (of small radius) in $\R^3$ with $Q'(x_1,x_2,x_3) = 0$. The latter is either empty, a cone (and then $L_\eps$ is two circles rather than one), two intersecting planes (and then $L_\eps$ is two intersecting circles), or a double plane and then only one of the $a_i$'s is non-zero. In the latter case the intersection of $\Sph^3$ with $Q'(x)=0$ is an equator, so has genus $0$. Thus, we may assume that none of the $a_i$ vanish. Without loss of generality, at most two of them are negative. If $a_i>0$ for all $i$, then $\{Q'(x)=0\}$ is the origin and the $\Phi(M)$ is empty, a contradiction. If only one $a_i<0$, say $a_4<0$, then $\{Q'(x)=0\}\cap \{x_4=0\}\cap\Sph^3$ is empty, hence, $\{Q'(x)=0\}\cap\Sph^3$ splits into two connected components according to the sign of $x_4$ and $\Phi(M)$ is one of those connected components. This leads to a contradiction as $x_4$ is a non-constant Laplace eigenfunction and, therefore, has to change sign on $\Phi(M)$. 

It remains to consider the case where two of the $a_i$'s are positive and the other two are negative, say $a_1,a_2>0$ and $a_3,a_4<0$. We claim that all intersections with such values of $a_i$ are isotopic. Indeed, consider the map $F\colon \R^8\to \R^2$ given by $F(\vec a,\vec x) = \{Q'(\vec a,\vec x), |\vec x|^2-1\}$. The determinant of the matrix $A_{ij} = \left(\frac{\bd F}{\bd x_k\bd x_l}\right)_{k,l=i,j}$ equals $4(a_i-a_j)x_ix_j$. If $(\vec a,\vec x)$ is such that $F(\vec a,\vec x) = 0$, then there exists $i\in\{1,2\}$, $j\in\{3,4\}$ such that $x_i\ne 0$, $x_j\ne 0$. At the same time, $a_i$ and $a_j$ have different signs and, hence, $\det A_{ij}\ne 0$. By the implicit function theorem, $x_i$ and $x_j$ are locally smooth functions of the remaining arguments and a standard partition of unity argument implies that the set $\{\vec x : \,F(\vec a,\vec x) = 0 \}$ varies smoothly with $\vec a$ as long as none of $a_i$ change sign. Finally, observe that if $a_1=a_2 = 1$, $a_3=a_4 = -1$, then $\{Q'(x)=1\}\cap \Sph^3$ is a Clifford torus, hence has genus $1$.     
\end{proof}

Using these two observations, we can prove the existence result for $\bar{\lambda}_1$-maximal metrics for a large class of surfaces with group actions. The proof uses an induction argument based on a numerical invariant which plays a role analogous to that of the genus for the trivial group $\Gamma=1$.

\begin{definition}\label{cdef}
Let $(M,T)$ be a closed, oriented surface equipped with a group action $T: \Gamma\times M\to M$. We define the \emph{topological complexity} $c(M,T)$ to be the largest integer $k$ for which there exists a chain of degenerations $(M_k,T_k)\prec \cdots\prec (M_1,T_1)\prec (M_0,T_0)=(M,T)$.
\end{definition}
Note that we always have $c(M,T)\leq \genus(M)$, since $\genus(M')< \genus(M)$ whenever $(M',T')\prec (M,T)$. Moreover, it's clear that $c(M',T')<c(M,T)$ whenever $(M',T')\prec (M,T)$.

\begin{theorem}
\label{thm:group_existence}
Let $(M,T)$ be one of the surfaces with group actions considered in  Section~\ref{Ss:fund1} %Section~\ref{sec:actions_ex}. 
Then for any $(M',T')\prec (M,T)$ one has 
\begin{equation}
\label{ineq:toprove_1}
\Lambda_1^{T'}(M')<\Lambda_1^T(M).
\end{equation}

In particular, there exists a $T$-invariant $\bar\lambda_1^T$-maximal metric on $M$ and a corresponding equivariant minimal embedding $\Phi_{(M,T)}\colon M\hookrightarrow\mathbb{S}^3$. 
\end{theorem} 
\begin{proof}
We prove the theorem by induction on $c(M,T)$. Assume first that $c(M,T)=0$. One then has the following cases depending on the family of examples:
\begin{enumerate}
\item  $(M,T)$ is of type $M(a)$.  Then $a=1$, hence $M$ is diffeomorphic to $\Sph^2$ and  Corollary~\ref{cor:max_sphere} implies $\Lambda_1^{T}(M)$ is attained on the round metric.
\item  $(M,T)$ is of type $M(\mathbb{Z}_2, f+e_1\rho_1)$.  Then $f=0$, $e_1=1$, hence $M$ is once again diffeomorphic to $\Sph^2$ and Corollary~\ref{cor:max_sphere} implies that $\Lambda_1^{T}(M)$ is attained on the round metric.
\item  $(M,T)$ is of type $M(D_k, \bunder)$.  Then $\bunder = \rho_1\rho_2$, so that $M$ is a sphere, and we conclude as in the previous cases. 
%$M(D_k, f + e_1\rho_1 + e_2\rho_2 + v_{12}\rho_1\rho_2)$.  Then one has $f=e_1=e_2=0$, $v_{12}=1$.
\item $(M,T)$ is of the type $M(G,\ub)$, where $G$ is one of the remaining groups with $3$ generators.  Then $\ub$ must be equal to $\rho_i \rho_j$ for some $1\leq i < j \leq 3$. %has exactly one non-zero components, which is one of $v_{ij}$. 
We claim that for such surfaces $\delta_{or} = 0$, so that $\Lambda_1^{T}(M)$ is attained by Theorem~\ref{thm:existence}.
Indeed, by Proposition~\ref{prop:geo_FundDomain} any simple closed $2$-sided curve must meet the fundamental domain $\Omega(\Gamma,\ub)$ in a piece of the boundary or in a segment connecting two boundary points. But then it is easy to observe that none of these situations give rise to a single invariant curve.
\end{enumerate}

Suppose now that  the theorem holds for all $(M',T')$ with $c(M',T')\leq k$. Let $c(M,T) = k+1$ and 
$(M',T')\prec (M,T)$ be a topological degeneration, which we can always assume to be elementary.
Then $c(M',T')\leq k$, so there exists a $T'$-invariant $\bar\lambda_1^{T'}$-maximal metric $g$ on $M'$ and a corresponding minimal embedding $\Phi_{(M',T')}$ into $\mathbb{S}^3$. If the genus of $M'$ is at least two, then Proposition~\ref{prop:uniq_map} implies that $\Phi_{(M',T')}$ is the unique map by $\lambda_1(M,g)$-eigenfunctions to $\mathbb{S}^3$ (up to elements of $O(4)$) and since it is an embedding, it separates points. Thus, combining Lemma~\ref{lemma:collapsed_sets} with  Theorem~\ref{thm:general_existence} yields inequality~\eqref{ineq:toprove_1}.

Suppose that the genus of $M'$ is $0$. By Lemma~\ref{lemma:collapsed_sets}, we can choose a collapsed set $P$ with a partition $P=\sqcup P_i$ such that each $P_i$ has more than $1$ element. Since all maps $\Sph^2\to\mathbb{S}^n$ by $\lambda_1$-eigenfunctions separate points, Theorem~\ref{thm:general_existence} implies that $\Lambda_1^T(M)>\Lambda_1^{T'}(M') = 8\pi$.

Suppose now that $M'$ is a torus.  Recall from Example \ref{EMGtori} that
the only tori in our families are $M(2)$, $M(\mathbb{Z}_2, 1)$, $M(D_k,2\rho_1\rho_2)$, $M(D_2,\rho_1) = M(D_2,\rho_2)$ and $M(\mathbb{Z}_2\times D_k,\rho_2\rho_3)$ (for $k=2$, by symmetry, one also has $M(\mathbb{Z}_2\times D_2,\rho_1\rho_3) = M(\mathbb{Z}_2\times D_k,\rho_1\rho_2)$). There is a corresponding minimal embedding to $\mathbb{S}^3$ by first eigenfunctions. According to the classification of minimal embeddings of tori by first eigenfunctions~\cite[Theorem 4]{MR}, the minimal surface must be the Clifford torus, hence, $\Lambda_1^T(M')=4\pi^2$.
Using results of Section~\ref{sec:actions_ex} and our classification of tori in the families in \ref{EMGtori}, for any elementary topological degeneration $(M',T')\prec (M,T)$ one has the following cases for $(M,T)$:
\begin{enumerate}
\item $(M,T)$ is $M(3)$ or $M(\mathbb{Z}_2,1+\rho_1)$. In this case $M$ has genus $2$ and we observe that the metric on the Lawson's surface $\xi_{2,1}$ is of the corresponding type, so that by the resolution of the Willmore conjecture~\cite{MN} one has
\[
\Lambda_1^T(M)\geq 2\area(\xi_{2,1})>4\pi^2 = \Lambda^{T'}_1(M').
\]
\item $(M,T)$ is $M(D_2, 1)$. In this case $M$ has genus $3$ and we observe that the metric on the Lawson's surface $\xi_{3,1}$ is of the corresponding type, so that by~\cite{MN}
\[
\Lambda_1^T(M)\geq 2\area(\xi_{3,1})>4\pi^2 = \Lambda^{T'}_1(M').
\]
\item $(M,T)$ is $M(D_k,\rho_1 + \rho_1\rho_2) = M(D_k,\rho_2 + \rho_1\rho_2)$.
\begin{enumerate}
\item If $k=2$, then $\xi_{2,1}$ is of the corresponding type and we conclude as in the case 1);
\item If $k>2$, then $\xi_{k-1,1}$ is of type $M(D_k,\rho_1)$, and $M(D_k,\rho_1)\prec M(D_k,\rho_1 + \rho_1\rho_2)$. Thus, by~\eqref{ineq:MM'} and~\cite{MN} one has
\[
\Lambda_1^T(M)\geq 2\area(\xi_{k-1,1})>4\pi^2 = \Lambda^{T'}_1(M').
\]
\end{enumerate}
\item $(M,T)$ is $M(\mathbb{Z}_2\times D_k, \rho_2)=M(\mathbb{Z}_2\times D_k,\rho_3)$.
%In this case $M$ has genus $2n-1$ and...
\begin{enumerate}
\item If $k=2$, then $\xi_{3,1}$ is of the corresponding type and we conclude as in the case 2);
\item If $k> 2$, then $\xi_{k-1,1}$ is of type $M(\mathbb{Z}_2\times D_k, \rho_1\rho_2)$, and $M(\mathbb{Z}_2\times D_k, \rho_1\rho_2)\prec M(\mathbb{Z}_2\times D_k, \rho_2)$. Thus, by~\eqref{ineq:MM'} and~\cite{MN} one has
\[
\Lambda_1^T(M)\geq 2\area(\xi_{k-1,1})>4\pi^2 = \Lambda^{T'}_1(M').\qedhere
\]
\end{enumerate}
\end{enumerate}
% 
%The rest of the proof is by induction. The statement we are proving is that $(M',T')\prec (M,T)$ implies~\eqref{ineq:toprove_1} for all $M$ of genus $\gamma$. The base case $\gamma=1$ is established above. Let us now prove the induction step, so assume that $M$ has genus $\gamma+1$ and that $(M',T')\prec (M,T)$. Without loss of generality, we may assume that $(M',T')\prec (M,T)$ is an elementary degeneration. Furthermore, if $M'$ has genus $0$ or $1$, then the inequality~\eqref{ineq:toprove_1} has been established in the previous paragraph, hence, we assume that genus of $M'$ is at least $2$. Since genus of $M'$ is always strictly smaller than genus of $M$, the induction step and REF imply that there exists a $T'$-invariant $\bar\lambda_1^{T'}$-maximal metric $g$ on $M'$ and the corresponding minimal embedding $\Phi_{(M',T')}$ into $\mathbb{S}^3$. By Proposition~\ref{prop:uniq_map} $\Phi_{(M',T')}$ is the unique map by $\lambda_1(M,g)$ eigenfunctions to $\mathbb{S}^n$ (up to an element of orthogonal group) and since it is embedding, it separates points. Thus, combining Lemma~\ref{lemma:collapsed_sets} with  Theorem~\ref{thm:general_existence} yields inequality~\eqref{ineq:toprove_1}
\end{proof}

Observe that many surfaces covered in Section~\ref{sec:actions_ex} have the same topology. For example, $M(\mathbb{Z}_2,a+b\rho_1)$ has genus $\gamma = 2a+b-1$. As a result, there are at least $\lfloor \frac{\gamma}{2}\rfloor+1$ topologically different actions of this type on a surface of genus $\gamma$. The following proposition asserts that equivariant eigenvalue optimization in these classes often yields distinct maximal metrics, hence, distinct minimal surfaces.

\begin{proposition}
\label{prop:distinct_MS}
Let $\gamma+1 = 2a+b = 2a'+b'$, and $g_{a,b}$, $g_{a',b'}$ be metrics achieving $\Lambda_1(M(\mathbb{Z}_2,a+b\rho_1))$, $\Lambda_1(M(\mathbb{Z}_2,a'+b'\rho_1))$ respectively. If $(a',b')\ne (a,b)$ and $\gamma>2(a+a')+1$, then $g_{a,b}$ and $g_{a',b'}$ are not isometric.

In particular, there exists at least $\lfloor\frac{\gamma-1}{4}\rfloor+1$ embedded non-isometric minimal surfaces of genus $\gamma$ and area below $8\pi$ in $\mathbb{S}^3$.
\end{proposition}
\begin{remark}
One can use other types of group actions to improve the lower bound on the number of different minimal surfaces of given genus. However, such arguments do not result in a superlinear lower bound for $\gamma$, so we choose to not include them.
\end{remark}
\begin{proof}
Suppose that $g_{a,b}$ and $g_{a',b'}$ are isometric, and denote this metric by $g$. Then $g$ admits an isometric action $T$ of $\Z_2\times\Z_2 = \langle\tau,\rho_1\rangle$ of type $M(\mathbb{Z}_2,a+b\rho_1)$, where $(M,\tau)$ is a basic reflection surface, and, at the same time, $g$ admits an isometric action $T'$ of $\Z_2\times\Z_2 = \langle\tau',\rho'_1\rangle$ of type $M(\mathbb{Z}_2,a'+b'\rho_1)$, where $(M,\tau')$ is a basic reflection surface. Suppose that $\tau\ne\tau'$. Then by Lemma~\ref{lem:2basic} $(M,\tau)$ has type $M(\Z_2,(\gamma+1)\tau')$. Consider now $G = \langle \rho_1,\tau'\rangle$ and the action of $\langle\tau\rangle\times G$ and suppose that  this action is of type $M(G,f+e_1\rho_1 + e'\tau' + v_1'\rho_1\tau')$. Since the fundamental domain for $\langle\tau,\tau'\rangle$ is obtained by a repeated reflection of $\Omega(G,f+e_1\rho_1 + e'\tau' + v_1'\rho_1\tau')$, one observes that $f=0=e_1$.
 At the same time, it is given that $M$ is of type $M(\Z_2,a+b\rho_1)$ with respect to the action of $\langle\tau\rangle\times\langle\rho_1\rangle$ and similar arguments yield $b\leq 2$. However, $\gamma>2(a+a')+1$ implies $b>2a'+2\geq 2$, a contradiction.

Hence, $\tau = \tau'$ and $\tau$ commutes with both $\rho_1$ and $\rho_1'$. Since $(a',b')\ne(a,b)$ one has $\rho_1\ne\rho_1'$. Set $G = \langle\rho_1,\rho_2\rangle$, so that $\Gamma = \langle\tau\rangle\times G$ acts on $M$, and $M$ has type $M(G,\ub)$. By arguments analogous to the previous paragraph, one observes that the number of connected components of $M^\tau$ fixed by $\rho_1'$ does not exceed $2a+2$. At the same time, this number is $b'$ by definition, so we conclude $b'\leq 2a+2$. Since $\gamma>2(a+a')+1$ implies $b>2a'+2\geq 2$, a contradiction.
 
If $0\leq a,a'\leq \lfloor\frac{\gamma-1}{4}\rfloor$ are two distinct numbers, then $2(a+a')<\gamma-1$ and $4a, 4a'\leq \gamma-1$. The latter inequality implies that $b = \gamma-2a\geq 0$ and analogously for $b'$ and, hence, metrics achieving 
$\Lambda_1(M(\mathbb{Z}_2,a+b\rho_1))$ for distinct values of $a$, where $0\leq a\leq \lfloor\frac{\gamma-1}{4}\rfloor$ yield distinct minimal surfaces.  The embeddedness in $\Sph^3$ and area bound follows from Lemma~\ref{Lgraphsph}. 
\end{proof}
\subsection{Proof of Theorem~\ref{thm:general_existence}}
\label{sec:proof_general_existence}

%
%Suppose that $(M',T')\prec (M,T)$. Let $\mC$ be a $T'$-invariant conformal class on $M'$. Then 
%$$
%\Lambda_1^T(M)>\Lambda^{T'}_1(M',\mC)
%$$ 
%provided that there exists a collapsed set $P'\subset M'$, $P'=\sqcup P'_i$ and $T'$-invariant $\bar\lambda_1^{T'}$-conformally maximal metric $g'\in \mC$ such that for any corresponding $\Gamma$-equivariant $\bar\lambda_1^{T'}$-maximal harmonic map $\Phi\colon (M',\mC)\to\mathbb{S}^n$ the image $\Phi(P'_i)$ contains more than one point for some $i$.

Let $(M',T')\prec (M,T)$ with collapsed set $P'$ and the partition $P' =\sqcup P'_i$. We make use of the Riemannian inverse surgery construction discussed in Sections~\ref{sec:inverse_surgery} and~\ref{sec:cc_limits}. Recall that 
 there exist connected surfaces $M_i''$ and collections of points $P_i''\subset M_i''$ together with the bijections $\alpha_i\colon P'_i\to P_i''$ such that
\begin{itemize}
\item The (possibly disconnected) surface $M'' = \sqcup M_i''$ and the collection $P'' = \sqcup P_i''\subset M''$ admit action $T''$ of $\Gamma$ such that such that $P''$ is preserved and  $\alpha = \sqcup \alpha_i$ is equivariant.
\item Performing equivariant surgery at pairs $\{p,\alpha(p)\}$ for all $p\in P'$ on $(M',T')\sqcup (M'',T'')$ leads to a surface with a group action topologically equivalent to $(M,T)$. To simplify the notation we identify this surface with $(M,T)$.
\end{itemize}

Let $g'\in\mC$ be a unit area $T'$-invariant metric achieving $\Lambda^{T'}(M',\mC')$ and let $g''$ be any $T''$-invariant metric on $M''$. We construct the metric $g_{\eps, L}$ on $M$ as follows. For each $p\in P'\sqcup P''$ there are neighborhoods $U_p$ of $p$ and smooth conformal factors $\frac{1}{2}\leq \phi_p\leq 2$ depending equivariantly on $p$, such that $g_p=\phi_pg$ is flat on $U_p$.
Denote by $D_\eps(p)$ the disk of radius $\eps$ around $p$ in the metric $g_p$. The metric $g_{\eps,L}$ on $M$ is obtained by gluing a flat cylinder  $C_{\eps, L}(p) = \mathbb{S}^1_\eps\times[0,L\eps]$ to $\partial D_\eps(p)$ and $\partial D_\eps(\alpha(p))$ in an equivariant way, where $\mathbb{S}^1_\eps$ denotes a circle of radius $\eps$. Then $g_{\eps,L}$ is $T$-invariant, and we set $\mC_{\eps,L}=[g_{\eps,L}]$.

We ntroduce the following notation:
\[
D_\eps(P'):=\bigcup\limits_{p\in P'} D_\eps(p);\quad D_\eps(P''):=\bigcup\limits_{p\in P''} D_\eps(p);\quad C_{\eps,L}:=\bigcup\limits_{p\in P'} C_{\eps,L}(p).
\]

For now, we fix $\eps$ and $L$; later, $\eps>0$ will be chosen to be small, whereas $L$ will be large depending on $\eps$. In the following we sometimes suppress the dependence on $\eps, L$ to simplify the notation. Moreover, we use $C$ to denote any positive constant that is independent of parameters $\eps, L$. The constant may change from line to line.

\begin{lemma}
\label{lem:good_map}
Given $\eps, L$ there exists $n\geq 2$ and a $T$-equivariant map $F\in W^{1,2}_{g_{\eps,L}}(M, \mathbb{B}^{n+1})$ such that 
\begin{equation}
\label{ineq:gm1}
\int_M |dF|^2_{g_{\eps,L}}\,dv_{g_{\eps,L}}<\Lambda^T_1(M,\mC_{\eps,L}) + \eps^2,\quad
\int_{M} (1-|F|^2)^2\,dv_{g_{\eps,L}}<\eps^2.
\end{equation}
Furthermore, the harmonic extension $\hat F\in W^{1,2}_{g'}(M', \mathbb{B}^{n+1})$ of the restriction $F|_{M'\setminus D_\eps(P')}$ to $M'$ is a $T'$-equivariant map satisfying
\begin{equation}
\label{ineq:gm2}
\Lambda_1^{T'}(M',\mC')\leq \int_{M'}|d\hat F|^2_{g'}\,dv_{g'} + C\eps,\quad 
\int_{M'} (1-|\hat F|^2)^2\,dv_{g'}<C\eps^2,
\end{equation}
\begin{equation}
\label{ineq:gm3}
\int_{M'}\hat F\,dv_{g'} = 0.
\end{equation}
\end{lemma}
\begin{proof}

First, by Theorem \ref{lap.mm.char}, there exists $m\in \mathbb{N}$ such that 
$$\Lambda_1^T(M,\mC_{\eps,L})=2\mathcal{E}_m^T(M,\mC_{\eps,L})=2\sup_{\delta>0}\mathcal{E}_{m,\delta}^T(M,g_{\eps,L}),$$
so in particular, for any $\delta>0$, we can find a family $(F_y)_{y\in \mathbb{B}^{m|\Gamma|}}\in \mathcal{B}_m$ such that 
$$
\sup_{y\in \mathbb{B}^{m|\Gamma|}}\int_M |dF_y|_{g_{\eps,L}}^2\,dv_{g_{\eps,L}}\leq \Lambda_1(M,\mC_{\eps,L})+\delta
$$
and
\begin{equation}
\label{ineq:almostSn}
\int_{M} (1-|F_y|^2)^2\,dv_{g_{\eps,L}}<\delta.
\end{equation}
Set $\delta=\eps^2$, and fix such a family $(F_y)\in \mathcal{B}_m$; without loss of generality, we may also assume $|F_y|\leq 1$ pointwise for all $y\in \mathbb{B}^{m|\Gamma|}$.

Next, consider the bounded linear functional $\ell: W^{1,2}_{g_{\epsilon,L}}(M)\to \mathbb{R}$ defined as follows: for any $\phi\in W^{1,2}(M)$, let $\hat{\phi}\in W^{1,2}_{g'}(M')$ be given by restricting $\phi$ to $M'\setminus D_{\epsilon}(P')\subset M$ and extending it harmonically into $D_{\epsilon}(P')$, then set $\ell(\phi):=\int_{M'}\hat{\phi}dv_{g'}$. Applying Remark \ref{gen.hersch} to the family $(F_y)$ and the functional $\ell\in (W^{1,2}(M))^*$, it follows that there exists $y\in \mathbb{B}^{m|\Gamma|}$ such that the map $\hat{F}_y\in W^{1,2}_{\Gamma}(M',A_{\Gamma}^m)$ satisfies 
\begin{equation}
\label{ineq:hat_balance}
\int_{M'} \hat F_y\,dv_{g'} = 0.
\end{equation}
Moreover, $|F_y|\leq 1$ implies $|\hat F_y|\leq 1$, so that by~\eqref{ineq:almostSn}
$$
\int_{M'} (1-|\hat F_y|^2)^2\,dv_{g'}< C\eps^2.
$$
Combining this estimate with~\eqref{ineq:hat_balance}, the same computation used in the proof of Proposition \ref{mm.lower} yields
$$
\Lambda^{T'}_1(M',\mC') = \bar\lambda_1(M',g')\leq \int_{M'}|d\hat F|_{g'}^2\,dv_{g'} + C\eps.
$$ 
Thus, the map $F_y$ satisfies all the conclusions of the lemma with $n+1=m|\Gamma|$, and the proof is complete. 
\end{proof}

Note that inequalities in Lemma~\ref{lem:good_map} imply that 
\begin{equation}
\label{ineq:fhllb}
\begin{split}
\int_{M'}|\hat F|^2\,dv_{g'} &= 1-\int_{M'}(1-|\hat F|^2)\,dv_{g'}\\
&\geq 1 - \left(\int_{M'}(1-|\hat F|^2)^2\,dv_{g'}\right)^{1/2}\geq 1 - C\eps.
\end{split}
\end{equation}

We would like to obtain an upper bound for the energy of $\hat F$ in terms of the energy of $F$. To that end, we record the following lemma.

\begin{lemma}
\label{lem:ext_bound}
Let $C_L = \mathbb{S}^1\times[0,L]$ be the flat cylinder with $L \geq \frac{3}{2}\log(2)$, and let 
$$
H\colon W^{1,2}(C_L)\to W^{1,2}(\mathbb{B}^2)
$$
be the map sending $\phi$ to the harmonic extension $H(\phi) = \hat\phi$ of $\phi|_{\mathbb{S}^1\times\{0\}}$. Then
\[
\|dH(\phi)\|^2_{L^2(\mathbb{B}^2)}\leq (1+Ce^{-2L})\|d\phi\|^2_{L^2(C_L)}
\]
for some universal constant C.
\end{lemma}
\begin{proof}
It suffices to consider the case where $\phi$ is the energy-minimizing extension of $\beta:=\phi|_{\mathbb{S}^1\times\{0\}}$ to $C_L$. This extension is the unique solution of the boundary value problem
\begin{equation}
\label{eq:Hbeta}
\begin{cases}
\Delta \phi = 0 &\text{ on }C_L;\\
\phi  \phantom{\Delta}=\beta &\text{ on }\mathbb{S}^1\times \{0\};\\
\frac{\partial\phi}{\partial t}\phantom{ \,\,}=0 &\text{ on }\mathbb{S}^1\times \{L\},
\end{cases}
\end{equation}
where $t$ denotes the coordinate along $[0,L]$ and we use $\theta\in[0,2\pi]$ for the coordinate on $\mathbb{S}^1$.  

Let us first consider the case where $\beta(\theta)=\beta_k(\theta):= a\sin(k\theta)+b\cos(k\theta)$ is an eigenfunction of the Laplacian on $\mathbb{S}^1$ with eigenvalue $k^2$. The case $k=0$ is trivial, so we assume $k\geq 1$. In this case the solution $\phi_k$ to \eqref{eq:Hbeta} can be determined by separation of variables, and is given by
\[
\phi_k(\theta,t) = \psi_k(t)\beta_k(\theta) = \left(\frac{e^{kt}}{1+e^{2kL}} + \frac{e^{-kt}}{1+e^{-2kL}} \right)\beta_k(\theta).
\]  
We then see that
\begin{equation}
\label{eq:cyl_energy}
\begin{split}
\int_{C_L}|d\phi_k|^2d\theta dt&=-\int_{\mathbb{S}^1}\psi_k(0)\frac{\partial \psi_k}{\partial t}(0)\beta_k^2(\theta)d\theta\\
&=\left(\frac{k}{1+e^{-2kL}}-\frac{k}{1+e^{2kL}}\right)\int_{\mathbb{S}^1}\beta_k^2(\theta)d\theta.
\end{split}
\end{equation}
Using that $k\geq 1$, we observe
\[
\frac{1}{1+e^{-2kL}}-\frac{1}{1+e^{2kL}}\geq \frac{1}{1+e^{-2L}}-\frac{1}{1+e^{2L}} = \frac{1-e^{-2L}}{1+e^{-2L}}\geq 1-4e^{-2L},
\]
hence
\[
\int_{C_L}|d\phi_k|^2d\theta dt\geq k(1-4e^{-2L})\int_{\mathbb{S}^1}\beta_k^2(\theta)d\theta.
\]
On the other hand, the harmonic extension of $\beta_k$ into $\mathbb{B}^2$ is given in polar coordinates by $\hat\phi_k(r,\theta) = r^k\beta(\theta)$, for which
\[
|d\hat{\phi}_k|^2(r,\theta)=k^2r^{2k-2}\beta_k(\theta)^2+r^{2k-2}(\partial_\theta\beta_k)^2,
\]
and therefore
\begin{eqnarray*}
\int_{\mathbb{B}^2} |d\hat\phi_k|^2 r dr d\theta&=&\int_0^1 \left(k^2r^{2k-1}\int_{\mathbb{S}^1}\beta_k^2d\theta+r^{2k-1}\int_{\mathbb{S}^1}(\partial_\theta\beta_k)^2d\theta\right)dr\\
&=&\frac{k^2}{2k}\int_{\mathbb{S}^1}\beta_k^2d\theta+\frac{k^2}{2k}\int_{\mathbb{S}^1}\beta_k^2d\theta
=k\int_{\mathbb{S}^1}\beta_k^2d\theta.
\end{eqnarray*}
As a result, we see that
\begin{equation}
\label{ineq:for_bases}
(1-4e^{-2L})\int_{\mathbb{B}^2}|dH(\phi_k)|^2\leq \int_{C_L}|d\phi_k|^2,
\end{equation}
from which the desired estimate follows for $L\geq \frac{3}{2}\log(2)$.

By the Sobolev trace inequality, for arbitrary $\phi\in W^{1,2}(C_L)$ the function 
$\beta(\theta) = \phi(\theta, 0)$ belongs to the space $W^{\frac{1}{2},2}(\mathbb{S}^1)$. Any element in this space can be written as 
\[
\beta(\theta) = \sum_{k=0}^\infty c_k\beta_k(\theta),
\]
where $\sum kc^2_k<\infty$. Similarly to~\eqref{eq:cyl_energy} we have
\begin{equation*}
\begin{split}
\int_{C_L}\langle d\phi_k,d\phi_l\rangle d\theta dt&=-\int_{\mathbb{S}^1}\psi_k(0)\frac{\partial \psi_l}{\partial t}(0)\beta_k(\theta)\beta_l(\theta)d\theta = 0,
\end{split}
\end{equation*}
so that the minimal energy extension of $\beta$ to $C_L$ is given by $\sum_k c_k\phi_k$ and has the energy equal to 
\[
\sum_{k=0}^\infty c_k^2\int_{C_L}|d\phi_k|^2d\theta dt.
\]
In a similar fashion, since the $\beta_k$'s form a basis of eigenfunctions on $\mathbb{S}^1$,
\[
\int_{\mathbb{B}^2}|dH(\phi)|^2 = \sum_{k=0}^\infty c_k^2\int_{\mathbb{B}^2}|dH(\phi_k)|^2.
\]
Thus,~\eqref{ineq:for_bases} implies that 
\[
(1-4e^{-2L})\int_{\mathbb{B}^2}|dH(\phi)|^2\leq \int_{C_L}|d\phi|^2d\theta dt,
\]
from which the conclusion of the lemma follows.
\end{proof}
\begin{remark}
\label{rem:ext_bound}
Note that the inequality of the lemma is conformally invariant, so it can be applied to any domain conformally equivalent to a cylinder. Below we mostly use it on $C_{\eps,L}(p)$, which is simply a scaled copy of $C_L$. However, one can also apply the lemma to annular neighbourhoods in flat metrics $g'_p$ to yield
\[
\|d\hat \phi\|^2_{L^2(D_{\eps}(p))}\leq C\|d \phi\|^2_{L^2(D_{2\eps}(p)\setminus D_{\eps}(p))}.
\] 
\end{remark}

Let us introduce the notation 
\[
\Lambda_1' = \Lambda_1^{T'}(M',\mC');\qquad \Lambda_1 = \Lambda_1^{T}(M,\mC_{\eps,L}),
\]
where we continue to omit the dependence on $\eps, L$. Let $F\in W^{1,2}(M,\mathbb{R}^{n+1})$ be the map supplied by Lemma~\ref{lem:good_map}.
For all points $p\in P'$, Lemma~\ref{lem:ext_bound} implies that
\[
\int_{D_\eps(p)} |d\hat F|^2_{g'}\,dv_{g'}\leq (1 + Ce^{-2L})\int_{C_{\eps,L}(p)}|dF|^2.
\]
Summing over $p\in P$ we arrive at
\begin{equation}
\label{ineq:fhueb}
\begin{split}
\int_{M'}|d\hat{F}|_{g'}^2\,dv_{g'}&\leq \int_{(M'\setminus D_\eps(P'))\cup C_{\eps,L}}|dF|^2\,dv_{g_{\eps,L}}+Ce^{-2L}\int_{C_{\eps,L}}|dF|^2\\
&\leq \Lambda_1+Ce^{-2L}\int_{C_{\epsilon,L}}|dF|^2 +\eps^2,
\end{split}
\end{equation}
where we used Lemma~\ref{lem:good_map} in the last inequality.

Let $Q$ denote the  nonnegative-definite quadratic form on $W^{1,2}(M', \mathbb{R}^{n+1})$ given by
\[
Q(\Phi,\Phi):=\int_{M'}|d\Phi|_{g'}^2\,dv_{g'}-\Lambda_1'\int_{M'}\left|\Phi-\int_{M'}\Phi\,dv_{g'}\right|^2\,dv_{g'},
\]
where we recall that $g'$ is a unit-area metric such that $\lambda_1(M',g')=\Lambda_1'$.
Applying the Cauchy-Schwarz inequality to the corresponding bilinear form, 
we observe that for any $\Phi\in W^{1,2}(M', \mathbb{R}^{n+1})$ satisfying $\int_{M'}\Phi\,dv_{g'} = 0$ one has
\begin{equation}
\label{ineq:CS1}
 \int_{M'}\langle d\hat{F},d\Phi\rangle_{g'}\,dv_{g'}-\Lambda_1'\int_{M'}\langle \hat{F},\Phi\rangle\,dv_{g'}= Q(\hat F,\Phi) \leq \sqrt{Q(\hat{F},\hat{F})}\sqrt{Q(\Phi,\Phi)},
\end{equation}
where we used~\eqref{ineq:gm3} in the first equality.

In particular, it follows from~\eqref{ineq:fhueb} and~\eqref{ineq:fhllb} that
\[
Q(\hat{F},\hat{F})\leq (\Lambda_1-\Lambda_1') + Ce^{-2L}\|dF\|_{L^2(C_{\eps,L})}^2+C\eps,
\]
and the Cauchy-Schwarz estimate~\eqref{ineq:CS1} gives
\begin{equation}
\label{ineq:CS2}
\int_{M'}(\langle d\hat{F},d\Phi\rangle_{g'}-\Lambda_1'\langle \hat{F},\Phi\rangle)\,dv_{g'}\leq C((\Lambda_1-\Lambda_1')^{\frac{1}{2}}+e^{-L}\|dF\|_{L^2(C_{\eps,L})}+\eps^{\frac{1}{2}})\sqrt{Q(\Phi,\Phi)}
\end{equation}
for all $\Phi$ satisfying $\int_{M'}\Phi\,dv_{g'} = 0$.

\begin{lemma}
\label{lem:Feb_cyl}
One has 
\[
\|dF\|^2_{L^2(C_{\eps,L})}\leq \Lambda_1-\Lambda_1' + \frac{C}{|\log\eps|}.
\]
\end{lemma}
\begin{proof}
Let $\rho_{\eps,p}\in W^{1,2}(M')$ be a logarithmic cut-off function given in polar coordinates centered at $p$ with respect to $g_p$ as
\[
\rho_{\eps,p} (r) = 
\begin{cases}
1, &\text{ if } r\leq\eps;\\
\frac{\log (r/\eps^{1/2})}{\log(\eps^{1/2})}, &\text{ if }\eps^{\frac{1}{2}}\geq r\geq\eps;\\
0, &\text{ if } r\geq\eps^{\frac{1}{2}}.\\
\end{cases}
\]
Choose $\eps$ small enough so that the supports of $\rho_{\eps,p}$ are disjoint for $p\in P'$ and define $\psi_\eps:= 1 - \sum_{p\in P}\rho_{\eps,p}$. Since the Dirichlet integral is conformally invariant, one has
\[
\int_{M'}|d\psi_\eps|^2_{g'}\,dv_{g'}\leq \frac{C}{|\log\eps|}.
\]

Next, observe that
\begin{eqnarray*}
\left|\int_{M'}\psi_{\eps}F\,dv_{g'}\right|&=&\left|\int_{M'}\hat{F}\,dv_{g'}+\int_{M'}(\psi_{\eps}-1)\hat{F}\,dv_{g'}\right|\\
&=&\left|\int_{M'}(\psi_{\eps}-1)\hat{F}\,dv_{g'}\right|
\leq \area_{g'}(\mathrm{supp}(\psi_\eps-1))\leq C\eps,
\end{eqnarray*}
so that
\begin{equation}
\label{ineq:fllbM'}
\int_{M'}|d(\psi_{\eps}F)|_{g'}^2\,dv_{g'}\geq \Lambda_1'\left(\int_{M'}\psi_{\eps}^2|F|^2\,dv_{g'}-C\eps\right)\geq \Lambda_1'-C'\eps,
\end{equation}
where we used~\eqref{ineq:fhllb} in the last inequality.

On the other hand, using conformal invariance of the Dirichlet integral, on $A_\eps(p):= D_{\sqrt{\eps}}(p)\setminus D_{\eps}(p) =\mathrm{supp}(d\rho_{\eps,p})$ one has
\begin{equation*}
\begin{split}
&\int_{A_\eps(p)}\langle d(|\psi_{\eps}|^2),d(|F|^2)\rangle_{g'}\,dv_{g'} = \int_{A_\eps(p)} |F|^2 \Delta_{g'}
(|\psi_{\eps}|^2)\,dv_{g'} - \\
2&\int_{\partial D_{\sqrt{\eps}}(p)}\frac{1}{r\log(\eps^{1/2})}\left( 1- \frac{\log (r/\eps^{1/2})}{\log(\eps^{1/2})}\right)|F|^2d\theta +\\ 2&\int_{\partial D_{\eps}(p)}\frac{1}{r\log(\eps^{1/2})}\left( 1- \frac{\log (r/\eps^{1/2})}{\log(\eps^{1/2})}\right)|F|^2d\theta = \\
-2&\int_{A_{\eps}(p)}|F|^2|d\psi_{\eps}|^2_{g'}\,dv_{g'} + \frac{2}{\eps^{1/2}\log(\eps^{-1/2})}\int_{\partial D_{\sqrt{\eps}}(p)} |F|^2\,d\theta\leq \frac{C}{|\log\eps|},
\end{split}
\end{equation*}
where we used $|F|\leq 1$ in the last step. Therefore, we have 
\begin{equation*}
\begin{split}
\int_{M'}|d(\psi_{\eps}F)|_{g'}^2\,dv_{g'}&=\int_{M'} \psi_{\eps}^2|dF|_{g'}^2+\frac{1}{2}\langle d(|\psi_{\eps}|^2),d(|F|^2)\rangle_{g'} + |d\psi_{\eps}|_{g'}^2|F|^2\,dv_{g'}\leq \\
&\leq\int_{M'\setminus D_{\eps}(P')} |dF|^2_{g'}\,dv_{g'}+\frac{C}{|\log\eps|}.
\end{split}
\end{equation*}
Together with~\eqref{ineq:fllbM'} this implies  
\[
\int_{M'\setminus D_{\eps}(P')}|dF|^2_{g_{\eps,L}}\,dv_{g_{\eps,L}} = \int_{M'\setminus D_{\eps}(P')}|dF|^2_{g'}\,dv_{g'}\geq \Lambda'_1-\frac{C}{|\log\epsilon|}.
\]
At the same time, by Lemma~\ref{lem:good_map} one has
\[
\int_M|dF|^2_{g_{\eps,L}}\,dv_{g_{\eps,L}}\leq \Lambda_1 +\eps^2,
\]
so the combining the last two inequalities concludes the proof.
\end{proof}

Lemma~\ref{lem:Feb_cyl} allows us to rewrite inequality~\eqref{ineq:CS2} as
\begin{equation}
\label{ineq:CS3}
\int_{M'}(\langle d\hat{F},d\Phi\rangle_{g'}-\Lambda_1'\langle \hat{F},\Phi\rangle)\,dv_{g'}\leq C\left((\Lambda_1-\Lambda_1')^{\frac{1}{2}}+\frac{e^{-L}}{|\log\eps|^{\frac{1}{2}}}+\eps^{\frac{1}{2}}\right)\sqrt{Q(\Phi,\Phi)}.
\end{equation}

We also record the following corollary of Lemma~\ref{lem:Feb_cyl}.
\begin{corollary}
\label{cor:''bound}
One has
\[
\int_{M''\setminus D_\eps(P'')}|dF|^2\,dv_{g_{\eps,L}} \leq (\Lambda_1 - \Lambda_1') + \eps + Ce^{-2L}\left(\Lambda_1-\Lambda_1' + \frac{1}{|\log\eps|}\right).
\]
\end{corollary}
\begin{proof}
Using the first line of~\eqref{ineq:fhueb}, Lemma~\ref{lem:Feb_cyl} implies
\[
\int_{M'}|d\hat{F}|_{g'}^2\,dv_{g'}\leq \int_{(M'\setminus D_{\eps}(P'))\cup C_{\eps,L}}|dF|^2\,dv_{g_{\eps,L}}+Ce^{-2L}\left(\Lambda_1-\Lambda_1' + \frac{1}{|\log\eps|}\right).
\]
Together with~\eqref{ineq:gm2}, one arrives at
\[
\int_{(M'\setminus D_\eps(P'))\cup C_{\eps,L}}|dF|^2\,dv_{g_{\eps,L}}\geq \Lambda_1' - \eps - Ce^{-2L}\left(\Lambda_1-\Lambda_1' + \frac{1}{|\log\eps|}\right).
\]
Using the upper bound~\eqref{ineq:gm1} on the energy of $F$, we finally obtain
\[
\int_{M''\setminus D_\eps(P'')}|dF|^2\,dv_{g_{\eps,L}}\leq (\Lambda_1 - \Lambda_1') + \eps + Ce^{-2L}\left(\Lambda_1-\Lambda_1' + \frac{1}{|\log\eps|}\right).\qedhere
\]
\end{proof}

By~\eqref{ineq:gm3} we can decompose 
\[
\hat F=\Phi_{\eps,L}+R_{\eps,L},
\]
where $\Phi_{\eps,L}$ is a map by $\lambda_1(M',g')$-eigenfunctions, and $R_{\eps,L}$ is the projection of $\hat F$ onto the higher-frequency eigenspaces with eigenvalues $\lambda\geq \lambda_m(M',g')$, where $\lambda_m(M',g')>\lambda_{m-1}(M',g')=\ldots=\lambda_1(M',g')$. 

Since the kernel of $Q$ is exactly the $\lambda_1(M',g')$-eigenspace, we have $Q(\hat F,R_{\eps,L})=Q(R_{\eps,L},R_{\eps,L})$,
and \eqref{ineq:CS3} gives
\[
Q(\hat F,R_{\eps,L})\leq C\left((\Lambda_1-\Lambda_1')^{\frac{1}{2}}+\frac{e^{-L}}{|\log\eps|^{\frac{1}{2}}}+\eps^{\frac{1}{2}}\right)\sqrt{Q(R_{\eps,L},R_{\eps,L})},
\]
so that
\[
Q(R_{\eps,L},R_{\eps,L})\leq C\left(\Lambda_1-\Lambda_1' + \frac{e^{-2L}}{|\log\epsilon|} + \eps\right).
\]
On the other hand, since $R$ is orthogonal to the  constants and the $\lambda_1(M',g')$-eigenspace, one has
\[
Q(R_{\eps,L},R_{\eps,L})\geq \left(1-\frac{\lambda_1(M',g')}{\lambda_m(M',g')}\right)\|dR_{\eps,L}\|_{L^2(M')}^2,
\]
hence,
\begin{equation}
\label{ineq:Reub}
\|R_{\eps,L}\|^2_{W^{1,2}(M')} \leq C\left(\Lambda_1-\Lambda_1' + \frac{e^{-2L}}{|\log\epsilon|} + \eps\right).
\end{equation}

\begin{lemma} 
\label{lem:average}
Let $(N,h)$ be a closed surface, $p\in N$ and let $h_p$ be a flat metric conformal to $h$ defined in a neighbourhood $U_p$ of $p$. If $D_\eps(p)$ is a disk of radius $\eps$ around $p$ in the metric $h_p$, then
\[
\left|\frac{1}{\eps}\int_{\partial D_{\eps}(p)}u\,ds_{h_p}\right|\leq C|\log\eps|^{\frac{1}{2}}\|u\|_{W^{1,2}(U_p\setminus D_{\eps}(p),h)},
\]
where $C$ depends only on $h$ and the choice of $h_p$.
\end{lemma}

\begin{proof} At a fixed positive radius $r_0>0$, the estimate
\begin{equation}
\label{ineq:tr}
\left|\frac{1}{r_0}\int_{\partial D_{r_0}(p)}u\,ds_{h_p}\right|\leq C\left|\frac{1}{r_0}\int_{\partial D_{r_0}(p)}u\,ds_{h}\right|\leq C\|u\|_{W^{1,2}(U_p\setminus D_{r_0}(p),h)}
\end{equation}
follows, for example,  from the boundedness of the trace map 
\[
W^{1,2}(D_{2r_0}(p)\setminus D_{r_0}(p),h)\to L^2(\partial D_{r_0}(p),h).
\]
On the flat annulus $A_\eps(p) = D_{r_0}(p)\setminus D_{\eps}(p)$, it's easy to see that
\[
\left(\frac{1}{r_0}\int_{\partial D_{r_0}(p)}u\,ds_{h_p}-\frac{1}{\eps}\int_{\partial D_{\eps}(p)}u\,ds_{h_p}\right)=\int_{A_\eps(p)}\langle du, d \log r\rangle_{h_p}\,dv_{h_p},
\]
where $r$ is the distance function to point $p$ in the metric $h_p$. At the same time,
\[
\int_{A_\eps(p)}|d\log r|^2_{h_p}\,dv_{h_p}=2\pi |\log\eps|,
\]
which together with an application of Cauchy-Schwarz yields
\[
\left|\frac{1}{\eps}\int_{\partial D_{\eps}}u\,ds_{h_p}\right|\leq C|\log\eps|^\frac{1}{2}\|d u\|_{L^2(A_\eps,h_p)}+\left|\frac{1}{r_0}\int_{\partial D_{r_0}}u\,ds_{h_p}\right|.
\]
Combining with~\eqref{ineq:tr} completes the proof.
\end{proof}

We are now ready to prove the theorem. Assume by contradiction that $\Lambda_1\leq \Lambda_1'$ for all $\eps, L$. 
Let $P_i'\subset P'$ be an equivalence class, then by definition the set $\alpha(P'_i)=P_i''\subset P''\subset M''$ lies on the same connected component $M''_i$ of $M''$. Therefore, $\lambda_1(M_i'',g'')>0$. Let $\tilde F_i$ denote the harmonic extension of $F|_{M_i''\setminus D_\eps(P'')}$ to $M_i''$ and define
\[
\bar F_i: = \frac{1}{\area(M_i'',g'')}\int_{M_i''} \tilde F_i\,dv_{g''}.
\]
Then for any $q\in P_i''$, by Lemma~\ref{lem:average} one has 
\begin{equation*}
\begin{split}
\left|\frac{1}{\eps}\int_{\partial D_\eps(q)} F\,ds_{g''_p} - 2\pi\bar F_i \right|^2&\leq C|\log\eps|\|\tilde F_i-\bar F_i\|^2_{W^{1,2}(M_i'',g'')}\\
\text{(by Definition of $\bar F_i$)}&\leq C\left(1 + \frac{1}{\lambda_1(M_i'',g'')}\right)|\log\eps|\|d\tilde F_i\|^2_{L^2(M_i'',g'')}\\
\text{(by Remark~\ref{rem:ext_bound})}&\leq C|\log\eps|\|d F_i\|^2_{L^2(M_i''\setminus D_\eps(P''),g'')}\\
\text{(by Corollary~\ref{cor:''bound})}&\leq C|\log\eps| \left(\eps + \frac{e^{-2L}}{|\log\eps|}\right).
\end{split}
\end{equation*}
Similarly, by Lemma~\ref{lem:average} and~\eqref{ineq:Reub} one has for all $p\in P_i'$
\[
\left|\frac{1}{\eps}\int_{\partial D_\eps(p)}R_{\eps,L}\,ds_{g'_p}\right|^2\leq C\left(e^{-2L} + \eps|\log\eps|\right).
\]
Finally, one has
\begin{equation*}
\begin{split}
&\left|\frac{1}{\eps}\int_{\partial D_\eps(p)}F\,dv_{g'_p} - \frac{1}{\eps}\int_{\partial D_\eps(\alpha(p))}F\,dv_{g''_p}\right| = 
\frac{1}{\eps}\left|\int_{C_{\eps,L(p)}}\frac{\partial F}{\partial t}\,dtd\theta\right|\leq\\
&\frac{1}{\eps}\|dF\|_{L^2(C_{\eps,L}(p))}\area(C_{\eps,L}(p))^{\frac{1}{2}}\leq \frac{C\sqrt{L}}{|\log\eps|^\frac{1}{2}},
\end{split}
\end{equation*}
where we applied Lemma~\ref{lem:Feb_cyl} in the last step.

Combining all these inequalities, one obtains for all $p\in P_i'$
\begin{equation*}
\left|\frac{1}{\eps}\int_{\partial D_\eps(p)}\Phi_{\eps,L}\,dv_{g_p'} - 2\pi \bar F_i\right|\leq C\left(\frac{\sqrt{L}}{|\log\eps|^\frac{1}{2}} + e^{-L} + (\eps|\log\eps|)^{\frac{1}{2}}\right),
\end{equation*}
and, therefore, for all $p,q\in P_i'$ one has
\begin{equation}
\label{ineq:Phi_bound}
\left|\frac{1}{\eps}\int_{\partial D_\eps(p)}\Phi_{\eps,L}\,dv_{g_p'} - \frac{1}{\eps}\int_{\partial D_\eps(q)}\Phi_{\eps,L}\,dv_{g_q'}\right|\leq C\left(\frac{\sqrt{L}}{|\log\eps|^\frac{1}{2}} + e^{-L} + (\eps|\log\eps|)^{\frac{1}{2}}\right).
\end{equation}

Choose $L = L(\eps)$ such that as $\eps\to 0$ one has $L(\eps)\to\infty$ while $L = o(|\log\eps|)$. This way the right hand side of~\eqref{ineq:Phi_bound} tends to $0$ as $\eps\to 0$. 

Consider the maps $\Phi_{\eps,L(\eps)}\colon M'\to \mathbb{R}^{n_{\eps,L}+1}$ by first eigenfunctions. Since the $\lambda_1(M',g')$-eigenspace has dimension $m$ (independent of $\eps,L$) there exists a rotation $A_{\eps}$ of $\mathbb{R}^{n_{\eps,L}+1}$ such that the only non-vanishing components of the map $A_\eps\Phi_{\eps,L(\eps)}$ are the first $m$ and, hence, it can be viewed as a map to $\mathbb{R}^m$. We set 
\[
\Phi_\eps:=A_\eps\Phi_{\eps,L(\eps)}\colon M'\to \mathbb{R}^m.
\] 
By~\eqref{ineq:Reub},~\eqref{ineq:fhueb} and~\eqref{ineq:gm2} one has
\[
\int_{M'}|d\Phi_\eps|^2_{g'}\,dv_{g'} = \Lambda_1' + o(1)
\]
as $\eps\to 0$. Since $|1-|u||\leq|1-|v|| + |u-v|$,~\eqref{ineq:Reub} and~\eqref{ineq:gm2} imply that
\[
\int_{M'}|1-|\Phi_\eps||\,dv_{g'}\leq\int_{M'} |1-|\hat F|| + |R_{\eps,L(\eps)}|\,dv_{g'} = o(1).
\]
In particular, $\Phi_\eps$ are bounded in $W^{1,2}(M',g')$ and, since $\Phi_\eps$ are maps by first eigenfunctions, they are 
precompact in $C^k$ for any $k$. Thus, we can extract a $C^\infty$-convergent subsequence $\Phi_\eps\to \Phi$. 
By construction
$\Phi$ is a map by first eigenfunctions. Passing to the limit as $\eps\to 0$ in the previous inequalities yields that $\Phi\colon (M',\mC")\to\mathbb{S}^{m-1}$ is a map to a sphere of energy $E_{g'}(\Phi) = 2\Lambda_1'$. From~\eqref{ineq:Phi_bound} we obtain that for all $p,q\in P_i'$ one has $\Phi(p) = \Phi(q)$, i.e. the image $\Phi(P_i)$ contains only one point. Moreover, we may take $\Phi$ to be $T'$-equivariant, if necessary by replacing $\Phi$ with the $T'$-equivariant map $\Phi_\Gamma\colon M'\to\mathbb{R}^{m|\Gamma|}$ given by
\[
\Phi_\Gamma(x) = \frac{1}{\sqrt{|\Gamma|}}(\Phi(\g_1 \cdot x),\Phi(\g_2\cdot x),\ldots,\Phi(\g_{|\Gamma|}\cdot x)),
\]
where $\g_i$ is an arbitrary enumeration of elements in $\Gamma$. Since the first eigenspace is $\Gamma$-invariant, $\Phi_\Gamma$ is a map by $\lambda_1(M',g')$-eigenfunctions, and maps to the unit sphere with $E(\Phi) = E(\Phi_\Gamma)$. Moreover, if $p,q\in P'$ belong to the same equivalence class, then $\g\cdot p, \g\cdot q\in P'$ belong to the same equivalence class. Hence, if $p,q\in P_i'$, then $\Phi(\g\cdot p) = \Phi(\g\cdot q)$ for all $\g\in \Gamma$, i.e. the image $\Phi_\Gamma(P_i')$ still contains only one point for all $i$. Thus, $\Phi_{\Gamma}:(M',g')\to \Sph^n$ is a $T'$-equivariant map by first eigenfunctions such that $\Phi_{\Gamma}(P_i)$, violating the condition of Theorem~\ref{thm:general_existence} on the images of $P_i$. 
 This contradiction completes the proof. \qed

\section{Existence of maximizers for surfaces with boundary}

\label{sec:global_Sexistence}
In this section we provide results for equivariant Steklov eigenvalue optimization analogous to those of Section~\ref{sec:global_existence}. The section is structured in the same way: we first formulate the Steklov version of Theorem~\ref{thm:general_existence}, then explore its corollaries, and, finally, prove the theorem at the end of the section.  

For a collapsed set $P'\sqcup P'_i$ of the degeneration $(N',T')\prec (N,T)$ recall the definition of the set $P'^\iota$ from Section~\ref{sec:top_degen_boundary}.
One has $p'\in P_i'\cap P'^{\iota}$ if $p'$ is an interior point of $N'$ and the surface $N''_i$, which is being attached to $P_i'$ to form $N$, has non-empty boundary, see Section~\ref{sec:top_degen_boundary} for details.

\begin{theorem}
\label{thm:general_Sexistence}
Suppose that $(N',T')\prec (N,T)$. Let $\mC'$ be a $T'$-invariant conformal class on $N'$. Then 
\[
\Sigma_1^T(N)>\Sigma^{T'}_1(N',\mC')
\]
provided that there exists a collapsed set $P'\subset M'$ with partition $P'=\sqcup P'_i$ and $T'$-invariant $\bar\sigma_1^{T'}$-maximal metric $g'\in\mC'$ such that either
\begin{enumerate}
\item $P'^{\iota}$ is nonempty or
\item for any $T'$-equivariant map $\Psi\colon (N',\partial N')\to(\mathbb{B}^{n+1},\Sph^n)$ by first Steklov eigenfunctions for $g'$, the image $\Psi(P_i)$ contains more than one point for some $i$.
\end{enumerate}
\end{theorem}

%Before proving this theorem, let us provide its corollaries for some simple choices of group actions.

\subsection{Partial existence without group action} Theorem~\ref{thm:general_Sexistence} yields new results evenfor the trivial group, i.e. for the classical $\bar{\sigma}_1$-maximization problem. Let $\Sigma_1(\gamma,b)=\Sigma_1(N)$ for the orientable surface $N$ with genus $\gamma$ and $b$ boundary components.

\begin{corollary}
\label{cor:classical_Sexistence}
Let $N$ be an orientable surface of genus $\gamma$ with $b$ boundary components admitting a $\bar\sigma_1(N)$-maximal metric $g'$. Then 
\begin{equation}
\label{ineq:classical_Sexistence1}
\Sigma_1(\gamma,b+1)>\Sigma_1(\gamma,b).
\end{equation}
Suppose, additionally, that $b\geq 2$ and there exist two points $p,q$ lying on different connected components of $\partial N$ such that for any $T'$-equivariant map $\Psi\colon (N',\partial N')\to(\mathbb{B}^{n+1},\Sph^n)$ by first Steklov eigenfunctions for $g'$ one has $\Psi(p)\ne \Psi(q)$. Then
\begin{equation}
\label{ineq:classical_Sexistence2}
\Sigma_1(\gamma+1,b-1)>\Sigma_1(\gamma,b).
\end{equation}
\end{corollary} 

\begin{proof} The proof is similar to that of Corollary~\ref{cor:classical_existence}. Choose $\mC'$ to be the conformal class of the metric $g'$. To prove~\eqref{ineq:classical_Sexistence1} observe that the surface with $b+1$ boundary components can be obtained from $N$ by gluing a cylinder at an interior point of $N$; the corresponding collapsed set has a single point and it belongs to $P'^\iota$, so that condition (1) of Theorem~\ref{thm:general_Sexistence} is satisfied.

Similarly, to prove inequality~\eqref{ineq:classical_Sexistence2} observe that the surface of genus $\gamma+1$ and $b-1$
boundary components can be obtained from $N$ by gluing a half-cylinder at two points $p,q$ from different boundary components of $N$. For such a degeneration $p\sim q$, so that the condition of the corollary implies that (2) of Theorem~\ref{thm:general_Sexistence} is satisfied.
%Choose $\mC'$ to be the conformal class of the $\bar\lambda_1(M)$-maximal metric $g'$.
%The surface of genus $\gamma+1$ can be obtained from a surface of genus $\gamma$ by gluing a cylinder at two arbitrary points $p,q\in M$. For such degeneration $P=\{p,q\}$ and $p\sim q$, hence if $\Phi(p)\ne\Phi(q)$ the image of the equivalence class contains more than one point. Thus, the conditions of Theorem~\ref{thm:general_existence} are satisfied and applying it yields the conclusion of the corollary.
\end{proof}

\begin{corollary}
\label{cor:classical_Sexistence2}
If $\gamma = 0$, then for any $b\geq 1$ inequality~\eqref{ineq:classical_Sexistence1} holds and, as a result, the quantity $\Sigma_1(0,b)$ is achieved by a smooth metric. Moreover, for any $\gamma>0$ there exists $B(\gamma)$ such that for all $b\geq B(\gamma)$ at least one of $\Sigma_1(\gamma,b)$, $\Sigma_1(\gamma+1, b)$ is achieved by a smooth metric.
\end{corollary}
\begin{proof}
The first conclusion of the corollary is known to follow from~\eqref{ineq:classical_Sexistence1} and Theorem~\ref{thm:Sexistence} (which for $\Gamma=\{1\}$ was previously proved in~\cite{FraserSchoen, PetridesS}) by induction on $b$. Namely, for $b=1$ one has that $\Sigma_1(0,1)$ is achieved by a standard disk, which is the base of induction. If the result is known for $b$, then existence of the metric achieving $\Sigma_1(0,b)$ implies~\eqref{ineq:classical_Sexistence1}, which in turn implies existence of the metric achieving $\Sigma_1(0,b+1)$ by Theorem~\ref{thm:Sexistence}.

For $\gamma>0$~\cite[Theorem 1]{PetridesS} gives existence of the maximum for $\Sigma_1(\gamma,b)$ provided
\begin{equation}
\label{ineq:Petrides_existence}
\Sigma_1(\gamma,b)>\max\{\Sigma_1(\gamma,b-1),\Sigma_1(\gamma-1,b+1)\}.
\end{equation}
The arguments in~\cite[Theorem 5.9]{KSminmax} establish that if $\Lambda_1(\gamma)>\Lambda_1(\gamma-1)$, then for $b\geq B(\gamma)$ one has $\Sigma_1(\gamma,b)>\Sigma_1(\gamma-1,b+1)$ and $\Sigma_1(\gamma,B(\gamma))>\Sigma_1(\gamma,B(\gamma)-1)$. This observation together with Corollary~\ref{cor:classical_Sexistence} allow for the same type of induction argument starting with $b=B(\gamma)$. 
\end{proof}

\begin{corollary}
For any $\gamma\geq 0$ and any $b\geq 1$ one has
\[
\Sigma_1(\gamma,b+1)> \Sigma_1(\gamma,b).
\]
\end{corollary}
\begin{proof}
Let $N$ be a surface of genus $\gamma$, $b+1$ boundary components; $N'$ be a surface of genus $\gamma$, $b$ boundary components. Let $N''\preccurlyeq N'$ be the smallest surface (w.r.t. the order induced by $\prec$) such that $\Sigma_1(N'') = \Sigma_1(N')$. Then by definition, for any $N'''\prec N''$ one has $\Sigma_1(N''')<\Sigma_1(N'')$ and, hence, there is a $\bar\sigma_1$-maximal metric on $N''$. Since $N''\preccurlyeq N'\prec N$, we observe that for the degeneration $N''\prec N$ there exists a collapsed set $P'$ with non-empty $P'^\iota$, so that by Theorem~\ref{thm:general_Sexistence} one has
\[
\Sigma_1(\gamma,b+1) = \Sigma_1(N)>\Sigma_1(N'') = \Sigma_1(N') = \Sigma_1(\gamma,b).
\]
\end{proof}

%\begin{corollary}
%For any $\gamma\geq 0$, $b\geq 2$ one has
%\[
%\Sigma_1(\gamma+2,b-1)>\Sigma_1(\gamma,b).
%\]
%%In particular, for all $\gamma\geq 0$ either $\Lambda_1(\gamma)$ or $\Lambda_1(\gamma+1)$ is achieved on a smooth (up to a finite number of conical singularities) metric.
%\end{corollary}
%\begin{proof}
%NOT SURE IF WE WANT TO INCLUDE THIS PROPOSITION, IT HAS NO COROLLARIES FOR EXISTENCE
%\end{proof}

\subsection{Existence of group-invariant maximal metrics}
Let us now consider surfaces with group actions introduced in Section~\ref{sec:actions_Sex}. 
\begin{lemma}		
\label{lemma:Scollapsed_sets}
Let $(N,T)$ be one of the surfaces with group actions considered in Section~\ref{sec:actions_ex}. Then for any elementary degeneration $(N',T')\prec (N,T)$ there exists a choice of a collapsed set $P'$ and its partition $P'=\sqcup P'_i$, such that either all $P'_i$ have at least two elements or $P'^\iota$ is not empty.
\end{lemma}
\begin{proof}
Follows immediately from Lemma~\ref{lemma:collapsed_sets} using the considerations from Section~\ref{sec:top_degen_boundary}. Indeed, Lemma~\ref{lemma:collapsed_sets} implies that there exists a collapsed set $\wt P' = \sqcup \wt P'_j$ for the degeneration of the doubles $(\wt N',\wt T')\prec (\wt N,\wt T)$ such that each $ \wt P'_j$ has at least two points. Suppose that for some $i$ the set $P'_i$ has a single point $p'$, then its preimage in $\wt P$ has at most two points. Since this preimage is either a single equivalence class or a union of two classes interchanged by $\iota$, we conclude that it has to be a single equivalence class with two elements. This implies $p'\in P'^\iota$.
\end{proof}

\begin{proposition}
\label{prop:Suniq_map}
Let $N$ be an orientable surface whose number of boundary components does not equal $2$.
Suppose that $\Psi\colon (N,g)\to\mathbb{B}^3$ is a free boundary minimal embedding by $\sigma_1(N,g)$-eigenfunctions such that multiplicity of $\sigma_1(N,g)$ equals $3$. If $\Xi$ is another free boundary harmonic map to $\mathbb{B}^3$ by   
$\sigma_1(N,g)$-eigenfunctions, then $\Xi = A\Psi$ for some $A\in O(3)$. 
\end{proposition}
\begin{remark}
\label{rmk:disk}
Note that by the multiplicity condition one has that $N$ is not diffeomorphic to a disk~\cite{Kokarev}. Moreover, if $N$ is a disk, then as follows from the arguments below, an analogous statement holds. Namely, if $\Psi\colon (N,g)\to\mathbb{B}^2$ is a free boundary minimal embedding by $\sigma_1(N,g)$-eigenfunctions, then any other free boundary harmonic map by $\sigma_1(N,g)$-eigenfunctions is $A\Psi$ for some $A\in O(2)$.
\end{remark}
\begin{proof}
%Let $N, \Psi$, and $\Xi$ be as in the Proposition.  
%Let $\Xi\colon (N,g)\to\mathbb{B}^3$ be another harmonic map by $\sigma_1(N,g)$-eigenfunctions. 
Since $\Psi$ is an embedding and $N$ is not diffeomorphic to the disk, the components of $\Psi$ are linearly independent and, hence, span the $\sigma_1(N,g)$-eigenspace. As a result, the components of $\Xi$ are linear combinations of components of $\Psi$, and since $|\Xi|^2 \equiv 1$ on $\partial N$, the image $\Psi(\partial N)$ lies on the intersection of $\mathbb{S}^2$ with a level set $Q(x)=1$ of a non-negative definite quadratic form $Q$ on $\mathbb{R}^3$. If $Q(x) = |x|^2$, then $\Xi = A\Psi$ for some $A\in O(3)$. Otherwise, the intersection $I = \mathbb{S}^2\cap \{Q(x)=1\}$ is at most $1$-dimensional. Since $\Psi(\partial N)\subset I$  one has that $\Psi(\partial N)$ is, in fact, a union of connected components of $I$. We claim that this implies that $N$ has exactly two boundary components, which would contradict the assumptions of the proposition.

Indeed, by diagonalizing $Q$ in an orthonormal basis, we may assume that 
\[
Q(x) = ax_1^2+bx_2^2+cx_3^2,
\]
where $a\geq b\geq c$. Subtracting the equation of the sphere from $Q$, one has that $I = \mathbb{S}^2\cap \{Q'(x)=0\}$, where
\[
Q'(x) = (a-1)x_1^2 + (b-1)x_2^2 + (c-1)x_3 =: a'x_1^2+b'x_2^2+c'x_3^2,
\]
where $a'\geq b'\geq c'$. If $c'> 0$ or $a'<0$, then $I$ is empty, which is impossible. Without loss of generality, we may assume that $a'>0$. If $c'=0$, then since $b'\geq c'=0$, $I = \{0,0,\pm1\}$ unless $b'=0$. In the latter case $I$ is an equator in $\mathbb{S}^2$ and, hence, $\Psi(N)$ is a disk, which contradicts the fact that $N$ is not diffeomorphic to a disk. Assume now that $c'<0$. Up to changing $Q'\mapsto (-Q')$, one has $b'\geq 0$. Then $x_3\ne 0$ on $I$ and it is easy to observe that $I$ has two connected components $I^\pm$ according to the sign of $x_3$. If $\Psi(\partial N) = I^+$, then $x_3\circ\Psi$ is a $\sigma_1(N,g)$-eigenfunction, which is positive on $\partial N$, hence, not orthogonal to constants, a contradiction. Thus, $\Psi(\partial N) = I$ and $\Psi(\partial N)$ (therefore, $N$ as well) has exactly two boundary components.
\end{proof}
\begin{remark}
\label{rmk:quadric_catenoid}
It is easy to see that the boundary of the critical catenoid lies in the intersection of a sphere with another quadric, e.g. two parallel planes. And, indeed, one observes that there are other free boundary harmonic maps from the critical catenoid to $\mathbb{B}^3$ by $\sigma_1$-eigenfunctions. If the coordinates are chosen so that the catenoid is rotationally symmetric in $xy$-plane, then all such harmonic maps can be written as $(ax,ay,bz)$ for some $a,b$.
\end{remark}

To illustrate the value of Proposition~\ref{prop:Suniq_map}, let us show the following.

\begin{corollary}
\label{cor:classical_Sexistence3}
If $\gamma=0$ then for any $b\geq 2$ inequality~\eqref{ineq:classical_Sexistence2} holds and, as a result, the quantity $\Sigma_1(1,b-1)$ is achieved by a smooth metric.
\end{corollary}
\begin{proof}
According to~\cite[Theorem 1]{PetridesS} it sufficient to check that inequality~\eqref{ineq:Petrides_existence} holds for $\gamma=1$ and any $b\geq 1$. Similarly to Corollary~\ref{cor:classical_Sexistence}, the proof is by induction on $b$. For $b=1$ it is sufficient to prove $\Sigma_1(1,1)>\Sigma_1(0,2)$. 
By Remark~\ref{rmk:quadric_catenoid}, if we choose $p,q$ on different connected components of the catenoid and with different $(x,y)$-coordinates, then for any free boundary harmonic map $\Psi$ from the catenoid to $\mathbb{B}^3$ by $\sigma_1$-eigenfunctions, one has $\Psi(p)\ne\Psi(q)$. Therefore, Corollary~\ref{cor:classical_Sexistence} implies that 
$\Sigma_1(1,1)>\Sigma_1(0,2)$.

Suppose that the Corollary is proved for $b\geq 1$ boundary components. Then there exists a smooth metric achieving $\Sigma_1(1,b)$ and, hence, by~\eqref{ineq:classical_Sexistence1} one $\Sigma_1(1,b+1)>\Sigma_1(1,b)$. It remains to show $\Sigma_1(1,b+1)>\Sigma_1(0,b+2)$. By Corollary~\ref{cor:classical_Sexistence2} $\Sigma_1(0,b+2)$ is achieved by a smooth metric and by~\cite[Proposition 8.1]{FraserSchoen} the corresponding free boundary minimal surface is embedded in $\mathbb{B}^3$. Since $b+2>2$, Proposition~\ref{prop:Suniq_map} together with Corollary~\ref{cor:classical_Sexistence} implies $\Sigma_1(1,b+1)>\Sigma_1(0,b+2)$.
\end{proof}

In view of Remark~\ref{rmk:quadric_catenoid} it is natural to conjecture that the critical catenoid is the only free boundary minimal surface in $\mathbb{B}^3$ violating the conclusion of Proposition~\ref{prop:Suniq_map}, but we are unable to confirm it. A successful resolution of this conjecture would significantly simplify some of the arguments below. The best we can say is the following.

\begin{proposition}
\label{prop:non_uniq}
Let  $N$ be an orientable surface with two boundary components. Suppose that $\Psi\colon (N,g)\to\mathbb{B}^3$ is a free boundary minimal embedding by $\sigma_1(N,g)$-eigenfunctions such that multiplicity of $\sigma_1(N,g)$ equals $3$. If there exists another free boundary harmonic map to $\mathbb{B}^3$ by   
$\sigma_1(N,g)$-eigenfunctions, such that $\Xi \ne A\Psi$ for all $A\in O(3)$, then $(N,g)$ admits the action of the group $\Gamma = \mathbb{Z}_2\times D_2$ of type $N_{\rho_1}(\Gamma, e_2\rho_2+v_{23}\rho_2\rho_3)$ 
\end{proposition}

\begin{remark}
Note that up to changing the roles of $\tau$ and $\rho_2$ one has that $N_{\rho_1}(\Gamma, e_2\rho_2+v_{23}\rho_2\rho_3)$ is topologically equivalent to $N_{\rho_1}(\Gamma, (e_2-1+v_{23})\rho_2 + v_{12}\rho_1\rho_2 + (1-v_{23})\rho_2\rho_3)$. 
\end{remark}
\begin{proof}
Using the proof of Proposition~\ref{prop:Suniq_map} we have that $\Psi(\partial N) = I$, where $I = \mathbb{S}^2\cap \{Q'(x) = 0\}$,
\[
Q'(x)=a'x_1^2+b'x_2^2+c'x_3^2,
\]
$a'>0$, $b'\geq 0$, $c<0$. Since $\Psi$ is an embedding, for the rest of the proof we identify $N$ with $\Psi(N)$ so that $\partial N = I$. 

The main ingredient of the proof is a result of Schoen in~\cite{Schoen}. To state it, we recall the following notions from~\cite{Schoen}. Split $\mathbb{R}^3 = \mathbb{R}_{x,y}^2\times \mathbb{R}_t$ and view $t$ as a distinguished coordinate. For a subset $S\subset \mathbb{R}^3$ we write $S_\pm = S\cap \{\pm t>0\}$. If $S$ is a (not necessarily closed) manifold, we  say that $S$ is a graph of locally finite slope if the projection onto $\{t=0\}$ plane is one-to-one and for any interior point $p\in S$ the tangent plane $T_pS$ does not contain the vertical vector $(0,0,1)$. The following is a weaker version of~\cite[Theorem 2]{Schoen}, which is sufficient for our purposes.

\begin{theorem}[Schoen~\cite{Schoen}] 
\label{thm:Schoen_uniq}
Let $\Omega\subset \mathbb{R}^2$ be a smooth convex domain and $N\subset\mathbb{R}^3$ be a minimal surface with boundary such that (i) $\partial N\subset \partial\Omega\times\mathbb{R}$; (ii) $\partial N$ is symmetric with respect to reflection across $\{t=0\}$; (iii) $(\partial N)_+$ is a graph of locally bounded slope.

Then $N$ is symmetric with respect to reflection across $\{t=0\}$ and $N_+$ is a graph of locally bounded slope.
\end{theorem}
\begin{remark}
Note that convex hull property for minimal surfaces implies that if $\partial N\subset \partial\Omega\times\mathbb{R}$, then $N\subset \Omega\times\mathbb{R}$.
\end{remark}

We start by setting $x_3=t$ to be the distinguished coordinate and observe that one can write $I = \mathbb{S}^2\cap \{Q''=0\}$, where
\[
Q''(x)=(a'-c')x_1^2+(b'-c')x_2^2+c'
\]
and $(a'-c'), (b'-c')>0$ by our assumptions. Hence, $\Omega = \{(a'-c')x_1^2+(b'-c')x_2^2+c'<0\}\subset \mathbb{R}^2$ is a smooth convex domain with $\partial N\subset \Omega\times \mathbb{R}$. It is easy to see that $\partial N$ is invariant with respect to reflection $x_3\mapsto (-x_3)$. Finally, since $\mathbb{S}^2_+$ is a graph of locally bounded slope, so is $(\partial N)_+\subset \mathbb{S}^2_+$. Therefore, Theorem~\ref{thm:Schoen_uniq} implies that $N$ is symmetric with respect to $x_3\mapsto (-x_3)$ and $N_+$ is graph and, hence, has genus $0$. In particular, if we denote by $\tau$ the restriction of $x_3\mapsto (-x_3)$ to $N$, then $\tau$ is a separating involution such that $N\setminus N^\tau$ has genus $0$ and $N^\tau\cap \partial N = \varnothing$.
\begin{remark}
In fact, one does not require Theorem~\ref{thm:Schoen_uniq} to conclude that $N$ is preserved by the reflection across $\{t=0\}$. Instead, one can use Bj{\"o}rling's Theorem~\cite{Hildebrandt} stating that different free boundary minimal surfaces can not have common boundary.
\end{remark} 
Furthermore, if $a'=b'$, then $\{Q''=0\}$ is a round cylinder and, hence, $I$ consists of two circles lying in parallel planes. Then~\cite[Corollary 3]{Schoen} implies that $N$ is rotationally symmetric, hence, a catenoid. Thus, in the following we may assume $a'>b'$.

Let us now set $x_1=t$ to be the distinguished coordinate. One can write  $I = \mathbb{S}^2\cap \{Q'''=0\}$, where
\[
Q'''(x)=(a'-b')x_2^2+(a'-c')x_3^2-a'
\]
and $(a'-b'), (a'-c')>0$ by our assumptions. Denoting by $\rho_2$ the restriction of $x_1\mapsto (-x_1)$, the same arguments as above can be used to conclude that $N\setminus N^{\rho_2}$ has two connected components of genus $0$. However, this time $N^\rho_2$ does intersect the boundary $\partial N$. In particular, this implies that the action of $\mathbb{Z}_2\times\mathbb{Z}_2$ on $N$ generated by $\tau$ and $\rho_2$ is of type $N_{\rho_1}(\mathbb{Z}_2\times\mathbb{Z}_2,(\gamma+1)\rho_2)$, where $\gamma$ is the genus of $N$. 

Finally, denoting by $\rho_3$ the restriction of $x_2\mapsto (-x_2)$ and analyzing the symmetries of the fundamental domain of $N_{\rho_1}(\mathbb{Z}_2\times\mathbb{Z}_2,(\gamma+1)\rho_2)$, we arrive at the statement of the theorem. 
\end{proof}

Using these observations we can prove the existence result for $\sigma_1$-maximal metrics for a large class of surfaces with group actions. Note that the arguments involved in the proof of the theorem below are significantly more complicated compared to Theorem~\ref{thm:group_existence} in the closed case. There are two main reasons for this. First, compared to Proposition~\ref{prop:uniq_map}, Proposition~\ref{prop:Suniq_map} does not completely determine the topology of surfaces with multiple different free boundary harmonic maps by $\sigma_1$-eigenfunctions. Second, even if we take for granted that the critical catenoid is the only such surface, then there would still be difficulties with starting the induction argument as in the proof of Theorem~\ref{thm:group_existence} due the lack of examples of free boundary minimal surfaces with required symmetry groups and area above that of the critical catenoid. Below, we are able to circumvent these complications in most cases by using some ad hoc arguments together with an additional information provided by Proposition~\ref{prop:non_uniq}. Still, our results are not quite complete and new ideas are required to prove the existence in greater generality.

\begin{theorem}
\label{thm:group_Sexistence}
Let $(N,T)$ be one of the following surfaces with group actions considered in Section~\ref{sec:actions_Sex}: 
\begin{enumerate}
\item $N_\tau(\mathbb{Z}_2, f+e_1\rho_1)$;
\item $N_\tau(D_2, f + e_1\rho_1 + e_2\rho_2 + v_{12}\rho_1\rho_2)$, where $f+e_1+e_2+v_{12}>1$;
\item $N_\tau(D_k, f + e_1\rho_1 + e_2\rho_2 + v_{12}\rho_1\rho_2)$, where $k>2$ and $f+e_1+e_2+v_{12}>2$;
\item $N_{\tau}(\mathbb{Z}_2\times D_k,\ub)$, where $|\ub|>1$;
\item $N_\tau(\Gamma,\ub)$, where $\Gamma$ is one of the platonic groups;
\item $N_{\rho_1}(\mathbb{Z}_2,f+e_1\rho_1)$;
\item $N_{\rho_1}(\mathbb{Z}_2\times D_k,\ub)$, where $k>2$, $\ub\ne f,\rho_2,\rho_3$, and if $v_{23}\ne 0$, then $|\ub-\rho_2\rho_3|\ne 1$.
\item $N_{\rho_1}(D_2, f + e_1\rho_1 + e_2\rho_2 + v_{12}\rho_1\rho_2)$, where either $f,e_1 = 0$ or $e_1\geq 1$, $f<e_2-1$.
%\item $N_{\rho}(D_2\times \mathbb{Z}_2,\ub)$, where $e_1\geq 1$.
\end{enumerate}
Then for any $(N',T')\prec (N,T)$ one has 
\begin{equation}
\label{ineq:toprove_S}
\Sigma_1^{T'}(N')<\Sigma_1^T(N).
\end{equation}
In particular, there exists a $T$-invariant $\bar\sigma_1^T$-maximal metric on $N$ induced by an equivariant free boundary minimal embedding $\Psi_{(N,T)}\colon N\hookrightarrow\mathbb{B}^3$. 
\end{theorem} 
\begin{remark}
\label{Rstek8}
The conditions in item (8) are not optimal and it is likely that clever modifications of our ad hoc arguments could widen the range of parameters. We chose not to pursue this direction any further for the sake of brevity, and because we expect that future developments in free boundary minimal surface theory may allow for a much shorter universal argument similar to the one in Theorem~\ref{thm:group_existence}.
\end{remark}
\begin{proof}
The overall strategy of the proof is analogous to that of Theorem~\ref{thm:group_existence}, where the main difference is in the treatment of the cases to which Proposition~\ref{prop:Suniq_map} can not be applied.
To begin, we define the complexity as in Definition \ref{cdef}--equivalently, we set $c(N,T) = c(\wt{N},\wt{T})$, where $(\wt{N}, \wt{T})$ is a double of $(N,T)$--and prove the theorem by induction on $c(N,T)$. Some of the cases below require a modified definition of complexity, which we introduce later.

Assume first that $c(N,T)=0$. This happens iff $c(\wt{N}, \wt{T}) = 0$, so the cases 2)-4) described in the proof of Theorem~\ref{thm:group_existence} are the possibilities for $(\wt{N}, \wt{T})$. In cases 2) and 3) one has that $N$ is a disk, hence, we can conclude with the help of Corollary~\ref{cor:max_disk}. In case 4) $\wt{N}$ has no $\wt\Gamma$-invariant closed curves, therefore, the existence of a $\bar\sigma_1^T(N)$-maximizer follows from Theorem~\ref{thm:Sexistence}.

Suppose now that the theorem holds for all $(N',T')$ with $c(N',T')\leq k$. Let $c(N,T) = k+1$ and let
$(N',T')\prec (N,T)$ be a topological degeneration, which we can always assume to be elementary.
Then $c(N',T')\leq k$, so that there exists a $T'$-invariant $\bar\sigma_1^{T'}$-maximal metric $g$ on $N'$ induced by a free boundary minimal embedding $\Psi_{(N',T')}$ into $\mathbb{B}^3$. Suppose that $N'$ does not have exactly $2$ boundary components. Then
by Proposition~\ref{prop:Suniq_map}, $\Psi_{(N',T')}$ is the unique map by $\sigma_1(N',g)$ eigenfunctions to $\mathbb{B}^n$ (up to an element of orthogonal group) and since it is embedding, it separates points. Thus, combining Lemma~\ref{lemma:Scollapsed_sets} with  Theorem~\ref{thm:general_Sexistence} yields inequality~\eqref{ineq:toprove_S}.

It remains to consider the case when $N'$ has exactly two boundary components. This is where the particular properties of each family of group actions come in. We consider these families separately.

{\bf Case 1.} $(N',T') = N_\tau(\mathbb{Z}_2, f'+e_1'\rho_1)$. Since $N'$ has $2$ boundary components, one has $(f',e_1') = (1,0)$ or $(0,2)$. In both cases, the critical catenoid admits a $\mathbb{Z}_2$ action of the corresponding type, hence, $\Sigma_1^{T'}(N') = \Sigma_1(\mathbb{A})$ and $\Sigma_1^{T'}(N')$ admits a maximal metric. If $(N',T')\prec (N,T)$ is an elementary degeneration, then $(N,T) = N_\tau(\mathbb{Z}_2,1+\rho_1)$. The collapsed set for the degeneration $N_\tau(\mathbb{Z}_2,1)\prec N_\tau(\mathbb{Z}_2,1+\rho_1)$ can be chosen to consist of a single point in $P'^{\iota}$, hence, by Theorem~\ref{thm:general_Sexistence}
\[
\Sigma_1^{T'}(N') = \Sigma_1(N_\tau(\mathbb{Z}_2,2\rho_1)) = \Sigma_1(N_\tau(\mathbb{Z}_2,1))<\Sigma_1(\mathbb{Z}_2,1+\rho_1) = \Sigma_1^T(N).
\]  

{\bf Case 2.} $(N',T') = N_\tau(D_2, f' + e'_1\rho_1 + e'_2\rho_2 + v'_{12}\rho_1\rho_2)$. Then $(f',e_1',e_2',v_{12}') = (0,1,0,0), (0,0,1,0)$ or $(0,0,0,2)$ and in all cases the critical catenoid admits a $D_2$ action of the corresponding type, so that $\Sigma_1^{T'}(N') = \Sigma_1(\mathbb{A})$ is realized by the critical catenoid. In this case, we can not show~\eqref{ineq:toprove_S} for the degeneration $N_\tau(D_2,\rho_2)\prec N_\tau(D_2,1)$, so we need to modify the induction argument by choosing a different base case.
If $(N,T) = N_\tau(D_2, f + e_1\rho_1 + e_2\rho_2 + v_{12}\rho_1\rho_2)$ with $f+e_1+e_2+v_{12}>1$, we can assume that $f+e_1+e_2>0$, since otherwise $(N,T) = N_\tau(D_2,2\rho_1\rho_2)$, a case that is covered above. Then one has $N_\tau(D_2,\rho_1)\prec N_\tau(D_2,\rho_1+\rho_1\rho_2)\preccurlyeq (N,T)$ or $N_\tau(D_2,\rho_2)\prec N_\tau(D_2,\rho_2 + \rho_1\rho_2) \preccurlyeq (N,T)$ and in both cases there exists a maximal metric on the first surface in the chain, and the collapsed set in the first degeneration can be chosen to have $P'^\iota\neq \varnothing$, so that by Theorem~\ref{thm:general_Sexistence},
\begin{equation}
\label{ineq:cat_comp1}
\Sigma_1(\mathbb{A}) = \Sigma_1^{T'}(N')<\Sigma_1^T(N)
\end{equation}
Let now $(\hat N,\hat T)$ be a surface in the family such that $\Sigma_1^{\hat T}(\hat N)>\Sigma_1(\mathbb{A})$, but for all $(\hat N',\hat T')\prec(\hat N,\hat T)$ one has $\Sigma_1^{\hat T'}(\hat N') \leq \Sigma_1(\mathbb{A})$. In particular, by Theorem~\ref{thm:Sexistence} there exists a $\bar\sigma_1^{\hat T}(\hat N)$-maximal metric. The idea now is to use such $(\hat N,\hat T)$ as a base of induction, namely, we define $\hat c(N,T)$ to be the maximal length of the decreasing chain starting at $(N,T)$ and terminating at one of the $(\hat N,\hat T)$. 
%Second, by~\eqref{ineq:cat_comp1} one has that $(N,T)\not\prec (\hat N,\hat T)$.   

\begin{lemma}
\label{lem:hatc_1}
Let $(N,T) = N_\tau(D_2, f + e_1\rho_1 + e_2\rho_2 + v_{12}\rho_1\rho_2)$ with $f+e_1+e_2+v_{12}>1$. Then  $\hat c(N,T)$ is well-defined. In particular, for any $(N'',T'')\prec (N,T),$ either $\Sigma_1^{T''}(N'') \leq \Sigma_1(\mathbb{A})$ or the quantity $\hat c(N'',T'')$ is well-defined.
\end{lemma}
\begin{proof}
% Clearly, any maximal descending chain terminates at a minimal element. The only minimal element of the partially ordered set in question is $N_{\tau,0,0,0,1}$ and it is easy to see that the set of all $(N',T')$ is exactly the set of all $(N'',T'')$ admitting an elementary degeneration $N_{\tau,0,0,0,1}\prec (N',T')$.
 To prove the first assertion, we need to show that there exists a decreasing chain starting at $(N,T)$ and terminating at one of the $(\hat N,\hat T)$. By~\eqref{ineq:cat_comp1}, one has $\Sigma_1^T(N)>\Sigma_1(\mathbb{A})$. If for all $(N_1,T_1)\prec (N,T)$ one has $\Sigma_1^{T_1}(N_1)\leq\Sigma_1(\mathbb{A})$, then $(N,T)$ is one of $(\hat N, \hat T)$ and $\hat c(N,T) = 0$. Otherwise, there exists $(N_1,T_1)\prec (N,T)$ such that $\Sigma_1^{T_1}(N_1)>\Sigma_1(\mathbb{A})$ and we can repeat the process with $(N_1,T_1)$. The process stops only if we arrive at $(\hat N,\hat T)$ and it is guaranteed to stop after $c(N,T)$ steps.
The second assertion is proved in the same way. Indeed, in the previous paragraph we proved that if $\Sigma_1^T(N)>\Sigma_1(\mathbb{A})$, then $\hat c(N,T)$ is well-defined. 
\end{proof}
We can now finish the proof by induction on $\hat c(N,T)$ in the same way as before. Indeed, as we mentioned before, by construction surfaces with $\hat c(N,T)=0$ satisfy the conclusion of the theorem. Moreover, similarly to $c(N,T)$ one has $\hat c(N',T')<\hat c(N,T)$ as soon as $(N',T')\prec (N,T)$. And, finally, again by construction, surfaces for which $\hat c(N,T)$ is defined have more than $2$ boundary components, so that one can combine Lemma~\ref{prop:Suniq_map} with Theorem~\ref{thm:general_Sexistence} to prove the induction step.

%\begin{remark}
%The main reason we needed to modify the structure of the proof is that the current methods do not allow us to prove inequalty~\eqref{ineq:toprove_S} for the degeneration $N_{\tau,0,0,1,0}\prec N_{\tau, 1,0,0,0}$.
%\end{remark}

{\bf Case 3.} $(N',T')=N_\tau(D_k, f' + e'_1\rho_1 + e'_2\rho_2 + v'_{12}\rho_1\rho_2)$ with $k>2$. Then $(f',e_1',e_2',v_{12}') = (0,0,0,2)$ and the critical catenoid admits a $D_k$ action of the corresponding type, so that $\Sigma_1^{T'}(N') = \Sigma_1(\mathbb{A})$ is realized by the critical catenoid. The strategy of the proof is similar to the previous case in that we need to modify the definition of complexity to account for some degenerations where we can not directly show~\eqref{ineq:toprove_S}, for example $N_\tau(D_k,2\rho_1\rho_2)\prec N_\tau(D_k,\rho_1+\rho_1\rho_2)$.

Let $(N,T) = N_\tau(D_k, f + e_1\rho_1 + e_2\rho_2 + v_{12}\rho_1\rho_2)$, where $k>2$ and $f+e_1+e_2+v_{12}>2$. Then one has $N_\tau(D_k,2\rho_1\rho_2)\prec N_\tau(D_k,\rho_1+2\rho_1\rho_2)\preccurlyeq (N,T)$ or $N_\tau(D_k,\rho_2+ 2\rho_1\rho_2)\prec N_\tau(D_k,2\rho_1\rho_2)\preccurlyeq (N,T)$ and in both cases the collapsed set of the first degeneration can be chosen to have non-empty $P'^\iota$, so that by Theorem~\ref{thm:general_Sexistence}
\begin{equation*}
\label{ineq:cat_comp2}
\Sigma_1(\mathbb{A}) = \Sigma_1^{T'}(N')<\Sigma_1^T(N).
\end{equation*}
With \eqref{ineq:cat_comp2} established, the rest of the proof is exactly as in Case 2.

{\bf Case 4.} $(N',T') = N_\tau(\mathbb{Z}_2\times D_k,\ub')$. Assume for now that $k>2$, then $\ub' = \rho_2\rho_3$ and the critical catenoid admits a $\mathbb{Z}_2\times D_k$ action of the corresponding type, so that $\Sigma_1^{T'}(N') = \Sigma_1(\mathbb{A})$ is realized by the critical catenoid.

If $(N,T) = N_{\tau}(\mathbb{Z}_2\times D_k,\ub)$ with $k>2$, $|\ub|>1$, then one has $N_\tau(\mathbb{Z}_2\times D_k,\rho_2\rho_3)\prec N_\tau(\mathbb{Z}_2\times D_k,\rho_1\rho_2+\rho_2\rho_3)\preccurlyeq (N,T)$ or $N_\tau(\mathbb{Z}_2\times D_k,\rho_2\rho_3)\prec N_\tau(\mathbb{Z}_2\times D_k,\rho_1\rho_3+\rho_2\rho_3)\preccurlyeq (N,T)$. In both cases the collapsed set can be chosen to have non-empty $P'^\iota$, so that by Theorem~\ref{thm:general_Sexistence}
\begin{equation*}
\label{ineq:cat_comp2}
\Sigma_1(\mathbb{A}) = \Sigma_1^{T'}(N')<\Sigma_1^T(N).
\end{equation*}
The rest of the proof is the same as in Case 2.

If $k=2$ the only difference is that the $\rho_i$ are all interchangeable, so that one could also have $\ub'=\rho_1\rho_2$ or $\rho_2\rho_3$. However, it does not affect the remainder of the proof.

{\bf Case 5.} $(N',T') = N_\tau(\Gamma,\ub')$, where $\Gamma$ is a platonic group. In this case there are no values of $\ub'$ for which $N'$ has exactly two boundary components, so the induction on $c(N,T)$ is enough to conclude.

{\bf Case 6.} $(N',T') = N_{\rho_1}(\mathbb{Z}_2,f'+e'_1\rho_1)$, then either $e_1'=2$ or $e_1'=0$. The key new observation is the following lemma.

\begin{lemma}
One has
\begin{equation}
\label{ineq:gap_cond1}
\begin{split}
\Sigma_1(N_{\rho_1}(\mathbb{Z}_2,f+2\rho_1))&<\Sigma_1(N_{\rho_1}(\mathbb{Z}_2,(f+1)+\rho_1));\\
\Sigma_1(N_{\rho_1}(\mathbb{Z}_2,f))&<\Sigma_1(N_{\rho_1}(\mathbb{Z}_2,f+\rho_1)),
\end{split}
\end{equation}
and 
\begin{equation}
\label{ineq:gap_cond2}
\begin{split}
\Sigma_1(N_{\rho_1}(D_2,e_2\rho_2+2\rho_1\rho_2))&<\Sigma_1(N_{\rho_1}(D_2,(e_2+1)\rho_2 + \rho_1\rho_2));\\
\Sigma_1(N_{\rho_1}(D_2,e_2\rho_2))&<\Sigma_1(N_{\rho_1}(D_2,e_2\rho_2+\rho_1\rho_2)),
\end{split}
\end{equation}
provided that the quantities on the l.h.s. are achieved by a smooth metric.
\end{lemma}
\begin{proof}
%M: I PUT MY OLD PROOF HERE FOR MY PERSONAL CONVENIENCE. PETER'S PROOF SEEMS TO BE MORE CONCEPTUAL, MAYBE WE REPLACE IT IN THE LATER VERSION.
We provide the argument for the second inequality in~\eqref{ineq:gap_cond1}; the others are proved in the same way.
If the strict inequality fails, then by Proposition~\ref{prop:non_uniq} there is a $\bar \sigma_1^T$-maximal metric $g$ for $(N,T) = N_{\rho_1}(\mathbb{Z}_2,f)$ with a $\mathbb{Z}_2\times D_2$ action of type $N_{\rho_2}(\mathbb{Z}_2\times D_2,\lfloor f/2\rfloor\rho_2 + [f]_2\rho_2\rho_3)$, where $[f]_2$ is a residue of $f$ modulo $2$ and $D_2$ is generated by $\rho_2,\rho_3$, see Figure~\ref{FM0}.

\begin{figure}[h]
\centering
\def\ra{4}
\begin{tikzpicture}[scale=.7, rotate=90,transform shape] % Desingularization of disks with 3 symmetry planes
\draw [draw=black, fill=light-gray] (0, 0) ellipse ({\ra/sqrt(2)} and {\ra});
\draw[dashed] (0, -\ra)--(0, \ra); %vertical dashed line
\draw[dashed] (-{\ra/sqrt(2)}, 0)--({\ra/sqrt(2)}, 0); %horizontal dashed line

%\draw[] (0,-\ra) node[below]{}; %The plane labels
%\draw[] ({\ra/sqrt(2)}, 0) node[right]{};
%\draw[] (-{\ra/sqrt(2)}, 0) node[left]{\phantom{$P_3$}};

\foreach \j in {0, 1, 2}
{
\foreach \i in {-1, 1}
{
\fill [fill=white] (0, {\i*(\j*\ra/3)}) circle ({\ra/12}); %interior white circle.
\draw [draw=black, densely dashed] (0, {\i*(\j*\ra/3)}) circle ({\ra/12}); %outline of dotted boundary
}
}
\end{tikzpicture} 
\caption{A fundamental domain for the action of $\tau$ on $N = N_{\rho_1}(\Z_2,5)$.  Fixed-point sets are dashed, with $N^\tau, N^{\rho_2}$, and $N^{\rho_3}$ respectively consisting of the circular curves, horizontal segments, and the vertical segments. A $\Z_2\times D_2$-fundamental domain $U$ is one of the quarters.}
\label{FM0}
\end{figure}
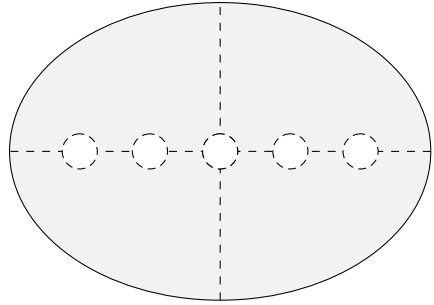

From the proof of Proposition~\ref{prop:non_uniq} one has that 
\begin{align}
\label{EM0}
\sigma_1(N, g) = \sigma_1^{++-}(N, g) = 
\sigma_1^{+-+}(N, g)
=
\sigma_1^{-++}(N, g),
\end{align}
where each $\sigma^{\pm \pm \pm}_k(N, g)$ denotes the $k$-th nonzero Steklov eigenvalues in the space of 
 even ($+$) or odd ($-$)  functions with respect to the generators $\tau, \rho_2, \rho_3$ in that order. %Note that here continue to adhere to the convention that trivial $0$ eigenvalue is denoted by $\sigma_0$, i.e. $\sigma_0(N,g) = \sigma_0^{+++}(N,g) = 0$ and enumeration of eigenvalues with other sign combination starts with $\sigma_1$.

To prove the inequality~\eqref{ineq:gap_cond1} we construct a metric $h$ on a surface $(\hat N, \hat T) = N_{\rho_1}(D_2,f\rho_2 + \rho_1\rho_2)$ such that $\bar\sigma_1(\hat N,h)>\bar\sigma_1(N,g)$. Since a surface of type $N_{\rho_1}(D_2,f\rho_2 + \rho_1\rho_2)$ is also a surface of type $N_{\rho_1}(\Z_2, f+\rho_1)$, this would complete the proof.

To each edge of the boundary $\partial U$ of fundamental domain $U$ of $N$ assign a \emph{label}, given by the name of the set, either $\partial N, N^{\tau}, N^{\rho_2}$, or $N^{\rho_3}$, containing it. This way, for example, $\sigma_1^{++-}$ corresponds to the first eigenvalue of the mixed Steklov problem on $U$ with Steklov condition on $\bd N$, Neumann conditions on $N^\tau, N^{\rho_2}$ and Dirichlet condition on $N^{\rho_3}$.
Now let $U^*$ be another copy of $U$, but with the labels on the $N^{\tau}$ and $N^{\rho_2}$ edges interchanged, and let $\Uhat$ be the quotient of $U\cup U^*$ obtained by identifying the points along the side with label $N^{\rho_3}$. Now define $\hat N$ to be the surface obtained from doubling $\Uhat$ along the sides with label $N^\tau$, and then doubling the result along the sides with label $N^{\rho_2}$, see Figure~\ref{Fbb}. 

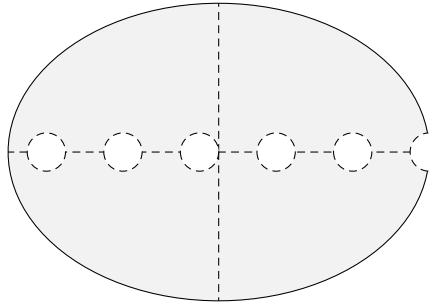
\begin{figure}[h]
\begin{tikzpicture}[scale=.7, rotate=180, transform shape]
\centering
\def\ra{4} %size for ellipse picture

   \fill [fill=light-gray, domain=-83:180+83, samples=100] plot ( {\ra*sin(\x)},{\ra/sqrt(2)*cos(\x)}); 
   \draw [black, domain=-83:180+83, samples=100] plot ( {\ra*sin(\x)},{\ra/sqrt(2)*cos(\x)}); 

%the nodes
	%\draw[] (0, {.5*\ra/sqrt(2)}) node[right]{$M^{\tau}$};
	%\draw[] (0, {-.5*\ra/sqrt(2)}) node[right]{$M^{\rho\tau \rho}$};
	%\draw[] (1.15, -{.1*\ra/sqrt(2)}) node[]{$M^{\omega}$};
	%\draw[] (1.15, -{.1*\ra/sqrt(2)}) node[]{$M^{\omega}$};
	%\draw[] (-1, {.6}) node[]{$M^{\rho}$};
	%\draw[] ({\ra*sin(-40)}, {\ra*cos(-40)-1.1}) node[below]{$\partial M$};

\draw[densely dashed, black] (-\ra+.1, 0)--(\ra, 0 ); %horizontal dashed line
\draw[densely dashed, black] (0, 0)--(0, {\ra/sqrt(2)}); %vertical L line
\draw[densely dashed, black] (0, 0)--(0, {-\ra/sqrt(2)}); %vertical L line

   \fill [fill=white, domain=3:177, samples=100] plot ( {-4+\ra/11*sin(\x)},{\ra/11*cos(\x)});
      \draw [black, densely dashed, domain=3:177, samples=100] plot ( {-4+\ra/11*sin(\x)},{\ra/11*cos(\x)});

	%most of the side circle. 

\foreach \j in {0, 1, 2}% The right handles
{
\fill [fill=white]({\j*4*\ra/11+\ra/11}, 0 ) circle ({\ra/11}); %filled white circle
\draw [black, densely dashed]({\j*4*\ra/11+\ra/11}, 0 ) circle ({\ra/11}); % the boundary

}
\foreach \j in {0, 1}% The left handles
{
\fill [fill=white] ( {-\j*4*\ra/11-3*\ra/11}, 0) circle ({\ra/11});
\draw [black, densely dashed] ( {-\j*4*\ra/11-3*\ra/11}, 0) circle ({\ra/11});

}
\end{tikzpicture}
\caption{A fundamental domain for $\tau$ on $\hat N$ with $f=5$. Pieces on the left are isomorphic to $U$, while pieces on the right are isomorphic to $U^*$, the vertical line indicates the line, where $U$ is attached to $U^*$.}
\label{Fbb}
\end{figure}

 Let $\tau,\rho_2$ be the obvious involutions on $\hat N$ arising from the successive doubling and we consider the action $\hat T$ of $D_2$ on $\hat N$ generated by these involutions. It is easy to see that $(\hat N,\hat T)$ is of type $N_{\rho_1}(D_2,f\rho_2+\rho_1\rho_2)$. The metric $h$ on $\hat N$ is defined to be the $D_2$-invariant metric that agrees with $g$ on $U$ and $U^*$.

It follows from Proposition~\ref{Lasymcl}.(ii) that $\sigma_1(\hat N,h)$-eigenfunction can not be $\tau$-odd  and $\rho_2$-odd, i.e. $\sigma_1(\hat N,h)<\sigma_1^{--}(\hat N,h)$. In particular,
\begin{equation}
\label{Eb1m}
\sigma_1(\hat N, h) \geq \min (\sigma^{+-}_1(\hat N, h),\sigma^{-+}_1(\hat N,h), \sigma^{++}_1(\hat N, h)).
\end{equation}

\emph{Claim 1:} $\sigma_1^{-+}(\hat N,h) = \sigma_1^{+-}( \hat N, h) > \sigma_1(N, g)$. The first equality follows from the fact that the reflection across the image of $N^{\rho_3}$ in $\Uhat$ (the vertical line on Figure~\ref{Fbb}) is an isometry of $\hat U$ that interchanges the labels for $N^{\tau},N^{\rho_2}$. Next, one has
\begin{equation}
\label{EM1}
\sigma_1^{+-}(\hat N,h)\geq \sigma_1^{+-+}(N, g) = \sigma_1^{-++}(N,g) = \sigma_1(N,g),
\end{equation}
where the first inequality follows from the classical bracketing argument, stating that introducing Neumann boundary condition along the image of $N^{\rho_3}$ can not increase the eigenvalue, as the Neumann condition widens the space of admissible test functions~\cite[\S 3.2.1]{LMP}, and the equalities follow from~\eqref{EM0}. Furthermore, if the inequality in~\eqref{EM1} is an equality, then the restrictions of $\sigma_1^{+-+}(N, g)$ and $\sigma_1^{-++}(N,g)$-eigenfunctions to $U$ glue together to the restriction of $\sigma_1^{+-}(\hat N,h)$-eigenfunction to $\Uhat$. In particular, the former two eigenfunctions agree on $N^{\rho_3}$. Since they both satisfy Neumann conditions on $N^{\rho_3}$, they have the same Dirichlet and Neumann data there, which contradicts unique continuation for harmonic functions. This proves Claim 1.

\emph{Claim 2:} $\sigma_1^{++}(\hat N,h) = \sigma_1^{++-}(N,g)$. As we noted in the proof of Claim 1, reflection across $N^{\rho_3}$ is an isometry of $\Uhat$ that interchanges the labels. Since in order to compute $\sigma_1^{++}(\hat N,h)$ both labels are assigned the same boundary conditions on $\Uhat$, one has that the corresponding eigenspace splits into $\rho_3$-odd and $\rho_3$-even eigenfunctions, i.e. $\sigma_1^{++}(\hat N,h) = \min\{\sigma_1^{++-}(N,g),\sigma_1^{+++}(N,g)\}$. At the same time, by~\eqref{EM0} $\sigma_1^{++-}(N,g) = \sigma_1(N,g)<\sigma_1^{+++}(N,g)$ and the claim is proved.

\emph{Claim 3:} There exists a $\tau,\rho_2$-invariant perturbation of $h$ with the property that $\sigmabar^{++}_1( \hat N, h) > \sigmabar_1(N, g)$.  If this were not the case, then by standard perturbation results \cite[Proposition 5.2]{FraserSchoen}, there would exist a collection of $\tau,\rho_2$-even eigenfunctions on $(\hat N, h)$ whose sum of squares is equal to $1$ on $\partial \hat N$.  This is impossible by Claim $2$ since any such eigenfunction vanishes on $\bd\hat N\cap N^{\rho_3}$.

By choosing the perturbation in Claim 3 to be small enough, we may ensure that the inequality from Claim 1 still holds.  Combining these facts with \eqref{Eb1m} completes the proof of the lemma. 
\end{proof}

The proof in Case 6 is concluded by observing that~\eqref{ineq:gap_cond1} are the only conditions left to check. Indeed, if $N_{\rho_1}(\mathbb{Z}_2,f'+2\rho_1)\prec N_{\rho_1}(\mathbb{Z}_2,f+e_1\rho_1)$ is an elementary deformation, then $(f,e_1) = (f'+1,1)$ and~\eqref{ineq:toprove_S} follows from~\eqref{ineq:gap_cond1}. If $N_{\rho_1}(\mathbb{Z}_2,f')\prec N_{\rho_1}(\mathbb{Z}_2,f+e_1\rho_1)$ is an elementary deformation, then $(f,e_1) = (a',1)$ and once again~\eqref{ineq:toprove_S} follows from~\eqref{ineq:gap_cond1}.

{\bf Case 7.} $(N',T') = N_{\rho_1}(\mathbb{Z}_2\times D_k,\ub')$, $k>2$. Since $N'$ has $2$ boundary components, one has $e_1' = v'_{12} = v'_{13}= 0$. Furthermore, using Proposition~\ref{prop:non_uniq} we can conclude that~\eqref{ineq:toprove_S} holds unless $(N',T')$ admits an additional $\mathbb{Z}_2$-symmetry. It is then easy to see that the only surface admitting an action of the required type is $N_{\rho_1}(\mathbb{Z}_2\times D_k,\rho_2\rho_3)$ in which case the critical catenoid is of this topological type, so that $\Sigma_1(N_{\rho_1}(\mathbb{Z}_2\times D_k,\rho_2\rho_3)) = \Sigma_1(\mathbb{A})$. We can not show~\eqref{ineq:toprove_S} for the degeneration $N_{\rho_1}(\mathbb{Z}_2\times D_k,\rho_2\rho_3)\prec N_{\rho_1}(\mathbb{Z}_2\times D_k,\rho_2)$, so we follow the strategy similar to Case 2.

First, we have the following lemma, which is an analogue of~\eqref{ineq:cat_comp1}.
\begin{lemma} One has the following
\begin{equation}
\label{ineq:gap_cond3}
\begin{split}
\Sigma_1(\mathbb{A})&<\Sigma_1(N_{\rho_1}(\mathbb{Z}_2\times D_k,\rho_2+\rho_1\rho_2)) =\Sigma_1(N_{\rho_1}(\mathbb{Z}_2\times D_k,\rho_3+\rho_1\rho_3));\\
\Sigma_1(\mathbb{A})&<\Sigma_1(N_{\rho_1}(\mathbb{Z}_2\times D_k,\rho_2+\rho_1\rho_3)) =\Sigma_1(N_{\rho_1}(\mathbb{Z}_2\times D_k,\rho_3+\rho_1\rho_2)).
\end{split}
\end{equation}
\end{lemma}
\begin{proof}
Both inequalities are proved in the same way; we provide details for the first one. Consider the degeneration $(N',T')=N_{\rho_1}(\mathbb{Z}_2\times D_k,\rho_2\rho_3)\prec N_{\rho_1}(\mathbb{Z}_2\times D_k,\rho_2+\rho_1\rho_2)$. Let $\Gamma = \mathbb{Z}_2\times D_k$, then the collapsed set $P'$ and its partition can be chosen as $P' = P_0'\sqcup \bigsqcup_{i=1}^kP_i'$, where $P_0'$ is the orbit $\Gamma p$ of a point $p\in N'^\tau\cap N'^{\rho_3}$ and $P_i'$ are formed by $\tau$-orbits of points in $N'^{\rho_2}\cap N'^{\rho_1}$. 

%Picture?

The required inequality follows from the conclusion of Theorem~\ref{thm:general_Sexistence}, if we can show that condition 2) is satisfied. Assume the contrary, i.e. there exists a free boundary harmonic map from the catenoid to $\mathbb{B}^3$ by $\sigma_1$-eigenfunctions such that each of $P_0',P_i'$ are sent to a single point.
By Remark~\ref{rmk:quadric_catenoid} any such map has the form $(ax,ay,bz)$, where $z$ is a $\tau$-odd eigenfunction whereas $x,y$ are $\tau$-even. Since $z$ never vanishes on $\bd N'$ and, in particular, has opposite signs on the two elements of $P_i'$, we must have $b=0$. At the same time, the map $(x,y)$ assumes different values on elements of $P_0'$, a contradiction.
\end{proof}

The rest of the proof proceeds as in Case 2: we use the same definition of complexity $\hat c(N,T)$ and induct on $\hat c$. To be precise, we consider two possibilities for $(N,T) = N_{\rho_1}(\mathbb{Z}_2\times D_k,\ub)$. If $f+e_2+e_3+v_{23} = 0$, then $N_{\rho_1}(\mathbb{Z}_2\times D_k,\rho_2\rho_3)\not\prec (N,T)$, hence, the existence in this case follows from the induction on $c(N,T)$. If $f+e_2+e_3+v_{23} >0$, then conditions on $\ub$ imposed in the formulation of the theorem imply that $(\hat N,\hat T)\prec (N,T)$, where $(\hat N,\hat T)$ is one of the surfaces in the r.h.s. of~\eqref{ineq:gap_cond3}. Then~\eqref{ineq:NN'} together with~\eqref{ineq:gap_cond3} imply that $\Sigma_1^T(N)>\Sigma_1(\mathbb{A})$ and, hence, $\hat c(N,T)$ is defined by Lemma~\ref{lem:hatc_1}.

{\bf Case 8.} $(N',T') = N_{\rho_1}(D_2,f'+e'_1\rho_1+e_2'\rho_2 + v'_{12}\rho_1\rho_2)$, then $(e_1',v_{12}') = (1,0), (0,0)$ or $(0,2)$. Furthermore, using Proposition~\ref{prop:non_uniq} we can conclude that~\eqref{ineq:toprove_S} holds unless $(N',T')$ satisfies an additional $\mathbb{Z}_2$-symmetry. The only surfaces that satisfy the additional symmetry of the required type are $N_{\rho_1}(D_2,e_2'\rho_2 + 2\rho_1\rho_2)$ or $N_{\rho_1}(D_2,f'+e'_1\rho_1+e_2'\rho_2)$, where in the latter case $e_1',e_2'\leq 1$. Let us refer to such surfaces as {\em bad}. Similarly to case 2, to complete the proof we modify the induction procedure. An additional challenge is that there is a whole countable family of surfaces for which we can not show inequality~\eqref{ineq:toprove_S} directly, so the arguments required are more involved. 

Consider the quantity
\[
B^T(N):=\max_{\substack{(N',T')\prec (N,T),\\ \text{$(N',T')$ is bad}}}\Sigma_1^{T'}(N').
\]
It is easy to see that $B^T(N)$ is defined for all $(N,T)$, except for $N_{\rho_1}(D_2,\rho_1\rho_2)$, which is diffeomorphic to a disk and is, thus, covered by the Remark~\ref{rmk:disk}. Set $B(N_{\rho_1}(D_2,\rho_1\rho_2)) = 0$ for convenience. 
It is easy to see that $(N',T')\prec(N,T)$ implies $B^{T'}(N')\leq B^T(N)$. 
%Furthermore, by~\eqref{ineq:NN'} one has $\Sigma_1^T(N)\geq B^T(N)$ and, furthermore, for bad surfaces $\Sigma_1^T(N) = B^T(N)$. 
We call $(N,T)$ {\em initial} if $\Sigma^T_1(N)>B^T(N)$, but for all $(N',T')\prec (N,T)$ one has $\Sigma_1^{T'}(N') = B^{T'}(N')$. Note that with this definition $N_{\rho_1}(D_2,\rho_1\rho_2)$ is initial. Furthermore, if $(N,T)$ is initial, then for all $(N',T')\prec (N,T)$ one has 
\[
\Sigma_1^T(N)>B^T(N)\geq B^{T'}(N') = \Sigma_1^{T'}(N'),
\] 
hence, by Theorem~\ref{thm:Sexistence}, there exists $\bar\sigma_1^T(N)$-maximal metric. The idea is to use initial surfaces as a base of induction.

\begin{lemma}
\label{lem:hatc_2}
If $\Sigma^T_1(N)>B^T(N)$, then $\Sigma^T_1(N)$ is achieved by a smooth metric.
\end{lemma}
\begin{proof}
Define the induction parameter $\hat c(N,T)$ in the following way. It is $0$ for initial surfaces, and $c(N,T) = k>0$ if $k$ is the maximal number for which there exists a chain $(N_0,T_0)\prec (N_1,T_1)\prec \ldots \prec (N_{k-1},T_{k-1})\prec (N,T)$, where $(N_0,T_0)$ is initial and $\Sigma^{T_i}_1(N_i)>B^{T_i}(N)$.

Similarly to the proof of Lemma~\ref{lem:hatc_1} we observe that $\hat c(N,T)$ is defined. Indeed, if $\Sigma^T_1(N)>B^T(N)$, then either $(N,T)$ is initial (so that $\hat c(N,T)=0$) or there exists $(N',T')\prec (N,T)$ such that $\Sigma_1^{T'}(N')>B^{T'}(N')$. Repeat the same argument with $(N',T')$. This process has to terminate and it can only terminate at an initial surface.

We proceed by induction on $\hat c(N,T)$. We observed above that for surfaces with $\hat c(N,T)=0$, i.e. initial surfaces, the conclusion of the lemma is satisfied. Suppose that we have proved the lemma for all $(N',T')$ with $\Sigma^{T'}_1(N')>B^{T'}(N')$ and  $\hat c(N',T')\leq k$. Consider $(N,T)$ with $\Sigma^T_1(N)>B^T(N)$ and $\hat c(N,T) = k+1$ and let $(N'',T'')\prec (N,T)$. We show that $\Sigma_1^{T''}(N'')<\Sigma^T_1(N)$, then the conclusion follows from Theorem~\ref{thm:Sexistence}.

If $\Sigma_1(N'',T'') = B^{T''}(N'')$, then one has
\[
\Sigma_1^T(N)>B^T(N)\geq B^{T''}(N'') = \Sigma_1^{T''}(N'').
\] 
If $\Sigma_1(N'',T'')>B^{T''}(N'')$ and $(N'',T'')$ is bad, then by definition of $B^T(N)$ one has $B^T(N)\geq \Sigma_1(N'',T'')$ and, therefore,
\[
\Sigma_1^T(N)>B^T(N)\geq \Sigma_1(N'',T'').
\]
Finally, if $\Sigma_1(N'',T'')>B^{T''}(N'')$ and $(N'',T'')$ is not bad, then $\hat c(N'',T'')\leq k$, so by the induction hypothesis $\Sigma_1(N'',T'')$ is achieved by a smooth metric. Hence, by Theorem~\ref{thm:general_Sexistence} and Proposition~\ref{prop:non_uniq} one has  $\Sigma_1^{T''}(N'')<\Sigma^T_1(N)$.
\end{proof}

The rest of the proof is devoted to proving $\Sigma_1^T(N)>B^T(N)$ for surfaces $(N,T)$ described in the formulation of the theorem.

\begin{lemma}
\label{lem:ooc}
Let $(N,T) = N_{\rho_1}(D_2,e_2\rho_2+v_{12}\rho_1\rho_2)$. Then $\Sigma_1^T(N)>B^T(N)$ and, in particular, inequalities~\eqref{ineq:gap_cond2} hold unconditionally. Furthermore, one has
\begin{equation}
\label{ineq:gap_cond4}
\Sigma_1(N_{\rho_1}(D_2, e_2\rho_2 + 2\rho_1\rho_2))<\Sigma_1(N_{\rho_1}(D_2, \rho_1 + e_2\rho_2 + \rho_1\rho_2))
\end{equation}
\end{lemma}
\begin{proof}
Again, we use induction on $\hat c(N,T)$. The main observation is that if $(N',T')\prec (N,T)$, then $(N',T')$ is in the same class as $(N,T)$, i.e. $(N',T') = N_{\rho_1}(D_2,e_2'\rho_2+v_{12}'\rho_1\rho_2)$. Thus, to complete the induction argument it is sufficient to prove that $\Sigma_1^{T'}(N')<\Sigma_1^T(N)$ provided that $(N',T')$ is bad and $\Sigma_1^{T'}(N')$ is achieved. But this easily follows from~\eqref{ineq:gap_cond2}.

To show inequality~\eqref{ineq:gap_cond4} observe that the first part of the lemma implies that $\Sigma_1(N_{\rho_1}(D_2, e_2\rho_2 + 2\rho_1\rho_2))$ is achieved. Furthermore, the collapsed set $P'$ for the degeneration $N_{\rho_1}(D_2, e_2\rho_2 + 2\rho_1\rho_2)\prec N_{\rho_1}(D_2, \rho_1 + e_2\rho_2 + \rho_1\rho_2)$ can be chosen to be $P'=P'^\iota=\{q\}$, where $q\in N^\tau\cap N^\rho_2$, so that~\eqref{ineq:gap_cond4} follows from Theorem~\ref{thm:general_Sexistence}. 
\end{proof}

\begin{lemma}
\label{lem:gap_cond5}
Let $(N',T') = N_{\rho_1}(D_2, f'+e_1'\rho_1 + e_2'\rho_2)$, where $e_1',e_2'\leq 1$, and suppose that $\Sigma_1^{T'}(N')$ is achieved. Then either $\Sigma_1^{T'}(N')<\Sigma_1^T(N)$ for all $(N',T')\prec (N,T)$ or  one has
\[
\Sigma_1^{T'}(N')\leq \Sigma_1(N_{\rho_1}(D_2, (2f'+e_1')\rho_1 + 2e_2'\rho_2)).
\]
\end{lemma}
\begin{proof}
Let $\Sigma_1^{T'}(N') = \bar\sigma_1(N',g')$ for some $T'$-invariant metric $g'$. Then by Proposition~\ref{prop:non_uniq} either there is a unique free boundary harmonic map to a ball by $\sigma_1(N',g')$-eigenfunctions, and then Lemma~\ref{lemma:Scollapsed_sets} and Theorem~\ref{thm:general_Sexistence} give $\Sigma_1^{T'}(N')<\Sigma_1^T(N)$ or $g'$ is invariant under an additional reflection $\rho_3$. Under the action $T''$ of the group generated by $\tau$ and $\rho_3$ (which is isomorphic $D_2$) one has that $(N',T'')$ is of type $N_{\rho_1}(D_2, (2f'+e_1')\rho_1 + 2e_2'\rho_2)$ and since $g'$ is $T''$-invariant, the inequality follows.
\end{proof}

\begin{lemma}
\label{lem:Sigma_B}
Let $(N,T) = N_{\rho_1}(D_2, f + e_1\rho_1 + e_2\rho_2 + v_{12}\rho_1\rho_2)$ such that $e_1\geq 1$ and $f<e_2-1$. Then $\Sigma^T_1(N)>B^T(N)$.
\end{lemma}
\begin{proof}
Let $N_{\rho_1}(D_2,\ub') = (N',T')\prec (N,T)$ be a bad surface, we need to show $\Sigma_1^{T'}(N')<\Sigma_1^T(N)$. We may assume that $\Sigma_1^{T'}(N')>B^{T'}(N')$ and, hence, $\Sigma_1^{T'}(N')$ is achieved by a smooth metric. Indeed, otherwise let $(N'',T'')\prec (N',T')$ be the smallest (in the sense of ordering $\prec$) bad surface satisfying $\Sigma_1^{T''}(N'') = \Sigma_1^{T'}(N')$. Then for all bad surfaces $(N''',T''')\prec (N'',T'')$ one has $\Sigma_1^{T'''}(N''')<\Sigma_1^{T'}(N') = \Sigma_1^{T''}(N'')$ and, hence, $\Sigma_1^{T''}(N'')> B^{T''}(N'')$, so that 
$\Sigma_1^{T''}(N'')$ is achieved. Therefore, in order to prove $\Sigma_1^{T'}(N')<\Sigma_1^T(N)$ it is sufficient to prove  
$\Sigma_1^{T''}(N'')<\Sigma_1^T(N)$, where $\Sigma_1^{T''}(N'')$ is achieved.

Assume first that $f'\ne 0$, so that $\ub' = f'+e_1'\rho_1+e_2'\rho_2$, where $e_1',e_2'\leq 1$. Then by Lemma~\ref{lem:gap_cond5} either  
$\Sigma_1^{T'}(N')<\Sigma_1^T(N)$ or 
$\Sigma_1^{T'}(N') \leq \Sigma_1(N_{\rho_1}(D_2,(2f'+e_2')\rho_2+2e'_1\rho_1 \rho_2))$. If $e_1' = 0$, then one has
\begin{equation*}
\begin{split}
&N_{\rho_1}(D_2,(2f'+e_2')\rho_2)\prec N_{\rho_1}(D_2,(2f'+e_2')\rho_2+\rho_1\rho_2)\prec \\ 
&N_{\rho_1}(D_2,f' + (f'+e_2'+1)\rho_2)\preccurlyeq (N,T),
\end{split}
\end{equation*}
where the last degeneration follows from the fact that $f'\leq f < e_2-1$ and, hence, $f'+e_2'+1\leq e_2$. Since we can apply~\eqref{ineq:gap_cond2} to the first degeneration in the chain, it follows from Lemma~\ref{lem:ooc} that $\Sigma_1^{T'}(N')<\Sigma_1^T(N)$. If $e_1'=1$, then one has
 \begin{equation*}
\begin{split}
&N_{\rho_1}(D_2,(2f'+e_2')\rho_2+2\rho_1\rho_2)\prec N_{\rho_1}(D_2,(2f'+e_2'+1)\rho_2+\rho_1\rho_2)\prec \\ 
& N_{\rho_1}(D_2,f' + \rho_1 + (f'+e_2'+1)\rho_2)\preccurlyeq (N,T),
\end{split}
\end{equation*}
where the last degeneration follows from $e_1\geq 1$ and $f'\leq f, f'+e_2'+1\leq e_2$. We once again apply~\eqref{ineq:gap_cond2} to the first degeneration in the chain to conclude $\Sigma_1^{T'}(N')<\Sigma_1^T(N)$.

Suppose now that $f'=0$, therefore, $e_1'=0$. If $v_{12}'=0$, then 
\[
(N',T') = N_{\rho_1}(D_2, e_2'\rho_2)\prec N_{\rho_1}(D_2, e_2'\rho_2 + \rho_1\rho_2)\prec N_{\rho_1}(D_2, \rho_1+ e_2'\rho_2)\preccurlyeq (N,T),
\]
where the last degeneration follows from $e_1\geq 1$.
Applying~\eqref{ineq:gap_cond2} to the first degeneration yields $\Sigma_1^{T'}(N')<\Sigma_1^T(N)$. Finally, if $v_{12}'=2$, then 
\[
(N',T') = N_{\rho_1}(D_2, e_2'\rho_2 + 2\rho_1\rho_2)\prec N_{\rho_1}(D_2, \rho_1 + e_2'\rho_2 + \rho_1\rho_2)\preccurlyeq (N,T),
\]
where the last degeneration follows from $e_1\geq 1$. Applying~\eqref{ineq:gap_cond4} to the first degeneration yields $\Sigma_1^{T'}(N')<\Sigma_1^T(N)$ and completes the proof.
\end{proof}

%\begin{lemma}
%Let $(N,T) = N_{\rho_1}(D_n\times \mathbb{Z}_2,\ub)$, $n>2$. If $n$ is suffi then $\Sigma_1^T(N)>\Sigma_1(\mathbb{A})$.
%\end{lemma}
%\begin{proof}
%
%\end{proof}
\end{proof}

\begin{proposition}
\label{Pzeronum}
Let $b = 2f+e_1 = 2f'+e_1'$ and $g,g'$ be metrics achieving $\Sigma_1(N_{\tau}(\mathbb{Z}_2, f+e_1\rho_1))$ and $\Sigma_1(N_{\tau}(\mathbb{Z}_2, f'+e'_1\rho_1))$ respectively. If $(f',e_1')\ne (f',e_1')$ and $b>2(f+f'+1)$, then $g$ and $g'$ are not isometric.

In particular, there exists at least $\lfloor\frac{b-2}{4}\rfloor+1$ embedded non-isometric free boundary minimal surfaces of genus $0$ with $b$ boundary components and area below $4\pi$ in $\mathbb{B}^3$.
\end{proposition}
\begin{proof}
The proof is analogous to the proof of Proposition~\ref{prop:distinct_MS}. The area bound follows from~\cite[Proposition 8.1]{FraserSchoen}.
\end{proof}

\begin{proposition}
\label{Pnumbd}
Let $e_1\geq 1$, $v_{12} = 0,1$, $\gamma = 2f+e_2 = 2f'+e'_2$
 and $g$, $g'$ be metrics achieving $\Sigma_1(N_{\rho_1}(\mathbb{Z}_2\times \mathbb{Z}_2, f + e_1\rho_1 + e_2\rho_2 + v_{12}\rho_1\rho_2))$, $\Sigma_1(N_{\rho_1}(\mathbb{Z}_2\times \mathbb{Z}_2, f' + e_1\rho_1 + e'_2\rho_2 + v_{12}\rho_1\rho_2))$ respectively. If $(f',e_2')\ne (f',e_2')$ and $\gamma>2(f+f'+1)$, then $g$ and $g'$ are not isometric.

In particular, if $b\geq 2$ there exists at least $\lfloor\frac{\gamma-2}{4}\rfloor+1$ embedded non-isometric free boundary minimal surfaces of genus $\gamma$ with $b$ boundary components and area below $2\pi$ in $\mathbb{B}^3$.
\end{proposition}
\begin{proof}
The proof is analogous to the proof of Proposition~\ref{prop:distinct_MS}. The main difference is that we only know existence of the corresponding maximal metrics for $e_1\geq 1$, $f<e_2-1$. Since in order to prove the lower bound on the number of minimal surfaces we only consider $f\leq\lfloor\frac{\gamma-2}{2}\rfloor$, the latter condition only causes problems when $e_2\leq 1$ (as $f$ can not be negative), but then $f=0$, which means that $\gamma=0,1$. The lower bound then asserts that there is at least one surface with such topology, which is known from Theorem~\ref{thm:group_Sexistence}.
 The number $b$ of boundary components equals $2e_1 + v_{12}$, so that for any $b\geq 2$ there exist $v_{12}$ and $e_1\geq 2$ realising the required number of boundary components. The area bound follows from Lemma~\ref{Lgraphbd}.
\end{proof}

\subsection{Proof of Theorem~\ref{thm:general_Sexistence}}

Let $(N',T')\prec (N,T)$ be the topological degeneration with collapsed set $P'$ and the partition $P'=\sqcup P_i'$. 
Below we use the notation introduced in Section~\ref{sec:top_degen_boundary} and Section~\ref{sec:proof_general_existence}, namely for the degeneration $(\wt N',\wt T')\prec (\wt N,\wt T)$ we have
%
%
%Then there is a corresponding degeneration of the doubles $(\wt N',\wt T')\prec (\wt N,\wt T)$ and the partition $\wt P' = \sqcup \wt P_j'$. Then one can use the inverse surgery construction to recover $(\wt N,\wt T)$ from $(\wt N',\wt T')$, where $(\wt N,\wt T)$ (or $(\wt N',\wt T')$) have a separating involution $\iota$ such that $(N,T)$ (or $(N',T')$) is identified with the closure of the connected component of $\wt N\setminus \wt N^\iota$ (or $\wt N'\setminus \wt N'^\iota$).
\[
D_\eps(\wt P'):=\bigcup\limits_{p\in \wt P'} D_\eps(p);\quad D_\eps(\wt P''):=\bigcup\limits_{p\in \wt P''} D_\eps(p);\quad \wt C_{\eps,L}:=\bigcup\limits_{p\in \wt P'} C_{\eps,L}(p),
\]
and for the original degeneration $(N',T')\prec (N,T)$,
\[
D_\eps(P') = D_\eps(\wt P')/ \wt T'(\iota);\quad D_\eps(P'') = D_\eps(\wt P'')/ \wt T''(\iota);\quad C_{\eps,L} = \wt C_{\eps,L}/\wt T(\iota).
\]

The first steps of the proof below follow closely the exposition in Section~\ref{sec:proof_general_existence} of the proof of Theorem~\ref{thm:general_existence}.

\begin{lemma}
\label{lem:good_Smap}
Given $\eps,L\in (0,\infty),$ there exists $n\geq 2$ and a $T$-equivariant map $F\in W^{1,2}_{g_{\eps,L}}(N, \mathbb{B}^{n+1})$ such that 
\begin{equation}
\label{ineq:gSm1}
\int_N |dF|^2_{g_{\eps,L}}\,dv_{g_{\eps,L}}<\Sigma^T_1(N,\mC_{\eps,L}) + \eps^2,\quad
\int_{\partial N} (1-|F|^2)^2\,ds_{g_{\eps,L}}<\eps^2.
\end{equation}
Furthermore, the harmonic extension $\hat F\in W^{1,2}_{g'}(N', \mathbb{B}^{n+1})$ of $F$ restricted to $N'\setminus D_\eps(P')\subset N'$ is a $T'$-equivariant map satisfying
\begin{equation}
\label{ineq:gSm2}
\Sigma_1^{T'}(N',\mC')\leq \int_{N'}|d\hat F|^2_{g'}\,dv_{g'} + C\eps^\frac{1}{2},\quad 
\int_{\partial N'} (1-|\hat F|^2)^2\,dv_{g'}<C\eps,
\end{equation}
\begin{equation}
\label{ineq:gSm3}
\int_{\partial N'}\hat F\,ds_{g'} = 0.
\end{equation}
\end{lemma}
\begin{proof}
The proof is similar to the proof of Lemma~\ref{lem:good_map}.

By Theorem \ref{stek.mm.char}, there exists $m\in \mathbb{N}$ such that
$$\Sigma_1^T(N,\mC_{\eps,L})=2\mathcal{F}_m^T(N,\mC_{\eps,L})=2\sup_{\delta>0}\mathcal{F}_{m,\delta}^T(N,g_{\eps,L}).$$
Thus, for any $\delta>0$, there exists a family $(F_y)_{y\in \mathbb{B}^{m|\Gamma|}}\in \mathcal{B}_m$ for which
$$\sup_{y\in \mathbb{B}^{m|\Gamma|}}\int_N|dF_y|^2_{g_{\epsilon,L}}dv_{g_{\epsilon,L}}\leq \Sigma_1^G(\mC_{\epsilon,L})+\delta$$
and
\begin{equation}
\label{ineq:SalmostSn}
\int_{\partial N} (1-|F_y|^2)^2\,dv_{g_{\eps,L}}<\delta
\end{equation}
for all $y\in \mathbb{B}^{m|\Gamma|}$. Set $\delta=\epsilon^2$, and fix such a family $(F_y)\in \mathcal{B}_m$; as in the closed case, we may also assume that $|F_y|\leq 1$ pointwise.

Next, denote by $f\mapsto \hat{f}$ the assignment $W^{1,2}_{g_{\eps,L}}(N)\to W^{1,2}_{g'}(N')$ given by: 1) extending $f$ to a symmetric function $\wt{f}\in W^{1,2}_{\wt{g}_{\eps,L}}(\wt{N})$; 2) restricting $\wt{f}$ to $\wt{N'}\setminus\cup_{p\in \wt{P'}} D_\eps(p)$; 3) extending it harmonically into $\cup_{p\in \wt{P'}} D_\eps(p)$; 4) and finally restricting it back to $W^{1,2}_{g'}(N')$. Then, defining $\beta(f):=\int_{\partial N'}\hat{f}$, it is straightforward to check that $\beta$ defines a bounded linear functional on $W^{1,2}(N,g_{\epsilon,L})$. Applying Remark \ref{gen.hersch} to $\beta$, we deduce that there exists $y\in \mathbb{B}^{m|\Gamma|}$ such that
\begin{equation}
\label{ineq:Shat_balance}
\int_{\partial N'} \hat F_y\,ds_{g'} = 0.
\end{equation}
Moreover, $|F_y|\leq 1$ implies $|\hat F_y|\leq 1$, so that by~\eqref{ineq:SalmostSn}
$$
\int_{\partial N'} (1-|\hat F_y|^2)^2\,ds_{g'}< C\eps.
$$
Note the difference from the Laplace case in the last inequality. 
Combining this with~\eqref{ineq:Shat_balance} and estimates identical to those in the proof of Lemma \ref{fb.lbd} gives
$$
\Sigma^{T'}_1(N',\mC') = \bar\sigma_1(N',g')\leq \int_{N'}|d\hat F|_{g'}^2\,dv_{g'} + C\eps^\frac{1}{2}.
$$ 
Thus, the map $F_y$ satisfies all the conclusions of the lemma and the proof is complete, with $n+1=m|\Gamma|$.
\end{proof}

Note that the inequalities in Lemma~\ref{lem:good_Smap} imply that 
\begin{equation}
\label{ineq:Sfhllb}
\begin{split}
\int_{\bd N'}|\hat F|^2\,ds_{g'} &= 1-\int_{\bd N'}(1-|\hat F|^2)\,ds_{g'}\\
&\geq 1 - \left(\int_{\bd N'}(1-|\hat F|^2)^2\,ds_{g'}\right)^{1/2}\geq 1 - C\eps^\frac{1}{2}.
\end{split}
\end{equation}

Let us introduce the notation 
\[
\Sigma_1' = \Sigma_1^{T'}(N',\mC');\qquad \Sigma_1 = \Sigma_1^{T}(N,\mC_{\eps,L}),
\]
where we continue to omit the dependence on $\eps, L$.
For all points $p\in P'$, Lemma~\ref{lem:ext_bound} implies that
\[
\int_{D_\eps(p)} |d\hat F|^2_{g'}\,dv_{g'}\leq (1 + Ce^{-2L})\int_{C_{\eps,L}(p)}|dF|^2.
\]
Denoting $C_{\eps,L}:=\cup_p C_{\eps,L}(p)$ and summing over $p\in P$ we arrive at
\begin{equation}
\label{ineq:Sfhueb}
\begin{split}
\int_{N'}|d\hat{F}|_{g'}^2\,dv_{g'}&\leq \int_{(N'\setminus D_\eps(P'))\cup C_{\eps,L}}|dF|^2\,dv_{g_{\eps,L}}+Ce^{-2L}\int_{C_{\eps,L}}|dF|^2\\
&\leq \Sigma_1+Ce^{-2L}\int_{C_{\epsilon,L}}|dF|^2 +\eps^2
\end{split}
\end{equation}
where we used Lemma~\ref{lem:good_Smap} in the last inequality.

Let $Q$ denote the  nonnegative-definite quadratic form on $W^{1,2}(N', \mathbb{R}^{n+1})$ given by
\[
Q(\Psi,\Psi):=\int_{N'}|d\Psi|_{g'}^2\,dv_{g'}-\Sigma_1'\int_{\partial N'}\left|\Psi-\int_{\partial N'}\Psi\,ds_{g'}\right|^2\,ds_{g'},
\]
where we recall that $g'$ is a unit-area metric such that $\sigma_1(N',g')=\Sigma_1'$.
Applying Cauchy-Schwarz inequality to the corresponding bilinear form, 
we observe that for any $\Psi\in W^{1,2}(N', \mathbb{R}^{n+1})$ satisfying $\int_{\bd N'}\Psi\,ds_{g'} = 0$ one has
\begin{equation}
\label{ineq:SCS1}
 \int_{N'}\langle d\hat{F},d\Psi\rangle_{g'}\,dv_{g'}-\Sigma_1'\int_{\partial N'}\langle \hat{F},\Psi\rangle\,ds_{g'}= Q(\hat F,\Psi) \leq \sqrt{Q(\hat{F},\hat{F})}\sqrt{Q(\Psi,\Psi)},
\end{equation}
where we used~\eqref{ineq:gSm3} in the first equality.

In particular, it follows from~\eqref{ineq:Sfhueb} and~\eqref{ineq:Sfhllb} that
\[
Q(\hat{F},\hat{F})\leq (\Sigma_1-\Sigma_1') + Ce^{-2L}\|dF\|_{L^2(C_{\eps,L})}^2+C\eps^\frac{1}{2},
\]
and the Cauchy-Schwarz estimate~\eqref{ineq:SCS1} gives
\begin{equation}
\label{ineq:SCS2}
\int_{N'}\langle d\hat{F},d\Psi\rangle_{g'}\,dv_{g'}-\Sigma_1'\int_{\bd N'}\langle \hat{F},\Psi\rangle\,ds_{g'}\leq C((\Sigma_1-\Sigma_1')^{\frac{1}{2}}+e^{-L}\|dF\|_{L^2(C_{\eps,L})}+\eps^{\frac{1}{4}})\sqrt{Q(\Psi,\Psi)}
\end{equation}
for all $\Psi$ satisfying $\int_{\bd N'}\Psi\,ds_{g'} = 0$.

\begin{lemma}
\label{lem:SFeb_cyl}
One has 
\[
\|dF\|^2_{L^2(C_{\eps,L})}\leq \Sigma_1-\Sigma_1' + \frac{C}{|\log\eps|}.
\]
\end{lemma}
\begin{proof}
The proof is very similar to Lemma~\ref{lem:Feb_cyl}; one essentially applies the same arguments to the symmetric function $\wt{F}$ on the double $\wt{N}$.

Let $\rho_{\eps,p}\in W^{1,2}(N')$ be a logarithmic cut-off function given in polar coordinates centered at $p$ with respect to $g_p$ as
\[
\rho_{\eps,p} (r) = 
\begin{cases}
1, &\text{ if } r\leq\eps;\\
\frac{\log (r/\eps^{1/2})}{\log(\eps^{1/2})}, &\text{ if }\eps^{\frac{1}{2}}\geq r\geq\eps;\\
0, &\text{ if } r\geq\eps^{\frac{1}{2}}.\\
\end{cases}
\]
Choose $\eps$ small enough so that supports of $\rho_{\eps,p}$ are disjoint for $p\in P'$ and define $\psi_\eps:= 1 - \sum_{p\in P}\rho_{\eps,p}$. Since the Dirichlet integral is conformally invariant, one has
\[
\int_{N'}|d\psi_\eps|^2_{g'}\,dv_{g'}\leq \frac{C}{|\log\eps|}.
\]

Next, observe that
\begin{eqnarray*}
\left|\int_{\bd N'}\psi_{\eps}F\,ds_{g'}\right|&=&\left|\int_{\bd N'}\hat{F}\,ds_{g'}+\int_{\bd N'}(\psi_{\eps}-1)\hat{F}\,ds_{g'}\right|\\
&=&\left|\int_{\bd N'}(\psi_{\eps}-1)\hat{F}\,ds_{g'}\right|
\leq \length_{g'}(\mathrm{supp}(\psi_\eps-1)|_{\bd N'})\leq C\eps^\frac{1}{2},
\end{eqnarray*}
so that
\begin{equation}
\label{ineq:SfllbM'}
\int_{N'}|d(\psi_{\eps}F)|_{g'}^2\,dv_{g'}\geq \Sigma_1'\left(\int_{\bd N'}\psi_{\eps}^2|F|^2\,ds_{g'}-C\eps^\frac{1}{2}\right)\geq \Sigma_1'-C'\eps^\frac{1}{2},
\end{equation}
where we used~\eqref{ineq:Sfhllb} in the last inequality. 

On the other hand, by conformal invariance of the Dirichlet integral, on $A_\eps(p):= D_{\sqrt{\eps}}(p)\setminus D_{\eps}(p) =\mathrm{supp}(d\rho_{\eps,p})$ one has
\begin{equation*}
\begin{split}
&\int_{A_\eps(p)}\langle d(|\psi_{\eps}|^2),d(|F|^2)\rangle_{g'}\,dv_{g'} = \int_{A_\eps(p)} |F|^2 \Delta_{g'}
(|\psi_{\eps}|^2)\,dv_{g'} - \\
2&\int_{S_{\sqrt{\eps}}(p)}\frac{1}{r\log(\eps^{1/2})}\left( 1- \frac{\log (r/\eps^{1/2})}{\log(\eps^{1/2})}\right)|F|^2d\theta +\\ 2&\int_{\partial D_{\eps}(p)}\frac{1}{r\log(\eps^{1/2})}\left( 1- \frac{\log (r/\eps^{1/2})}{\log(\eps^{1/2})}\right)|F|^2d\theta = \\
-2&\int_{A_{\eps}(p)}|F|^2|d\psi_{\eps}|^2_{g'}\,dv_{g'} + \frac{2}{\eps^{1/2}\log(\eps^{-1/2})}\int_{S_{\sqrt{\eps}}(p)} |F|^2\,d\theta\leq \frac{C}{|\log\eps|},
\end{split}
\end{equation*}
where we used $|F|\leq 1$ in the last step. Note that for $p\in \bd N'$ we additionally used that $\psi_\eps$ is a radial function, so that the normal derivative of $|\psi_\eps|^2$ vanishes along $\bd N'$. 
Therefore, the following inequality holds
\begin{equation}
\begin{split}
\int_{N'}|d(\psi_{\eps}F)|_{g'}^2\,dv_{g'}&=\int_{N'} \psi_{\eps}^2|dF|_{g'}^2+\frac{1}{2}\langle d(|\psi_{\eps}|^2),d(|F|^2)\rangle_{g'} + |d\psi_{\eps}|_{g'}^2|F|^2\,dv_{g'}\leq \\
&\leq\int_{N'\setminus D_{\eps}(P')} |dF|^2_{g'}\,dv_{g'}+\frac{C}{|\log\eps|}.
\end{split}
\end{equation}
Together with~\eqref{ineq:SfllbM'} this implies  
\[
\int_{N'\setminus D_{\eps}(P')}|dF|^2_{g_{\eps,L}}\,dv_{g_{\eps,L}} = \int_{N'\setminus D_{\eps}(P')}|dF|^2_{g'}\,dv_{g'}\geq \Sigma'_1-\frac{C}{|\log\epsilon|}.
\]
At the same time, by Lemma~\ref{lem:good_Smap} one has
\[
\int_M|dF|^2_{g_{\eps,L}}\,dv_{g_{\eps,L}}\leq \Sigma_1 +\eps^2,
\]
so combining the last two inequalities concludes the proof.
\end{proof}

Lemma~\ref{lem:SFeb_cyl} allows us to rewrite inequality~\eqref{ineq:SCS2} as
\begin{equation}
\label{ineq:SCS3}
\int_{N'}\langle d\hat{F},d\Psi\rangle_{g'}\,dv_{g'}-\Sigma'_1\int_{\bd N'}\langle \hat{F},\Psi\rangle\,ds_{g'}\leq C\left((\Sigma_1-\Sigma_1')^{\frac{1}{2}}+\frac{e^{-L}}{|\log\eps|^{\frac{1}{2}}}+\eps^{\frac{1}{4}}\right)\sqrt{Q(\Phi,\Phi)}
\end{equation}

We also record the following corollary of Lemma~\ref{lem:SFeb_cyl}.
\begin{corollary}
\label{cor:S''bound}
One has
\[
\int_{N''\setminus D_\eps(P'')}|dF|^2\,dv_{g_{\eps,L}} \leq (\Sigma_1 - \Sigma_1') + \eps^\frac{1}{2} + Ce^{-2L}\left(\Sigma_1-\Sigma_1' + \frac{1}{|\log\eps|}\right).
\]
\end{corollary}

\begin{proof}
Using the first line of~\eqref{ineq:Sfhueb}, Lemma~\ref{lem:SFeb_cyl} implies
\[
\int_{N'}|d\hat{F}|_{g'}^2\,dv_{g'}\leq \int_{(N'\setminus D_\eps(P'))\cup C_{\eps,L}}|dF|^2\,dv_{g_{\eps,L}}+Ce^{-2L}\left(\Sigma_1-\Sigma_1' + \frac{1}{|\log\eps|}\right).
\]
Together with~\eqref{ineq:gSm2}, one arrives at
\[
\int_{(N'\setminus D_\eps(P'))\cup C_{\eps,L}}|dF|^2\,dv_{g_{\eps,L}}\geq \Sigma_1' - \eps^\frac{1}{2} - Ce^{-2L}\left(\Sigma_1-\Sigma_1' + \frac{1}{|\log\eps|}\right).
\]
Using the upper bound~\eqref{ineq:gSm1} on the energy of $F$, we finally obtain
\[
\int_{N''\setminus D_\eps(P'')}|dF|^2\,dv_{g_{\eps,L}}\leq (\Sigma_1 - \Sigma_1') + \eps^\frac{1}{2} + Ce^{-2L}\left(\Sigma_1-\Sigma_1' + \frac{1}{|\log\eps|}\right).
\]
\end{proof}

By~\eqref{ineq:gSm3} we can decompose 
\[
\hat F=\Psi_{\eps,L}+R_{\eps,L},
\]
where $\Psi_{\eps,L}$ is a map by $\sigma_1(N',g')$-eigenfunctions, and $R_{\eps,L}$ is the projection of $\hat F$ onto the higher-frequency eigenspaces with eigenvalues $\sigma\geq \sigma_m(N',g')$, where $\sigma_m(N',g')>\sigma_{m-1}(N',g')=\ldots=\sigma_1(N',g')$. 

Since the kernel of $Q$ is exactly the $\sigma_1(N',g')$-eigenspace, we have 
$Q(\hat F,R_{\eps,L})=Q(R_{\eps,L},R_{\eps,L})$,
and \eqref{ineq:SCS3} gives
\[
Q(\hat F,R_{\eps,L})\leq C\left((\Sigma_1-\Sigma_1')^{\frac{1}{2}}+\frac{e^{-L}}{|\log\eps|^{\frac{1}{2}}}+\eps^{\frac{1}{4}}\right)\sqrt{Q(R_{\eps,L},R_{\eps,L})},
\]
so that
\[
Q(R_{\eps,L},R_{\eps,L})\leq C\left(\Sigma_1-\Sigma_1' + \frac{e^{-2L}}{|\log\eps|} + \eps^\frac{1}{2}\right).
\]

On the other hand, since $R$ is orthogonal to constants and the $\sigma_1(N',g')$-eigenspace, one has
\[
Q(R_{\eps,L},R_{\eps,L})\geq \left(1-\frac{\sigma_1(N',g')}{\sigma_m(N',g')}\right)\|dR_{\eps,L}\|_{L^2(N')}^2,
\]
hence,
\begin{equation}
\label{ineq:SReub}
\|R_{\eps,L}\|^2_{W^{1,2}(N')} \leq C\left(\Sigma_1-\Sigma_1' + \frac{e^{-2L}}{|\log\eps|} + \eps^{\frac{1}{2}}\right).
\end{equation}

We are now ready to prove the theorem.
 Assume by contradiction that $\Sigma_1\leq \Sigma_1'$ for all $\eps, L$. 
 If the collapsed set $P'$ satisfies the condition 2), then the proof is similar to the proof of Theorem~\ref{thm:general_existence}.
Let $P_i'\subset P'$ be an equivalence class, then by definition the set $\alpha(P'_i)=P_i''\subset P''\subset N''$ lies on the same connected component $N''_i$ of $N''$. Therefore, the first  eigenvalue $\lambda_1(N_i'',g'')>0$ (with Neumann boundary conditions if $\bd N_i''\ne\varnothing$). Let $\tilde F_i$ denote the harmonic extension to $N_i''$ of $F|_{N_i''\setminus D_\eps(P'')}$ and define
\[
\bar F_i: = \frac{1}{\area(N_i'',g'')}\int_{N_i''} \tilde F_i\,dv_{g''}.
\]
Then for any $q\in P_i''$, by Lemma~\ref{lem:average} applied to the symmetric extensions, one has 
\begin{equation*}
\begin{split}
\left|\frac{1}{\eps}\int_{S_\eps(q)} F\,ds_{g''_p} - 2\pi\bar F_i \right|^2&\leq C|\log\eps|\|\tilde F_i-\bar F_i\|^2_{W^{1,2}(N_i'',g'')}\\
\text{(by Definition of $\bar F_i$)}&\leq C\left(1 + \frac{1}{\lambda_1(N_i'',g'')}\right)|\log\eps|\|d\tilde F_i\|^2_{L^2(N_i'',g'')}\\
\text{(by Remark~\ref{rem:ext_bound})}&\leq C|\log\eps|\|d F_i\|^2_{L^2(N_i''\setminus D_\eps(P''),g'')}\\
\text{(by Corollary~\ref{cor:S''bound})}&\leq C|\log\eps| \left(\eps^\frac{1}{2} + \frac{e^{-2L}}{|\log\eps|}\right).
\end{split}
\end{equation*}
Similarly, by Lemma~\ref{lem:average} and~\eqref{ineq:SReub} one has for all $p\in P_i'$
\begin{equation*}
%\label{ineq:SR_avg}
\left|\frac{1}{\eps}\int_{S_\eps(p)}R_{\eps,L}\,ds_{g'_p}\right|^2\leq C\left(e^{-2L} + \eps^\frac{1}{2}|\log\eps|\right).
\end{equation*}
Finally, one has
\begin{equation*}
\begin{split}
&\left|\frac{1}{\eps}\int_{S_\eps(p)}F\,dv_{g'_p} - \frac{1}{\eps}\int_{S_\eps(\alpha(p))}F\,dv_{g''_p}\right| = 
\frac{1}{\eps}\left|\int_{C_{\eps,L(p)}}\frac{\partial F}{\partial t}\,dtd\theta\right|\leq\\
&\frac{1}{\eps}\|dF\|_{L^2(C_{\eps,L}(p))}\area(C_{\eps,L}(p))^{\frac{1}{2}}\leq \frac{C\sqrt{L}}{|\log\eps|^\frac{1}{2}},
\end{split}
\end{equation*}
where we applied Lemma~\ref{lem:SFeb_cyl} in the last step.

Combining all these inequalities, one obtains for all $p\in P_i'$
\begin{equation}
\label{ineq:PSi_avg_bound}
\left|\frac{1}{\eps}\int_{S_\eps(p)}\Psi_{\eps,L}\,dv_{g_p'} - 2\pi \bar F_i\right|\leq C\left(\frac{\sqrt{L}}{|\log\eps|^\frac{1}{2}} + e^{-L} + \eps^\frac{1}{4}|\log\eps|^{\frac{1}{2}}\right),
\end{equation}
and, therefore, for all $p,q\in P_i'$ one has
\begin{equation}
\label{ineq:PSi_bound}
\left|\frac{1}{\eps}\int_{S_\eps(p)}\Psi_{\eps,L}\,dv_{g_p'} - \frac{1}{\eps}\int_{S_\eps(q)}\Psi_{\eps,L}\,dv_{g_q'}\right|\leq C\left(\frac{\sqrt{L}}{|\log\eps|^\frac{1}{2}} + e^{-L} + \eps^\frac{1}{4}|\log\eps|^{\frac{1}{2}}\right).
\end{equation}
With this estimate in place, the proof is concluded in the same way as in Theorem~\ref{thm:general_existence}.

Suppose now that $p\in P'^\iota$; then $\bd N_i''\ne\varnothing$, so that $\sigma_1(N'',g'')>0$, and we define
\[
\check F_i:=\frac{1}{\length(\bd N_i'')}\int_{\bd N_i''}|\tilde{F}_i|^2\,dv_{g''}.
\]
Note that by an argument similar to~\eqref{ineq:Sfhllb} one has
\begin{equation}
\label{ineq:cFilb}
|\check F_i|\geq 1 - C\eps^\frac{1}{2}.
\end{equation}
Furthermore, for $q = \alpha(p)$ one has
\begin{equation*}
\begin{split}
\left|\frac{1}{\eps}\int_{S_\eps(q)} |F|^2\,ds_{g''_p} - 2\pi\check F_i \right|^2&\leq C|\log\eps|\||\tilde F_i|^2-\check F_i\|^2_{W^{1,2}(N_i'',g'')}\\
\text{(by Definition of $\check F_i$)}&\leq C\left(1 + \frac{1}{\sigma_1(N_i'',g'')}\right)|\log\eps|\|d(|\tilde F_i|^2)\|^2_{L^2(N_i'',g'')}\\
\text{(by Cauchy-Schwarz, $|\tilde F_i|\leq 1$)}&\leq C|\log\eps|\|d\tilde F_i\|^2_{L^2(N_i'',g'')}\\
\text{(by Remark~\ref{rem:ext_bound})}&\leq C|\log\eps|\|d F_i\|^2_{L^2(N_i''\setminus D_\eps(P''),g'')}\\
\text{(by Corollary~\ref{cor:S''bound})}&\leq C|\log\eps| \left(\eps^\frac{1}{2} + \frac{e^{-2L}}{|\log\eps|}\right).
\end{split}
\end{equation*}
Similarly to case 1), one also has a bound similar to~\eqref{ineq:PSi_avg_bound}
\[
\left|\frac{1}{\eps}\int_{S_\eps(p)}\Psi_{\eps,L}\,ds_{g_p'} - \frac{1}{\eps}\int_{S_\eps(q)}F\,ds_{g_p'}\right|\leq C\left(\frac{\sqrt{L}}{|\log\eps|^\frac{1}{2}} + e^{-L} + \eps^\frac{1}{4}|\log\eps|^{\frac{1}{2}}\right)
\] 
A novel ingredient is the following lemma.
\begin{lemma}
Let $C_L = \mathbb{S}^1\times [0,L]$ be the flat cylinder as in Lemma~\ref{lem:ext_bound}. Consider a mixed Steklov-Neumann eigenvalue problem  
\begin{equation*}
\begin{cases}
\Delta \phi = 0 &\text{ on }C_L;\\
\frac{\partial\phi}{\partial t}\phantom{ \,\,}=0 &\text{ on }\mathbb{S}^1\times \{0\};\\
\frac{\partial\phi}{\partial t}\phantom{ \,\,}=\sigma \phi &\text{ on }\mathbb{S}^1\times \{L\}.
\end{cases}
\end{equation*}
Then its first nontrivial eigenvalue is bounded from above by $1$. In particular, for any $\phi\in W^{1,2}(C_L)$ one has
\begin{equation}
\label{ineq:SN_cyl}
\left|\int_{t=L}\phi^2\,d\theta - \frac{1}{2\pi}\left(\int_{t=L} \phi\,d\theta\right)^2\right|\leq \int_{C_L}|d\phi|^2
\end{equation}
\end{lemma}
\begin{proof}
By separation of variables we find that the nonconstant eigenfunctions are of the form $\sin(k\theta)\cosh (kt)$, $\cos(k\theta)\cosh (kt)$ for $k\in\mathbb{N}$. The corresponding eigenvalues are $k\tanh(kL)$ and, in particular, the first nontrivial eigenvalue is at most $\tanh(L)<1$. Inequality~\eqref{ineq:SN_cyl} follows from the min-max characterization of the eigenvalues.
\end{proof}

Applying~\eqref{ineq:SN_cyl} to each component of $F$ restricted to $C_{\eps,L}(p)$, we arrive at
\[
\left|\frac{1}{\eps}\int_{S_\eps(q)} |F|^2 ds_{g_q}- \frac{1}{2\pi}\left|\frac{1}{\eps}\int_{S_\eps(q)}F\right|^2\,ds_{g_q}\right|\leq \|dF\|^2_{L^2(C_{\eps,L})}\leq \frac{C}{|\log \eps|}.
\]
Combining all the previous inequalities, one obtains
\[
\left|\left|\frac{1}{2\pi \eps}\int_{S_\eps(p)}\Psi_{\eps,L}\,ds_{g_p'}\right|^2 - \check F_i \right|\leq C\left(\frac{1}{|\log\eps|}+\frac{\sqrt{L}}{|\log\eps|^\frac{1}{2}} + e^{-L} + \eps^\frac{1}{4}|\log\eps|^{\frac{1}{2}}\right)
\]
To finish the proof, we choose $L=L(\eps)$ so that the r.h.s. tends to $0$ as $\eps\to 0$. As in the proof of Theorem~\ref{thm:general_existence}, up to a rotation and a choice of a subsequence, $\Psi_{\eps,L(\eps)}$ converges in $C^{\infty}$ to a map $\Psi\colon (N',g')\to\mathbb{B}^n$ by $\sigma_1(N',g')$-eigenfunctions. In particular, one has
\[
\left|\frac{1}{2\pi \eps}\int_{S_\eps(p)}\Psi_{\eps,L(\eps)}\,ds_{g_p'}\right|^2\to |\Psi(p)|^2<1.
\]
where the last inequality follows from the maximum principle and the fact that $p$ is an interior point of $N''$. At the same time, by~\eqref{ineq:cFilb} $\check F_i\to 1$, and we arrive at a contradiction.

%===================================================
%===================================================
%===================================================
%===================================================
%===================================================
%===================================================

%===================================================
%===================================================
%===================================================
%===================================================
%===================================================
%===================================================

\appendix

\section{Compactness of the moduli space.}
\label{app:moduli_space}

In this section we outline the ideas behind the results of~\cite{BSS} and provide modifications of their arguments necessary for our purposes as is explained in Remark~\ref{rmk:app_BSS}. Given $M$ a closed surface and $T\colon\Gamma\times M\to M$ an action of the group $\Gamma$ on $M$, we are interested in the compactness properties of the moduli space of $T$-invariant constant curvature metrics on $M$. In~\cite{BSS} Theorems~\ref{thm:Mahler} and~\ref{them_EquivDM} are proved provided $M$ is orientable, $\chi(M)<0$ and the action $T$ is by orientation preserving transformations. Below, we first sketch the proof in~\cite{BSS} and then show how to modify it in order to remove these assumptions.

\subsection{Sketch of Buser-Sepp\"al\"a-Silhol arguments} 
\label{sec:BSS}
In the case of $\chi(M)<0$ we are working with the space of hyperbolic metrics on $M$. The proof of Theorem~\ref{thm:Mahler} consists in constructing a geometrically controlled $T$-invariant geodesic triangulation of $M$, i.e. such that the length of each side is bounded from above and bounded from below away from zero; all angles are bounded away from $0$ and $\pi$; and the number of triangles is bounded from above, where the bounds only depend on $(M,T)$ and $\eps$, and independent of the hyperbolic metric $h$ with $\mathrm{inj}(h)\geq \eps$. One then argues that, up to a choice of a subsequence, each triangle has a limit (a triangle with sides and angles equal to the limits of the corresponding quantities). Furthermore, since the number of triangles (and, hence, possible ways to glue them together and act on them by the group) is finite, up to a further choice of subsequence, the combinatorial data of gluing triangles together and the way group permutes the triangles is independent of $h$. Gluing the limit triangles and defining the group action according to the this data yields the limit surface. 
It is explained in~\cite[Theorem 1]{BSS} that in order to construct a geometrically controlled triangulation it is sufficient to find a maximal $\eps/2$-separated collection of points $\mP$, i.e. a collection such that (a) any two points of $\mP$ lie at distance $\geq\eps/2$ from one another and (b) for any point $x\in M$ there exists a point in $\mP$ at distance $\leq\eps/2$ from it. Similarly, in order to find $T$-invariant triangulation, it is sufficient to find $T$-invariant maximal $\eps/2$-separated set $\mP$. The proof of Theorem~\ref{them_EquivDM} is similar: one constructs geometrically controlled triangulation of the thick part and treats thin part separately.

Let $h_n$ be a sequence of $T$-invariant metrics on $M$. Let $\{c_{n,i}\}$, $i= 1,\ldots, m$, $m\leq 3\gamma-3$ be a collection of disjoint simple closed geodesics in $(M,h_n)$ with $\ell_{h_n}(c_{n,i})\to 0$. Up to a choice of a subsequence, all other closed geodesics in $(M,h_n)$ have length at least $2\rho$, where $\rho\leq 1/16$ for convenience. For each geodesic we denote by $\mC_{n,i}$ its (truncated) collar chosen so that boundary curves have length $1$ and distance between different collars is at least $1/2$. The union $\cup_i\mC_{n,i}$ represents the thin part of $(M,h_n)$. As is explained in Section~\ref{sec:cc_limits}, for a fixed $n$ the action $T$ permutes $c_{n,i}$ and their collars $\mC_{n,i}$. For a fixed $n,i$, one now construct a tessellation of $\mC_{n,i}$ as follows. Let $\Gamma'\leq \Gamma$ be a subgroup that maps $\mC_{n,i}$ to itself. By a classification of isometries of a cylinder, all fixed points of elements of $\Gamma'$ are located on $c_{n,i}$ (recall that for now all isometries preserve the orientation) and are equidistant there. One declares those points to be vertices, and then adds additional vertices to form a set of $N$ equidistant points on $c_{n,i}$. Draw geodesics perpendicular to $c_{n,i}$ based at these points and declare their intersection with $\bd \mC_{n,i}$ to be vertices. Choose $N$ such that the distance $d$ between adjacent vertices is $\eps/2\leq d\leq\eps$, where $\eps\leq\rho/|\Gamma|$ is a small fixed constant to be precised later. Connecting them by a geodesic gives rise to a $\Gamma'$-invariant tessellation of $\mC_{n,i}$ by long thin rectangles, see~\cite[Fig. 3]{BSS}. Assuming now that the way group acts on $\mC_{n,i}$ is independent of $n$, it is explained in~\cite[Fig. 4,5]{BSS} that each long rectangle converges to a degenerate triangle in a way that $\mC_{n,i}$ converges to a union of two cusps. Doing this for all $i$ and passing to a subsequence takes care of the limit of the thin part. Denote by $\mP'$ the set of all vertices introduced at the last step, so that pairwise distances between two points in $\mP'$ are at least $\eps/2$ and for any point $p\in\mP$, there is $q\in\mP$, $p\ne q$ such that $\dist(p,q)<\eps$.  

In the thick part, the goal is to construct a $T$-invariant maximal $\eps/2$-separated set $\mP$. At this point the set $\mP'$ (as a subset of thick part) is $T$-invariant and $\eps/2$-separated, so one would like to add points to $\mP'$ in order for it to become maximal. If the action is trivial, then one could simply add points one at time: if an $\eps/2$-separated set $\mQ$ is not maximal, then there is $q$ lying in the complement of $D_{\eps/2}(\mQ)$ in the thick part, so we replace $\mQ$ by $\mQ\cup\{q\}$ and repeat the process. If the action is non-trivial, then as soon as we add a point $q$, we are forced to also add the whole orbit $\Gamma(q)$ and there is a potential problem that $\dist(q,\gamma(q))<\eps/2$. Indeed, this does occur if $q$ is sufficiently close to a fixed point of $\gamma$, hence, we need to treat neighbourhoods of fixed points separately.

Thus, we declare all points fixed by some element of $\Gamma$ to be vertices. According to~\cite[Lemma 1]{BSS} if $\gamma_1,\gamma_2\in\Gamma$ are non-trivial elements, $\gamma_1(p_1) = p_1$, $\gamma_1(p_2) = p_2$ and at least one of $\gamma_1,\gamma_2$ has order bigger than 2, then $\dist(p_1,p_2)>1/8\geq 2\rho$. Moreover, by~\cite[Lemma 2]{BSS} if additionally $p_1,p_2$ lie in the thick part, then $\dist(p_1,p_2)\geq2\rho/|\Gamma|\geq\eps$. Finally, Collar theorem~\cite[Theorem 4.1.1]{Buser} implies that $\dist(p_1, \mP')\geq 1/4$. Denote by $\mP''$ the union of $\mP'$ with all fixed points in the thick part, so that $\mP''$ is still $T$-invariant and $\eps/2$-separated.

 Let $p$ be a fixed point in the thick part and $\gamma\in\Gamma$ be the generator of $\Stab_\Gamma(p)$. If $\gamma$ has order at most $6$, then for any $q$ with $\dist(p,q) = \eps/2$ one has $\dist(q,\gamma(q))\geq\eps/2$, so that adding such $q$ to $\mP''$ would preserve $\eps/2$-separation property. However, if $\gamma$ has order $k\geq 7$, it is no longer the case. Instead, we pick a point $q$ with $\dist(p,q) = r$ such that $\dist(p,\gamma(p)) = \eps/2$ and it is proved~\cite[p. 218]{BSS} that, up to decreasing $\eps$ if needed, one has
\begin{equation}
\label{app1:ineq1}
\frac{\eps}{2}<r<\frac{k}{8}\eps\leq\frac{|\Gamma|}{8}\eps\leq \frac{\rho}{8}. 
\end{equation}
The idea now is to consider the ``crown'' of $p$ to be $p\cup \Gamma(q)$ with its natural triangulation~\cite[Fig. 7]{BSS} and treat separately, similar to the thin part. Indeed, the triangles in the crown only depend on $k$, which is independent of $n$, so the limit of the crown is itself. Inequality~\eqref{app1:ineq1} implies that points in the crown are at distance at least $\rho$ from $\mP'$ and different crowns are at distance at least $\rho$ from one another. As a result, if we define $\mP'''$ to be the union of $\mP''$ with all the crowns, it continues to be $T$-invariant and $\eps/2$-separated. 

One then observes that for any point $q$ lying in the thick part away from the crowns and at least $\eps/2$-away from $\mP''$ one has $\dist(q,\gamma(q))\geq\eps/2$ for all non-trivial $\gamma\in\Gamma$. Indeed, if $\dist(q,\gamma(q))<\eps/2$, then, assuming that $\gamma$ has order $k>2$, by consequetively connecting $p$ to $\gamma(p)$, then $\gamma(p)$ to $\gamma^2(p)$ etc by geodesics, we obtain a closed curve $S$ of length $\leq 2\rho$. Since any point in the thick part has injectivity radius at least $1/4>2\rho$, we conclude that $S$ is contractible. Since any diffeomorphism of a disk preserving the boundary has a fixed point, we conclude that there exists $p$ such that $\gamma(p) = p$. If $k=2$, then there is a possibility that the geodesic segment $[q,\gamma(q)]$ is mapped to itself, but then we conclude that the middle point $p$ of the segment is a fixed point. For a fixed point $p$ one necessarily has $\dist(p,q) = \dist(p,\gamma(q))$.
Using that $\dist(q,\gamma(q))<\eps/2$ the same computations that led to~\eqref{app1:ineq1} yield that either $\dist(p,q)<\eps/2$ (if order of $k\leq 6$), or $\dist(p,q)<r$ (if order of $k\geq 7$). Both of these result in a contradiction. Having established that  $\dist(q,\gamma(q))\geq\eps/2$, we can proceed as in the case of trivial action by simply adding orbits of points to $\mP''$ until we arrive at a $T$-invariant maximal $\eps/2$-separated set in the thick part away from the crowns. This completes the argument.

\subsection{Non-orientable transformations} In this section we make necessary adjustments to the arguments in the previous section to account for group actions that do not necessarily preserve orientation. For now we still assume that $M$ is orientable and $\chi(M)<0$. The main new feature is that fixed point sets of $\gamma\in\Gamma$ are no longer isolated and could consist of several disjoint closed geodesics. The strategy is to treat isolated fixed points in the same way as before and add an extra step in the construction of $\mP$ to accommodate for fixed curves of reflections. it would be convenient to take $\rho\leq 1/32$ as opposed to $\rho\leq 1/16$.

The tessellation of the thin part remains largely the same. The only difference is in the classification of isometries of the cylinder: additional isometries include reflection about $c_{n,i}$, reflections about two equidistant geodesics orthogonal to $c_{n,i}$ and the fixed-point free antipodal involution. This time, in addition to the isolated fixed points on $c_{n,i}$, we also declare the fixed points of the second kind of reflections on $c_{n,i}$ to be vertices. Thus declared vertices are equidistant on $c_{n,i}$ and we repeat the same procedure as before to form $T$-invariant $\eps/2$-separated set $\mP'$ in the thick part. 

Note that if  the action of a non-trivial $\gamma\in\Gamma$ has a fixed curve, then $\gamma$ is of order $2$, hence, a reflection (this follows from looking at the differential of the transformation at any point on the curve). 
One then has the following generalisation of~\cite[Lemma 1]{BSS}: if $\gamma_1$ is a reflection and $S\subset \mathrm{Fix}(\gamma_1)$ and $\gamma_2$ is an element of order $>2$ with $\gamma_2(p_2)=p_2$, then $\dist(S,p_2)>1/16\leq 2\rho$. Indeed, if $p_2$ is a fixed point of $\gamma_2$, then $p_1 = \gamma_1(p_2)$ is a fixed point of an orientation preserving transformation 
$\gamma_1\gamma_2\gamma_1$, so that by ~\cite[Lemma 1]{BSS}  $\dist(p_2, S) = \dist(p_1,p_2)/2>1/16$. In the same way~\cite[Lemma 2]{BSS} generalises to the case, where exactly one of $\gamma_1$, $\gamma_2$ is a reflection. Note, however, that fixed points sets for different reflections can intersect in the thick part, so~\cite[Lemma 2]{BSS} can not hold for two reflections.   

As before, we set $\mP''$ to be the set of isolated fixed points. Any such point $p$ is necessarily a fixed point of an orientation-preserving transformation, but it could also be preserved by some reflections, i.e. $\Stab_{\Gamma}(p)$ is either a cyclic group (and then it is not fixed by any reflections) or a dihedral group. In the former case we do the same procedure as before, whereas in the latter we form a crown depending on whether the order $2k = \Stab_{\Gamma}(p)$ exceeds $6$ or not. If $k\geq 4$, then the crown consist of an orbit of a point $q$ with $\dist(p,q) = r$, where $q,r$ are chosen so that $\Stab_{\Gamma}(p)$-orbit of $q$ forms $2k$-gon with sidelength equal to $\eps$. Then the same calculation as in~\eqref{app1:ineq1} gives $r<\rho/4$. As before, the set $\mP'''$ formed by the union of $\mP''$ and the crowns is $T$-invariant and $\eps/2$-separated. 

Let now $q$ be a point in the thick part outside of the crowns and at least $\eps/2$-away from $\mP''$. We already know from Section~\ref{sec:BSS} that for any non-trivial orientation preserving $\gamma\in \Gamma$ one has $\dist(q,\gamma(q))>\eps/2$. 
Let $S_i$ denote all connected components of fixed point sets of reflections lying in the thick part outside of the crowns and at least $\eps/2$-away from $\mP''$. We now claim that there is at most one $S_i$ at distance $\leq\eps/2$ from $q$. Indeed, suppose there are two such components $S_1$ and $S_2$. Let $p_i\in S_i$ be such that $\dist(q,p_i)\leq \eps/2$. Then repeatedly reflecting the curve $[p_1,q]\cup[q,p_2]$ across $S_1$ and $S_2$ we obtain a closed curve $S$ of length $\leq 2\rho$, which then has to be contractible. If reflections across $S_1$ and $S_2$ are distinct, then the argument in Section~\ref{sec:BSS} yields that  their (orientation preserving) composition has a fixed point at distance $<r$ from $q$, which leads to a contradiction as before. If  reflections across $S_1$ and $S_2$ is the of the same element $\tau$, then the action of $\tau$ on the disk bounding $S$ is by reflection, so it has a fixed point $S_0$ inside a disk such that $S_0\cup S_1\cup S_2$ is connected. Since $S_1$ and $S_2$ are different connected components, this is only possible if $S_0$ lies inside the crown, or  $\eps/2$-close to $\mP''$. Then there is an isolated fixed point on $S_0$ and we arrive at the same contradiction as before. As a result, $\eps/4$-tubular neighbourhoods $T_{\eps/4}(S_i)$ are at least $\eps/2$-away from each other and for any point on $p\in \bd T_{\eps/4}(S_i)$ one has $\dist(p,\gamma(p))\geq \eps/2$ for all non-trivial $\gamma\in\Gamma$. We can then add orbits of points from $\bd T_{\eps/4}(S_i)$ to $\mP'''$ to obtain $T$-invariant, $\eps/2$-separated sett $\mP^{(4)}$ such that $D_{\eps/2}(\mP^{(4)})\supset T_{\eps/4}(S_i)$ for all $i$.

Finally, for any point $p$ in the thick part outside of crowns, $D_{\eps/2}(\mP'')$ or $\cup_i T_{\eps/4}(S_i)$, one has $\dist(p,\gamma(p))\geq \eps/2$ for all non-trivial $\gamma\in\Gamma$. Thus, we can finish the construction of the set $\mP$ by adding the orbits of such points one by one.

\subsection{Non-orientable surfaces} In this section we continue to assume $\chi(M)<0$, but now $M$ could be non-orientable. 

Let $(M,T,h)$ be a triple consisting of a non-orientable surface $M$ with the action $T$ of group $\Gamma$ and $T$-invariant hyperbolic metric $h$. To any such triple we can associate $(\wt M, \wt T, \wt h)$, where $\pi\colon \wt M\to M$ is the orientable double cover; $\wt T$ is the action of the group $\wt\Gamma = \Z_2\times \Gamma$, where the generator $\tau$ of $\Z_2$ acts by a fixed-point free orientation reversing involution interchanging the leaves of $\pi$ and satisfying $T\circ\pi = \pi\circ\wt T$; and $\wt h = \pi^*h$ is a $\wt T$-invariant hyperbolic metric. Conversely any triple $(\wt M, \wt T, \wt h)$ can be factorised by $\tau$ to obtain $(M,T,h)$.

Given a sequence $(M,T,h_n)$ the natural way to obtain its limit is to consider the limit $(\wt M_\infty, \wt T_\infty,\wt h_\infty)$ of $(\wt M, \wt T, \wt h_n)$ and then factorise by $\tau$. First, observe that $\tau $ continues to act by fixed-point free orientation reversing involution. The statement about orientation follows directly from the construction. The only way for $\tau$ to gain fixed points is if one of the cusps is preserved by $\tau$. In that case $\tau$ has to preserve the corresponding collar and has to act by a fixed-point free orientation reversing involution, hence, by an antipodal map. However, antipodal map interchanges the two boundary components of the collar, hence, it interchanges the two corresponding cusps. Thus, the factor space is a smooth surface.

Note that since $\wt M_\infty$ is not necessarily connected, the limit of $(M,T,h_n)$ can have orientable and non-orientable connected components.

\subsection{Torus and Klein bottle} Finally, we consider the case $\chi(M)=0$. According to the previous section it is in fact sufficient to consider the case $M=\mathbb{T}^2$. There are two main differences compared to the hyperbolic case: the angles of Euclidean triangles do not determine their lengths (it makes a difference for the proof of~\cite[Lemma 1]{BSS}), and the definition of collapsing geodesics is not immediately obvious as a whole family of geodesics have lengths going to $0$.

Let $h_n$ be a sequence of flat metrics of unit area on $M$. Such metrics are in one-to one correspondence with lattices in $\mathbb{R}^2$ with fundamental parallelogram of unit area. 
 First assume that there is a uniform lower bound on the injectivity radius of $h_n$. Suppose that $\gamma_1,\gamma_2\in\Gamma$ have distinct fixed points $p_1,p_2\in M$ respectively. If $\gamma_1$ and $\gamma_2$ both have order $2$, then the same argument as in~\cite[Lemma 2]{BSS} gives a uniform lower bound on $\dist(p_1,p_2)$. If $\gamma_1$, $\gamma_2$ have orders $P_1,P_2$ respectively, where, for example $P_1>2$, then the same arguments as in~\cite[Lemma 1]{BSS} give existence of $Q\in\N$ such that $P_1^{-1} + P_2^{-1} + Q^{-1} = 1$, i.e. $\{P_1,P_2,Q\} = \{2,4,4\}, \{2,3,6\}$ or $\{3,3,3\}$. However, the only conformal class of tori that admits rotation by $\pi/2$ is the square torus, whereas the only class that admits a rotation by $2\pi/3$ is the equilateral torus. Hence, in such a situation the sequence of conformal classes is constant up to a choice of a subsequence, so that Theorem~\ref{thm:Mahler} follows. In the remaining cases, all fixed points correspond to rotations by $\pi$ with distances between centers of different rotations uniformly bounded from below. Then, the triangulation argument analogous to the one described in~\ref{sec:BSS} completes the proof of Theorem~\ref{thm:Mahler} for the torus.

Assume now that $\mathrm{inj}(h_n)\to 0$, so that without loss of generality there is a unique shortest geodesic passing through every point on $\mathbb{T}^2$ and they are all parallel to one another. Let $\mU_n$ denote the set of all such geodesics. The idea of the proof is to find a finite $T$-invariant subset $\{u_{n,i}\}_{i=1}^m$ of equidistant elements in $\mU_n$. Then since the length of those geodesics goes to $0$ and the area is fixed to be $1$, the distance $d_n$ between them goes to infinity. The tubular neighbourhoods $T_{d_n/4}(u_{n,i})$ are then long cylinders that are permuted by the action of $\Gamma$. These long cylinders can then be equivariantly and conformally identified with larger and larger subsets of the infinite cylinder, which can in turn be identified with $\mathbb{S}^2\setminus \{S,N\}$, where $N,S$ are north and south pole respectively. The complement to the tubular neighbourhoods can then be seen as a union of collars, so that $M_n$ (in notation of Theorem~\ref{them_EquivDM}) is a union of $m$ long cylinders separated by collars,  $(M_\infty,h_\infty)$ is a union of $m$ infinite flat cylinders and $(\wh M_\infty,\wh h_\infty)$ is a union of $m$ spheres.

It remains to construct $\mU_n$. To do this, we first observe that the image of a shortest geodesic under an isometry is a shortest geodesic, so $\Gamma$ acts on $\mU_n$. Introduce the distance function on $\mU_n$ induced by the distance between the parallel lines, this way $\mU_n$ becomes isometric to a circe of circumference $1/\mathrm{inj}(h_n)$ and $\Gamma$ is a finite group acting by isometries on it. Using the classification of isometries of a circle, it is easy to find the subset $\{u_{n,i}\}_{i=1}^m\subset \mU_n$ with the required properties.  

%===================================================
%===================================================
%===================================================
%===================================================
%===================================================
%===================================================

%===================================================
%===================================================
%===================================================
%===================================================
%===================================================
%===================================================

\end{document}